\documentclass[11pt,twoside]{book}
\usepackage{project}
\addto\captionsenglish{}

\usepackage{amsfonts, amssymb, epsfig,subfigure,amstext}
\usepackage{enumerate}
\usepackage{mathrsfs}
\usepackage{stmaryrd}
\usepackage{graphicx,amsmath}
\usepackage{bm}
\usepackage{latexsym}
\usepackage{verbatim}
\usepackage{diagbox}
\usepackage{color}
\usepackage{mdwlist}
\usepackage{cleveref}
\newtheorem{thm}{Theorem}[chapter]
\newtheorem{lemma}{Lemma}[chapter]
\newtheorem{coro}{Corollary}[chapter]
\newtheorem{prop}{Proposition}[chapter]

\newtheorem{assp}{Assumption}[chapter]
\newtheorem{expl}{Example}[chapter]
\newtheorem{rem}{Remark}[chapter]
\crefname{assp}{Assumption}{Assumptions}
\crefname{prop}{Proposition}{Propositions}
\crefname{thm}{Theorem}{Theorems}
\crefname{lemma}{Lemma}{Lemmas}
\crefname{coro}{Corollary}{Corollaries}
\crefname{rem}{Remark}{Remarks}
\crefname{expl}{Example}{Examples}
\crefname{defn}{Definition}{Definitions}

\newcommand{\lbracket}{[}
\newcommand{\rbracket}{]}

\newcommand{\E}{\mathbb{E}}
\newcommand{\PP}{\mathbb{P}}
\newcommand{\RR}{\mathbb{R}}

\def\tr{\Delta}
\def\e{\varepsilon}
\def\<{\langle}
\def\>{\rangle}
\def\l{\lambda}

\def\nn{\nonumber}

\allowdisplaybreaks
\makeatletter

\def\projectauthor{Chuchu Chen (chenchuchu@lsec.cc.ac.cn)\\
\vspace*{20pt}
 Tonghe Dang (dangth@lsec.cc.ac.cn)\\
 \vspace*{20pt}
 Jialin Hong (hjl@lsec.cc.ac.cn)\\
 \vspace*{20pt}
  Guoting Song (songgt606@lsec.cc.ac.cn)}
\def\projecttitle{Longtime behaviors of $\theta$-Euler--Maruyama  method for stochastic functional differential equations}
\def\projectmonth{~}
\def\projectdepartment{LSEC, ICMSEC, Academy of Mathematics and Systems Science, Chinese Academy of Sciences, Beijing 100190, China; 
School of Mathematical Sciences, University of Chinese Academy of Sciences, Beijing 100049, China}
\begin{document}
    \prefrontmatter
    \frontmatter
    \pagenumbering{roman} 
    \cleardoublepage
    \chapter*{Abstract}

This paper investigates longtime behaviors of the $\theta$-Euler--Maruyama method for the stochastic functional differential equation  with superlinearly growing coefficients.  
 We focus on the longtime convergence analysis in mean-square sense and weak sense of the $\theta$-Euler--Maruyama method, the convergence of the numerical invariant measure, the existence and convergence of the numerical density function, and the Freidlin--Wentzell large deviation principle of the  method. 
The main contributions are outlined as follows. 
First, we obtain the longtime mean-square convergence of the $\theta$-Euler--Maruyama method and show that the mean-square convergence rate is $\frac12$.
A key step in the proof  is to establish the time-independent boundedness of high-order moments of the numerical functional solution.
Second, based on the technique of   the  Malliavin calculus, we present the longtime weak convergence of the  $\theta$-Euler--Maruyama method, which implies that the  invariant measure of the $\theta$-Euler--Maruyama  functional solution converges to the exact one with   rate $1.$ 
Third, by the  analysis of  the  test-functional-independent weak convergence  and  negative moment estimates of the determinant of the corresponding Malliavin covariance matrix, we derive the existence,  convergence, and the logarithmic  estimate   of the  density function  of the $\theta$-Euler--Maruyama  solution. At last,  utilizing the weak convergence method, we obtain the Freidlin--Wentzell large deviation principle for the $\theta$-Euler--Maruyama solution  on the infinite time horizon. 

\noindent {\bf Keywords:} Stochastic functional differential equation; $\theta$-Euler--Maruyama method; longtime behaviors; mean-square and weak convergence; invariant measure;  density function; large deviation principle

        \markboth{Abstract}{}
      \cleardoublepage
    \pdfbookmark{\contentsname}{toc}
    \renewcommand{\sectionmark}[1]{\markright{#1}}
    \addtolength{\parskip}{-\baselineskip}  
    \tableofcontents
    \begingroup
        \let\clearpage\relax
        \vspace*{20pt}
        \vspace*{20pt}
        \vspace*{20pt}
         \vspace*{20pt}
         \vspace*{20pt}
    \endgroup
    \addtolength{\parskip}{\baselineskip}
    \renewcommand{\sectionmark}[1]{\markright{\thesection\ #1}}
    \mainmatter
%
%
%


\chapter{Introduction}

The stochastic functional differential equation (SFDE)
is commonly used to model  systems
that depend not only on the present state but also on the past ones, 
which has wide applications in  diverse fields such as biology, epidemiology, mechanics, neural networks, economics, finance and so on. 
The numerical method of the SFDE  
serves as  a powerful  tool to  qualitatively and quantitatively study stochastic phenomena and time-lag effect  prevalent in these fields. 
There is a growing interest in  the study of longtime behaviors of numerical methods, which  provides effective approaches for investigating the intrinsic properties  of the exact solution. However, many problems  on this topic  remain unsolved. 
This paper focuses on the investigation of  longtime behaviors for the $\theta$-Euler--Maruyama ($\theta$-EM) method of the SFDE.


\section{Stochastic functional differential equation}\label{PR}

In this paper, we consider the following SFDE on 
$t\in[-\tau,+\infty)$,
\begin{align}\label{FF}
\left\{
\begin{array}{ll}
\mathrm{d}x^{\xi}(t)=b(x^{\xi}_t)\mathrm{d}t +\sigma(x^{\xi}_t)\mathrm{d}W(t),\quad t>0,\\
~~x^{\xi}(t)=\xi(t),\quad t\in[-\tau,0],\\
\end{array}
\right.
\end{align}
where the delay $\tau>0$, the initial datum $\xi\in \mathcal C([-\tau,0];\RR^d),$  the drift coefficient  $b: \mathcal C([-\tau,0];\RR^d) \rightarrow \RR^{d}$ is locally Lipschitz continuous, the diffusion coefficient $\sigma: \mathcal C([-\tau,0];\RR^d)  \rightarrow \RR^{d\times m}$ is  globally Lipschitz continuous, $\{W(t),t\ge 0\}$ is an $m$-dimensional standard Brownian motion on a filtered probability space $(\Omega,\mathcal F,\{\mathcal F_t\}_{t\ge 0},\mathbb P),$ and $x^{\xi}_t:r\mapsto x^{\xi}(t+r)$ is the $\mathcal C([-\tau,0];\mathbb R^d)$-valued functional solution for $t\ge 0.$ The well-posedness of \eqref{FF} has been deeply studied, see e.g.  \cite{Mohammed84, Mao2007, XuYang, RM10}  and references therein for details.


The longtime behavior  of the SFDE on the infinite time horizon is a recent and  ongoing research subject. There are now a certain number of papers devoted to this research.  
For instance, 
the stability of the exact solution has been established in e.g. \cite{BDH18, LMS11, Jiang2023, Shaikhet2013}; The theory on the existence and uniqueness of the invariant measure of the functional solution has  been well studied in e.g. 
 \cite{BS17, Bao-Shao-Yuan2023, MYM17, Baobook, WWM2019}.  
We 
recall some  results on  the functional solution of \eqref{FF}, including  the second moment boundedness, the exponential attractiveness, and the existence and uniqueness of the invariant measure.
Suppose that the coefficients satisfy the following two assumptions, which are globally  Lipschitz condition of the diffusion  coefficient $\sigma$ and the dissipative condition of the drift coefficient $b$.

\begin{assp}\label{a1}
There exist a constant   $L>0$ and a probability measure $\nu_1$ on $[-\tau,0]$ such that for any $\phi_1,~\phi_2\in \mathcal C([-\tau,0];\mathbb R^d)$,
\begin{align*}
   |\sigma(\phi_1)-\sigma(\phi_2)|^2
  \leq{}
  L\Big(|\phi_1(0)-\phi_2(0)|^2+\int_{-\tau}^{0}|\phi_1(r)- \phi_2(r)|^2\mathrm d \nu_1(r)\Big).
  \end{align*}
\end{assp}

\begin{assp}\label{a2}
The drift coefficient $b$ is continuous and there exist a probability measure $\nu_2$ on $[-\tau,0]$ and positive constants $a_1,a_2$ satisfying $a_1>a_2+L$  such that for any $\phi_1,~\phi_2\in \mathcal C([-\tau,0];\mathbb R^d)$,
\begin{align*}
\langle\phi_1(0)\!-\!\phi_2(0),b(\phi_1)\!-\!b(\phi_2)\rangle
\leq \!-a_1|\phi_1(0)\!-\!\phi_2(0)|^{2}\!+a_2\int_{-\tau}^{0}|\phi_1(r)\!-\!\phi_2(r)|^2\mathrm d\nu_2(r).
\end{align*}
\end{assp}

Note that under  \cref{a2},  the drift coefficient $b$ may grow superlinearly, in particular, the classical example  $b(x)=\frac1\tau\int_{-\tau}^0x(s)\mathrm ds-|x(0)|^2x(0)-x(0),\,x\in\mathcal C([-\tau,0];\mathbb R^d)$ is included. 

\begin{rem}\label{r1}
By Assumption {\ref{a2}}, we obtain that for any $\phi\in \mathcal C([-\tau,0];\mathbb R^d)$,
\begin{align}\label{F1r1}
\langle \phi(0), b(\phi)\rangle
\leq |\phi(0)||b(0)|-a_1|\phi(0)|^{2}+a_2\int_{-\tau}^{0}|\phi(r)|^2\mathrm d\nu_2(r).
\end{align}
Since $a_1>a_2+L$, by the Young inequality, we have  
\begin{align}\label{F2r1}
2|\phi(0)||b(0)|\leq \frac{1}{a_1-a_2-(1+\e)L}|b(0)|^2+ (a_1-a_2-(1+\e)L)|\phi(0)|^2,
\end{align}
where the positive constant $\e$ satisfies  $a_1-a_2-(1+\e)L>0$.
Inserting \eqref{F2r1} into \eqref{F1r1} leads to
\begin{align*}
2\langle \phi(0), b(\phi)\rangle
\leq{}&\frac{1}{a_1-a_2-(1+\e)L}|b(0)|^2-(a_1+a_2+(1+\e)L)|\phi(0)|^{2}\nn\\
&\quad+2a_2\int_{-\tau}^{0}|\phi(r)|^2\mathrm d\nu_2(r).
\end{align*}
This, along with Assumption {\ref{a1}} implies that
\begin{align*}
&2\langle \phi(0), b(\phi)\rangle+|\sigma(\phi)|^2\nn\\
\leq{}& \frac{1}{a_1-a_2-(1+\e)L}|b(0)|^2+(1+\frac{1}{\e})|\sigma(0)|^2-(a_1+a_2)|\phi(0)|^{2}\nn\\
&\quad+(1+\e)L\int_{-\tau}^{0}|\phi(r)|^2\mathrm d\nu_1(r)+2a_2\int_{-\tau}^{0}|\phi(r)|^2\mathrm d\nu_2(r).
\end{align*}
Letting 
$
\e=\frac{a_1-(a_2+L)}{2L},
$
we have
\begin{align}\label{F2r1+1}
2\langle \phi(0), b(\phi)\rangle
\leq K-\frac{3a_1+a_2+L}{2}|\phi(0)|^{2}+2a_2\int_{-\tau}^{0}|\phi(r)|^2\mathrm d\nu_2(r)
\end{align}
and
\begin{align}\label{F3r1}
2\langle \phi(0), b(\phi)\rangle+|\sigma(\phi)|^2
\leq{}&K-(a_1+a_2)|\phi(0)|^{2}+\frac{a_1-a_2+L}{2}\int_{-\tau}^{0}|\phi(r)|^2\mathrm d\nu_1(r)\nn\\
&\quad+2a_2\int_{-\tau}^{0}|\phi(r)|^2\mathrm d\nu_2(r),
\end{align}
where the positive constant $K$ depends on $a_1$,~$a_2$,~$L$, $b(0)$,  and $\sigma(0)$.
\end{rem}

Under Assumptions {\ref{a1}} and {\ref{a2}},  the  SFDE \eqref{FF} with the initial datum $\xi\in \mathcal C([-\tau,0];\mathbb R^d)$ has a unique strong solution $x^{\xi}(\cdot)$ on $[-\tau, \infty)$ according to \cite[Theorem 2.3]{RM10}. 
 Below, we list 
 the second moment boundedness, the exponential attractiveness,  and the existence and uniqueness of the invariant measure for the functional solution of \eqref{FF}. 
 The proofs are similar to \cite[Lemma 3.1, Lemma 3.2, and Theorem 1.3]{Bao-Shao-Yuan2023}, which are omitted.

\begin{lemma}\label{l1}
Under Assumptions {\ref{a1}} and {\ref{a2}}, the functional solution of \eqref{FF} satisfies
\begin{align*}
\sup_{t\geq0}\E[\|x_t^{\xi}\|^2]\leq K(1+\|\xi\|^2)
\end{align*}
and 
\begin{align*}
\E[\|x^{\xi}_t-x^{\eta}_t\|^2]\leq K\|\xi-\eta\|^2e^{-\lambda t}\quad \forall ~t\geq0,
\end{align*}
where $\xi,\eta\in \mathcal C([-\tau,0];\mathbb R^d),\,K>0$, and $\lambda>0$ is a sufficiently small constant satisfying 
\begin{align}\label{lamb}
c_{\l}:=2a_1-L-(2a_2+L)e^{\lambda\tau}-\lambda>0.
\end{align}
\end{lemma}

For any  $\xi\in \mathcal C([-\tau,0];\mathbb R^d)$ and $t\geq 0$, denote by $\mu_t^{\xi}$ the probability measure generated by  $x_t^{\xi}$, namely,
\begin{align*}
\mu_t^{\xi}(A)=\PP\big\{\omega\in\Omega: x_{t}^{\xi}\in A\big\}\quad \forall A\in\mathcal{B}(\mathcal C([-\tau,0];\mathbb R^d)).
\end{align*}
Define the bounded-Wasserstein distance  $\mathbb{W}_q~(q\in[1,\infty))$ as
\begin{align}\label{2.1+}
\mathbb{W}_q(\mathfrak u_1,\mathfrak u_2):=\Big(\inf_{\pi\in\Pi(\mathfrak u_1,\mathfrak u_2)}\int_{\mathcal C([-\tau,0];\mathbb R^d)\times \mathcal C([-\tau,0];\mathbb R^d)}1\wedge\|\phi_1-\phi_2\|^q\pi(\mathrm d\phi_1,\mathrm d\phi_2)\Big)^{\frac{1}{q}},
\end{align}
where $\Pi(\mathfrak u_1,\mathfrak u_2)$ denotes the collection of all probability measures on $\mathcal C([-\tau,0];\mathbb R^d)\times\mathcal C([-\tau,0];\mathbb R^d)$ with
marginal measures $\mathfrak u_1$ and $\mathfrak u_2$; \textup{see \cite[Chapter 6]{Villani}} for details.

\begin{lemma}\label{l2.4}
Under Assumptions {\ref{a1}} and {\ref{a2}}, the functional solution $x^{\xi}_{\cdot}$
is asymptotically stable in distribution and admits a unique invariant measure $\mu$ satisfying
\begin{align*}
\mathbb W_2(\mu_t^{\xi},\mu)\leq K(1+\|\xi\|)e^{-\frac{\lambda}{2} t}.
\end{align*}
\end{lemma}



\section{$\theta$-Euler--Maruyama  method}

The numerical study of the SFDE on the infinite time horizon has attracted a great deal of attention in the past decades. For instance, authors in  \cite{Zhou15, LD21} obtain 
the stability
of the Euler--Maruyama (EM) method and the backward Euler--Maruyama (BEM) method  for the linearly growing coefficients case and  the superlinearly growing coefficients case, respectively;  Authors in  \cite{Bao-Shao-Yuan2023} present the existence, uniqueness, and convergence of the invariant  measure of the EM method for the SFDE with Markov switching; \cite{Wufuke2024} proves that for the SFDE with superlinearly growing coefficients, the numerical invariant measure of the  BEM method converges to the exact one.

In this paper, we introduce a  general class of  numerical methods, namely,  
the $\theta$-Euler--Maruyama ($\theta$-EM) method ($\theta\in[0,1]$),
and  investigate  the longtime behaviors of numerical solutions for the SFDE. The  $\theta$-EM method
    includes the  
 EM method (for which $\theta=0$) and the BEM method (for which $\theta=1$).
Concerning that for the SFDE with superlinearly growing coefficients, 
the explicit method such as  the EM method often fails to converge strongly or weakly to the exact solution; see e.g. \cite{diverge}. To overcome the divergence, we consider the implicit case with $\theta\in(\frac12,1]$ in this paper. We also remark that for the case of $\theta\in[0,\frac12],$ results in this paper can also be obtained under a stronger condition on coefficients $b$ and $\sigma$.
Without loss of generality, we  take a sufficiently large integer  $N\geq\tau$  
such that $\tr=\frac{\tau}{N}\in(0,1]$. Let $t_k=k\tr$ for $k=-N,-N+1,\ldots$
  Define the $\theta$-EM method   as follows: for any
$\xi\in \mathcal C([-\tau,0];\mathbb R^d)$, 
\begin{align}\label{thEM}
\left\{
\begin{array}{ll}
y^{\xi,\tr}(t_k)=\xi(t_k),~~~k=-N,-N+1,\ldots, 0,\\
y^{\xi,\tr}(t_{k+1})=y^{\xi,\tr}(t_k)+(1-\theta)b(y^{\xi,\tr}_{t_k})\tr+\theta b(y^{\xi,\tr}_{t_{k+1}})\tr +\sigma(y^{\xi,\tr}_{t_k})\delta W_k,~~~k\in\mathbb N,
\end{array}
\right.
\end{align}
where $\theta\in(\frac12,1]$, $\delta W_k=W(t_{k+1})-W(t_k)$, and $y^{\xi,\tr}_{t_k}$ is a $\mathcal C([-\tau,0];\mathbb R^d)$-valued random variable defined by 
\begin{align}\label{thL}
y^{\xi,\tr}_{t_k}(r)=\frac{t_{j+1}-r}{\tr}y^{\xi,\tr}(t_{k+j})
+\frac{r-t_{j}}{\tr}y^{\xi,\tr}(t_{k+j+1})
\end{align}
for $r\in[t_j,t_{j+1}], \;j\in\{-N,\ldots,-1\}.$ 
We call $\{y^{\xi,\tr}(t_k)\}_{k=-N}^{\infty}$ the \emph{$\theta$-EM solution} and $\{y^{\xi,\tr}_{t_k}\}_{k=0}^{\infty}$ the \emph{$\theta$-EM functional solution}. 
Without  particular clarifications, 
we always assume $\theta\in(\frac12,1].$

The following lemma shows the solvability of the $\theta$-EM method. 

\begin{lemma}\label{solvability}
Under \cref{a1,a2}, the solution of the implicit method \eqref{thEM}   exists uniquely  for any $\tr\in(0,1]$.
\end{lemma}
\begin{proof}
Fix $\tr\in(0,1]$ and $\theta\in(\frac12,1]$.
For any $\phi\in \mathcal C([-\tau,0];\mathbb R^d)$ and $u\in\RR^d$, define  continuous functions  $\Phi^{\tr}_{\phi,u}: [-\tau,0]\rightarrow \RR^d$ and $F^{\tr}_{\theta,\phi}:\RR^{d}\rightarrow \RR^d$, respectively,  
by
\begin{align}\label{movephi}
\Phi^{\tr}_{\phi,u}(r)=\left\{
\begin{array}{ll}
\phi(r+\tr),~~~~~~~~~~r\in[-\tau,-\tr),&\\
\frac{-r}{\tr}\phi(0)+\frac{\tr+r}{\tr}u,~~~r\in[-\tr,0]&
\end{array}
\right.
\end{align}
and 
\begin{align*}
F^{\tr}_{\theta,\phi}(u)=&\;u-\theta b(\Phi^{\tr}_{\phi,u})\tr.
\end{align*}
Moreover, by virtue of \eqref{F2r1+1} and the convex property of $|\cdot|^2,$ we have
\begin{align*}
&\lim_{|u|\rightarrow\infty}\frac{\langle u, F^{\tr}_{\theta,\phi}(u)\rangle}{|u|^2}
=1-\theta\tr\lim_{|u|\rightarrow\infty}\frac{\langle \Phi^{\tr}_{\phi,u}(0), b(\Phi^{\tr}_{\phi,u})\rangle}{|u|^2}\nn\\
\geq&\,1-\theta\Delta\lim_{|u|\rightarrow\infty}\frac{K-\frac{3a_1+a_2+L}{2}|\Phi^{\tr}_{\phi,u}(0)|^2
+2a_2\int_{-\tau}^{0}|\Phi^{\tr}_{\phi,u}(r)|^2\mathrm d\nu_2(r) }{2|u|^2}\nn\\
\geq &\,1-\theta\tr \lim_{|u|\rightarrow\infty}\frac{K-\frac{3a_1+a_2+L}{2}|u|^2+2a_2(\|\phi\|^2+|u|^2)}{2|u|^2}\nn\\
=&\,1+\theta\tr\frac{3(a_1-a_2)+L}{4}>1,
\end{align*}
which implies that $F^{\tr}_{\theta,\phi}$ is coercive.
In addition, for any $u_1,u_2\in\RR^d$ with $u_1\neq u_2$, by Assumption \ref{a2} and \eqref{movephi}, we derive
\begin{align*}
&\langle u_1-u_2, F^{\tr}_{\theta,\phi}(u_1)-F^{\tr}_{\theta,\phi}(u_2) \rangle\nn\\
=&\,|u_1-u_2|^2-\theta\tr\langle \Phi^{\tr}_{\phi,u_1}(0)-\Phi^{\tr}_{\phi,u_2}(0), b(\Phi^{\tr}_{\phi,u_1})-b(\Phi^{\tr}_{\phi,u_2}) \rangle\nn\\
\geq&\, |u_1-u_2|^2+\theta\tr a_1|\Phi^{\tr}_{\phi,u_1}(0)-\Phi^{\tr}_{\phi,u_2}(0)|^2
-\theta\tr a_2\int_{-\tau}^{0}|\Phi^{\tr}_{\phi,u_1}(r)-\Phi^{\tr}_{\phi,u_2}(r)|^2 \mathrm d\nu_2(r)\nn\\
=&\,(1+\theta\tr a_1)|u_1-u_2|^2-\theta\tr a_2\int_{-\tr}^{0}\frac{\tr+r}{\tr}|u_1-u_2|^2\mathrm d\nu_2(r)\nn\\
\geq&\, (1+\theta\tr a_1-\theta\tr a_2)|u_1-u_2|^2>|u_1-u_2|^2>0,
\end{align*}
which shows the monotone property of $F^{\tr}_{\theta,\phi}$.
Hence according to \cite[Lemma 3.1]{Mao-Szpruch2013}, we arrive at that for any $v\in\RR^d$, the equation
$F^{\tr}_{\theta,\phi}(u)=v$
has a unique solution $u\in\RR^d,$ and the inverse $(F^{\tr}_{\theta,\phi})^{-1}$ of $F^{\tr}_{\theta,\phi}$ exists.
For the fixed integer  $k\geq 0$ and $y^{\xi,\tr}_{t_k}$, we observe
\begin{align*}
F^{\tr}_{\theta,y^{\xi,\tr}_{t_k}}(y^{\xi,\tr}(t{_{k+1}}))=y^{\xi,\tr}(t_{k+1})-\theta b(\Phi^{\tr}_{y^{\xi,\tr}_{t_{k}},y^{\xi,\tr}(t_{k+1})})\tr =y^{\xi,\tr}(t_{k+1})-\theta b(y^{\xi,\tr}_{t_{k+1}})\tr.
\end{align*}
Then  \eqref{thEM} can be rewritten as
$$F^{\tr}_{\theta,y^{\xi,\tr}_{t_k}}(y^{\xi,\tr}(t{_{k+1}}))=y^{\xi,\tr}(t_k)+(1-\theta)b(y^{\xi,\tr}_{t_k})\tr
+\sigma(y^{\xi,\tr}_{t_k})\delta W_k.$$
By virtue of the existence of $(F^{\tr}_{\theta,y^{\xi,\tr}_{t_k}})^{-1}$, we derive that $y^{\xi,\tr}(t{_{k+1}})$ exists uniquely and satisfies
\begin{align*}
y^{\xi,\tr}(t{_{k+1}})=(F^{\tr}_{\theta,y^{\xi,\tr}_{t_k}})^{-1}\big(y^{\xi,\tr}(t_k)+(1-\theta)b(y^{\xi,\tr}_{t_k})\tr
+\sigma(y^{\xi,\tr}_{t_k})\delta W_k\big)\quad \text{a.s}.
\end{align*}
 Thus, the proof is finished. 
\end{proof}
The following lemma shows the Markov property of the numerical functional solution, whose proof is similar to that of \cite[Theorem 6.1]{Wufuke2024} and is omitted. 
\begin{lemma}
The $\theta$-EM functional solution is a time-homogenous Markov chain.
\end{lemma}

\section{Outline and notations}

\subsection{Outline}
Our objective is to investigate 
 longtime  behaviors   of the  numerical solution of  $\theta$-EM method \eqref{thEM}, including the longtime mean-square convergence and weak convergence, the convergence of the numerical invariant measure, 
 the existence and convergence of the numerical density function, and the large deviation principle of the numerical solution. To be specific, 
 \begin{itemize}
 \item[(1)] In \Cref{chap2}, we obtain the longtime mean-square convergence  of the $\theta$-EM functional solution for the SFDE \eqref{FF},  and show that the mean-square convergence rate is  $\frac12$. 
 \item[(2)] In Chapter 3, we present a quantitative estimate of the invariant measure of the   $\theta$-EM functional solution, based on the investigation  of  the  time-independent  weak convergence rate of the considered numerical method. 
 \item[(3)] In Chapter 4, we show the existence of density functions of the exact solution  and $\theta$-EM solution for SFDE \eqref{FF}, 
and further  establish the  convergence  of the numerical density function.
 \item[(4)] In Chapter 5, we investigate the Freidlin--Wentzell large deviation principle (LDP) of the $\theta$-EM solution, and further  present the logarithmic estimate of the numerical density function for the SFDE \eqref{FF} with small noise. 
 \end{itemize}

We proceed with the main results obtained in this paper:

\textbf{Chapter 2:} 
The mean-square convergence  is one of the  crucial  features  for the accuracy of a stochastic numerical method. 
Over the past decades, much attention has been paid to this topic  for SFDEs on the finite time horizon. 
For instance,  for SFDE with the globally Lipschitz continuous coefficients, when the delay argument is of the discrete type, authors in \cite{Huetal2004} develop a strong Milstein scheme and show that the strong convergence rate is 1; 
Authors in \cite{caowr} use the Wong--Zakai approximation as an intermediate step to derive three numerical schemes and analyze their mean-square convergence. For the distributed type of delay argument, the author in \cite{Mao2003LMS} obtains the mean-square convergence of the EM method and further proves that the convergence rate is $\frac12$; The author in \cite{Buckwar2006} provides a systematic approach to investigating the mean-square convergence of drift-implicit one-step schemes. 
For the superlinear growing coefficients case,  when the delay argument is of the discrete type,  authors in \cite{TamedEM2016} construct the tamed EM method and analyze its mean-square convergence. Recently, for the distributed type of delay argument, authors in \cite{LMS2024spa} design a truncated EM method and show that the mean-square convergence rate is $\frac12$.
By contrast, the research on the longtime 
mean-square convergence for  numerical methods of  SFDEs remains rare. And we 
 are aware of a recent paper \cite{chen2023probabilistic}, in which the  mean-square convergence rate of the EM method for the SFDE on the infinite time horizon is proved  to be $\frac12$ under the globally Lipschitz condition.  However, for the superlinearly growing coefficients case, how can one  obtain the longtime mean-square convergence of numerical methods for the SFDE?
 In  Chapter 2, we focus on   the $\theta$-EM method of the SFDE \eqref{FF} with superlinearly growing coefficients, and  study the longtime mean-square convergence of the method.  
 We first establish   
  the time-independent boundedness of high-order moments of the $\theta$-EM functional solution.
Then we derive the  inheritance of the exponential attractiveness   and the exponential stability of the exact functional solution
by the $\theta$-EM method. 
Armed with these preparations, we prove that the  mean-square convergence rate of the $\theta$-EM method on the infinite time horizon is  $\frac12.$ As applications, we derive the probabilistic limit behaviors, including the strong law of large numbers and the central limit theorem,  for the  $\theta$-EM functional solution.

\textbf{Chapter 3:} Invariant measure is an important characteristic for describing longtime dynamical behaviors of Markov processes. 
For  SFDEs with  
either discrete or distributed 
delay arguments, it is known that the solution is non-Markovian due to the dependence on the past. 
Hence, the functional solution, which is proved to be Markovian, becomes the objective to study the existence and  uniqueness of the invariant measure of the underlying  SFDE.
Numerical approximation of the unique invariant measure of the functional solution  for SFDEs with different  types of delay arguments has been studied in the existing literature. For example, for a kind of  discrete delay arguments case, authors in  \cite{luyulan} prove that the numerical solution of the BEM method at integer time admits a unique invariant measure, and further prove that the convergence rate of the numerical invariant measure is $1$.  
For the case of the  distributed delay argument, 
authors in \cite{Bao-Shao-Yuan2023} demonstrate that the  functional solution of the EM method of the SFDE with Markov switching admits a unique numerical invariant  measure which converges  to the exact one in the Wasserstein distance; Authors in \cite{Wufuke2024} concern the  approximation of invariant measure of the functional solution for the SFDE by the BEM method 
 and reveal the convergence of the numerical invariant measure in the Wasserstein distance. Natural questions   arise:    
 For the distributed delay argument case,  does the numerical invariant measure of a numerical method converge to the original one with rate 1? In particular,  is this convergence  result  applicable  to  the superlinealy growing coefficients case? 
 Less is known about these questions  to our knowledge. 

To take a step towards answering  the above  questions, in Chapter 3, we focus on the $\theta$-EM method for the SFDE \eqref{FF} with superlinearly growing coefficients, and investigate the existence, uniqueness, and the convergence rate of the invariant measure for the numerical functional solution. 
We first obtain the existence and uniqueness of the numerical invariant measure by means of  the time-independent moment boundedness  and the exponential attractiveness of the $\theta$-EM functional solution. 
A common approach to obtaining  the convergence rate of the numerical invariant measure is based on the longtime  weak convergence of the numerical method, which can be established by means of the Malliavin calculus.  
The prerequisite   is to derive the time-independent estimates for Malliavin derivatives and G\^ateaux derivatives of both the exact and numerical functional solutions. We show that the moments of these derivatives decay exponentially in time. 
Based on these properties, we obtain  
 the weak convergence rate with order 1 of the $\theta$-EM method on the infinite time horizon, which  is  twice the mean-square one. As a by-product, the  convergence rate for the numerical invariant measure of the $\theta$-EM functional solution is derived, which is the same as the weak convergence rate.


\textbf{Chapter 4:} 
The density function of the solution of a stochastic differential equation, characterizing  all the relevant probabilistic information, is one of the most essential characteristics that reveals the probabilistic behavior of the underlying solution. 
The Malliavin calculus, as   a powerful tool to prove the  existence and smoothness  of density functions, was first introduced by P. Malliavin in \cite{PMall} with the aim of giving a probabilistic proof of the H\"ormander hypoellipticity theorem; see  e.g. \cite{memo_SHEden, Pardoux, Nualart_Quer, Friz, Nualart} and references therein for further development and application of the Malliavin calculus. 
In order to approximate the density function of the stochastic differential equation, the study on the numerical density function of a numerical method has received much attention recently. For instance, for the stochastic  ordinary differential equation (SODE) with globally Lipschitz coefficients,  the convergence of the density of the numerical solution is proved in e.g. \cite{SODE_density, talay_density2, Hu_Watanabe, Ito_Taylor} based on It\^o–Taylor type dicretizations under the H\"ormander condition, and in e.g. \cite{KAKA08} based on the kernel density estimation and the Malliavin calculus.  For the stochastic Langevin equation with non-globally  coefficients, \cite{Sheng_density_SODE} proves  the existence, smoothness, and convergence of the numerical density function of the splitting method. For the stochastic partial differential equation, authors in \cite{Sheng_density_SPDE, hong2023density} investigate the existence and convergence of the numerical density function for the stochastic heat equation and the stochastic Cahn--Hilliard  equation,  respectively. 
However, less is known about the case of SFDEs. 

The main aim in Chapter 4 is to investigate the existence and convergence  of the density function of  the $\theta$-EM method for the SFDE \eqref{FF}.  
We first present the nondegeneracy of the numerical solution and prove the existence of the corresponding numerical density function by means of the Malliavin calculus. 
For the SFDE driven by  the multiplicative noise, 
combining the mean-square convergence result of the $\theta$-EM method, the convergence of the numerical density function is proved. 
While for the case of the additive noise,  
the convergence rate of the numerical density function is derived,  
 based on the test-functional-independent weak convergence analysis for the $\theta$-EM method.
The major obstacle in the proof lies in  the estimates of the negative moments of the determinant of the corresponding Malliavin covariance matrix of the numerical solution, which is overcome by presenting a discrete comparison principle for the SFDE with additive noise. Combining the full use of the Malliavin integration  by parts formula, the test-functional-independent weak convergence rate and further the convergence rate of numerical density function are both proved to be $1.$ 

\textbf{Chapter 5:} The LDP for the stochastic differential equation with small noise is also called the Freidlin--Wentzell LDP, which characterizes the exponential decay probabilities that sample paths of the solution of the stochastic differential equation deviate from that of the corresponding deterministic equation as the intensity of the noise tends to zero. Fruitful results appear on the  Freidlin--Wentzell LDP for both the SODE and the stochastic partial differential equation (SPDE). In contrast, the research for the case of the  numerical method is still at its infancy. We are aware of \cite{sode_LDP, jin_langevinldp} for the numerical study of the LDP for the SODE case and of \cite{chenzh2021large, chen2022adaptive, maxwell_LDP, jin2023asymptotics} for the SPDE case. 
The Freidlin--Wentzell LDP for SFDEs on the finite time horizon  has received much attention in recent years; see e.g. \cite{AAdensity, Baobook, jin2022large} and references therein. 
To our knowledge, there are few results for  both the SFDE \eqref{FF} and its numerical method on the infinite time horizon case.

The main contribution of Chapter 5 is to present  the  Freidlin--Wentzell LDP for the  $\theta$-EM method of the SFDE with small noise on the infinite time horizon. 
A well-known approach to presenting  the LDP,  through its equivalence to the Laplace principle, is proposed by \cite{Weakapproach}, where the finite time horizon case is considered.  Such method is also known as the weak convergence method, and has further been developed to prove the LDP on the infinite time horizon (see \cite{ZhengZH}).
The key in the analysis lies in  establishing the compactness of solutions for both the skeleton equation and the stochastic controlled equation  
in Banach space $\mathcal C_{\xi}([-\tau,+\infty);\mathbb R^d)$. 
In this regard, by analyzing the time-independent  moment estimats of the solution of the stochastic controlled equation, we prove that the $\theta$-EM solution  satisfies the LDP on $\mathcal C_{\xi}([-\tau,+\infty);\mathbb R^d)$ with the rate function given by the corresponding skeleton equation. Moreover, combining the technique of the Malliavin
calculus, we present the logarithmic estimate of the numerical density function 
and establish the relation between the logarithmic limit and the rate function of the LDP. 

\subsection{Notations}

\begin{basedescript}{\desclabelstyle{\pushlabel}\desclabelwidth{12em}}
\item[\rm{SFDE}] Stochastic functional differential equation
\item[\rm{EM}] Euler--Maruyama
\item[\rm{BEM}] Backward Euler--Maruyama
\item[$\theta$\rm{-EM}] $\theta$-Euler--Maruyama
\item[\rm{LDP}] Large deviation principle
\item[$\mathbb N$] Set of all non-negative integers 
\item[$\mathbb{N}_{+}$] Set of all positive integers

\item[${\mathbb R}^d$] $d$-dimensional Euclidean space
\item[$|x|$] Euclidean norm of $x\in\mathbb R^d$ or  $x\in\mathbb R^{d\times m}$
\item[$\<x,y\>$] Inner product  for $x,y\in\mathbb R^d$
\item[$\mathrm{Id}_{d\times d}$] $d$-dimensional identity operator
\item[$\otimes$] Kronecker inner product 
\item[$|\alpha|$]  $\sum_{i=1}^d\alpha_i$ for a multi-index $\alpha=(\alpha_1,\ldots,\alpha_d)$

\item[$i!$]  factorial of $i\in\mathbb N_+$
\item[$i!!$]  double factorial of $i\in\mathbb N_+$
\item[$C_{p}^{i}$] combinatorial number, i.e., $C_{p}^{i}=\frac{p!}{i!(p-i)!}$ 
\item[$a\vee b$ \rm{(resp.} $a\wedge b$)] Maximum (resp. minimum) of $a,b\in\mathbb{R}$
\item[$\mathbf 1_{A}$] $\mathbb R$-valued indicator function of a set $A$  with $\mathbf 1_{A}(x)=1$ for $x\in A$ and $\mathbf 1_{A}(x)=0$ for $x\in A^c$ 

\item [$\mathcal C^d:=(\mathcal {C}(\lbracket-\tau, 0\rbracket;\mathbb {R}^d),\|\cdot\|)$] Space of  continuous functions $\phi:[-\tau,0]\to\mathbb R^d$ equipped with the norm $\|\phi\|=\sup_{s\in[-\tau,0]}|\phi(s)|$

\item[$\mathcal B(\mathcal {C}(\lbracket-\tau,0\rbracket ;\mathbb R^d))$] Borel $\sigma$-algebra of  $\mathcal C([-\tau,0];\mathbb R^d)$

\item[$\mathcal P$] Space of probability measures on $(\mathcal C^d,\mathcal B(\mathcal C^d))$

\item[$\mathcal C_b^k$] Space of continuous  functionals $f:\mathbb R^d\to\mathbb R$ with bounded and continuous derivatives of orders up to $k$ 
\item[$\mathcal C_{pol}^{\infty}(\mathbb R^d;\mathbb R)$] Space of smooth functions $f:\RR^d\rightarrow \RR$ whose partial  derivatives have at most  polynomial growth
\item[$\mathcal C([0,\infty);\RR^m)$]
Space of continuous functions  $w:[0,+\infty)\rightarrow \RR^m$ endowed with the norm
$\|w\|_{\mathcal C([0,\infty);\RR^m)}:= \int_{0}^{\infty} |w(t)|e^{-\mathfrak a_1 t}\mathrm dt$ for some $\mathfrak a_1>0$

\item[$\mathcal C_0(\lbracket 0,+\infty);\mathbb R^m)$] Space of continuous functions  $w:[0,+\infty)\rightarrow \RR^m$ with initial value being 0

\item[$\mathcal C_{\xi}(\lbracket -\tau,+\infty);\mathbb R^d)$] Space of continuous functions $u:[-\tau,+\infty)\to\mathbb R^d$ with $u(r)=\xi(r)$ for $r\in[-\tau,0]$, endowed with the norm 
$$\|u\|_{\mathcal C_{\xi}}:=\sum_{k=1}^{\infty}e^{-\mathfrak a t_k}(\sup_{t\in[-\tau,t_k]}|u(t)|)\Delta$$ for some $\mathfrak a>0,$
which is also written as $\mathcal C_{\xi}$ for short

\item[$(\Omega,{\mathcal F},\{{\mathcal F}_{t}\}_{t\ge 0},{\mathbb P})$] Filtered probability space 

\item[${\mathbb E}\lbracket u\rbracket$] Expectation of a random variable $u$ 
\item[${\mathbb E}\lbracket u|{\mathcal G}\rbracket$] Conditional expectation of $u$ with respect to a $\sigma$-algebra ${\mathcal G}\subset{\mathcal F}$

\item[$\mathcal N(0,v^2)$] Normal random variable with mean 0 and variance $v^2$

\item[$\Delta$] Step size of the $\theta$-EM method
\item[$K$] A generic positive constant independent of $\Delta$ whose value may vary at different occurrences 
\item[$\lfloor t \rfloor $] Maximal grid point that no larger than $t,$ i.e., $\max \{t_k\leq t: \,t_k=k\Delta,\,k=0,1,\ldots\}$
\item[$\mu$] Invariant measure of the SFDE
\item[$\mu^{\Delta}$] Invariant measure of the $\theta$-EM method
\item[$\xi$] Deterministic initial value in $\mathcal C^d$ 
\item[$x^{\xi}(\cdot)$ \rm{(resp.} $x^{\xi}_{\cdot}$)] Solution \rm{(resp.} functional solution) of the SFDE
\item[$y^{\xi,\Delta}(\cdot)$ \rm{(resp.} $y^{\xi,\Delta}_{\cdot}$)] Solution (resp. functional solution) of the $\theta$-EM method
\item[$\varphi(t;t_i,\eta)$] Solution of the SFDE at time $t$ with initial value $\eta$ at $t_i$
\item[$\varphi^{Int}(\cdot;t_i,\eta)$] Linear interpolation of $\varphi(\cdot;t_i,\eta)$ by grids $\{(t_k,\varphi(t_k;t_i,\eta)),\\k=i,i+1,\ldots\}$
\item[$\Phi(t_k;t_i,\eta^{Int})$] Solution of the $\theta$-EM method at time $t_k$ with initial value $\eta^{Int}$ at $t_i$

\item[$\mu^{\xi}_t$] Probability measure generated by $x^{\xi}_t$
\item[$P^{\xi}_t$] Markov semigroup generated by $\mu^{\xi}_t$
\item[$\mu^{\xi,\Delta}_t$] Probability measure generated by $y^{\xi,\Delta}_t$
\item[$P^{\xi,\Delta}_t$] Markov semigroup generated by $\mu^{\xi,\Delta}_t$ 

\item[$\mathfrak p$] Density function of the solution of the SFDE
\item[$\mathfrak p^{\Delta}$] Density function of the solution of the $\theta$-EM method

\item[$L^{p}(\Omega, E)$] Space of all $E$-valued random variables which are $p$-integrable with respect to $\mathbb{P}$

\item[$(H,\<\cdot,\cdot\>_{H},\|\cdot\|_H)$] Hilbert space $L^2([0,+\infty);\mathbb R^m)$

\item[$L^2_{\Delta}(\lbracket0,+\infty);\mathbb R^m)$] Space of functions $v:[0,\infty)\to\mathbb R^m$ satisfying that $v(\lfloor \cdot\rfloor)$ is square integrable, which is also written as $L^2_{\Delta}$ for short
\item[$\|\cdot\|_{L^2_{\Delta}}$] Norm in $L^2_{\Delta}$ 
\item[$\mathcal S_{M}$] Subspace of $L^2_{\Delta}([0,+\infty);\mathbb R^m)$ for functions $v$ satisfying $\|v\|^2_{L^2_{\Delta}}\leq M$
\item[$\mathcal P_M$] Space of $\mathcal F_t$-measurable functions $v:\Omega\times [0,+\infty)\to\mathbb R^m$ satisfying that $v\in\mathcal S_M$ a.s.
\item[$(\mathcal L(E_1;E_2),\|\cdot\|_{\mathcal L(E_1;E_2)})$] Space of  bounded linear operators from $E_1$ to $E_2$ with usual operator norm

\item[$\mathcal D$] G\^ateaux derivative operator
\item[$\mathcal D^{\alpha}$] G\^ateaux  derivative operator of  order $\alpha$
\item[$D$] Malliavin derivative operator
\item[$D^{\alpha}$] Malliavin derivative operator of  order $\alpha$
\item[$(\mathbb D^{\alpha,p}(\mathbb R^d),\|\cdot\|_{\alpha,p})$] Malliavin Sobolev space consisting of all random variables $f:\Omega\to {\mathbb R^d}$ whose Malliavin derivatives up to orders $\alpha$  belong to $L^p(\Omega)$


\end{basedescript}

 \section{Acknowledgement}

 This work is funded by the National key R\&D Program of China under Grant (No. 2020YFA0713701),   National Natural Science Foundation of China (No. 12031020, No. 12022118), and by Youth Innovation Promotion Association CAS, China.

    \cleardoublepage
%
%
%


\chapter{Mean-square convergence analysis} \label{chap2}
In this chapter, we analyze the  longtime mean-square convergence  of the $\theta$-EM method \eqref{thEM} of the SFDE \eqref{FF} with superlinearly growing coefficients. A prerequisite is to prove the time-independent moment   boundedness of the numerical solution. Then we derive the exponential attractiveness  and the exponential stability 
of the $\theta$-EM method. 
With these preparations, the time-independent 
  mean-square convergence rate of the $\theta$-EM method is shown to be $\frac12.$

\section{Time-independent moment boundedness of numerical solution} 
In this section, we investigate the time-independent moment boundedness of the  $\theta$-EM functional solution,
which is the key  to obtain the longtime mean-square convergence rate of the considered numerical method. 

\subsection{Time-independent boundedness of second moment} 

This subsection is devoted to proving the time-independent boundedness of the second moment of  the $\theta$-EM functional solution. Before that, we show that the $p$th moment for any $p>0$ of the numerical solution in any finite time horizon is bounded, which ensures that the stochastic integral with respect to the numerical solution is a martingale.

For this purpose, we define an auxiliary process as follows:
\begin{align}
\begin{cases}\label{ST}
z^{\xi,\Delta}(t_k)=\;\xi(t_k),\quad k=-N,-N+1,\ldots,-1,\\
z^{\xi,\Delta}(t_{k})=\;y^{\xi,\Delta}(t_k)-\theta b(y^{\xi,\Delta}_{t_{k}})\Delta,\quad k=0,\\
z^{\xi,\Delta}(t_{k})=\;z^{\xi,\Delta}(t_{k-1})+b(y^{\xi,\Delta}_{t_{k-1}})\Delta 
+\sigma(y^{\xi,\Delta}_{t_{k-1}})\delta W_{k-1},\quad k\in\mathbb N_+.
\end{cases}
\end{align}
The continuous version $\{z^{\xi,\Delta}(t)\}_{t\geq-\tau}$ is given by  
\begin{align*}
z^{\xi,\Delta}(t)=
\begin{cases}
z^{\xi,\Delta}(t_{k})\!+\!b(y^{\xi,\Delta}_{t_k})(t\!-\!t_k)
\!+\!\sigma(y^{\xi,\Delta}_{t_k})(W(t)\!-\!W(t_k)),~~ t\in[t_k,t_{k+1}),~~k\in\mathbb N,\\[2mm]
\frac{t_{j+1}-t}{\Delta}z^{\xi,\Delta}(t_{j})
+\frac{t-t_{j}}{\Delta}z^{\xi,\Delta}(t_{j+1}),\quad t\in[t_j,t_{j+1}],~~j\in\{-N,\ldots,-1\}.
\end{cases}
\end{align*}
We see clearly that
\begin{align}\label{zcont}
z^{\xi,\Delta}(t)=z^{\xi,\Delta}(0)+\int_{0}^{t}b(y^{\xi,\Delta}_{s})\mathrm ds
+\int_{0}^{t}\sigma(y^{\xi,\Delta}_{s})\mathrm dW(s), ~~~t>0,
\end{align}
where
$$y^{\xi,\Delta}_{s}:=\sum_{k=0}^{\infty}y^{\xi,\Delta}_{t_{k}}\mathbf 1_{[t_k,t_{k+1})}(s).$$
Denote by $\{z^{\xi,\Delta}_t\}_{t\geq0}$ the segment process  with respect to $\{z^{\xi,\Delta}(t)\}_{t\geq-\tau}$, namely, for any $t\geq0$ and $r\in[-\tau,0]$,
we have $z^{\xi,\Delta}_t(r)=z^{\xi,\Delta}(t+r).$

\begin{prop}\label{p2.1}
Under Assumptions {\ref{a1}} and {\ref{a2}}, 
 for any $T>0$ and integer $p\ge 2,$ it holds that
\begin{align*}
\sup_{\Delta\in(0,1]}\E\Big[\sup_{t_k\in[0,T]}\big(|z^{\xi,\Delta}(t_k)|^{2p}
+|y^{\xi,\Delta}(t_k)|^{2p}\big)\Big]\leq K_T(1+\|\xi\|^{2p}),
\end{align*}
where the positive constant $K_T$ depends on $p$ and $T$.
\end{prop}
\begin{proof}
We divide the proof into two steps.

\underline{Step 1.} We shall show that for any $T>0$ and $p\in \mathbb{N}_+$,
\begin{align}\label{BoundT}
\sup_{\Delta\in(0,1]}\sup_{t_k\in[0,T]}\E\Big[|z^{\xi,\Delta}(t_k)|^{2p}
+|y^{\xi,\Delta}(t_k)|^{2p}\Big]<K\|\xi\|^{2p}e^{KT}.
\end{align}
Fix $\Delta\in(0,1]$ and $k\in \mathbb N$.
For any $M>\|\xi\|$, define a random time $\zeta^{\xi,\Delta}_{M}$ by
\begin{align*}
\zeta^{\xi,\Delta}_{M}=\inf\{i\in \mathbb N: |y^{\xi,\Delta}(t_i)|\vee|z^{\xi,\Delta}(t_i)|>M\}.
\end{align*}
Obviously, $$\{\omega\in\Omega:\zeta^{\xi,\Delta}_{M}(\omega)>i+1\}
=\{\omega\in\Omega:\zeta^{\xi,\Delta}_{M}(\omega)\leq i\}^c\in \mathcal F_{t_i}\quad\forall~i\in\mathbb N.$$
To simplify notations, we introduce the following abbreviations
\begin{align*}
&z(t_k)=z^{\xi,\Delta}(t_k), ~~z_{t_k}=z^{\xi,\Delta}_{t_k},\\
&y(t_k)=y^{\xi,\Delta}(t_k), ~~y_{t_k}=y^{\xi,\Delta}_{t_k},\\
&b_{k}=b(y^{\xi,\Delta}_{t_k}),~~\sigma_k=\sigma(y^{\xi,\Delta}_{t_k}),
~~\zeta_{M}=\zeta^{\xi,\Delta}_{M}.
\end{align*}
In fact, it follows from \eqref{F2r1+1} and \eqref{ST} that
\begin{align}\label{12p2.1}
&\quad |z(t_{k+1})|^2\nn\\
&=|z(t_k)|^2+|b_k|^2\Delta^2+|\sigma_k\delta W_k|^2
+2\langle y(t_k)-\theta b_k\Delta, b_k\rangle\Delta\nn\\
&\quad +2\langle z(t_k), \sigma_k\delta W_k\rangle+2\langle b_k, \sigma_k\delta W_k\rangle\Delta\nn\\
&\leq |z(t_k)|^2+(1-2\theta)|b_k|^2\Delta^2+|\sigma_k\delta W_k|^2
+K\Delta-\frac{3a_1+a_2+L}{2}\Delta|y(t_k)|^2\nn\\
&\quad +2a_2\Delta\int_{-\tau}^{0}|y_{t_k}(r)|^2\mathrm d\nu_2(r)+2\langle z(t_k), \sigma_k\delta W_k\rangle\!+\!\frac{2}{\theta}\langle y(t_k)\!-\!z(t_k), \sigma_k\delta W_k\rangle.
\end{align}
By $\theta\in(0.5,1]$, we have
\begin{align}\label{10p2.1}
|z(t_{k+1})|^2
\leq |z(t_k)|^2+|\sigma_k\delta W_k|^2
+K\Delta+2a_2\Delta\int_{-\tau}^{0}|y_{t_k}(r)|^2\mathrm d\nu_2(r)+\mathcal M_k,
\end{align}
where  $\{\mathcal M_{k}\}_{k=0}^{\infty}$ is a local martingale defined by
\begin{align}\label{mart}
\mathcal M_k=\frac{2(\theta-1)}{\theta}\langle z(t_k), \sigma_k\delta W_k\rangle+\frac{2}{\theta}\langle y(t_k), \sigma_k\delta W_k\rangle.
\end{align}
If $\zeta_{M}\geq k+1$, we have $k\wedge \zeta_{M}=k$ and
 \begin{align*}
 &|z(t_{(k+1)\wedge \zeta_{M}})|^2=|z(t_{k+1})|^2\nn\\
\leq{} &|z(t_{k\wedge \zeta_{M}})|^2+|\sigma_k\delta W_k|^2
+K\Delta+2a_2\Delta\int_{-\tau}^{0}|y_{t_k}(r)|^2\mathrm d\nu_2(r)+\mathcal M_k.
\end{align*}
If $\zeta_{M}<k+1$, we have $\zeta_{M}\leq k$ and 
\begin{align*}
 |z(t_{(k+1)\wedge \zeta_{M}})|^2=|z(t_{\zeta_{M}})|^2
=|z(t_{k\wedge \zeta_{M}})|^2.
\end{align*}
The above two cases imply  that
\begin{align*}
 &|z(t_{(k+1)\wedge \zeta_{M}})|^2\nn\\
\leq {}&|z(t_{k\wedge \zeta_{M}})|^2+\Big(|\sigma_k\delta W_k|^2
+K\Delta+2a_2\Delta\int_{-\tau}^{0}|y_{t_k}(r)|^2\mathrm d\nu_2(r)
+\mathcal M_k\Big)\textbf 1_{\{\zeta_M\geq k+1\}}.
\end{align*}
Then for any $p\in\mathbb N$,
\begin{align}\label{1p2.1}
&\E[|z(t_{(k+1)\wedge \zeta_{M}})|^{2p}]\nn\\
\leq{}&\E [|z(t_{k\wedge \zeta_{M}})|^{2p}]+\sum_{l=1}^{p}C_{p}^{l}\E\Big[|z(t_{k\wedge \zeta_{M}})|^{2(p-l)}\big(|\sigma_k\delta W_k|^2
+K\Delta\nn\\
&\quad+2a_2\Delta\int_{-\tau}^{0}|y_{t_k}(r)|^2\mathrm d\nu_2(r)+\mathcal M_k\big)^{l}\mathbf 1_{\{\zeta_M\geq k+1\}}\Big]\nn\\
=: {}&\E [|z(t_{k\wedge \zeta_{M}})|^{2p}]+\sum_{l=1}^p\mathscr I_{l},
\end{align}
where the positive constant $C_{p}^l$ is the  binomial coefficient.
Note that $$\delta W_k\textbf 1_{\{\zeta_M\geq k+1\}}=W(t_{(k+1)\wedge \zeta_M})-W(t_{k\wedge  \zeta_M}).$$
By virtue of the Doob martingale stopping
 time theorem, 
 we obtain
\begin{align}\label{o3.35}
&\E\Big[A(t_k) \delta W_k\textbf 1_{\{ \zeta_M\geq k+1\}}\big|
\mathcal{F}_{t_{k\wedge \zeta_M}}\Big]=0,\nn\\
&\E\Big[\big|A(t_k)\delta W_k\big|^{2}\textbf 1_{\{ \zeta_M\geq k+1\}}\big|
\mathcal{F}_{t_{k\wedge \zeta_M}}\Big]
=|A(t_k)|^2\Delta\textbf 1_{\{ \zeta_M\geq k+1\}},
\end{align}
where $A(t_k)$ is a $\RR^{d\times m}$-valued  $\mathcal{F}_{t_k}$-measurable random variable for $k\in\mathbb N$. 
Then
\begin{align*}
\mathscr I_{1}
&=C^1_p\E\bigg[\E\Big[|z(t_{k\wedge \zeta_{M}})|^{2(p-1)}\big(|\sigma_k\delta W_k|^2
+K\Delta+2a_2\Delta\int_{-\tau}^{0}|y_{t_k}(r)|^2\mathrm d\nu_2(r)\nn\\
&\quad +\mathcal M_k\big)\textbf 1_{\{\zeta_M\geq k+1\}}\big|
\mathcal{F}_{t_{k\wedge \zeta_M}}\Big]\bigg]\nn\\
&=C^1_p\Delta\E\Big[|z(t_{k\wedge \zeta_{M}})|^{2(p-1)}\big(|\sigma_k|^2
+K+2a_2\int_{-\tau}^{0}|y_{t_k}(r)|^2\mathrm d\nu_2(r)\big)\textbf 1_{\{\zeta_M\geq k+1\}}\Big].
\end{align*}
Making use of Assumption \ref{a1} and applying the Young inequality, we derive
\begin{align}\label{2p2.1}
\mathscr I_{l}
\leq{} &C_{p}^{1}\Delta\E\Big[|z(t_{k\wedge \zeta_{M}})|^{2(p-1)}\big(2L|y(t_k)|^2+2L\int_{-\tau}^{0}|y_{t_k}(r)|^2\mathrm d\nu_1(r)
+K\nn\\
&\quad+2a_2\int_{-\tau}^{0}|y_{t_k}(r)|^2\mathrm d\nu_2(r)\big)\textbf 1_{\{\zeta_M\geq k+1\}}\Big]\nn\\
\leq{} &K\Delta\!\!+\!\!K\Delta\E[|z(t_{k\wedge \zeta_{M}})|^{2p}]\!+\!K\Delta\E[|y(t_{k\wedge \zeta_{M}})|^{2p}]\!+\!K\Delta\E \Big[\big( \int_{-\tau}^{0}|y_{t_k}(r)|^{2p}\mathrm d\nu_1(r)\nn\\
&\quad+\int_{-\tau}^{0}|y_{t_k}(r)|^{2p}\mathrm d\nu_2(r)\big)\textbf 1_{\{\zeta_M\geq k+1\}}\Big].
\end{align}
Moreover, by the elementary inequality we obtain
\begin{align*}
&\big(|\sigma_k\delta W_k|^2
+K\Delta+2a_2\Delta\int_{-\tau}^{0}|y_{t_k}(r)|^2\mathrm d\nu_2(r)+\mathcal M_k\big)^{l}\nn\\
\leq{}& K\big(|\sigma_k\delta W_k|^{2l}
\!+\!\Delta^l+\Delta^l\int_{-\tau}^{0}|y_{t_k}(r)|^{2l}\mathrm d\nu_2(r)\!+\!|z(t_k)|^l|\sigma_k\delta W_k|^l\!+\!|y(t_k)|^l|\sigma_k\delta W_k|^l\big).
\end{align*}
This, along with the following fact
\begin{align*}
&\E\Big[\big|A(t_k)\delta W_k\big|^{i}\textbf 1_{\{ \zeta_M\geq k+1\}}\big|
\mathcal{F}_{t_{k\wedge\zeta_M}}\Big]
\leq K |A(t_k)|^{i}\Delta^{\frac{i}{2}}\textbf 1_{\{ \zeta_M\geq k+1\}}~~~~\forall \,i\in\mathrm N ~\mbox{with} ~i\geq2
\end{align*}
implies that
 \begin{align*}
\sum_{l=2}^{p}\mathscr I_{l}
\leq{}&\sum_{l=2}^{p} K\E\bigg[\E\Big[|z(t_{k\wedge \zeta_{M}})|^{2(p-l)}\big(|\sigma_k\delta W_k|^{2l}
+\Delta^l+\Delta^l\int_{-\tau}^{0}|y_{t_k}(r)|^{2l}\mathrm d\nu_2(r)\nn\\
&\quad+|z(t_k)|^l|\sigma_k\delta W_k|^l+|y(t_k)|^l|\sigma_k\delta W_k|^l\big)\textbf 1_{\{\zeta_M\geq k+1\}}\big|
\mathcal{F}_{t_{k\wedge\zeta_M}}\Big]\bigg]\nn\\
\leq{}& K\Delta\sum_{l=2}^{p}\E\Big[|z(t_{k\wedge \zeta_{M}})|^{2(p-l)}\big(|\sigma_k|^{2l}
+1+\int_{-\tau}^{0}|y_{t_k}(r)|^{2l}\mathrm d\nu_2(r)\nn\\
&\quad+|z(t_k)|^l|\sigma_k|^l+|y(t_k)|^l|\sigma_k|^l\big)\textbf 1_{\{\zeta_M\geq k+1\}}\Big].
\end{align*}
Applying  Assumption \ref{a1} and the Young inequality again, we arrive at
 \begin{align}\label{3p2.1}
\sum_{l=2}^{p}\mathscr I_{l}
\leq{} &K\Delta+K\Delta\E[|z(t_{k\wedge \zeta_{M}})|^{2p}]+K\Delta\E[|y(t_{k\wedge \zeta_{M}})|^{2p}]\nn\\
&\quad+K\Delta\E \Big[\big( \int_{-\tau}^{0}|y_{t_k}(r)|^{2p}\mathrm d\nu_1(r)+\int_{-\tau}^{0}|y_{t_k}(r)|^{2p}\mathrm d\nu_2(r)\big)\textbf 1_{\{\zeta_M\geq k+1\}}\Big].
\end{align}
Inserting \eqref{2p2.1} and \eqref{3p2.1} into \eqref{1p2.1} and by the recursive calculation, we obtain 
\begin{align}\label{4p2.1}
&\E[|z(t_{(k+1)\wedge \zeta_{M}})|^{2p}]\nn\\
\leq{}&\E [|z(t_{k\wedge \zeta_{M}})|^{2p}]+K\Delta+K\Delta\E[|z(t_{k\wedge \zeta_{M}})|^{2p}]+K\Delta\E[|y(t_{k\wedge \zeta_{M}})|^{2p}]\nn\\
&\quad+K\Delta\E \Big[\big( \int_{-\tau}^{0}|y_{t_k}(r)|^{2p}\mathrm d\nu_1(r)+\int_{-\tau}^{0}|y_{t_k}(r)|^{2p}\mathrm d\nu_2(r)\big)\textbf 1_{\{\zeta_M\geq k+1\}}\Big]\nn\\
\leq{} &\E [|z(0)|^{2p}]\!+\!K(k+1)\Delta\!+\!K\Delta\sum_{i=0}^{k}\E[|z(t_{i\wedge \zeta_{M}})|^{2p}]\!+\!K\Delta\sum_{i=0}^{k}\E[|y(t_{i\wedge \zeta_{M}})|^{2p}]\nn\\
&\quad+K\Delta\sum_{i=0}^{k}\E \Big[\big( \int_{-\tau}^{0}|y_{t_i}(r)|^{2p}\mathrm d\nu_1(r)
+\int_{-\tau}^{0}|y_{t_i}(r)|^{2p}\mathrm d\nu_2(r)\big)\textbf 1_{\{\zeta_M\geq i+1\}}\Big].
\end{align}
It follows from \eqref{thL} and the convex property of $|\cdot|^{2p}$ that
\begin{align}\label{5p2.1}
&\sum_{i=0}^{k}\int_{-\tau}^{0}|y_{t_i}(r)|^{2p}\mathrm d\nu_1(r)\textbf 1_{\{\zeta_M\geq i+1\}}\nn\\
=&\sum_{j=-N}^{-1}\sum_{i=0}^{k}\int_{t_j}^{t_{j+1}}
|\frac{t_{j+1}-r}{\Delta}y(t_{i+j})+\frac{r-t_{j}}{\Delta}y(t_{i+j+1})|^{2p}\mathrm d\nu_1(r)\textbf 1_{\{\zeta_M\geq i+1\}}\nn\\
\leq{}&\sum_{j=-N}^{-1}\sum_{i=0}^{k}\int_{t_j}^{t_{j+1}}
\frac{t_{j+1}-r}{\Delta}\mathrm d\nu_1(r)|y(t_{i+j})|^{2p}\textbf 1_{\{\zeta_M\geq i+j+1\}}\nn\\
&\quad+\sum_{j=-N}^{-1}\sum_{i=0}^{k}\int_{t_j}^{t_{j+1}}
\frac{r-t_{j}}{\Delta}\mathrm d\nu_1(r)|y(t_{i+j+1})|^{2p}\textbf 1_{\{\zeta_M\geq i+j+2\}}\nn\\
\leq{}&\sum_{j=-N}^{-1}\sum_{l=-N}^{k}\int_{t_j}^{t_{j+1}}
\frac{t_{j+1}-r}{\Delta}\mathrm d\nu_1(r)|y(t_{l})|^{2p}\textbf 1_{\{\zeta_M\geq l+1\}}\nn\\
&\quad+\sum_{j=-N}^{-1}\sum_{l=-N}^{k}\int_{t_j}^{t_{j+1}}
\frac{r-t_{j}}{\Delta}\mathrm d\nu_1(r)|y(t_{l})|^{2p}\textbf 1_{\{\zeta_M\geq l+1\}}\nn\\
={}&\sum_{j=-N}^{-1}\sum_{l=-N}^{k}\int_{t_j}^{t_{j+1}}
\mathrm d\nu_1(r)|y(t_{l})|^{2p}\textbf 1_{\{\zeta_M\geq l+1\}}\nn\\
\leq{}&N\|\xi\|^{2p}+\sum_{i=0}^{k}|y(t_{i})|^{2p}\textbf 1_{\{\zeta_M\geq i+1\}}.
\end{align}
Similarly,
 \begin{align}\label{6p2.1}
\sum_{i=0}^{k}\int_{-\tau}^{0}|y_{t_i}(r)|^{2p}\mathrm d\nu_2(r)\textbf 1_{\{\zeta_M\geq i+1\}}
 \leq N\|\xi\|^{2p}+\sum_{i=0}^{k}|y(t_{i})|^{2p}\textbf 1_{\{\zeta_M\geq i+1\}}.
\end{align}
Inserting \eqref{5p2.1} and \eqref{6p2.1} into \eqref{4p2.1}, we have
 \begin{align}\label{7p2.1}
\E[|z(t_{(k+1)\wedge \zeta_{M}})|^{2p}]
\leq{} &\E [|z(0)|^{2p}]+K(k+1)\Delta+K\Delta\sum_{i=0}^{k}\E[|z(t_{i\wedge \zeta_{M}})|^{2p}]+K\tau\|\xi\|^{2p}\nn\\
&\quad+K\Delta\sum_{i=0}^{k}\E\big[|y(t_{i})|^{2p}\textbf 1_{\{\zeta_M\geq i+1\}}\big].
\end{align}
It is straightforward to see from $z(t_i)=y(t_i)-\theta b_{i}\Delta$  and \eqref{F2r1+1} that
\begin{align}\label{9p2.1}
|z(t_i)|^{2}={}& |y(t_i)|^2+\theta^{2}\Delta^2b_i^2-2\theta\Delta \langle y(t_i),b_i\rangle\nn\\
\geq{}&-K\Delta+(1+\theta\Delta\frac{3a_1+a_2+L}{2})|y(t_i)|^2-2a_2\theta\Delta\int_{-\tau}^{0}|y_{t_i}(r)|^2\mathrm d\nu_2(r).
\end{align}
According to \cite[Lemma 4.1]{Mao2007}, we know that for any $u,v\in\RR$, ~$q\geq1$, and $\e>0$,  there exists a constant $K_{\e}>1$ such that
\begin{align*}
|u+v|^{q}\leq K_{\e}|u|^{q}+(1+\e)|v|^{q}.
\end{align*}
This, along with  \eqref{9p2.1} implies that
\begin{align*}
&\Big(1+\big(\theta\Delta\frac{3a_1+a_2+L}{2}\big)^{p}\Big)\sum_{i=0}^{k}|y(t_{i})|^{2p}\textbf 1_{\{\zeta_M\geq i+1\}}\nn\\
\leq{}& K_{\e}\sum_{i=0}^{k}(K\Delta+|z(t_{i})|^{2}\textbf 1_{\{\zeta_M\geq i+1\}})^{p}\\
&\quad+(1+\e)(2a_2\theta\Delta)^{p}\sum_{i=0}^{k}
\int_{-\tau}^{0}|y_{t_{i}}(r)|^{2p}\mathrm d\nu_2(r)\textbf 1_{\{\zeta_M\geq i+1\}}.
\end{align*}
By \eqref{6p2.1}, we derive
\begin{align*}
&\Big(1+\big(\theta\Delta\frac{3a_1+a_2+L}{2}\big)^{p}\Big)
\sum_{i=0}^{k}|y(t_{i})|^{2p}\textbf 1_{\{\zeta_M\geq i+1\}}\nn\\
\leq{}&KK_{\e}(k+1)\Delta+KK_{\e}\sum_{i=0}^{k}|z(t_{i\wedge\zeta_{M}})|^{2p}+K\Delta^{p-1}\|\xi\|^{2p}\nn\\
&\quad+(1+\e)(2a_2\theta\Delta)^{p}\sum_{i=0}^{k}
|y(t_{i})|^{2p}\textbf 1_{\{\zeta_M\geq i+1\}}.
\end{align*}
Choose  a sufficiently small $\e>0$   such that
$$\Big(\frac{3a_1+a_2+L}{2}\Big)^{p}>(1+\e)(2a_2)^p.$$
Then
\begin{align}\label{yz-relat}
&\sum_{i=0}^{k}|y(t_{i})|^{2p}\textbf 1_{\{\zeta_M\geq i+1\}}
\leq K(k+1)\Delta+K\sum_{i=0}^{k}|z(t_{i\wedge\zeta_{M}})|^{2p}+K\Delta^{p-1}\|\xi\|^{2p}\nn\\
&-\big((\frac{3a_1+a_2+L}{2})^{p}-(1+\e)(2a_2)^{p}\big)\theta^p\Delta^p\sum_{i=0}^{k}
|y(t_{i})|^{2p}\textbf 1_{\{\zeta_M\geq i+1\}}\nn\\
\leq{}& K(k+1)\Delta+K\sum_{i=0}^{k}|z(t_{i\wedge\zeta_{M}})|^{2p}+K\|\xi\|^{2p}.
\end{align}
Inserting this into \eqref{7p2.1} yields
\begin{align*}
\E[|z(t_{(k+1)\wedge \zeta_{M}})|^{2p}]
\leq K\|\xi\|^{2p}+K(k+1)\Delta+K\Delta\sum_{i=0}^{k}\E[|z(t_{i\wedge \zeta_{M}})|^{2p}].
\end{align*}
By the discrete Gr\"onwall inequality, we have that for any $T>0$,
\begin{align*}
\sup_{0\leq t_k\leq T}\E[|z(t_{k\wedge\zeta_{M}})|^{2p}]\leq K\|\xi\|^{2p}e^{KT}.
\end{align*}
 Letting $M\rightarrow\infty$, one has
 $$\sup_{0\leq t_k\leq T}\E[|z(t_{k})|^{2p}]\leq K\|\xi\|^{2p}e^{KT}.$$
Combining \eqref{9p2.1}, we completes the proof of \underline{Step 1}.

\underline{Step 2.}
Similarly, by \eqref{10p2.1}
and the recursive calculation, we obtain
\begin{align*}
&|z(t_{k+1})|^{2p}\nn\\
\leq{}& |z(0)|^{2p}+\sum_{i=0}^{k}C_{p}^{1}|z(t_{i})|^{2(p-1)}\Big(|\sigma_i\delta W_i|^2
+K\Delta+2a_2\Delta\int_{-\tau}^{0}|y_{t_i}(r)|^2\mathrm d\nu_2(r)\Big)\nn\\
&\quad+\sum_{i=0}^{k}\sum_{l=2}^{p}C_{p}^{l}|z(t_{i})|^{2(p-l)}\Big||\sigma_i\delta W_i|^2
+K\Delta+2a_2\Delta\int_{-\tau}^{0}|y_{t_i}(r)|^2\mathrm d\nu_2(r)+\mathcal M_i\Big|^{l}\nn\\
&\quad+\sum_{i=0}^{k}C_{p}^{1}|z(t_{i})|^{2(p-1)}\mathcal M_i.
\end{align*}
Then
\begin{align}\label{11p2.1}
\E\Big[\sup_{0\leq t_k\leq T}|z(t_{k})|^{2p}\Big]
\leq |z(0)|^{2p}+I_1+I_2+I_3,
\end{align}
where
\begin{align*}
I_1:=&\sum_{k=0}^{\lfloor T\rfloor}\Big\{C_{p}^{1}\E\Big[|z(t_{k})|^{2(p-1)}\big(|\sigma_k\delta W_k|^2
+K\Delta+2a_2\Delta\int_{-\tau}^{0}|y_{t_k}(r)|^2\mathrm d\nu_2(r)\big)\Big]\nn\\
&+\sum_{l=2}^{p}\!C_{p}^{l}\E\Big[|z(t_ {k})|^{2(p-l)}\Big||\sigma_k\delta W_k|^2 \!+\!K\Delta\!+\!2a_2\Delta\int_{-\tau}^{0}|y_{t_k}(r)|^2\mathrm d\nu_2(r)\!+\!\mathcal M_k\Big|^{l}\Big]\Big\}\nn\\
I_2:=&\E\Big[\sup_{0\leq t_k\leq T}\sum_{i=0}^{k}C_{p}^{1}\frac{2(\theta-1)}{\theta}|z(t_{i})|^{2(p-1)}
\langle z(t_i), \sigma_i\delta W_i\rangle\Big],\nn\\
I_3:=&\E\Big[\sup_{0\leq t_k\leq T}\sum_{i=0}^{k}C_{p}^{1}\frac{2}{\theta}|z(t_{i})|^{2(p-1)}\langle y(t_i), \sigma_i\delta W_i\rangle\Big].
\end{align*}
By similar arguments to those in the proof of \underline{Step 1}, we derive
\begin{align}\label{I1p2.1}
I_1\leq K\Delta\sum_{k=0}^{\lfloor T\rfloor}\E[|z(t_{k})|^{2p}]+K\Delta+K\|\xi\|^{2p}
+K\Delta\sum_{k=0}^{\lfloor T \rfloor}\mathbb E[|y(t_{k})|^{2p}].
\end{align}
Applying the Burkholder--Davis--Gundy inequality,  we arrive at
\begin{align*}
I_2\leq{}&K\E\Big[\Big(\sum_{k=0}^{\lfloor T \rfloor}|z(t_{k})|^{4p-2}|\sigma_k|^2\Delta\Big)^\frac{1}{2}\Big]\nn\\
\leq{}&\E\Big[\sup_{0\leq t_i\leq\lfloor T \rfloor }|z(t_{i})|^{2p-1} \big(K\Delta\sum_{k=0}^{\lfloor T\rfloor}|\sigma_k|^2\big)^{\frac{1}{2}}\Big].
\end{align*}
Using the Young inequality, Assumption \ref{a1}, and \eqref{5p2.1},  one has
\begin{align}\label{I2p2.1}
I_2\leq{}&\frac{1}{4}\E\Big[\sup_{0\leq t_k\leq\lfloor T \rfloor }|z(t_{k})|^{2p}\Big]+K\Delta^p\E\Big[\big(\sum_{k=0}^{\lfloor T \rfloor}|\sigma_k|^2\big)^{p}\Big]\nn\\
\leq{}&\frac{1}{4}\E\Big[\sup_{0\leq t_k\leq\lfloor T \rfloor }|z(t_{k})|^{2p}\Big]+K\Delta(T^{p-1}+1)\sum_{k=0}^{\lfloor T \rfloor}\E[|\sigma_k|^{2p}]\nn\\
\leq{}&\frac{1}{4}\E\Big[\sup_{0\leq t_k\leq\lfloor T \rfloor }|z(t_{k})|^{2p}\Big]+K(T^{p-1}+1)\nn\\
&\quad+K\Delta(T^{p-1}+1)\sum_{k=0}^{\lfloor T \rfloor}\mathbb E\Big[|y(t_{k})|^{2p}+\int_{-\tau}^{0}|y_{t_{k}}(r)|^{2p}\mathrm d\nu_1(r)\Big]\nn\\
\leq{}&\frac{1}{4}\E\Big[\sup_{0\leq t_k\leq\lfloor T \rfloor }|z(t_{k})|^{2p}\Big]+K\Delta(T^{p-1}+1)\sum_{k=0}^{\lfloor T \rfloor}\mathbb E[|y(t_{k})|^{2p}]\nn\\
&\quad+K(T^{p-1}+1)(1+\|\xi\|^{2p}).
\end{align}
Similarly, for $p\ge 2,$ we obtain 
\begin{align}\label{I3p2.1}
I_3\leq{}&K\E\Big[\Big(\sum_{k=0}^{\lfloor T \rfloor}|z(t_{k})|^{4p-4}|y(t_{k})|^2|\sigma_k|^2\Delta\Big)^\frac{1}{2}\Big]\nn\\
\leq{}&\frac{1}{4}\E\Big[\sup_{0\leq t_k\leq\lfloor T \rfloor }|z(t_{k})|^{2p}\Big]+K\Delta^{\frac{p}{2}}\E\Big[\big(\sum_{k=0}^{\lfloor T \rfloor}|y(t_k)|^2|\sigma_k|^2\big)^{\frac{p}{2}}\Big]\nn\\
\leq{}&\frac{1}{4}\E\Big[\sup_{0\leq t_k\leq\lfloor T \rfloor }|z(t_{k})|^{2p}\Big]+K\Delta(T^{\frac{p}{2}-1}+1)\sum_{k=0}^{\lfloor T \rfloor}\E[|y(t_{k})|^{p}|\sigma_k|^p]\nn\\
\leq{}&\frac{1}{4}\E\Big[\sup_{0\leq t_k\leq\lfloor T \rfloor }|z(t_{k})|^{2p}\Big]+K\Delta(T^{\frac{p}{2}-1}+1)\sum_{k=0}^{\lfloor T \rfloor}\E[|y(t_{k})|^{2p}]\nn\\
&\quad+K\Delta(T^{\frac{p}{2}-1}+1)\sum_{k=0}^{\lfloor T \rfloor}\E[|\sigma_k|^{2p}]\nn\\
\leq{}&\frac{1}{4}\E\Big[\sup_{0\leq t_k\leq\lfloor T \rfloor }|z(t_{k})|^{2p}\Big]+K\Delta(T^{\frac{p}{2}-1}+1)\sum_{k=0}^{\lfloor T \rfloor}\E[|y(t_{k})|^{2p}]\nn\\
&\quad+K(T^{\frac{p}{2}-1}+1)(1+\|\xi\|^{2p}).
\end{align}
Plugging \eqref{I1p2.1}--\eqref{I3p2.1} into \eqref{11p2.1}, a direct computation derives
\begin{align*}
&\E\big[\sup_{0\leq t_k\leq T}|z(t_{k})|^{2p}\big]\nn\\
\leq{}& 2\Big(K+K\Delta\sum_{k=0}^{\lfloor T \rfloor}\E[|z(t_{k})|^{2p}]+K\Delta+K(T^{p-1}+1)(1+\|\xi\|^{2p})\nn\\
&\quad+K\Delta(T^{p-1}+1)\sum_{k=0}^{\lfloor T \rfloor}\E[|y(t_{k})|^{2p}]\Big)\nn\\
\leq{}&K(T^{p-1}+1)(1+\|\xi\|^{2p})+K\Delta(\frac{T}{\Delta}+1)\sup_{0\leq t_k\leq T}\E[|z(t_k)|^{2p}]\nn\\
&\quad+ K\Delta(T^{p-1}+1)(\frac{T}{\Delta}+1)\sup_{0\leq t_k\leq T}\E[|y(t_k)|^{2p}]\nn\\
\leq{}& K_{T}(1+\|\xi\|^{2p}),
\end{align*}
where the positive constant $K_T$  depends on $T$ and $p$, but is independent of $\Delta$.
 From the above inequality and \eqref{9p2.1}, one obtains that
 $\E\big[\sup_{0\leq t_k\leq T}|y(t_{k})|^{2p}\big]\leq K_T(1+\|\xi\|^{2p}).$ Thus the proof is completed. 
\end{proof}

\begin{prop}\label{p2.2}
Under Assumptions {\ref{a1}} and {\ref{a2}},  it holds that
\begin{align*}
\sup_{\Delta\in(0,1]}\sup_{k\geq 0}\E\big[\|z^{\xi,\Delta}_{t_k}\|^2
+\|y^{\xi,\Delta}_{t_k}\|^2\big]\leq K(1+\|\xi\|^2+|b(\xi)|^2\Delta^2).
\end{align*}
\end{prop}
\begin{proof}
We divide the proof into two steps.

\underline{Step 1.}
We shall prove
$$\sup_{\Delta\in(0,1]}\sup_{k\geq -N}(\E[|z(t_{k})|^2]+\E[|y(t_{k})|^2])\leq K(1+\|\xi\|^2+|b(\xi)|^2\Delta^2).$$
By $a_1>a_2+L$, one can choose a sufficiently small $\kappa>0$  such that
$$\tilde c_\kappa:=\big(a_1+a_2-(\frac{a_1-a_2+L}{2}+2a_2)e^{\kappa\tau}\big)>0.$$
For any $\Delta\in(0,1]$ and $k\in \mathrm N$, according to \eqref{12p2.1} and $z(t_k)=y(t_k)-\theta b_{k}\Delta$, we have
\begin{align*}
|z(t_{k+1})|^2
\leq{}&|z(t_k)|^2+\frac{1-2\theta}{\theta^2}|y(t_k)-z(t_k)|^2+|\sigma_k\delta W_k|^2
+K\Delta\nn\\
&-\frac{3a_1+a_2+L}{2}\Delta|y(t_k)|^2+2a_2\Delta\int_{-\tau}^{0}|y_{t_k}(r)|^2\mathrm d\nu_2(r)+\mathcal M_k\nn\\
={}&\frac{(1-\theta)^2}{\theta^2}|z(t_k)|^2+\frac{1-2\theta}{\theta^2}|y(t_k)|^2
+\frac{2(2\theta-1)}{\theta^2}\<z(t_k),y(t_k)\>+|\sigma_k\delta W_k|^2\nn\\
&+K\Delta-\frac{3a_1+a_2+L}{2}\Delta|y(t_k)|^2+2a_2\Delta\int_{-\tau}^{0}|y_{t_k}(r)|^2\mathrm d\nu_2(r)+\mathcal M_k,
\end{align*}
Applying the Young inequality yields
\begin{align*}
2\<z(t_k),y(t_k)\>
\leq \frac{1}{1+\tilde c_\kappa\Delta}|z(t_k)|^2+(1+\tilde c_\kappa\Delta)|y(t_k)|^2.
\end{align*}
Recall that  $\theta\in(0.5,1].$ Then we have 
\begin{align*}
|z(t_{k+1})|^2
\leq{}&\Big(\frac{(1-\theta)^2}{\theta^2}+\frac{2\theta-1}{\theta^2(1+\tilde c_\kappa\Delta)}\Big)|z(t_k)|^2+|\sigma_k\delta W_k|^2+K\Delta\nn\\
&-\Big(\frac{3a_1+a_2+L}{2}-\tilde c_{\kappa}\Big)\Delta|y(t_k)|^2+2a_2\Delta\int_{-\tau}^{0}|y_{t_k}(r)|^2\mathrm d\nu_2(r)+\mathcal M_k.
\end{align*}
Hence for any $\e\in(0,\kappa]$,
\begin{align}\label{1p2.2}
&e^{\e t_{k+1}}|z(t_{k+1})|^2=|z(0)|^2+\sum_{i=0}^{k}(e^{\e t_{i+1}}|z(t_{i+1})|^2-e^{\e t_{i}}|z(t_{i})|^2)\nn\\
\leq{}&|z(0)|^2+\Big(\frac{(1-\theta)^2}{\theta^2}\!+\!\frac{2\theta-1}{\theta^2(1\!+\!\tilde c_\kappa\Delta)}
-e^{-\e\Delta}\Big)\sum_{i=0}^{k}e^{\e t_{i+1}}|z(t_i)|^2\!+\!\sum_{i=0}^{k}e^{\e t_{i+1}}|\sigma_i\delta W_i|^2\nn\\
&\quad+K\Delta\sum_{i=0}^{k}e^{\e t_{i+1}}-(\frac{3a_1+a_2+L}{2}-\tilde c_{\kappa})\Delta\sum_{i=0}^{k}e^{\e t_{i+1}}|y(t_i)|^2\nn\\
&\quad+2a_2\Delta\sum_{i=0}^{k}e^{\e t_{i+1}}\int_{-\tau}^{0}|y_{t_i}(r)|^2\mathrm d\nu_2(r)+\sum_{i=0}^{k}e^{\e t_{i+1}}\mathcal M_i.
\end{align}
According to Proposition \ref{p2.1} we obtain that $\mathcal M_k$ is a martingale and
$$\E\Big[\sum_{i=0}^{k}e^{\e t_{i+1}}\mathcal M_i\Big]=0.$$
Taking expectation on both sides of  \eqref{1p2.2} and using Remark \ref{r1}, we arrive at
\begin{align}\label{2p2.2}
&e^{\e t_{k+1}}\E[|z(t_{k+1})|^2]\nn\\
\leq{}&|z(0)|^2\!+\!\big(\frac{(1\!-\!\theta)^2}{\theta^2}\!+\!\frac{2\theta\!-\!1}{\theta^2(1\!+\!\tilde c_\kappa\Delta)}
\!-\!e^{-\e\Delta}\big)\sum_{i=0}^{k}e^{\e t_{i+1}}\E[|z(t_i)|^2]\!+\!K\Delta\sum_{i=0}^{k}e^{\e t_{i+1}}\nn\\
&-(a_1+a_2-\tilde c_{\kappa})\Delta\sum_{i=0}^{k}e^{\e t_{i+1}}\E[|y(t_i)|^2]+\frac{a_1-a_2+L}{2}\Delta\sum_{i=0}^{k}e^{\e t_{i+1}}\times\nn\\
&\E\Big[\int_{-\tau}^{0}|y_{t_i}(r)|^2\mathrm d\nu_1(r)\Big]+2a_2\Delta\sum_{i=0}^{k}e^{\e t_{i+1}}\E\Big[\int_{-\tau}^{0}|y_{t_i}(r)|^2\mathrm d\nu_2(r)\Big].
\end{align}
Similar to \eqref{5p2.1}, one has
\begin{align*}
&\frac{a_1-a_2+L}{2}\Delta\sum_{i=0}^{k}e^{\e t_{i+1}}\int_{-\tau}^{0}|y_{t_i}(r)|^2\mathrm d\nu_1(r)
+2a_2\Delta\sum_{i=0}^{k}e^{\e t_{i+1}}\int_{-\tau}^{0}|y_{t_i}(r)|^2\mathrm d\nu_2(r)\nn\\
\leq{}&\big(\frac{a_1-a_2+L}{2}+2a_2\big)e^{\e\tau}\tau\|\xi\|^2+\big(\frac{a_1-a_2+L}{2}+2a_2\big)e^{\e\tau}\Delta\sum_{i=0}^{k}e^{\e t_{i+1}}|y(t_i)|^2.
\end{align*}
Inserting this into \eqref{2p2.2} yields
\begin{align*}
&e^{\e t_{k+1}}\E[|z(t_{k+1})|^2]\nn\\
\leq{}&|z(0)|^2+\Big(\frac{(1-\theta)^2}{\theta^2}+\frac{2\theta-1}{\theta^2(1+\tilde c_\kappa\Delta)}
-e^{-\e\Delta}\Big)\sum_{i=0}^{k}e^{\e t_{i+1}}\E[|z(t_i)|^2]+\frac{K}{\e}e^{\e t_{k+2}}\nn\\
&-\Big(a_1+a_2-\tilde c_{\kappa}-(\frac{a_1-a_2+L}{2}+2a_2)e^{\e\tau}\Big)\Delta\sum_{i=0}^{k}e^{\e t_{i+1}}\E[|y(t_i)|^2]\nn\\
&+(\frac{a_1-a_2+L}{2}+2a_2)e^{\e \tau}\tau\|\xi\|^2.
\end{align*}
It follows from the definition of $\tilde c_{\kappa}$ that
\begin{align*}
a_1+a_2-\tilde c_{\kappa}-(\frac{a_1-a_2+L}{2}+2a_2)e^{\e\tau}\geq 0.
\end{align*}
Moreover, using $e^{-\e\Delta}>1-\e\Delta$, we conclude  that
\begin{align}\label{varep}
&\frac{(1-\theta)^2}{\theta^2}+\frac{2\theta-1}{\theta^2(1+\tilde c_\kappa\Delta)}
-e^{-\e\Delta}\nn\\
< {}&\frac{(1-\theta)^2}{\theta^2}+\frac{2\theta-1}{\theta^2(1+\tilde c_\kappa\Delta)}
-(1-\e\Delta)\nn\\
={}&-\frac{(2\theta-1)\tilde c_{\kappa}}{\theta^2(1+\tilde c_\kappa\Delta)}\Delta+\e\Delta
< (\e-\frac{(2\theta-1)\tilde c_{\kappa}}{\theta^2(1+\tilde c_{\kappa})})\Delta.
\end{align}
Hence, letting
$\e=\frac{(2\theta-1)\tilde c_{\kappa}}{\theta^2(1+\tilde c_{\kappa})}\wedge \kappa$
yields
\begin{align*}
e^{\e t_{k+1}}\E[|z(t_{k+1})|^2]
\leq{} &K(\|\xi\|^2+|b(\xi)|^2\Delta^2)+K e^{\e t_{k+2}},
\end{align*}
which  implies
$$\sup_{\Delta\in(0,1]}\sup_{k\geq -N}\E[|z(t_{k})|^2]\leq K(1+\|\xi\|^2+|b(\xi)|^2\Delta^2).$$
In addition, by \eqref{9p2.1}, we obtain
\begin{align}\label{+6p2.2}
&(1+\theta\Delta\frac{3a_1+a_2+L}{2})\sup_{\Delta\in(0,1]}\sup_{k\geq -N}\E[|y(t_k)|^2]\nn\\
\leq{}& K\Delta+\sup_{\Delta\in(0,1]}\sup_{k\geq -N}\E[|z(t_k)|^{2}]+2a_2\theta\Delta\sup_{\Delta\in(0,1]}\sup_{k\geq -N}\E[|y(t_{k})|^2].
\end{align}
Together with $a_1>a_2+L$, we arrive at
$$\sup_{\Delta\in(0,1]}\sup_{k\geq -N}\E[|y(t_k)|^2]\leq K(1+\|\xi\|^2+|b(\xi)|^2\Delta^2).$$
\underline{Step 2}.
It follows from \eqref{1p2.2} that 
\begin{align}\label{3p2.2}
&\E\Big[\sup_{(k-N)\vee0\leq i\leq k}e^{\e t_{i+1}}|z(t_{i+1})|^2\Big]\nn\\
\leq{}&|z(0)|^2+K\|\xi\|^2+\E\Big[\sup_{(k-N)\vee0\leq i\leq k}\sum_{l=0}^{i}e^{\e t_{l+1}}|\sigma_l\delta W_l|^2\Big]+K\Delta\sum_{i=0}^{k}e^{\e t_{i+1}}\nn\\
&\quad+\E\Big[\sup_{(k-N)\vee0\leq i\leq k}\sum_{l=0}^{i}e^{\e t_{l+1}}\mathcal M_l\Big]\nn\\
\leq{}&|z(0)|^2+K\|\xi\|^2+\Delta\E\Big[\sum_{l=0}^{k}e^{\e t_{l+1}}|\sigma_l|^2\Big]+\frac{K}{\e}e^{\e t_{k+2}}\nn\\
&\quad +\E\Big[\sup_{(k-N)\vee0\leq i\leq k}\sum_{l=(k-N)\vee0}^{i}e^{\e t_{l+1}}\mathcal M_l\Big].
\end{align}
Using Assumption \ref{a1} and  \underline{Step 1}, we obtain
\begin{align}\label{4p2.2}
\Delta\E\Big[\sum_{l=0}^{k}e^{\e t_{l+1}}|\sigma_l|^2\Big]\leq{} & K\Big(\Delta\sum_{l=0}^{k}e^{\e t_{l+1}}\E\Big[|y(t_l)|^2+\int_{-\tau}^{0}|y_{t_l}(r)|^2\mathrm d\nu_1(r)\Big]+ e^{\e t_{k+2}}\Big)\nn\\
\leq{}&K\Big(\Delta\sup_{i\geq-N}\E[|y(t_i)|^2]\sum_{l=0}^{k}e^{\e t_{l+1}}+ e^{\e t_{k+2}}\Big)\nn\\
\leq{}& K(1+\|\xi\|^2+|b(\xi)|^2\Delta^2)e^{\e t_{k+2}}.
\end{align}
Similar to the estimates of $I_2$  and $I_3$ in  Proposition \ref{p2.1}, one has
\begin{align}\label{5p2.2}
&\E\Big[\sup_{(k-N)\vee0\leq i\leq k}\sum_{l=(k-N)\vee 0}^{i}e^{\e t_{l+1}}\mathcal M_l\Big]\nn\\
\leq{}&\E\Big[\Big(\big(\sup_{(k-N)\vee0\leq i\leq k}e^{\e t_{i}}|z(t_{i})|^2\big)\big( Ke^{\e\Delta}\Delta\sum_{l=(k-N)\vee0}^{k}e^{\e t_{l+1}}|\sigma_l|^2\big)\Big)^{\frac{1}{2}}\Big]\nn\\
&\quad+\E\Big[\Big(\big(\sup_{(k-N)\vee0\leq i\leq k}e^{\e t_{i}}|y(t_{i})|^2\big)\big( Ke^{\e\Delta}\Delta\sum_{l=(k-N)\vee0}^{k}e^{\e t_{l+1}}|\sigma_l|^2\big)\Big)^{\frac{1}{2}}\Big]\nn\\
\leq{}&\frac{1}{4}\E\Big[\sup_{(k-N)\vee0\leq i\leq k}e^{\e t_{i}}|z(t_{i})|^2\Big]+\frac{1}{4}\E\Big[\sup_{(k-N)\vee0\leq i\leq k}e^{\e t_{i}}|y(t_{i})|^2\Big]\nn
\\
&\quad+K\Delta\E\Big[\sum_{i=(k-N)\vee0}^{k}e^{\e t_{i+1}}|\sigma_i|^2\Big]\nn\\
\leq{}&   \frac{1}{4}\E\Big[\sup_{(k-N)\vee0\leq i\leq k}e^{\e t_{i}}|z(t_{i})|^2\Big]+\frac{1}{4}\E\Big[\sup_{(k-N)\vee0\leq i\leq k}e^{\e t_{i}}|y(t_{i})|^2\Big]\nn\\
&\quad+K(1+\|\xi\|^2+|b(\xi)|^2\Delta^2)e^{\e t_{k+2}}.
\end{align}
Note that
\begin{align*}
&(1+\theta\Delta\frac{3a_1+a_2+L}{2})\E\Big[\sup_{(k-N)\vee0\leq i\leq k}e^{\e t_{i}}|y(t_{i})|^2\Big]\nn\\
\leq{}&K\Delta e^{\e t_{k}}\!+\!\E\Big[\sup_{(k-N)\vee0\leq i\leq k}e^{\e t_{i}}|z(t_{i})|^2\Big]\!+\!2a_2\theta\Delta \E \Big[\sup_{(k-2N)\vee (-N)\leq i\leq k}e^{\e t_{i}}|y(t_{i})|^2\Big]\nn\\
\leq{}&K\Delta e^{\e t_{k}}+\E\Big[\sup_{(k-N)\vee0\leq i\leq k}e^{\e t_{i}}|z(t_{i})|^2\Big]+2a_2\theta\Delta e^{\e t_{k}}\sum_{i=(k-2N)\vee (-N)}^{k}\E [|y(t_{i})|^2]\nn\\
\leq{}& K\Delta e^{\e t_{k}}+\E\Big[\sup_{(k-N)\vee0\leq i\leq k}e^{\e t_{i}}|z(t_{i})|^2\Big]+4a_2\theta\tau e^{\e t_{k}}\sup_{i\geq-N}\E [|y(t_{i})|^2]\nn\\
\leq{}& K\Delta e^{\e t_{k}}+\E\Big[\sup_{(k-N)\vee0\leq i\leq k}e^{\e t_{i}}|z(t_{i})|^2\Big]+e^{\e t_{k}}K(1+\|\xi\|^2+|b(\xi)|^2\Delta^2),
\end{align*}
which implies
\begin{align}\label{6p2.2}
&\E\Big[\sup_{(k-N)\vee0\leq i\leq k}e^{\e t_{i}}|y(t_{i})|^2\Big]\nn\\
\leq{}& \E\Big[\sup_{(k-N)\vee0\leq i\leq k}e^{\e t_{i}}|z(t_{i})|^2\Big]+e^{\e t_{k}}K(1+\|\xi\|^2+|b(\xi)|^2\Delta^2).
\end{align}
Inserting \eqref{6p2.2} into \eqref{5p2.2} leads to
\begin{align}\label{7p2.2}
&\E\Big[\sup_{(k-N)\vee0\leq i\leq k}\sum_{l=(k-N)\vee 0}^{i}e^{\e t_{l+1}}\mathcal M_l\Big]\nn\\
\leq{}&\frac{1}{2}\E\Big[\sup_{(k-N)\vee0\leq i\leq k}e^{\e t_{i}}|z(t_{i})|^2\Big]+K(1+\|\xi\|^2+|b(\xi)|^2\Delta^2)e^{\e t_{k+2}}\nn\\
\leq{}&\frac{1}{2}\E\Big[\sup_{(k-N)\vee0\leq i\leq k}e^{\e t_{i+1}}|z(t_{i+1})|^2\Big]+K(1+\|\xi\|^2+|b(\xi)|^2\Delta^2)e^{\e t_{k+2}}\nn\\
&\quad+|z(0)|^2+\E[e^{\e t_{k-N}}|z(t_{k-N})|^2]\nn\\
\leq{}&\frac{1}{2}\E\Big[\sup_{(k-N)\vee0\leq i\leq k}e^{\e t_{i+1}}|z(t_{i+1})|^2\Big]+K(1+\|\xi\|^2+|b(\xi)|^2\Delta^2)e^{\e t_{k+2}}.
\end{align}
Substituting \eqref{4p2.2} and \eqref{7p2.2} into \eqref{3p2.2} we deduce
\begin{align*}
e^{\e t_{k-N}}\E\Big[\sup_{(k-N)\vee0\leq i\leq k}|z(t_{i+1})|^2\Big]\leq{} &\E\Big[\sup_{(k-N)\vee0\leq i\leq k}e^{\e t_{i+1}}|z(t_{i+1})|^2\Big]\nn\\
\leq{} &K(1+\|\xi\|^2+|b(\xi)|^2\Delta^2)e^{\e t_{k+2}},
\end{align*}
which implies
\begin{align*}
\sup_{\Delta\in(0,1]}\sup_{k\geq 0}\E\Big[\sup_{(k-N)\vee 0\leq i\leq k}|z(t_{i+1})|^2\Big]\leq K(1+\|\xi\|^2+|b(\xi)|^2\Delta^2).
\end{align*} 
Combining \eqref{6p2.2}, we finishes the proof.
\end{proof}

\subsection{Time-independent boundedness of high-order moment}
For the SFDE \eqref{FF} with superlinearly growing coefficients, 
the time-independent boundedness of high-order moment of the numerical solution is required  in the longtime mean-square convergence analysis of numerical method.
We introduce  the following assumptions on the  coefficient $b$ and the initial value $\xi.$ 

\begin{assp}\label{a4}
 There exist positive constants $K,a_3, a_4, \ell$, and a function  $\nu_2\in\mathcal{P}$ such that for any $\phi\in \mathcal C^d$,
\begin{align*}
\langle\phi(0),b(\phi)\rangle
\leq K-a_3|\phi(0)|^{2+\ell}+a_4\int_{-\tau}^{0}|\phi(r)|^2\mathrm d\nu_2(r).
\end{align*}
\end{assp}

\begin{assp}\label{a5}
There exist  positive constants $K$ and $\beta$ such that for any $\phi_1,\phi_2\in \mathcal C^d$,
\begin{align*}
|b(\phi_1)-b(\phi_2)|\leq K\|\phi_1-\phi_2\|(1+\|\phi_1\|^{\beta}+\|\phi_2\|^{\beta}).
\end{align*}
\end{assp}

\begin{assp}\label{a7}
There exist  positive constants $K$ and $\rho\geq1/2$ such that
\begin{align*}
|\xi(s_1)-\xi(s_2)|\leq K|s_1-s_2|^{\rho}.
\end{align*}
\end{assp}
 The following lemma shows that the moments of functional solution of \eqref{FF} are  bounded uniformly for $t\ge 0$,
whose proof is similar to that of \cite[Theorem 3]{LMS11} and is omitted.

 \begin{lemma}\label{l4.1}
Under Assumptions \ref{a1}  and \ref{a4}, the following estimate holds: for integer  $p\geq 2$, 
\begin{align}\label{l4.1.1}
\sup_{t\geq 0}\E[\|x^{\xi}(t)\|^p] \leq K(1+\|\xi\|^p).
\end{align}
Moreover, the invariant  measure $\mu$ given in Lemma \ref{l2.4} satisfies 
\begin{align}\label{l4.1.2}
\mu(\|\cdot\|^p):=\int_{\mathcal C^d}\|\phi\|^{p}\mathrm d\mu(\phi)\leq K.
\end{align}
\end{lemma}
\begin{lemma}\label{l4.2}
Under Assumptions \ref{a1}, \ref{a4}, and  \ref{a5}, for integer $p\geq 2$, 
it holds that
$$\sup_{\Delta\in(0,1]}\sup_{k\geq -N}(\E[|z^{\xi,\Delta}(t_k)|^p]+\E[|y^{\xi,\Delta}(t_k)|^p])\leq K(1+\|\xi\|^{p(\beta+1)}).$$
\end{lemma}
\begin{proof} We divide the proof into three steps.

\underline{Step 1.}
We shall  show that there exists  $\Delta^{*}\in(0,1)$ such that 
$$\sup_{\Delta\in(0,\Delta^*]}\sup_{k\geq 0}\E[|z^{\xi,\Delta}(t_k)|^p]\leq K(1+\|\xi\|^{p(\beta+1)}) \quad\forall\, p\geq2. $$
In fact, by Assumption \ref{a4}, we obtain that for any $\Delta\in(0,1]$ and $k\in\mathbb N$,
\begin{align}\label{l4.2.16+}
|z(t_{k+1})|^2\leq &\;A_{\theta,\Delta}|z(t_k)|^2+2K\Delta-2a_3\Delta|y(t_k)|^{2+\ell}+2a_4\Delta\int_{-\tau}^{0}|y_{t_k}(r)|^{2}\mathrm d\nu_2(r)\nn\\
&\quad+\Delta|y(t_k)|^{2}+|\sigma_{k}\Delta W_{k}|^2+\mathcal M_{k}\nn\\
\leq&\;A_{\theta,\Delta}|z(t_k)|^2+R_1\Delta-a_3\Delta|y(t_k)|^{2+\ell}+2a_4\Delta\int_{-\tau}^{0}|y_{t_k}(r)|^{2}\mathrm d\nu_2(r)\nn\\
&-R_2\Delta|y(t_k)|^{2}+|\sigma_{k}\delta W_{k}|^2+\mathcal M_{k},
\end{align}
where $A_{\theta,\Delta}:=\frac{(1-\theta)^2}{\theta^2}+\frac{2\theta-1}{\theta^2(1+\Delta)}$, and
$R_1,R_2$ are positive constants with $R_2>2L$.
For any integer $p\geq2$,
\begin{align}\label{l4.2.3}
&\E\Big[\big(|z(t_{k+1})|^2+a_3\Delta|y(t_k)|^{2+\ell}\big)^p\Big]
\leq A^p_{\theta,\Delta}\E|z(t_k)|^{2p}+\sum_{i=1}^{p}C_{p}^{i}\E [I_{i}],
\end{align}
where
\begin{align*}
I_i:=&\,A^{p-i}_{\theta,\Delta}|z(t_k)|^{2(p-i)}\Big(R_1\Delta-R_2\Delta|y(t_k)|^{2}+2a_4\Delta\int_{-\tau}^{0}|y_{t_k}(r)|^{2}\mathrm d\nu_2(r)\\
&+|\sigma_{k}\delta W_{k}|^2
+\mathcal M_{k}\Big)^{i}.
\end{align*}
It follows from Assumption \ref{a1} and the Young inequality that for any $\e_1\in(0,1)$,
\begin{align}\label{l4.2.7}
\E[ I_1]
\leq&\;\E\Big[A^{p-1}_{\theta,\Delta}|z(t_k)|^{2(p-1)}\big(R_1\Delta-(R_2-2L)\Delta|y(t_k)|^{2}\nn\\
&\quad+2a_4\Delta\int_{-\tau}^{0}|y_{t_k}(r)|^{2}\mathrm d\nu_2(r)+2L\Delta\int_{-\tau}^{0}|y_{t_k}(r)|^2\mathrm d\nu_1(r)\big)\Big]\nn\\
\leq&\;3\e_1\Delta A^{p}_{\theta,\Delta}\E[|z(t_k)|^{2p}]\!+\!R(\e_1)\Delta\!-\!(R_2\!-\!2L)\Delta\E\big[A^{p-1}_{\theta,\Delta}|z(t_k)|^{2(p-1)}|y(t_k)|^{2}\big]\nn\\
&\quad+\frac{(2a_4)^{p}\Delta}{\e_1^{p-1}}\E\Big[\int_{-\tau}^{0}|y_{t_k}(r)|^{2p}\mathrm d\nu_2(r)\Big]+\frac{(2L)^p\Delta}{\e_1^{p-1}}\E\Big[\int_{-\tau}^{0}|y_{t_k}(r)|^{2p}\mathrm d\nu_1(r)\Big].
\end{align}
By \eqref{ST} and Assumption \ref{a4}, one has
\begin{align}\label{l4.2.15+}
|y(t_k)|^2\leq& \,|z(t_k)|^2+2K\theta\Delta-2a_3\theta\Delta|y(t_k)|^{2+\ell}+2a_4\theta\Delta\int_{-\tau}^{0}|y_{t_k}(r)|^2\mathrm d\nu_2(r)\nn\\
\leq& \,|z(t_k)|^2+2K\theta\Delta+2a_4\theta\Delta\int_{-\tau}^{0}|y_{t_k}(r)|^2\mathrm d\nu_2(r).
\end{align}
It is straightforward to see that
\begin{align*}
|y(t_k)|^{2p}
\leq &\,\Big(|z(t_k)|^2+2K\theta\Delta+2a_4\theta\Delta\int_{-\tau}^{0}|y_{t_k}(r)|^2\mathrm d\nu_2(r)\Big)^{p-1}|y(t_k)|^{2}\nn\\
\leq&\,3^{p-1}|z(t_k)|^{2(p-1)}|y(t_k)|^{2}+(6K\theta\Delta)^{p-1}|y(t_k)|^{2}\nn\\
&\,+(6a_4\theta\Delta)^{p-1}|y(t_k)|^{2}\int_{-\tau}^{0}|y_{t_k}(r)|^{2(p-1)}\mathrm d\nu_2(r).
\end{align*}
This implies
 \begin{align}\label{l4.2.2}
-|z(t_k)|^{2(p-1)}|y(t_k)|^{2}
\leq&-\frac{1}{3^{p-1}}|y(t_k)|^{2p}
+(2K\theta\Delta)^{p-1}|y(t_k)|^2\nn\\
&\quad+(2a_4\theta\Delta)^{p-1}|y(t_k)|^2\int_{-\tau}^{0}|y_{t_k}(r)|^{2(p-1)}\mathrm d\nu_2(r)\nn\\
\leq& -\frac{1}{3^{p-1}}|y(t_k)|^{2p}
+(2K\Delta)^{p-1}+(2\Delta)^{p-1}(K^{p-1}+a_4^{p-1})|y(t_k)|^{2p}\nn\\
&\quad+(2a_4\Delta)^{p-1}\int_{-\tau}^{0}|y_{t_k}(r)|^{2p}\mathrm d\nu_2(r).
\end{align}
Inserting \eqref{l4.2.2} into \eqref{l4.2.7} and using $A_{\theta,\Delta}\in(0,1)$ yield
\begin{align}\label{l4.2.9}
\E [I_1]
\leq &\,3\e_1\Delta A^{p}_{\theta,\Delta}\E[|z(t_k)|^{2p}]+R(\e_1)\Delta-\frac{(R_2-2L)A^{p-1}_{\theta,\Delta}}{3^{p-1}}\Delta\E[|y(t_k)|^{2p}]\nn\\
&\,+(R_2-2L) 2^{p-1}(K^{p-1}+a_4^{p-1})\Delta^{p}\E[|y(t_k)|^{2p}]\nn\\
&\,+(R_2-2L)(2a_4)^{p-1}\Delta^p\E\Big[\int_{-\tau}^{0}|y_{t_k}(r)|^{2p}\mathrm d\nu_2(r)\Big]\nn\\
&\,+\frac{(2a_4)^p}{\e_1^{p-1}}\Delta\E\Big[\int_{-\tau}^{0}|y_{t_k}(r)|^{2p}\mathrm d\nu_2(r)\Big]
+\frac{(2L)^p\Delta}{\e_1^{p-1}}\E\Big[\int_{-\tau}^{0}|y_{t_k}(r)|^{2p}\mathrm d\nu_1(r)\Big].
\end{align}
For the term $I_i$ with $i\in\{2,\ldots,p\}$,
\begin{align*}
&\quad\E [I_i]
\leq\,6^{i}A^{p-i}_{\theta,\Delta}\E\Big[|z(t_k)|^{2(p-i)}\Big(R_1\Delta^i+R_2\Delta^i|y(t_k)|^{2i}+(2i-1)!!|\sigma_{k}|^{2i}\Delta^{i}\nn\\
&\,+(2a_4)^{i}\Delta^{i}\int_{-\tau}^{0}|y_{t_k}(r)|^{2i}\mathrm d\nu_2(r)+2^{i}(i-1)!!A^{\frac{i}{2}}_{\theta,\Delta} |z(t_k)|^i |\sigma_k|^i\Delta^{\frac{i}{2}}\nn\\
&\,+4^{i}(i-1)!!| y(t_k)|^i |\sigma_k|^i|\Delta^{\frac{i}{2}}\Big)\Big]\nn\\
\leq&\,6\e_1\Delta A^{p}_{\theta,\Delta}\E[|z(t_k)|^{2p}]+R(\e_1)\Delta+\frac{6^piR_2^{\frac{p}{i}}\Delta^2}{p\e_1^{p-i}}\E[|y(t_k)|^{2p}]\\
&\,+\frac{i(12a_4)^p\Delta}{p\e_1^{p-i}}\E\Big[\int_{-\tau}^{0}|y_{t_k}(r)|^{2p}\mathrm d\nu_2(r)\Big]+\frac{6^p((2i-1)!!)^{\frac{p}{i}}i\Delta}{p\e_1^{p-i}}\E[|\sigma_k|^{2p}]\nn\\
&\,+\frac{12^{2p}((i-1)!!)^{\frac{2p}{i}}i\Delta}{2p\e_1^{2p-i}}\E[|\sigma_k|^{p}|z(t_k)|^p]+\frac{24^p((i-1)!!)^{\frac{p}{i}}i\Delta}{p\e_1^{p-i}}\E[|y(t_k)|^p|\sigma_k|^{p}],
\end{align*}
where $i!!$ represents the double factorial of $i$.
Together with the Young inequality and Assumption \ref{a1}, we obtain
\begin{align}\label{l4.2.8}
\E [I_i]
\leq&\,6\e_1\Delta A^{p}_{\theta,\Delta}\E[|z(t_k)|^{2p}]+R(\e_1)\Delta+\frac{6^pR_2^{\frac{p}{i}}\Delta^2}{\e_1^{p-i}}\E[|y(t_k)|^{2p}]\nn\\
&\,+\frac{(12a_4)^p\Delta}{\e_1^{p-i}}\E\Big[\int_{-\tau}^{0}|y_{t_k}(r)|^{2p}\mathrm d\nu_2(r)\Big]+\frac{12^{2p}((2i-1)!!)^{\frac{2p}{i}}\Delta}{\e_1^{2p-i}}\E[|\sigma_k|^{2p}]\nn\\
&\,+\frac{24^p((i-1)!!)^{\frac{p}{i}}\Delta}{\e_1^{p-i}}\E\big[|y(t_k)|^{2p}+|\sigma_k|^{2p}\big]\nn\\
\leq&\,6\e_1\Delta A^{p}_{\theta,\Delta}\E[|z(t_k)|^{2p}]+R(\e_1)\Delta+\frac{6^pR_2^{\frac{p}{i}}\Delta^2}{\e_1^{p-i}}\E[|y(t_k)|^{2p}]\nn\\
&\,+\frac{(12a_4)^p\Delta}{\e_1^{p-i}}\E\Big[\int_{-\tau}^{0}|y_{t_k}(r)|^{2p}\mathrm d\nu_2(r)\Big]\nn\\
&\,+B_{i}\Delta\E[|y(t_k)|^{2p}]+B_i\Delta\E\Big[\int_{-\tau}^{0}|y_{t_k}(r)|^{2p}\mathrm d\nu_1(r)\Big],
\end{align}
where
$$B_i:=\frac{(24L)^{2p}((2i-1)!!)^{\frac{2p}{i}}}{\e_1^{2p-i}}+\frac{24^p(1+(2L)^{2p})((i-1)!!)^{\frac{p}{i}}}{\e_1^{p-i}}.$$
Plugging   \eqref{l4.2.8} and \eqref{l4.2.9}  into \eqref{l4.2.3} leads to
\begin{align*}
&\E[|z(t_{k+1})|^{2p}]\nn\\
\leq&\,\Big(1+3\e_1p\Delta+6\e_1\Delta\sum_{i=2}^{p}C_{p}^{i}\Big)A^{p}_{\theta,\Delta}\E[|z(t_k)|^{2p}]+R(\e_1)\Delta+J_{k,1}(R_2)+J_{k,2}(R_2),
\end{align*}
where
\begin{align*}
J_{k,1}(R_2)
&=-\frac{p(R_2-2L)A^{p-1}_{\theta,\Delta}}{3^{p-1}}\Delta\E[|y(t_k)|^{2p}]+\frac{p(2a_4)^p}{\e_1^{p-1}}\Delta\E\Big[\int_{-\tau}^{0}|y_{t_k}(r)|^{2p}\mathrm d\nu_2(r)\Big]\nn\\
&\quad+\frac{p(2L)^p\Delta}{\e_1^{p-1}}\E\Big[\!\int_{-\tau}^{0}|y_{t_k}(r)|^{2p}\mathrm d\nu_1(r)\Big]+\sum_{i=2}^{p}C_{p}^{i}\Delta \Big(B_{i}\E[|y(t_k)|^{2p}] \nn\\
&\quad+\frac{(12a_4)^p}{\e_1^{p-i}}\E\Big[\!\int_{-\tau}^{0}|y_{t_k}(r)|^{2p}\mathrm d\nu_2(r)\Big]+B_i\E\Big[\int_{-\tau}^{0}|y_{t_k}(r)|^{2p}\mathrm d\nu_1(r)\Big]\Big),\nn\\
J_{k,2}(R_2)
&=\;p2^{p-1}(R_2-2L)(K^{p-1}\!+\!a_4^{p-1})\Delta^{p}\E[|y(t_k)|^{2p}]+p(R_2\!-\!2L)(2a_4)^{p-1}\Delta^p\nn\\
&\quad\times\E\Big[\int_{-\tau}^{0}|y_{t_k}(r)|^{2p}\mathrm d\nu_2(r)\Big]
+\sum_{i=2}^{p}C_{p}^{i}\frac{6^pR_2^{\frac{p}{i}}\Delta^2}{\e_1^{p-i}}\E[|y(t_k)|^{2p}].
\end{align*}
It is clear that for any $\e>0$,
\begin{align*}
&e^{\e t_{k+1}}\E[|z(t_{k+1})|^{2p}]-e^{\e t_{k}}\E[|z(t_{k})|^{2p}]\nn\\
\leq&\,\big((1+2^p6\e_1\Delta)A^{p}_{\theta,\Delta}e^{\e \Delta}-1\big)e^{\e t_k}\E[|z(t_k)|^{2p}]+R(\e_1)e^{\e t_{k+1}}\Delta\nn\\
&\quad +e^{\e t_{k+1}}J_{k,1}(R_2)+e^{\e t_{k+1}}J_{k,2}(R_2),
\end{align*}
which leads to
\begin{align}\label{l4.2.9+}
e^{\e t_{k+1}}\E[|z(t_{k+1})|^{2p}]
\leq&\,|z(0)|^{2p}+\sum_{l=0}^{k}\big((1+2^p6\e_1\Delta)A^{p}_{\theta,\Delta}e^{\e \Delta}-1\big)e^{\e t_l}\E[|(z(t_l)|^{2p}]\nn\\
&\,+\sum_{l=0}^{k}\big(R(\e_1)e^{\e t_{l+1}}\Delta+e^{\e t_{l+1}}J_{l,1}(R_2)+e^{\e t_{l+1}}J_{l,2}(R_2)\big).
\end{align} 
Note that
\begin{align*}
(1+2^p\cdot6\e_1\Delta)A^{p}_{\theta,\Delta}e^{\e \Delta}=&\;(1+2^p\cdot6\e_1\Delta)\Big(\frac{(1-\theta)^2}{\theta^2}+\frac{2\theta-1}{\theta^2(1+\Delta)}\Big)e^{\e \Delta}\nn\\
=&\;(1+2^p\cdot6\e_1\Delta)(1-\frac{(2\theta-1)\Delta}{\theta^2(1+\Delta)})e^{\e \Delta}\nn\\
\leq&\;(1-\frac{(2\theta-1)\Delta}{2\theta^2})e^{(2^p\cdot6\e_1+\e) \Delta}\nn\\
\leq &\;e^{(-\frac{(2\theta-1)}{2\theta^2}+2^p\cdot6\e_1+\e) \Delta}.
\end{align*}
Letting
$$2^p\cdot6\e_1=\e=\frac{(2\theta-1)}{4\theta^2},$$
we derive
\begin{align}\label{l4.2.16}
(1+2^p\cdot6\e_1\Delta)A^{p}_{\theta,\Delta}e^{\e \Delta}\leq1.
\end{align}
This, along with \eqref{l4.2.9+} implies that
\begin{align}\label{l4.2.12}
 &\quad e^{\e t_{k+1}}\E[|z(t_{k+1})|^{2p}]\nn\\
&\leq|z(0)|^{2p}+\sum_{l=0}^{k}\big(R(\e_1)e^{\e t_{l+1}}\Delta+e^{\e t_{l+1}}J_{l,1}(R_2)+e^{\e t_{l+1}}J_{l,2}(R_2)\big).
\end{align} 
By the definition of $A_{\theta,\Delta}$ with $\theta\in(0.5,1]$, we have
$$0<A_{\theta}:=\frac{(1-\theta)^2}{\theta^2}+\frac{2\theta-1}{2\theta^2}<A_{\theta,\Delta}<1.$$
Similar to \eqref{5p2.1}, we deduce
\begin{align*}
\sum_{l=0}^{k}e^{\e t_{l+1}}J_{l,1}(R_2)
\leq&-\!\Big(\frac{p(R_2\!-\!2L)A^{p-1}_{\theta}}{3^{p-1}}\!-\!e^{\e\tau}\tilde B\!-\!\sum_{i=2}^{p}C_{p}^{i}B_i\Big)\Delta\sum_{l=0}^{k} e^{\e t_{l+1}}\E[|y(t_l)|^{2p}]\nn\\
&\quad+e^{\e\tau}\tilde B\tau\|\xi\|^{2p},
\end{align*}
where
$$\tilde B:=\frac{p(2a_4)^p}{\e_1^{p-1}}+\frac{p(2L)^p}{\e_1^{p-1}}
+\sum_{i=2}^{p}C_{p}^{i}\Big(\frac{(12a_4)^p}{\e_1^{p-i}}+B_i\Big).$$
Choose a sufficiently large number $R_2^*>0$  such that
$$\frac{p(R_2^*-2L)A^{p-1}_{\theta}}{2\cdot3^{p-1}}-e^{\e\tau}\tilde B-\sum_{i=2}^{p}C_{p}^{i}B_i>0.$$
This leads to
\begin{align}\label{l4.2.10}
&\quad \sum_{l=0}^{k}e^{\e t_{l+1}}J_{l,1}(R_2^*)\nn\\
&\leq -\frac{p(R_2^*-2L)A^{p-1}_{\theta}}{2\cdot3^{p-1}}\Delta\sum_{l=0}^{k}e^{\e t_{l+1}}\E[|y(t_l)|^{2p}]+e^{\e\tau}\tilde B\tau\|\xi\|^{2p}.
\end{align}
Similarly,
\begin{align*}
\sum_{l=0}^{k}e^{\e t_{l+1}}J_{l,2}(R_2^*)
\leq D_{\Delta}\Delta\sum_{l=0}^{k}e^{\e t_{l+1}}\E[|y(t_l)|^{2p}]
+e^{\e\tau}\tau p(R_2^*-2L)(2a_4)^{p-1}\|\xi\|^{2p},
\end{align*}
where
\begin{align*}  
D_{\Delta}:=&~ p2^{p-1}(R_2^*-2L)(K^{p-1}+a_4^{p-1})\Delta^{p-1}
+e^{\e\tau}p(R_2^*-2L)(2a_4)^{p-1}\Delta^{p-1}\nn\\
&\quad+\sum_{i=2}^{p}C_{p}^{i}\frac{6^p(R_2^*)^{\frac{p}{i}}\Delta}{\e_1^{p-i}}.
\end{align*}
Choose a sufficiently small number $\Delta^{*}\in(0,1)$   such that
\begin{align*}
D_{\Delta^*}
\leq \frac{p(R_2^*-2L)A^{p-1}_{\theta}}{2\cdot3^{p-1}}.
\end{align*}
Together with \eqref{l4.2.10}, we obtain that for any $\Delta\in(0,\Delta^*]$,
\begin{align}\label{l4.2.11}
\sum_{l=0}^{k}e^{\e t_{l+1}}(J_{l,1}(R_2^*)+J_{l,2}(R_2^*))
\leq e^{\e\tau}\tau\big(\tilde B+p(R_2^*-2L)(2a_4)^{p-1}\big)\|\xi\|^{2p}.
\end{align}
Inserting \eqref{l4.2.11} into \eqref{l4.2.12} and using Assumption \ref{a5} yield
\begin{align*}
&\quad e^{\e t_{k+1}}\E[|z(t_{k+1})|^{2p}]\nn\\
&\leq|\xi-\theta b({\xi})\Delta|^{2p}+\sum_{l=0}^{k}R(\e_1)e^{\e t_{l+1}}\Delta+e^{\e\tau}\tau\big(\tilde B+p(R_2^*-2L)(2a_4)^{p-1}\big)\|\xi\|^{2p}\nn\\
&\leq K(1+\|\xi\|^{2p(\beta+1)})+\frac{R(\e_1)e^{\e t_{k+2}}}{\e}+e^{\e\tau}\tau\big(\tilde B+p(R_2^*-2L)(2a_4)^{p-1}\big)\|\xi\|^{2p},
\end{align*}
which finishes the proof of \underline{step 1}. 

\underline{Step 2.}
 We  show that for $p\ge 2,$
 $$\sup_{\Delta\in[\Delta^*,1]}\sup_{k\geq 0}\E[|z^{\xi,\Delta}(t_k)|^p]\leq K(1+\|\xi\|^{p(\beta+1)}).$$
Similar to \underline{Step 1}, we have that for any $\kappa_{1}\in(0,1)$, 
\begin{align}\label{l4.2.4}
\E [I_1]
\leq~& 3\kappa_{1}A^{p}_{\theta,\Delta}\E[|z(t_k)|^{2p}]+R(\kappa_{1})\Delta-\frac{(R_2-2L)A^{p-1}_{\theta,\Delta}}{3^{p-1}}\Delta\E[|y(t_k)|^{2p}]\nn\\
&+(R_2-2L) 2^{p-1}(K^{p-1}\!+\!a_4^{p-1})\Delta^{p}\E[|y(t_k)|^{2p}]\nn\\
&+\Big((R_2-2L)(2a_4)^{p-1}+\frac{(2a_4)^p}{\kappa_{1}^{p-1}}\Big)\Delta^{p}\E\Big[\int_{-\tau}^{0}|y_{t_k}(r)|^{2p}\mathrm d\nu_2(r)\Big]
\nn\\
&+\frac{(2L\Delta)^p}{\kappa_{1}^{p-1}}\E\Big[\int_{-\tau}^{0}|y_{t_k}(r)|^{2p}\mathrm d\nu_1(r)\Big],
\end{align}
and
\begin{align}\label{l4.2.5}
\E [I_i]
\leq&~\,6\kappa_{1}A^{p}_{\theta,\Delta}\E[|z(t_k)|^{2p}]+R(\kappa_{1})\Delta+R(\kappa_{1})\Delta^p\E[|y(t_k)|^{2p}]\nn\\
&\,+\frac{(12a_4\Delta)^p}{\kappa_{1}^{p-i}}\E\Big[\int_{-\tau}^{0}|y_{t_k}(r)|^{2p}\mathrm d\nu_2(r)\Big]\nn\\
&\,+\frac{(24L)^{2p}\Delta^p}{\kappa_{1}^{2p-i}}\E\Big[\int_{-\tau}^{0}|y_{t_k}(r)|^{2p}\mathrm d\nu_1(r)\Big]
\!\!+\!\!\frac{24^p\Delta}{\kappa_{1}^{p-i}}\big(1\!+\!(2L)^{2p}\big)\E[|y(t_k)|^{2p}]\nn\\
&\,+\frac{24^p(2L)^{2p}\Delta}{\kappa_{1}^{p-i}}\E\Big[\int_{-\tau}^{0}|y_{t_k}(r)|^{2p}\mathrm d\nu_1(r)\Big]\quad \forall \,i\in\{2,\ldots, p\}.
\end{align}
Plugging   \eqref{l4.2.4} and \eqref{l4.2.5}  into \eqref{l4.2.3} leads to
\begin{align*}
&\E[|z(t_{k+1})|^{2p}]+\frac{(a_3\Delta)^p}{2^p}\E[|y(t_k)|^{(2+\ell)p}]\nn\\
\leq&~(1+3\kappa_{1}p+6\kappa_{1}\sum_{i=2}^{p}C_{p}^{i})A^{p}_{\theta,\Delta}\E[|z(t_k)|^{2p}]+R(\kappa_{1})\Delta+J_{k,3}+J_{k,4},
\end{align*}
where
\begin{align*}
J_{k,3}:=&-\Big(\frac{p(R_2-2L)A^{p-1}_{\theta,\Delta}}{3^{p-1}}-\sum_{i=2}^{p}C_{p}^{i}\frac{24^p}{\kappa_{1}^{p-i}}\big(1+(2L)^{2p}\big)\Big)\Delta\E[|y(t_k)|^{2p}]\nn\\
&\quad+\sum_{i=2}^{p}C_{p}^{i}\frac{24^p(2L)^{2p}\Delta}{\kappa_{1}^{p-i}}\E\Big[\int_{-\tau}^{0}|y_{t_k}(r)|^{2p}\mathrm d\nu_1(r)\Big],\nn\\
J_{k,4}:=&~R(\kappa_{1})\Delta^{p}\E[|y(t_k)|^{2p}]+\Delta^{p} \Big(p(R_2-2L)(2a_4)^{p-1}+p\frac{(2a_4)^p}{\kappa_{1}^{p-1}}\nn\\
&\quad +\sum_{i=2}^{p}C_{p}^{i}\frac{(12a_4)^p}{\kappa_{1}^{p-i}}\Big) \E\Big[\int_{-\tau}^{0}|y_{t_k}(r)|^{2p}\mathrm d\nu_2(r)\Big]\nn
\\
&\quad +\Big(\frac{p(2L)^p}{\kappa_{1}^{p-1}}+\sum_{i=2}^{p}C_{p}^{i} \frac{(24L)^{2p}}{\kappa_{1}^{2p-i}}\Big)\Delta^{p}\E\Big[\int_{-\tau}^{0}|y_{t_k}(r)|^{2p}\mathrm d\nu_1(r)\Big].
\end{align*}
It is clear that for any $\kappa>0$,
\begin{align}\label{l4.2.6}
&e^{\kappa t_{k+1}}\E[|z(t_{k+1})|^{2p}]\nn\\
\leq&~|z(0)|^{2p}\!+\!\sum_{l=0}^{k}\big((1\!+\!2^p6\kappa_{1})A^{p}_{\theta,\Delta}e^{\kappa \Delta}-1\big)e^{\kappa t_l}\E[|z(t_l)|^{2p}]
\!+\!\sum_{l=0}^{k}\Big(R(\kappa_{1})e^{\kappa t_{l+1}}\Delta\nn\\
&\quad+e^{\kappa t_{l+1}}J_{l,3}+e^{\kappa t_{l+1}}J_{l,4}-\frac{(a_3\Delta)^p}{2^p}e^{\kappa t_{l+1}}\E[|y(t_l)|^{(2+\ell)p}]\Big).
\end{align} 
By the definition of $A_{\theta,\Delta}$ again, we have that for any $\Delta\in[\Delta^*,1]$,
$$A_{\theta}=\frac{(1-\theta)^2}{\theta^2}+\frac{2\theta-1}{2\theta^2}<\frac{(1-\theta)^2}{\theta^2}+\frac{2\theta-1}{\theta^2(1+\Delta)}=A_{\theta,\Delta}\leq A_{\theta,\Delta^{*}}<1.$$
Choose a sufficiently small number $\kappa_1>0$   such that
$$(1+2^p6\kappa_1)A^{p}_{\theta,\Delta^*}<1.$$
For such $\kappa_1,$  choose another   sufficiently small  number $\kappa>0$  such that
\begin{align}\label{l4.2.15}
(1+2^p6\kappa_1)A^{p}_{\theta,\Delta^{*}}e^{\kappa}\leq 1.
\end{align}
This, along with \eqref{l4.2.6} implies that
\begin{align}\label{l4.2.12+}
e^{\kappa t_{k+1}}\E[|z(t_{k+1})|^{2p}]
\leq&~|z(0)|^{2p}+\sum_{l=0}^{k}\Big(R(\kappa_{1})e^{\kappa t_{l+1}}\Delta+e^{\kappa t_{l+1}}J_{l,3}\nn\\
&\quad+e^{\kappa t_{l+1}}J_{l,4}-\frac{(a_3\Delta)^p}{2^p}e^{\kappa t_{l+1}}\E[|y(t_l)|^{(2+\ell)p}]\Big).
\end{align} 
Similar to \eqref{5p2.1}, we deduce
\begin{align*}
&\sum_{l=0}^{k}e^{\kappa t_{l+1}}J_{l,3}
\leq -\Big(\frac{p(R_2-2L)A^{p-1}_{\theta}}{3^{p-1}}-\sum_{i=2}^{p}C_{p}^{i}\frac{24^p}{\kappa_1^{p-i}}\big(1+(2L)^{2p}\big)\\
&-\sum_{i=2}^{p}C_{p}^{i}\frac{24^p(2L)^{2p}e^{\kappa\tau}}{\kappa_1^{p-i}}\Big)\sum_{l=0}^k e^{\kappa t_{l+1}}\Delta\E[|y(t_k)|^{2p}] +\sum_{i=2}^{p}C_{p}^{i}\frac{24^p(2L)^{2p}e^{\kappa\tau}\tau}{\kappa_1^{p-i}}\|\xi\|^{2p}.
\end{align*}
Letting $R_2>0$ be large sufficiently such that 
$$\frac{p(R_2-2L)A^{p-1}_{\theta}}{3^{p-1}}-\sum_{i=2}^{p}C_{p}^{i}\frac{24^p}{\kappa_1^{p-i}}\big(1+(2L)^{2p}\big)-\sum_{i=2}^{p}C_{p}^{i}\frac{24^p(2L)^{2p}e^{\kappa\tau}}{\kappa_1^{p-i}}>0,$$
we obtain
\begin{align}\label{l4.2.13}
\sum_{l=0}^{k}e^{\kappa t_{l+1}}J_{l,3}
\leq~\sum_{i=2}^{p}C_{p}^{i}\frac{24^p(2L)^{2p}e^{\kappa\tau}\tau}{\kappa_1^{p-i}}\|\xi\|^{2p}.
\end{align}
Similarly,
\begin{align}\label{l4.2.14}
&\quad\sum_{l=0}^{k}\big(e^{\kappa t_{l+1}}J_{l,4}-\frac{(a_3\Delta)^p}{2^p}e^{\kappa t_{l+1}}\E[|y(t_l)|^{(2+\ell)p}]\big)\nn\\
&\leq G \Delta^{p}\sum_{l=0}^{k}e^{\kappa t_{l+1}}\E[|y(t_l)|^{2p}]
-\frac{(a_3\Delta)^p}{2^p}\sum_{l=0}^{k}e^{\kappa t_{l+1}}\E[|y(t_l)|^{(2+\ell)p}]\nn\\
&\quad +K\|\xi\|^{2p}\leq K\|\xi\|^{2p}+Ke^{\kappa t_{k+1}},
\end{align}
where
\begin{align*}
G:=&~R(\kappa_{1})+e^{\kappa\tau}\Big(p\big((R_2-2L)(2a_4)^{p-1}+\frac{(2a_4)^p}{\kappa_{1}^{p-1}}\big)+\sum_{i=2}^{p}C_{p}^{i}\frac{(12a_4)^p}{\kappa_{1}^{p-i}}\Big)\nn\\
&\quad+e^{\kappa\tau}\big(\frac{p(2L)^p}{\kappa_{1}^{p-1}}+\sum_{i=2}^{p}C_{p}^{i} \frac{(24L)^{2p}}{\kappa_{1}^{2p-i}}\big).
\end{align*}
Then the desired result follows from \eqref{l4.2.12+}--\eqref{l4.2.14}.

\underline{Step 3.}
It follows from \eqref{l4.2.15+} that
\begin{align*}
&\sup_{k\geq0}\E[|y(t_k)|^2]+2a_3\theta\Delta\big(\sup_{k\geq0}\E[|y(t_k)|^2]\big)^{\frac{2+\ell}{2}}\nn\\
\leq&~\sup_{k\geq0} \E[|z(t_k)|^2]+2K\theta\Delta+2a_4\theta\Delta\sup_{k\geq-N}\E[|y(t_k)|^2],
\end{align*}
which implies
\begin{align*}
\sup_{k\geq0}\E[|y(t_k)|^2]
\leq&~\sup_{k\geq0} \E[|z(t_k)|^2]\!+\!2a_4\theta\Delta\sup_{k\geq0}\E[|y(t_k)|^2]
\!-\!2a_3\theta\Delta\big(\sup_{k\geq0}\E[|y(t_k)|^2]\big)^{\frac{2+\ell}{2}}\nn\\
&\quad+2K\theta\Delta+2a_4\theta\Delta\|\xi\|^2\nn\\
\leq&\sup_{k\geq-N} \E[|z(t_k)|^2]+K(1+\|\xi\|^{2}).
\end{align*}
This, along with Lemma \ref{l4.2} leads to
\begin{align}\label{l4.3.1}
\sup_{\Delta\in(0,1]}\sup_{k\geq -N}\E[|y^{\xi,\Delta}(t_k)|^p]
\leq&~\sup_{\Delta\in(0,1]}\big(\sup_{k\geq-N} \E[|z(t_k)|^2]+K(1+\|\xi\|^{2})\big)^{\frac{p}{2}}\nn\\
\leq &~K(1+\|\xi\|^{p(1+\beta)})\quad\forall ~p\geq 2.
\end{align}
This proof is therefore completed.
\end{proof}
\begin{prop}\label{l4.3}
Under Assumptions \ref{a1}, \ref{a4}, and \ref{a5}, for any integer  $p\geq 2$, 
it holds that
\begin{align}\label{4.3l+1}
\sup_{\Delta\in(0,1]}\sup_{k\geq 0}\E\big[\|z^{\xi,\Delta}_{t_k}\|^p
+\|y^{\xi,\Delta}_{t_k}\|^p\big]\leq K(1+\|\xi\|^{p(\beta+1)}). 
\end{align}
\end{prop}
\begin{proof}
We divide the proof into two  cases. Recall that $\Delta^*\in(0,1)$ and $\e$ are defined in \underline{Step 1} of Lemma \ref{l4.2}.

\underline{Case 1: $\Delta\in(0,\Delta^*]$.} By \eqref{l4.2.16+}, we have that for any $\Delta\in(0,\Delta^*]$,
\begin{align}\label{l4.3.2}
&e^{\e t_{k+1}}|z(t_{k+1})|^{2p}\nn\\
\leq &\;|z(0)|^{2p}+\sum_{l=0}^{k}(A^p_{\theta,\Delta}e^{\e\Delta}-1)e^{\e t_{l}}|z(t_l)|^{2p}+\sum_{l=0}^{k}\sum_{i=1}^{p}C_{p}^{i}e^{\e t_{l+1}}\tilde I_{k,i},
\end{align}
where 
$$\tilde I_{k,i}:=A^{p-i}_{\theta,\Delta}|z(t_k)|^{2(p-i)}\Big(R_1\Delta+2a_4\Delta\int_{-\tau}^{0}|y_{t_k}(r)|^{2}\mathrm d\nu_2(r)+|\sigma_{k}\delta W_{k}|^2
+\mathcal M_{k}\Big)^{i}.$$
Applying the Young inequality yields
\begin{align}\label{l4.3.3}
\tilde I_{k,1}
\leq &\,3\e_1\Delta A^{p}_{\theta,\Delta}|z(t_k)|^{2p}+\frac{\Delta}{\e_1^{p-1}}\Big(R_1^p+(2a_4)^p\int_{-\tau}^{0}|y_{t_k}(r)|^{2p}\mathrm d\nu_2(r)+\frac{|\sigma_{k} \delta W_{k}|^{2p}}{\Delta^{p}}\Big)\nn\\
&\,+A^{p-1}_{\theta,\Delta}|z(t_k)|^{2(p-1)}\mathcal M_{k}
\end{align}
and
\begin{align}\label{l4.3.4}
\tilde I_{k,i}
\leq&\,5^{i}\Delta A^{p-i}_{\theta,\Delta}|z(t_k)|^{2(p-i)}\Big(R_1^i+(2a_4)^i\int_{-\tau}^{0}|y_{t_k}(r)|^{2i}\mathrm d\nu_2(r)+\frac{|\sigma_{k}\delta W_{k}|^{2i}}{\Delta}\nn\\
&\,+\frac{A_{\theta,\Delta}^{\frac{i}{2}}(2|z(t_k)||\sigma_k\delta W_k|)^{i}}{\Delta}+\frac{(4|y(t_k)||\sigma_k\delta W_k|)^{i}}{\Delta}\Big)\nn\\
\leq&\,5\e_1\Delta A^{p}_{\theta,\Delta}|z(t_k)|^{2p}+\frac{5^p\Delta}{\e_1^{p-i}}\Big(R_1^{p}+(2a_4)^{p}\int_{-\tau}^{0}|y_{t_k}(r)|^{2p}\mathrm d\nu_2(r)+\frac{|\sigma_{k}\delta W_{k}|^{2p}}{\Delta^{\frac{p}{i}}}\nn\\
&\,+\frac{(4|y(t_k)||\sigma_k\delta W_k|)^{p}}{\Delta^{\frac{p}{i}}}\Big)+\frac{(10|\sigma_k\delta W_k|)^{2p}}{\e_1^{2p-i}\Delta^{\frac{2p}{i}}}\quad \forall\,i\in\{2,\ldots,p\},
\end{align}
where $\e_1$ is given in  \underline{Step 1} of Lemma \ref{l4.2}.
Inserting \eqref{l4.3.3} and \eqref{l4.3.4} into \eqref{l4.3.2}, we arrive at
\begin{align*}
&e^{\e t_{k+1}}|z(t_{k+1})|^{2p}\\
\leq &\,|z(0)|^{2p}+\sum_{l=0}^{k}\big((1+3p\e_1\Delta+5\e_1\Delta\sum_{i=2}^{p}C_{p}^{i})A^p_{\theta,\Delta}e^{\e\Delta}-1\big)e^{\e t_{l}}|z(t_l)|^{2p}\nn\\
&\,+\sum_{l=0}^{k}e^{\e t_{l+1}}\tilde J_{l,1}
+p\sum_{l=0}^{k}A^{p-1}_{\theta,\Delta}|z(t_l)|^{2(p-1)}\mathcal M_{l},
\end{align*}
where
\begin{align*}
\tilde J_{l,1}:=&\,5^p\Delta\sum_{i=1}^{p}\frac{C_{p}^{i}}{\e_1^{p-i}}\Big(R_1^{p}+(2a_4)^{p}\int_{-\tau}^{0}|y_{t_l}(r)|^{2p}\mathrm d\nu_2(r)\Big)
+\frac{|\sigma_{l}\delta W_{l}|^{2p}}{(\e_1\Delta)^{p-1}}\nn\\
&\,+5^p\Delta\sum_{i=2}^{p}\frac{C_{p}^{i}}{\e_1^{p-i}}\big(\frac{|\sigma_{l}\delta W_{l}|^{2p}}{\Delta^{\frac{p}{i}}}+\frac{(4|y(t_l)||\sigma_l\delta W_l|)^{p}}{\Delta^{\frac{p}{i}}}\big)+\sum_{i=2}^{p}\frac{C_{p}^{i}(10|\sigma_l\delta W_l|)^{2p}}{\e_1^{2p-i}\Delta^{\frac{2p}{i}}}.
\end{align*}
It follows from \eqref{l4.2.16} that
\begin{align*}(1+3p\e_1\Delta+5\e_1\Delta\sum_{i=2}^{p}C_{p}^{i})A^p_{\theta,\Delta}e^{\e\Delta}-1
\leq(1+2^p\cdot 6\e_1\Delta)A^p_{\theta,\Delta}e^{\e\Delta}-1\leq0,
\end{align*}
which implies 
\begin{align*}
e^{\e t_{k+1}}|z(t_{k+1})|^{2p}
\leq &\,|z(0)|^{2p}+\sum_{l=0}^{k}e^{\e t_{l+1}}\tilde J_{l,1}
+p\sum_{l=0}^{k}A^{p-1}_{\theta,\Delta}|z(t_l)|^{2(p-1)}\mathcal M_{l}.
\end{align*}
It is clear that
\begin{align}\label{l4.3.5}
&\,\E\Big[\sup_{(k-N)\vee 0\leq j\leq k}e^{\e t_{j+1}}|z(t_{j+1})|^{2p}\Big]\nn\\
\leq &\,|z(0)|^{2p}+\E\Big[\sum_{l=0}^{k}e^{\e t_{l+1}}\tilde J_{l,1}\Big]+p\E\Big[\sup_{(k-N)\vee 0\leq j\leq k}\!\!\!\sum_{l=0}^{(k-N)\vee 0}e^{\e t_{l+1}}A^{p-1}_{\theta,\Delta}|z(t_l)|^{2(p-1)}\mathcal M_{l}\nn\\
&\,+\sup_{(k-N)\vee 0\leq j\leq k}\sum_{l=(k-N)\vee 0}^{j}e^{\e t_{l+1}}A^{p-1}_{\theta,\Delta}|z(t_l)|^{2(p-1)}\mathcal M_{l}\Big]\nn\\
=&\,|z(0)|^{2p}+\E\Big[\sum_{l=0}^{k}e^{\e t_{l+1}}\tilde J_{l,1}\Big]\nn\\
&\,+p\E\Big[\sup_{(k-N)\vee 0\leq j\leq k}\sum_{l=(k-N)\vee 0}^{j}e^{\e t_{l+1}}A^{p-1}_{\theta,\Delta}|z(t_l)|^{2(p-1)}\mathcal M_{l}\Big].
\end{align}
Similar to \eqref{l4.2.10} and by \eqref{l4.3.1}, one has
\begin{align}\label{l4.3.6}
&\E\Big[\sum_{l=0}^{k}e^{\e t_{l+1}}\tilde J_{l,1}\Big]\nn\\
\leq&\,K_{\e_1}\Delta \E\Big[\sum_{l=0}^{k}e^{\e t_{l+1}}\Big(1+ \int_{-\tau}^{0}|y_{t_l}(r)|^{2p}\mathrm d\nu_2(r) +|\sigma_{l}|^{2p} \Big)\Big]\nn\\
\leq&\,K_{\e_1}\Delta \E\Big[\sum_{l=0}^{k}e^{\e t_{l+1}}\big(1+ \int_{-\tau}^{0}|y_{t_l}(r)|^{2p}\mathrm d\nu_2(r) +|y(t_l)|^{2p}+ \int_{-\tau}^{0}|y_{t_l}(r)|^{2p}\mathrm d\nu_1(r)\big)\Big]\nn\\
\leq&\,K_{\e_1}\|\xi\|^{2p}+K_{\e_1}\Delta \sup_{j\geq -N}\E[|y(t_j)|^{2p}]\sum_{l=0}^{k}e^{\e t_{l+1}}\nn\\
\leq&\, K_{\e_1}(1+\|\xi\|^{2p(\beta+1)})\frac{e^{\e t_{k+2}}}{\e}.
\end{align} 
Recalling the definition of $\mathcal M_k$, we apply the Burkholder--Davis--Gundy inequality to obtain
\begin{align*}
&\,p\E\Big[\sup_{(k-N)\vee 0\leq j\leq k}\sum_{l=(k-N)\vee 0}^{j}e^{\e t_{l+1}}A^{p-1}_{\theta,\Delta}|z(t_l)|^{2(p-1)}\mathcal M_{l}\Big]\nn\\
\leq&\,p\E\Big[\Big(\sum_{l=(k-N)\vee 0}^{k}\big(2e^{\e t_{l+1}}A^{p-\frac{1}{2}}_{\theta,\Delta}|z(t_l)|^{2p-1}|\sigma_l|\big)^2\Delta\Big)^{\frac{1}{2}}\Big]\nn\\
&\,+p\E\Big[\Big(\sum_{l=(k-N)\vee 0}^{k}\big(4e^{\e t_{l+1}}A^{p-1}_{\theta,\Delta}|z(t_l)|^{2(p-1)}|y(t_l)||\sigma_l|\big)^2\Delta\Big)^{\frac{1}{2}}\Big].
\end{align*}
Using $A_{\theta,\Delta}\in(0,1)$ and applying the Young inequality, we obtain
\begin{align*}
&\,p\E\Big[\sup_{(k-N)\vee 0\leq j\leq k}\sum_{l=(k-N)\vee 0}^{j}e^{\e t_{l+1}}A^{p-1}_{\theta,\Delta}|z(t_l)|^{2(p-1)}\mathcal M_{l}\Big]\nn\\
\leq&\,2p\E\Big[\big(\sup_{(k-N)\vee 0\leq j\leq k}(e^{\e t_{j+1}})^{\frac{2p-1}{2p}}|z(t_j)|^{2p-1}\big)\big(\sum_{l=(k-N)\vee 0}^{k}(e^{\e t_{l+1}})^{\frac{1}{p}}|\sigma_l|^2\Delta\big)^{\frac{1}{2}}\Big]\nn\\
&\,+4p\E\Big[\big(\!\!\sup_{(k-N)\vee 0\leq j\leq k}(e^{\e t_{j+1}})^{\frac{(p-1)}{p}}|z(t_j)|^{2(p-1)}\big)\big(\!\!\sum_{l=(k-N)\vee 0}^{k}\!(e^{\e t_{l+1}})^{\frac{2}{p}}|y(t_l)|^2|\sigma_l|^2\Delta\big)^{\frac{1}{2}}\Big]\nn\\
\leq &\,\frac{1}{2e^{\e \Delta}}\E\Big[\sup_{(k-N)\vee 0\leq j\leq k}e^{\e t_{j+1}}|z(t_j)|^{2p}\Big]
\!+\!\frac{(2p)^{2p}e^{\e(2p-1)}}{4^{2p-1}}\E\Big[\!\big(\!\!\sum_{l=(k-N)\vee 0}^{k}\!\!(e^{\e t_{l+1}})^{\frac{1}{p}}|\sigma_l|^2\Delta\big)^{p}\Big]\nn\\
&\,+\frac{(4p)^pe^{\e(p-1)}}{4^{p-1}}\E\Big[\big(\sum_{l=(k-N)\vee 0}^{k}(e^{\e t_{l+1}})^{\frac{2}{p}}|y(t_l)|^2|\sigma_l|^2\Delta\big)^{\frac{p}{2}}\Big]\nn\\
\leq&\,e^{\e t_{(k-N)\vee0}}\E[|z(t_{(k-N)\vee0})|^{2p}]\!+\!\frac{1}{2}\E\Big[\sup_{(k-N)\vee 0\leq j\leq k}e^{\e t_{j+1}}|z(t_{j+1})|^{2p}\Big]\!+\!K(N+1)^{p-1}\Delta^p\nn\\
&\,\times\sum_{l=(k-N)\vee 0}^{k}e^{\e t_{l+1}}\E[|\sigma_l|^{2p}]+K(N+1)^{\frac{p}{2}-1}\Delta^{\frac{p}{2}}\sum_{l=(k-N)\vee 0}^{k}e^{\e t_{l+1}}\E[|y(t_l)|^{p}|\sigma_l|^{p}]\nn\\
\leq&\,e^{\e t_{(k-N)\vee0}}\E[|z(t_{(k-N)\vee0})|^{2p}]+\frac{1}{2}\E\Big[\sup_{(k-N)\vee 0\leq j\leq k}e^{\e t_{j+1}}|z(t_{j+1})|^{2p}\Big]\nn\\
&\,+K(\tau\!+\!1)^{p-1}\Delta\!\!\!\sum_{l=(k-N)\vee 0}^{k}\!\!e^{\e t_{l+1}}\E[|\sigma_l|^{2p}]\!+\!K(\tau\!+\!1)^{\frac{p}{2}-1}\Delta\!\!\!\sum_{l=(k-N)\vee 0}^{k}\!e^{\e t_{l+1}}\E[|y(t_l)|^{2p}].
\end{align*}
 According to Assumption \ref{a1}, Lemma \ref{l4.2}, and \eqref{l4.3.1}, we deduce
\begin{align}\label{l4.3.7}
&\,p\E\Big[\sup_{(k-N)\vee 0\leq j\leq k}\sum_{l=(k-N)\vee 0}^{j}e^{\e t_{l+1}}A^{p-1}_{\theta,\Delta}|z(t_l)|^{2(p-1)}\mathcal M_{l}\Big]\nn\\
\leq&\, K(1+\|\xi\|^{2p(\beta+1)})e^{\e t_{k+1}}+\frac{1}{2}\E\Big[\sup_{(k-N)\vee 0\leq j\leq k}e^{\e t_{j+1}}|z(t_{j+1})|^{2p}\Big]\nn\\
&\,+K\Delta\sum_{l=(k-N)\vee 0}^{k}e^{\e t_{l+1}}\E\Big[|y(t_l)|^{2p}+\int_{-\tau}^{0}|y_{t_l}(r)|^{2p}\mathrm d\nu_1(r)\Big]\nn\\
\leq&\, K(1+\|\xi\|^{2p(\beta+1)})e^{\e t_{k+1}}+\frac{1}{2}\E\Big[\sup_{(k-N)\vee 0\leq j\leq k}e^{\e t_{j+1}}|z(t_{j+1})|^{2p}\Big]\nn\\
&\,+K\sup_{j\geq-N}\E[|y_{t_j}(r)|^{2p}]\sum_{l=(k-N)\vee 0}^{k} e^{\e t_{l+1}}\Delta\nn\\
\leq &\, K(1+\|\xi\|^{2p(\beta+1)})e^{\e t_{k+2}}+\frac{1}{2}\E\Big[\sup_{(k-N)\vee 0\leq j\leq k}e^{\e t_{j+1}}|z(t_{j+1})|^{2p}\Big].
\end{align}
Inserting \eqref{l4.3.6} and \eqref{l4.3.7} into \eqref{l4.3.5} yields
\begin{align*}
\E\Big[\sup_{(k-N)\vee 0\leq j\leq k}e^{\e t_{j+1}}|z(t_{j+1})|^{2p}\Big]
\leq  K_{\e_1}(1+\|\xi\|^{2p(\beta+1)})e^{\e t_{k+2}}.
\end{align*}
This implies that for any $\Delta\in(0,\Delta^*]$,
\begin{align*}
\E\Big[\sup_{(k-N)\vee 0\leq j\leq k}|z(t_{j+1})|^{2p}\Big]
&\leq e^{-\e t_{{(k-N)\vee 0+1}}} K_{\e_1}(1+\|\xi\|^{2p(\beta+1)})e^{\e t_{k+2}}\\
&\leq K(1+\|\xi\|^{2p(\beta+1)}).
\end{align*}
\underline{Case 2: $\Delta\in[\Delta^*,1]$.} Similar arguments to  \underline{Case 1} yield that for any $\Delta\in[\Delta^*,1]$,
\begin{align}\label{l4.3.9}
e^{\kappa t_{k+1}}|z(t_{k+1})|^{2p}
\leq &\,|z(0)|^{2p}+\sum_{l=0}^{k}\big((1+3p\kappa_1+5\kappa_1\sum_{i=2}^{p}C_{p}^{i})A^p_{\theta,\Delta}e^{\kappa\Delta}-1\big)e^{\kappa t_{l}}|z(t_l)|^{2p}\nn\\
&\,+\sum_{l=0}^{k}e^{\kappa t_{l+1}}\tilde J_{l,2}
+p\sum_{l=0}^{k}A^{p-1}_{\theta,\Delta}|z(t_l)|^{2(p-1)}\mathcal M_{l},
\end{align}
where 
\begin{align*}
\tilde J_{l,2}:=&\,5^p\Delta^p\sum_{i=1}^{p}\frac{C_{p}^{i}}{\e_1^{p-i}}\Big(R_1^{p}+(2a_4)^{p}\int_{-\tau}^{0}|y_{t_l}(r)|^{2p}\mathrm d\nu_2(r)\Big)
+5^p\sum_{i=1}^{p}\frac{C_{p}^{i}}{\e_1^{p-i}}|\sigma_{l}\delta W_{l}|^{2p}\nn\\
&\,+5^p\sum_{i=2}^{p}\frac{C_{p}^{i}}{\e_1^{p-i}}(4|y(t_l)||\sigma_l\delta W_l|)^{p}+\sum_{i=2}^{p}\frac{C_{p}^{i}(10|\sigma_l\delta W_l|)^{2p}}{\e_1^{2p-i}}.
\end{align*}
Similar to \eqref{l4.3.6}  and \eqref{l4.3.7}, we derive 
\begin{align*}
&\,\E\Big[\sum_{l=0}^{k}e^{\kappa t_{l+1}}\tilde J_{l,2}\Big]
\leq K_{\kappa_1}(1+\|\xi\|^{2p(\beta+1)})\frac{e^{\kappa t_{k+2}}}{\kappa},\nn\\
&\,p\E\Big[\sup_{(k-N)\vee 0\leq j\leq k}\sum_{l=(k-N)\vee 0}^{j}
A^{p-1}_{\theta,\Delta}|z(t_l)|^{2(p-1)}\mathcal M_{l}\Big]\nn\\
\leq&\, K_{\kappa_1}(1+\|\xi\|^{2p(\beta+1)})\frac{e^{\kappa t_{k+2}}}{\kappa}+\frac{1}{2}\E\Big[\sup_{(k-N)\vee 0\leq j\leq k}e^{\kappa t_{j+1}}|z(t_{j+1})|^{2p}\Big].
\end{align*}
Plugging these into \eqref{l4.3.9}, and combining  \eqref{l4.2.15}, we have
\begin{align*}
&\, e^{\kappa t_{k+1}}|z(t_{k+1})|^{2p}\\
\leq &\,|z(0)|^{2p}+\sum_{l=0}^{k}\big(1+2^p\cdot 6\kappa_1)A^p_{\theta,\Delta^*}e^{\kappa\Delta}-1\big)e^{\kappa t_{l}}|z(t_l)|^{2p}\nn\\
&\,+K(1+\|\xi\|^{2p(\beta+1)})\frac{e^{\kappa t_{k+2}}}{\kappa}
+Ke^{\kappa t_{k+2}}+\frac{1}{2}\Big(\sup_{(k-N)\vee 0\leq j\leq k}e^{\kappa t_{j+1}}|z(t_{j+1})|^{2p}\Big)\nn\\
\leq&\,K(1+\|\xi\|^{2p(\beta+1)})\frac{e^{\kappa t_{k+2}}}{\kappa}+\frac{1}{2}\Big(\sup_{(k-N)\vee 0\leq j\leq k}e^{\kappa t_{j+1}}|z(t_{j+1})|^{2p}\Big).
\end{align*}
This implies that for any $\Delta\in[\Delta^*,1]$,
\begin{align*}
\E\Big[\sup_{(k-N)\vee 0\leq j\leq k}|z(t_{j+1})|^{2p}\Big]
\leq K(1+\|\xi\|^{2p(\beta+1)}).
\end{align*}
Consequently, 
\begin{align}\label{l4.3.10}
\sup_{\Delta\in(0,1]}\sup_{k\geq 0}\E\Big[\sup_{k-N\leq j\leq k}|z(t_{j+1})|^{2p}\Big]
\leq K(1+\|\xi\|^{2p(\beta+1)}).
\end{align}
In addition, by virtue of \eqref{l4.2.15+}, \eqref{l4.3.1}, and \eqref{l4.3.10}, we obtain that for any $\Delta\in(0,1]$,
\begin{align*}
&\E\Big[\sup_{(k-N)\vee0\leq j\leq k}|y(t_{j})|^{2p}\Big]\nn\\
\leq&\, K\E\Big[\sup_{(k-N)\vee0\leq j\leq k}|z(t_{j})|^{2p}\Big]+K\Delta
+K\Delta\E\Big[\sup_{(k-N)\vee0\leq j\leq k}\int_{-\tau}^{0}|y_{t_j}(r)|^{2p}\mathrm d\nu_2\Big]\nn\\
\leq&\,K(1+\|\xi\|^{2p(\beta+1)})+K\Delta\E\Big[\sup_{(k-2N)\vee (-N)\leq j\leq k}|y(t_j)|^{2p}\Big]\nn\\
\leq&\,K(1+\|\xi\|^{2p(\beta+1)})+K\Delta\cdot(2N)\sup_{k\geq-N}\E[|y(t_k)|^{2p}]\leq K(1+\|\xi\|^{2p(\beta+1)}).
\end{align*}
The proof is finished.
\end{proof}
\begin{lemma}\label{l4.4}
Let Assumptions \ref{a1}--\ref{a2}, \ref{a4}--\ref{a7} hold. Then  for any $\Delta\in(0,1]$ and $p\geq1$,
\begin{align*}
\sup_{t\in[0,\infty)}\E[\|z^{\xi,\Delta}_{t}-y^{\xi,\Delta}_{t}\|^p]\leq K(1+\|\xi\|^{p(1+\beta)^2})\Delta^{\frac{p}{2}}\quad \forall \,\xi\in \mathcal C^d.
\end{align*}
\end{lemma}
\begin{proof}
For any $t\geq0$ and $r\in[-\tau,0]$, there exist $k\in\mathbb N$ and $j\in\{-N,-N+1,\ldots,-1\}$ such that
$t\in[t_k,t_{k+1})$ and $r\in[t_j,t_{j+1}]$. It is straightforward to see that
$$t+r\in[t_{k+j},t_{k+j+2}),~~~~t_{k}+r\in[t_{k+j},t_{k+j+1}).$$
Recall
$z^{\xi,\Delta}_t(r)=z^{\xi,\Delta}(t+r)$ and $y^{\xi,\Delta}_t(r)=y^{\xi,\Delta}_{t_k}(r)=y^{\xi,\Delta}(t_k+r)$. We first give estimates of $|z^{\xi,\Delta}_t(r)-y^{\xi,\Delta}_{t}(r)|$, whose proof is 
divided into four cases.

\underline{Case 1:  $t+r\in[t_{k+j},t_{k+j+1})\in[0,\infty)$.} By \eqref{thL} and \eqref{zcont}, we derive 
\begin{align*}
&z^{\xi,\Delta}_t(r)-y^{\xi,\Delta}_{t}(r)=z^{\xi,\Delta}(t+r)-y^{\xi,\Delta}_{t_k}(r)\nn\\
=&~z^{\xi,\Delta}(t_{k+j})+b(y^{\xi,\Delta}_{t_{k+j}})(t+r-t_{k+j})+\sigma(y^{\xi,\Delta}_{t_{k+j}})(W(t+r)-W(t_{k+j}))\nn\\
&~-\big(\frac{t_{j+1}-r}{\Delta}y^{\xi,\Delta}(t_{k+j})+\frac{r-t_j}{\Delta}y^{\xi,\Delta}(t_{k+j+1})\big)\nn\\
=&~\frac{t_{j+1}-r}{\Delta}\big(z^{\xi,\Delta}(t_{k+j})-y^{\xi,\Delta}(t_{k+j})\big)+\frac{r-t_{j}}{\Delta}\big(z^{\xi,\Delta}(t_{k+j})-y^{\xi,\Delta}(t_{k+j+1})\big)\nn\\
&\quad+b(y^{\xi,\Delta}_{t_{k+j}})(t+r-t_{k+j})+\sigma(y^{\xi,\Delta}_{t_{k+j}})(W(t+r)-W(t_{k+j})).
\end{align*}
Combining  $z^{\xi,\Delta}(t_l)=y^{\xi,\Delta}(t_l)-\theta b(y^{\xi,\Delta}_{t_l})\Delta$ for all $l\in\mathbb N$ yields
\begin{align*}
&z^{\xi,\Delta}_t(r)-y^{\xi,\Delta}_{t}(r)\nn\\
=&-\frac{t_{j+1}-r}{\Delta}\theta b(y^{\xi,\Delta}_{t_{k+j}})\Delta-\frac{r-t_{j}}{\Delta}\theta b(y^{\xi,\Delta}_{t_{k+j+1}})\Delta+\frac{r-t_{j}}{\Delta}\big(z^{\xi,\Delta}(t_{k+j})-z^{\xi,\Delta}(t_{k+j+1})\big)\nn\\
&\quad+b(y^{\xi,\Delta}_{t_{k+j}})(t+r-t_{k+j})+\sigma(y^{\xi,\Delta}_{t_{k+j}})(W(t+r)-W(t_{k+j}))\nn\\
=&-\frac{t_{j+1}-r}{\Delta}\theta b(y^{\xi,\Delta}_{t_{k+j}})\Delta-\frac{r-t_{j}}{\Delta}\theta b(y^{\xi,\Delta}_{t_{k+j+1}})\Delta+b(y^{\xi,\Delta}_{t_{k+j}})(t-t_{k})\nn\\&\quad+\sigma(y^{\xi,\Delta}_{t_{k+j}})\big(W(t+r)-W(t_{k+j})-\frac{r-t_{j}}{\Delta}(W(t_{k+j+1})-W(t_{k+j})\big).
\end{align*}
This implies that for any $p\geq1$,
\begin{align*}
&|z^{\xi,\Delta}_t(r)\!-\!y^{\xi,\Delta}_{t}(r)|^p\\
\leq& \,K\Delta^p\big(|b(y^{\xi,\Delta}_{t_{k+j}})|^p\!+\!|b(y^{\xi,\Delta}_{t_{k+j+1}})|^p\big)+K\big|\sigma(y^{\xi,\Delta}_{t_{k+j}})\big(W(t+r)\!-\!W(t_{k+j})\big)\big|^p\nn\\
&\,+K\big|\sigma(y^{\xi,\Delta}_{t_{k+j}})\big(W(t_{k+j+1})-W(t_{k+j})\big)\big|^p.
\end{align*}

\underline{Case 2:  $t+r\in[t_{k+j},t_{k+j+1})\!\in\![-\tau,0]$.}
It is obvious that $y^{\xi,\Delta}_{t}(r)=y^{\xi,\Delta}_{t_k}(r)$, and
\begin{align*}
|z^{\xi,\Delta}_t(r)-y^{\xi,\Delta}_{t}(r)|^p
={}&|\frac{t_{k+j+1}-(t+r)}{\Delta}\xi(t_{k+j})+\frac{t+r-t_{k+j}}{\Delta}\xi(t_{k+j+1})\nn\\
&-\frac{t_{j+1}-r}{\Delta}\xi(t_{k+j})-\frac{r-t_{j}}{\Delta}\xi(t_{k+j+1})|^p\nn\\
={} &\frac{t-t_k}{\Delta}|\xi(t_{k+j+1})-\xi(t_{k+j})|^p\leq K^p\Delta^{\frac{p}{2}},
\end{align*}
where we used Assumption \ref{a7}.

\underline{Case 3: $t+r\in[t_{k+j+1},t_{k+j+2})\in[0,\infty)$.}
Similar to  \underline{Case 1}, we obtain
\begin{align*}
&|z^{\xi,\Delta}_t(r)-y^{\xi,\Delta}_{t}(r)|^p\nn\\
=&~\big|z^{\xi,\Delta}(t_{k+j+1})+b(y^{\xi,\Delta}_{t_{k+j+1}})(t+r-t_{k+j+1})+\sigma(y^{\xi,\Delta}_{t_{k+j+1}})\big(W(t+r)-W(t_{k+j+1})\big)\nn\\
&-\big(\frac{t_{j+1}-r}{\Delta}y^{\xi,\Delta}(t_{k+j})+\frac{r-t_j}{\Delta}y^{\xi,\Delta}(t_{k+j+1})\big)\big|^p\nn\\
\leq{}&K\Delta^p|b(y^{\xi,\Delta}_{t_{k+j+1}})|^p+K\big|\sigma(y^{\xi,\Delta}_{t_{k+j+1}})\big(W(t+r)-W(t_{k+j+1})\big)\big|^p
+K|z^{\xi,\Delta}(t_{k+j+1})\nn\\
&-y^{\xi,\Delta}(t_{k+j+1})|^p+K|z^{\xi,\Delta}(t_{k+j+1})-z^{\xi,\Delta}(t_{k+j})|^p+K|z^{\xi,\Delta}(t_{k+j})-y^{\xi,\Delta}(t_{k+j})|^p\nn\\
\leq{}&K\Delta^p\big(|b(y^{\xi,\Delta}_{t_{k+j}})|^p\!+\!|b(y^{\xi,\Delta}_{t_{k+j+1}})|^p\big)+K\big|\sigma(y^{\xi,\Delta}_{t_{k+j+1}})\big(W(t+r)-W(t_{k+j+1})\big)\big|^p\nn\\
&+K|z^{\xi,\Delta}(t_{k+j+1})-z^{\xi,\Delta}(t_{k+j})|^p\big(\textbf 1_{\{t_{k+j+1>0}\}}+\textbf 1_{\{t_{k+j+1=0}\}}\big)\nn\\
\leq{}&K\Delta^p\big(|b(y^{\xi,\Delta}_{t_{k+j}\vee 0})|^p\!+\!|b(y^{\xi,\Delta}_{t_{k+j+1}})|^p\!+\!|b(\xi)|^p\big)\!\!+\!K\big|\sigma(y^{\xi,\Delta}_{t_{k+j+1}})\big(W(t\!+\!r)\!\!-\!\!\!W(t_{k+j+1})\big)\big|^p\nn\\
&+K\big|\sigma(y^{\xi,\Delta}_{t_{k+j}})\big(W(t_{k+j+1})\!-\!W(t_{k+j})\big)\big|^p\textbf 1_{\{t_{k+j+1>0}\}}+K\Delta^{\frac{p}{2}},
\end{align*}
where we used Assumption \ref{a7}.
%

\underline{Case 4:  $t+r\in[t_{k+j+1},t_{k+j+2})\in[-\tau,0)$.}
It can be calculated that 
\begin{align*}
|z^{\xi,\Delta}_t(r)-y^{\xi,\Delta}_{t}(r)|^p
\leq K\Delta^{\frac p2}.
\end{align*}
Together with the above cases, we have 
\begin{align}\label{l4.4.1}
&\E\Big[\sup_{r\in[-\tau,0]}|z^{\xi,\Delta}_t(r)-y^{\xi,\Delta}_{t}(r)|^p\Big]\nn\\
\leq{}& K (1+\|\xi\|^{p(1+\beta)})\Delta^{\frac p2}+K(I_1+I_2+I_3),
\end{align}
where
\begin{align*}
I_1:=&\,\E\Big[\sup_{j\in\{-N,\ldots,-1\}}\big(|b(y^{\xi,\Delta}_{t_{k+j}\vee0})|^p\Delta^p+|b(y^{\xi,\Delta}_{t_{k+j+1}\vee0})|^p\Delta^p)\Big],\nn\\
I_2:=&\,\E\Big[\sup_{j\in\{-N,\ldots,-1\}}\sup_{r\in[t_j,t_{j+1}]}\big|\sigma(y^{\xi,\Delta}_{t_{k+j+1}\vee0})\big(W(t+r)-W(t_{k+j+1})\big)\big|^p\times\nn\\
&\,\textbf 1_{\{t+r\in[t_{k+j+1},t_{k+j+2})\}}\Big],\nn\\
I_3:=&\,\E\Big[\sup_{j\in\{-N,\ldots,-1\}}\big|\sigma(y^{\xi,\Delta}_{t_{k+j}\vee0})\big(W(t_{k+j+1}\vee0)-W(t_{k+j}\vee0)\big)\big|^p\Big].
\end{align*}
It follows from Assumption \ref{a5} that
\begin{align*}
I_1\leq&~2\Delta^p\E\Big[\sup_{j\in\{-N,\ldots,0\}}|b(y^{\xi,\Delta}_{t_{k+j}\vee0})|^p\Big]
\leq K\Delta^p\E \Big[\sup_{j\in\{-N,\ldots,0\}}(1+\|y^{\xi,\Delta}_{t_{k+j}\vee0}\|^{p(\beta+1)})\Big].
\end{align*}
Since
$$\sup_{j\in\{-N,\ldots,0\}}\|y^{\xi,\Delta}_{t_{k+j}}\|
\leq \|y^{\xi,\Delta}_{t_{k-N}}\|+\|y^{\xi,\Delta}_{t_{k}}\|,
$$
by \eqref{4.3l+1}, we have
\begin{align}\label{l4.4.2}
I_1\leq K\Delta^p\E[1+\|\xi\|^{p(\beta+1)}+\|y^{\xi,\Delta}_{t_{k-N}}\|^{p(1+\beta)}+\|y^{\xi,\Delta}_{t_{k}}\|^{p(1+\beta)}]\leq K (1+\|\xi\|^{p(1+\beta)^2})\Delta^p.
\end{align}
Applying the H\"older inequality yields
\begin{align*}
I_2\leq& \Big(\E\Big[\sup_{j\in\{-N,\ldots,-1\}}\big|\sigma(y^{\xi,\Delta}_{t_{k+j+1}\vee0})|^{2p}\Big]\Big)^{\frac{1}{2}}\Big(\E\Big[\sup_{s_1-s_2\in[0,\Delta]}\big|W(s_1)-W(s_2)\big|^{2p}\Big]\Big)^{\frac{1}{2}}.
\end{align*}
Similar to \eqref{l4.4.2}, by virtue of Assumption \ref{a1} and \eqref{4.3l+1}, we arrive at
\begin{align*}
I_2\leq K(1+\|\xi\|^{p(\beta+1)})\Big(\E\Big[\sup_{s_1-s_2\in[0,\Delta]}\big|W(s_1)-W(s_2)\big|^{2p}\Big]\Big)^{\frac{1}{2}}.
\end{align*}
Applying the Burkholder--Davis--Gundy inequality, we obtain
\begin{align}\label{l4.4.3}
I_2\leq K(1+\|\xi\|^{p(\beta+1)})\Delta^{\frac{p}{2}}.
\end{align}
Similarly,
\begin{align}\label{l4.4.4}
I_3\leq K(1+\|\xi\|^{p(\beta+1)})\Delta^{\frac{p}{2}}.
\end{align}
The desired argument follows from \eqref{l4.4.1}--\eqref{l4.4.4}. The proof is completed.
\end{proof}

\section{Attractiveness and stability of numerical solution}

In this subsection, we  reveal that numerical solutions of the $\theta$-EM method can inherit  the  the exponential  attractiveness and the exponential  stability of the exact ones.  

\begin{prop}\label{p2.3}
Under Assumptions {\ref{a1}} and {\ref{a2}}, for any  $\Delta\in(0,1]$ and $k\in\mathbb N$, one has 
\begin{align*}
\E[\|y^{\xi,\Delta}_{t_k}-y^{\eta,\Delta}_{t_k}\|^2]\leq K\Big(|b(\xi)-b(\eta)|^2+\|\xi-\eta\|^2\Big)e^{-\lambda_0 t_k}\quad \forall\,\xi,\eta\in \mathcal C^d,
\end{align*}
where 
 $\l_0\in(0,\frac{(2\theta-1)c_{\l}}{\theta^2}\wedge \l)$.
\end{prop}
\begin{proof}
Fix $\xi,\eta\in \mathcal C^d$, $\Delta\in(0,1]$, ~$k\in\mathbb N$, and $\l_0\in(0,\l]$, where $\l$ satisfies \eqref{lamb}.
Similar to \underline{Step 1} in Proposition \ref{p2.2}, and using Assumptions \ref{a1} and \ref{a2}, we deduce 
\begin{align*}
&e^{\lambda_0 t_{k+1}}\E [|z^{\xi,\Delta}(t_{k+1})-z^{\eta, \Delta}(t_{k+1})|^2]\nn\\
\leq{}&|z^{\xi,\Delta}(0)-z^{\eta, \Delta}(0)|^2+(L+2a_2)\tau e^{\l_0\tau}\|\xi-\eta\|^2\nn\\
&+(\frac{(1-\theta)^2}{\theta^2}+\frac{2\theta-1}{\theta^2(1+c_\l\Delta)}
-e^{-\l_0\Delta})\sum_{i=0}^{k}e^{\l_0 t_{i+1}}\E[|z^{\xi,\Delta}(t_i)-z^{\eta, \Delta}(t_i)|^2]\nn\\
&-(2a_1-L-c_{\l}-(L+2a_2)e^{\l_0\tau})\Delta\sum_{i=0}^{k}e^{\l_0 t_{i+1}}\E[|y^{\xi,\Delta}(t_i)-y^{\eta, \Delta}(t_i)|^2].
\end{align*}
It follows from \eqref{lamb} that
$$2a_1-L-c_{\l}-(L+2a_2)e^{\l_0\tau}>0.$$ 
Similar to \eqref{varep}, we obtain 
\begin{align*}
\frac{(1-\theta)^2}{\theta^2}+\frac{2\theta-1}{\theta^2(1+c_\l\Delta)}
-e^{-\l_0\Delta}
< (\l_0-\frac{(2\theta-1)c_{\l}}{\theta^2(1+c_{\lambda})})\Delta.
\end{align*}
Letting $\l_0\in(0,\frac{(2\theta-1)c_{\l}}{\theta^2(1+c_{\lambda})}\wedge \l]$, we have
\begin{align}\label{2p2.3}
&\E [|z^{\xi,\Delta}(t_{k})-z^{\eta, \Delta}(t_{k})|^2]\nn\\
\leq{}&(1\vee((L+2a_2)\tau e^{\l_0\tau}))\big( |z^{\xi,\Delta}(0)-z^{\eta, \Delta}(0)|^2+\|\xi-\eta\|^2\big)e^{-\l_0 t_k}.
\end{align}
Next, we shall show that $\{y^{\cdot,\Delta}(t_k)\}_{k\in\mathbb N}$  has the similar property by the induction argument. 
In fact,  it follows from $z^{\xi,\Delta}(t_k)=y^{\xi,\Delta}(t_k)-\theta b(y^{\xi,\Delta}_{t_k})\Delta$ and Assumption \ref{a2} that
\begin{align}\label{3p2.3}
&(1+2a_1\theta\Delta)|y^{\xi,\Delta}(t_k)-y^{\eta,\Delta}(t_k)|^2\nn\\
\leq{}& |z^{\xi,\Delta}(t_k)-z^{\eta,\Delta}(t_k)|^2+2a_2\theta\Delta\int_{-\tau}^{0}|y^{\xi,\Delta}_{t_k}(r)-y^{\eta,\Delta}_{t_k}(r)|^2\mathrm d\nu_2(r).
\end{align}
 By virtue of \eqref{thL}, we arrive at
 \begin{align}\label{1p2.3}
&(1+2a_1\theta\Delta)\E[|y^{\xi,\Delta}(t_k)-y^{\eta,\Delta}(t_k)|^2]\nn\\
\leq{}&\E[|z^{\xi,\Delta}(t_k)-z^{\eta,\Delta}(t_k)|^2]+2a_2\theta\Delta\int_{-\Delta}^{0}\frac{r+\Delta}{\Delta}\mathrm d\nu_2(r)\E[|y^{\xi,\Delta}(t_k)-y^{\eta,\Delta}(t_k)|^2]\nn\\
&+2a_2\theta\Delta\Big(\nu_2([-\tau,-\Delta])+\int_{-\Delta}^{0}\frac{-r}{\Delta}\mathrm d\nu_2(r)\Big)\sup_{k-N\leq l \leq k-1}\E[|y^{\xi,\Delta}(t_l)-y^{\eta,\Delta}(t_l)|^2].
\end{align}
For $k=1$, according to $a_2 e^{\lambda_0\tau}<a_1$ and \eqref{3p2.3}, we derive 
\begin{align*}
&\Big(1+2a_1\theta\Delta\big(1-\int_{-\Delta}^{0}\frac{r+\Delta}{\Delta}\mathrm d\nu_2(r)\big)\Big)\E[|y^{\xi,\Delta}(t_1)-y^{\eta,\Delta}(t_1)|^2]\nn\\
\leq{}&\mathbb E[|z^{\xi,\Delta}(t_1)-z^{\eta,\Delta}(t_1)|^2]+2a_2\theta\Delta\big(\nu_2([-\tau,-\Delta])+\int_{-\Delta}^{0}\frac{-r}{\Delta}\mathrm d\nu_2(r)\big)\|\xi-\eta\|^2\nn\\
\leq{}&\Big(1+2a_2\theta\Delta\big(1-\int_{-\Delta}^{0}\frac{r+\Delta}{\Delta}\mathrm d\nu_2(r)\big)\Big)\Big(1\vee\big((L+2a_2)\tau e^{\l_0\tau}\big)\Big)\nn\\
&\times\big( |z^{\xi,\Delta}(0)-z^{\eta, \Delta}(0)|^2+\|\xi-\eta\|^2\big)e^{-\l_0 t_1}.
\end{align*}
It is clear that
\begin{align}\label{MIstep1}
&\E[|y^{\xi,\Delta}(t_1)-y^{\eta,\Delta}(t_1)|^2]\nn\\
\leq{}& (1\vee((L+2a_2)\tau e^{\l_0\tau}))\big( |z^{\xi,\Delta}(0)-z^{\eta, \Delta}(0)|^2+\|\xi-\eta\|^2\big)e^{-\l_0 t_1}.
\end{align}
Assume that for some integer  $k>1$ and $l\in\{1,\ldots,k\}$,
\begin{align}\label{MIstep2}
&\E[|y^{\xi,\Delta}(t_l)-y^{\eta,\Delta}(t_l)|^2]\nn\\
\leq{}&(1\vee((L+2a_2)\tau e^{\l_0\tau}))\big( |z^{\xi,\Delta}(0)-z^{\eta, \Delta}(0)|^2+\|\xi-\eta\|^2\big)e^{-\l_0 t_l}.
\end{align} 
Then for $k+1$, taking into account \eqref{2p2.3}, \eqref{1p2.3}, and \eqref{MIstep2}, we deduce 
\begin{align*}
&\Big(1+2a_1\theta\Delta\big(1-\int_{-\Delta}^{0}\frac{r+\Delta}{\Delta}\mathrm d\nu_2(r)\big)\Big)\E[|y^{\xi,\Delta}(t_{k+1})-y^{\eta,\Delta}(t_{k+1})|^2]\nn\\
\leq{}&\E[|z^{\xi,\Delta}(t_{k+1})-z^{\eta,\Delta}(t_{k+1})|^2]+2a_2\theta\Delta\big(1-\int_{-\Delta}^{0}\frac{r+\Delta}{\Delta}\mathrm d\nu_2(r)\big)\nn\\
&\times\sup_{k-N+1\leq l \leq k}\E[|y^{\xi,\Delta}(t_l)-y^{\eta,\Delta}(t_l)|^2]\nn\\
\leq{}&\big(1\vee((L+2a_2)\tau e^{\l_0\tau})\big)\Big(1+2a_2\theta e^{\lambda_0\tau}\Delta\big(1-\int_{-\Delta}^{0}\frac{r+\Delta}{\Delta}\mathrm d\nu_2(r)\big)\Big)\nn\\
&\times\big( |z^{\xi,\Delta}(0)-z^{\eta, \Delta}(0)|^2+\|\xi-\eta\|^2\big)e^{-\l_0 t_{k+1}},
\end{align*}
which implies 
\begin{align}\label{MIstep3}
&\E[|y^{\xi,\Delta}(t_{k+1})-y^{\eta,\Delta}(t_{k+1})|^2]\nn\\
\leq{}&(1\vee((L+2a_2)\tau e^{\l_0\tau}))\big( |z^{\xi,\Delta}(0)-z^{\eta, \Delta}(0)|^2+\|\xi-\eta\|^2\big)e^{-\l_0 t_{k+1}}.
\end{align} 
Hence, we conclude  from \eqref{MIstep1}--\eqref{MIstep3} that 
\begin{align}\label{4p2.3}
&\E[|y^{\xi,\Delta}(t_k)-y^{\eta,\Delta}(t_k)|^2]\nn\\
\leq{}&(1\vee((L+2a_2)\tau e^{\l_0\tau}))\big( |z^{\xi,\Delta}(0)-z^{\eta, \Delta}(0)|^2+\|\xi-\eta\|^2\big)e^{-\l_0 t_k}.
\end{align} 
Furthermore, by similar arguments to those in \underline{Step 2} of Proposition \ref{p2.2}, we use \eqref{2p2.3} and \eqref{4p2.3} to obtain 
\begin{align*}
&\E\Big[\sup_{(k-N)\vee0\leq i\leq k}e^{\l_0 t_{i+1}}|z^{\xi,\Delta}(t_{i+1})-z^{\eta,\Delta}(t_{i+1})|^2\Big]\nn\\
\leq{}& K\big( |z^{\xi,\Delta}(0)-z^{\eta, \Delta}(0)|^2+\|\xi-\eta\|^2\big)
+K\Delta\sum_{i=0}^{k}e^{\l_0 t_{i+1}}\E[|y^{\xi,\Delta}(t_i)-y^{\eta, \Delta}(t_i)|^2]\nn\\
&\quad+2a_2\theta\Delta\sum_{i=(k-2N)\vee(-N)}^{k}e^{\l_0 t_{i+1}}\E[|y^{\xi,\Delta}(t_i)-y^{\eta, \Delta}(t_i)|^2]\nn\\
\leq{}&K( |b(\xi)-b(\eta)|^2+\|\xi-\eta\|^2)(1+t_{k}).
\end{align*}
This implies that for any $\e\in(0,\l_0)$,
\begin{align}\label{5p2.3}
\E\Big[\sup_{(k-N)\vee0\leq i\leq k}|z^{\xi,\Delta}(t_{i+1})-z^{\eta,\Delta}(t_{i+1})|^2\Big]
\leq K( |b(\xi)-b(\eta)|^2+\|\xi-\eta\|^2)e^{(\l_0-\e) t_{k}}.
\end{align}
Without loss of generality, we redefine  $\l_0\in(0,\frac{(2\theta-1)c_{\l}}{\theta^2(1+c_{\lambda})}\wedge \l)$.
Then, it follows from \eqref{4p2.3} that
\begin{align}\label{5p2.3+}
\E\Big[\sup_{(k-N)\vee0\leq i\leq k}|z^{\xi,\Delta}(t_{i+1})-z^{\eta,\Delta}(t_{i+1})|^2\Big]
\leq K( |b(\xi)-b(\eta)|^2+\|\xi-\eta\|^2)e^{-\l_0 t_{k}}.
\end{align}
Using \eqref{3p2.3} leads to
\begin{align}\label{6p2.3}
&(1+2a_1\theta\Delta)\E\Big[\sup_{(k-N)\vee0\leq i\leq k}|y^{\xi,\Delta}(t_i)-y^{\eta,\Delta}(t_i)|^2\Big]\nn\\
\leq{}& \E\Big[\sup_{(k-N)\vee0\leq i\leq k}|z^{\xi,\Delta}(t_{i})-z^{\eta,\Delta}(t_{i})|^2\Big]\nn\\
&\quad+2a_2\theta\Delta\E\Big[\sup_{(k-2N)\vee(-N)\leq i\leq k}|y^{\xi,\Delta}(t_i)-y^{\eta,\Delta}(t_i)|^2\Big]\nn\\
\leq{}&\E\Big[\sup_{(k-N)\vee0\leq i\leq k}|z^{\xi,\Delta}(t_{i})-z^{\eta,\Delta}(t_{i})|^2\Big]\nn\\
&\quad+4a_2\theta\tau\sup_{(k-2N)\vee(-N)\leq i\leq k}\E[|y^{\xi,\Delta}(t_i)-y^{\eta,\Delta}(t_i)|^2].
\end{align}
Then the desired argument follows from \eqref{4p2.3} and \eqref{5p2.3+}. We finish the proof.
\end{proof}

Numerical stability is an important characteristic  to measure the reliability of numerical methods in long-term computation. Below, we present  the exponential stability of the trivial solution of \eqref{FF}. We propose the following  condition to ensure the existence of the trivial solution.

\begin{assp}\label{stab}
The drift and diffusion coefficients satisfy $b(\textup{\textbf 0})=\sigma(\textup{\textbf 0})=0$.
\end{assp}
It follows from Assumption \ref{stab} that $x^{\textbf 0}(t)\equiv0$, which implies
$x^{\textbf 0}_t\equiv\textbf 0$. By virtue of Lemma \ref{l1},  it is straightforward to obtain the exponential  stability of the  functional solution of \eqref{FF}.
\begin{thm}
Let Assumptions \ref{a1}, \ref{a2}, and \ref{stab} hold. Then the  functional solution of \eqref{FF} with the initial datum $\xi\in \mathcal C^d$  is mean-square exponential stable and satisfies
\begin{align*}
\E[\|x^{\xi}_{t}\|^2]\leq K (1+\|\xi\|^2)e^{-\lambda t},
\end{align*}
where the constant $\lambda$ satisfies \eqref{lamb}.
\end{thm}
Similarly, by virtue of Proposition \ref{p2.3}, we can  deduce that the $\theta$-EM method inherits the mean-square exponential stability.
\begin{thm}
Suppose that Assumptions \ref{a1}, \ref{a2}, and \ref{stab} hold. Then for any  $\Delta\in(0,1]$,  the numerical functional solution is mean-square exponential stable and satisfies
\begin{align*}
\E[\|y^{\xi,\Delta}_{t_k}\|^2]\leq K(1+|b(\xi)|^2+\|\xi\|^{2})e^{-\lambda_0 t_{k}},
\end{align*}
where the constant $\lambda_0$ is defined in Proposition \ref{p2.3}.
\end{thm}
\begin{rem}
In fact, if the coefficients $b$ and $\sigma$ are locally  Lipschitz continuous, and 
 Assumptions \ref{a1} and \ref{a2}  are relaxed as 
\begin{align*}
2\langle\phi(0),b(\phi)\rangle+|\sigma(\phi)|^2
\leq -d_1|\phi(0)|^{2}+d_2\int_{-\tau}^{0}|\phi(r)|^2\mathrm d\nu(r),
\end{align*}
 we can also deduce the above stability results of the exact functional solution and the numerical one,  where  $d_1$ and $d_2$ are positive constants with $d_1>d_2$, and $\nu$ is a probability measure on $[-\tau, 0]$. 
The proof is similar to  those of Lemma \ref{l1} and Proposition \ref{p2.3}.
\end{rem}

\section{Mean-square convergence rate}
In this subsection, we aim to prove the 
mean-square convergence rate of $\theta$-EM method,  based on the estimates of high-order moments.

\begin{lemma}\label{l4.5}
Let Assumptions \ref{a1}--\ref{a2}, \ref{a4}--\ref{a7} hold. Then,
\begin{align*}
\sup_{t\geq 0}\E[\|x^{\xi}_t-z^{\xi,\Delta}_t\|^2 ]\leq K(1+\|\xi\|^{2(1+\beta)(2\beta+2)})\Delta\quad \forall \,\xi\in \mathcal C^d.
\end{align*}
\end{lemma}
\begin{proof}
It follows from \eqref{FF} and \eqref{zcont} that
\begin{align*}
x^{\xi}(t)-z^{\xi,\Delta}(t)=\theta b({\xi})\Delta+\int_{0}^{t}b(x^{\xi}_s)-b(y^{\xi,\Delta}_s)\mathrm ds+\int_{0}^{t}\sigma(x^{\xi}_s)-\sigma(y^{\xi,\Delta}_s)\mathrm dW(s).
\end{align*}
By the It\^o formula and the Young inequality, we obtain  that for any $\e\in(0,1)$,
\begin{align*}
&e^{\e t}|x^{\xi}(t)-z^{\xi,\Delta}(t)|^2\nn\\
=&\;\theta^2 |b({\xi})|^2\Delta^2\!+\!\int_{0}^{t}e^{\e s}\Big(\e|x^{\xi}(s)\!-\!z^{\xi,\Delta}(s)|^2
+2\<x^{\xi}(s)\!-\!z^{\xi,\Delta}(s),b(x^{\xi}_s)\!-\!b(y^{\xi,\Delta}_s)\>\nn\\&\quad+|\sigma(x^{\xi}_s)-\sigma(y^{\xi,\Delta}_s)|^2\Big)\mathrm ds+\mathcal N_t\nn\\
\leq&\,K(1\!+\!\|\xi\|^{2(\beta+1)})\Delta^2\!\!+\!\!\int_{0}^{t}\!e^{\e s}\Big(\!\e|x^{\xi}(s)\!-\!z^{\xi,\Delta}(s)|^2
\!+\!2\<x^{\xi}(s)\!-\!z^{\xi,\Delta}(s),b(x^{\xi}_s)\!-\!b(z^{\xi,\Delta}_s)\>\nn\\
&\quad+(1+\e)|\sigma(x^{\xi}_s)-\sigma(z^{\xi,\Delta}_s)|^2\Big)\mathrm ds+\int_{0}^{t}\e e^{\e s}|x^{\xi}(s)-z^{\xi,\Delta}(s)|^2\mathrm ds+\mathcal E_t+\mathcal N_t,
\end{align*}
where 
\begin{align*}
\mathcal E_t:=\int_{0}^{t}e^{\e s}\Big(\frac{1}{\e}|b(z^{\xi,\Delta}_s)-b(y^{\xi,\Delta}_s)|^2+(1+\frac{1}{\e})|\sigma(z^{\xi,\Delta}_s)-\sigma(y^{\xi,\Delta}_s)|^2\Big)\mathrm ds,
\end{align*}
and $\mathcal N_{t}$ is a martingale defined by
\begin{align*}
\mathcal N_{t}:=2\int_{0}^{t}e^{\e s}\<x^{\xi}(s)-z^{\xi,\Delta}(s),(\sigma(x^{\xi}_s)-\sigma(y^{\xi,\Delta}_s))\mathrm dW(s)\>.
\end{align*}
According to Assumptions \ref{a1} and \ref{a2}, we derive
\begin{align*}
&e^{\e t}|x^{\xi}(t)-z^{\xi,\Delta}(t)|^2\nn\\
\leq&\,K(1+\|\xi\|^{2(\beta+1)})\Delta^2+\int_{0}^{t}\Big(-\big(2a_1-2\e-(1+\e)L\big)e^{\e s}|x^{\xi}(s)-z^{\xi,\Delta}(s)|^2\nn\\
&+2a_2\!\!\int_{-\tau}^{0}\!e^{\e s}|x^{\xi}_s(r)\!-\!z^{\xi,\Delta}_s(r)|^2\mathrm d\nu_2(r)+(1+\e)L\!\int_{-\tau}^{0}\!e^{\e s}|x^{\xi}_s(r)-z^{\xi,\Delta}_s(r)|^2\mathrm d\nu_1(r)\Big)\mathrm ds\nn\\
&+\mathcal E_t+\mathcal N_t\nn\\
\leq&\,K(1+\|\xi\|^{2(\beta+1)})\Delta^2-\big(2a_1-2\e-(1+\e)L-(2a_2 +(1+\e)L)e^{\e\tau}\big)\int_{0}^{t}e^{\e s}|x^{\xi}(s)\nn\\
&-z^{\xi,\Delta}(s)|^2\mathrm ds
+(2a_2+(1+\e)L)e^{\e\tau}\int_{-\tau}^{0}e^{\e s}|x^{\xi}(s)-z^{\xi,\Delta}(s)|^2\mathrm ds+\mathcal E_t+\mathcal N_t.
\end{align*}
In view of $a_1>a_2+L$, choose a sufficiently small number $\e\in(0,1)$ such that
$$2a_1-2\e-(1+\e)L-(2a_2 +(1+\e)L)e^{\e\tau}\geq0.$$
By Assumptions \ref{a5} and \ref{a7}, we have
$$(2a_2+(1+\e)L)e^{\e\tau}\int_{-\tau}^{0}e^{\e s}|x^{\xi}(s)-z^{\xi,\Delta}(s)|^2\mathrm ds\leq K(1+\|\xi\|^{2(\beta+1)})\Delta.$$
Thus,
\begin{align}\label{l4.5.2}
e^{\e t}|x^{\xi}(t)-z^{\xi,\Delta}(t)|^2
\leq K(1+\|\xi\|^{2(\beta+1)})\Delta+\mathcal E_t+\mathcal N_t.
\end{align}
 By virtue of Assumptions \ref{a1} and \ref{a5}, we apply the H\"older inequality to deduce
\begin{align*}
\E[\mathcal E_t]
\leq&K\E\Big[\int_{0}^{t}e^{\e s}
\|z^{\xi,\Delta}_s-y^{\xi,\Delta}_s\|^2(1+\|z^{\xi,\Delta}_s\|^{2\beta}+\|y^{\xi,\Delta}_s\|^{2\beta})\mathrm ds\Big]\nn\\
\leq&K \int_{0}^{t}e^{\e s}\big(\E[\|z^{\xi,\Delta}_s-y^{\xi,\Delta}_s\|^4]\big)^\frac{1}{2}\big(\E[1+\|z^{\xi,\Delta}_s\|^{4\beta}+\|y^{\xi,\Delta}_s\|^{4\beta}]\big)^{\frac{1}{2}}\mathrm ds.
\end{align*}
Making use of \eqref{4.3l+1} and Lemma \ref{l4.4} yields
\begin{align*}
\E[\mathcal E_t] &\leq K(1+\|\xi\|^{2(1+\beta)^2})(1+\|\xi\|^{2\beta(1+\beta)})\Delta\int_{0}^{t}e^{\e s}\mathrm ds\\
&\leq K(1+\|\xi\|^{2(1+\beta)(2\beta+2)})e^{\e t}\Delta.
\end{align*}
This, along with \eqref{l4.5.2} implies that
\begin{align}\label{l4.5.3}
\E[|x^{\xi}(t)-z^{\xi,\Delta}(t)|^2]\leq {}&e^{-\e t}\big(K(1+\|\xi\|^{2(1+\beta)})\Delta+\E [\mathcal E_t]\big)\nn\\
\leq{}& K(1+\|\xi\|^{2(1+\beta)(2\beta+2)})\Delta,
\end{align}
where the positive constant $K$ is independent of $\Delta$ and $t$.
Furthermore, according to \eqref{l4.5.2}, we apply the Burkholder--Davis--Gundy inequality to derive
\begin{align*}
&\E\Big[\sup_{(t-\tau)\vee 0 \leq u\leq t}e^{\e u}|x^{\xi}(u)-z^{\xi,\Delta}(u)|^2\Big]\nn\\
\leq &\,K(1+\|\xi\|^{2(1+\beta)})\Delta+\E[\mathcal E_t]+\E \Big[\sup_{(t-\tau)\vee 0 \leq u\leq t}\mathcal N_u\Big]\nn\\
\leq&\,K(1+\|\xi\|^{2(1+\beta)(2\beta+2)})e^{\e t}\Delta\\
&\,+K\E\Big[\big(\int_{(t-\tau)\vee 0}^{t}e^{2\e s}|x^{\xi}(s)-z^{\xi,\Delta}(s)|^2|\sigma(x^{\xi}_s)-\sigma(y^{\xi,\Delta}_s)|^2\mathrm ds\big)^{\frac{1}{2}}\Big]\nn\\
\leq&\,K(1+\|\xi\|^{2(1+\beta)(2\beta+2)})e^{\e t}\Delta+K\E\Big[\Big(\big(\sup_{(t-\tau)\vee 0 \leq u\leq t}e^{\e u}|x^{\xi}(u)-z^{\xi,\Delta}(u)|^2\big)\nn\\
&\,\times\int_{(t-\tau)\vee 0}^{t}e^{\e s}|\sigma(x^{\xi}_s)-\sigma(y^{\xi,\Delta}_s)|^2\mathrm ds\Big)^{\frac{1}{2}}\Big].
\end{align*}
By virtue of the Young inequality and Assumption \ref{a1},  we arrive at
\begin{align*}
&\E\Big[\sup_{(t-\tau)\vee 0 \leq u\leq t}e^{\e u}|x^{\xi}(u)-z^{\xi,\Delta}(u)|^2\Big]\nn\\
\leq{}&K(1+\|\xi\|^{2(1+\beta)(2\beta+2)})e^{\e t}\Delta+\frac{1}{2}\E\Big[\sup_{(t-\tau)\vee 0 \leq u\leq t}e^{\e u}|x^{\xi}(u)-z^{\xi,\Delta}(u)|^2\Big]\nn\\
&\quad+K\E\Big[\int_{(t-\tau)\vee 0}^{t}e^{\e s}|\sigma(x^{\xi}_s)-\sigma(y^{\xi,\Delta}_s)|^2\mathrm ds\Big]\nn\\
\leq& \,K(1+\|\xi\|^{2(1+\beta)(2\beta+2)})e^{\e t}\Delta+\frac{1}{2}\E\Big[\sup_{(t-\tau)\vee 0 \leq u\leq t}e^{\e u}|x^{\xi}(u)-z^{\xi,\Delta}(u)|^2\Big]\nn\\
&\quad+K\E\Big[\int_{(t-\tau)\vee 0}^{t}e^{\e s}|\sigma(x^{\xi}_s)-\sigma(z^{\xi,\Delta}_s)|^2\mathrm ds\Big]\nn\\
&\quad+K\E\Big[\int_{(t-\tau)\vee 0}^{t}e^{\e s}|\sigma(z^{\xi,\Delta}_s)-\sigma(y^{\xi,\Delta}_s)|^2\mathrm ds\Big]\nn\\
\leq&\,K(1+\|\xi\|^{2(1+\beta)(2\beta+2)})e^{\e t}\Delta+\frac{1}{2}\E\Big[\sup_{(t-\tau)\vee 0 \leq u\leq t}e^{\e u}|x^{\xi}(u)-z^{\xi,\Delta}(u)|^2\Big]\nn\\
&\quad+K\int_{(t-\tau)\vee 0}^{t}e^{\e s}\E[|x^{\xi}(s)-z^{\xi,\Delta}(s)|^2]\mathrm ds+K\int_{(t-\tau)\vee 0}^{t}e^{\e s}\E[\|z^{\xi,\Delta}_s-y^{\xi,\Delta}_s\|^2]\mathrm ds
\nn\\
&\quad+\int_{(t-\tau)\vee 0}^{t}e^{\e s}\int_{-\tau}^{0}\E[|x^{\xi}_s(r)-z^{\xi,\Delta}_s(r)|^2]\mathrm d \nu_{1}(r)\mathrm ds.
\end{align*}
It follows from \eqref{l4.5.3} and Lemma \ref{l4.4} that
\begin{align*}
&\frac{1}{2}e^{\e(t-\tau)\vee 0}\E\Big[\sup_{(t-\tau)\vee 0 \leq u\leq t}|x^{\xi}(u)-z^{\xi,\Delta}(u)|^2\Big]\nn\\
\leq {}&\frac{1}{2}\E\Big[\sup_{(t-\tau)\vee 0 \leq u\leq t}e^{\e u}|x^{\xi}(u)-z^{\xi,\Delta}(u)|^2\Big]\nn\\
\leq{}&K(1+\|\xi\|^{2(1+\beta)(2\beta+2)})e^{\e t}\Delta+K(1+\|\xi\|^{2(1+\beta)(2\beta+2)})\Delta\int_{(t-\tau)\vee 0}^{t}e^{\e s}\mathrm ds\nn\\
\leq{}& K(1+\|\xi\|^{2(1+\beta)(2\beta+2)})e^{\e t}\Delta,
\end{align*}
which implies
$$\E\Big[\sup_{(t-\tau)\vee 0 \leq u\leq t}|x^{\xi}(u)-z^{\xi,\Delta}(u)|^2\Big]\leq {}K(1+\|\xi\|^{2(1+\beta)(2\beta+2)})\Delta.$$
The proof is completed.
\end{proof}

Utilizing  Lemmas \ref{l4.4} and \ref{l4.5}, we conclude the following mean-square convergence rate  of the $\theta$-EM functional  solution. 
\begin{thm}\label{th4.1}
Let Assumptions \ref{a1}--\ref{a2}, \ref{a4}--\ref{a7} hold. Then for $\Delta\in(0,1]$ and $\xi\in \mathcal C^d$, 
\begin{align*}
\sup_{k\geq 0}\E[\|x^{\xi}_{t_k}-y^{\xi,\Delta}_{t_k}\|^2] \leq K(1+\|\xi\|^{2(1+\beta)(2\beta+2)})\Delta.
\end{align*}
\end{thm}

In view of Proposition \ref{l4.3}, Proposition \ref{p2.3}, Theorem \ref{th4.1}, and \cite[Theorems 3.1--3.2]{chen2023probabilistic}, we present the strong law of large numbers (SLLN) and the central limit theorem (CLT) of the 
$\theta$-EM method.   Define the empirical measure as $\Pi^{m,\Delta}:=\frac{1}{m}\sum_{k=0}^{m-1}\delta_{y^{\xi,\Delta}_{t_k}}$ with $\delta$ being the Dirac function. $\Pi^{m,\Delta}(f):=\frac{1}{m}\sum_{k=0}^{m-1}f(y^{\xi,\Delta}_{t_k})$ is called the time-average of the $\theta$-EM functional solution. 
\begin{coro}\label{coro_SLLN}
Under Assumptions \ref{a1}--\ref{a2} and  \ref{a4}--\ref{a7},  the numerical functional solution $\{y^{\xi,\Delta}_{t_k}\}_{k\ge 0}$ fulfills the  strong law of large numbers (SLLN) and the central limit theorem (CLT),  namely,  for any $p\ge 1,\gamma\in(0,1],$ and $f\in\mathcal C_{p,\gamma}$,
\begin{subequations}
\begin{align}
&\lim_{\Delta\to0}\lim_{m_1\to\infty}\frac{1}{m_1}\sum_{k=0}^{m_1-1}f(y^{\xi,\Delta}_{t_k})=\mu(f)\quad a.s., \tag{SLLN}\\
&\lim_{\Delta\to0}\frac{1}{\sqrt{m_2\Delta}}\sum_{k=0}^{m_2-1}(f(y^{\xi,\Delta}_{t_k})-\mu(f))\Delta=\mathcal N(0,v^2)\quad \text{in distribution}, \tag{CLT}
\end{align}
\end{subequations}
where $m_2=\lceil \Delta^{-1-2\lambda}\rceil$ with $\lambda\in(0,\frac12\gamma),$ $v^2=2\mu\big((f-\mu(f))\int_0^{\infty}(P_tf-\mu(f))\big).$ Here, $\mathcal C_{p,\gamma}:=\mathcal C_{p,\gamma}(\mathcal C^d;\mathbb R)$ is the  space of continuous functions $f$ satisfying $$\sup_{\phi\in\mathcal C^d}\frac{|f(\phi)|}{1\!+\!\|\phi\|^{\frac p2}}+\!\sup_{\substack{\phi_1,\phi_2\in\mathcal C^d\\\phi_1\neq \phi_2}}\frac{|f(\phi_1)-f(\phi_2)|}{(1\wedge \|\phi_1-\phi_2\|^{\gamma})(1+\|\phi_1\|^p+\|\phi_2\|^p)^{\frac12}}<\infty.$$
\end{coro}

\section{Summary and outlook}
In this chapter, we present the longtime mean-square convergence analysis for the $\theta$-EM method for the SFDE \eqref{FF} with the superlinearly growing drift coefficient. The  mean-square convergence rate is proved to be $\frac12.$ The time-independent boundedness of high-order moment of the numerical solution is the key to obtain the desired mean-square convergence result.  
There are  some rooms to loose the conditions proposed in this chapter. 
For example, if we weaken  Assumption \ref{a5} to 
\begin{align*}
\langle b(\phi),\phi(0)\rangle \leq K-a_3|\phi(0)|^{\alpha+\epsilon}+a_4\int_{-\tau}^0|\phi(r)|^{\alpha}\mathrm d\nu_2(r),\quad \phi\in\mathcal C^d,
\end{align*}
the mean-square convergence result  still holds based on  
 the approach used in this chapter, where $\epsilon>0,\alpha\ge 2, a_3,a_4>0$. 
When both the drift and the diffusion coefficients are superlinearly growing, there are some longtime mean-square convergence  results of  numerical methods for stochastic ordinary differential equations (for BEM method see \cite{LXW23}).  However, less is known for the  SFDE case.

As applications of the longtime mean-square convergence result,  we present in Corollary \ref{coro_SLLN} the SLLN and the CLT of the time-average of the  $\theta$-EM functional solution. Following this line, 
it is interesting to further study the probabilistic asymptotics that  the empirical measure  deviates from the invariant measure. 
We leave this as our future work.

    \cleardoublepage
%
%
%


\chapter{Invariant measure and weak convergence analysis}

In this chapter, we investigate the existence and uniqueness of the invariant measure of the $\theta$-EM method.
In order to give a quantitative estimate between  the numerical invariant measure and the exact one, we analyze the longtime weak convergence  of the $\theta$-EM method. Finally, it is shown that the convergence rate of the invariant measure of the $\theta$-EM method is $1$, which is twice the mean-square convergence rate.  

\section{Numerical invariant measure}
In this  section, we  show that   the $\theta$-EM functional  solution admits a unique invariant measure. 
 Recall
$$y^{\xi,\Delta}_{t}:=\sum_{k=0}^{\infty}y^{\xi,\Delta}_{t_{k}}\mathbf 1_{[t_k,t_{k+1})}(t)$$
and define the measure $\mu_{t}^{\xi,\Delta}$ by
\begin{align*}
\mu_{t}^{\xi,\Delta}(A):=\PP\big\{\omega\in\Omega: y_{t}^{\xi,\Delta}\in A\big\}\quad \forall\, A\in\mathcal{B}(\mathcal C^d).
\end{align*}
Denote by $P_{t}^{\xi, \Delta}$ the Markov semigroup associated to $\mu^{\xi,\Delta}_{t}$.
According to Propositions \ref{p2.2} and \ref{p2.3}, we can obtain the tightness of the sequence $\{\mu_{t_k}^{\xi,\Delta}\}_{k=0}^{\infty}$.
\begin{lemma}\label{lemm2.1}
Under Assumptions \ref{a1} and \ref{a2}, for $\xi\in \mathcal C^d$ and $\Delta\in(0,1]$, the family $\{\mu_{t_k}^{\xi,\Delta}\}_{k=0}^{\infty}$ is tight. 
\end{lemma}
\begin{proof}
Since $\mathcal C^d$ is a Polish space and 
$$\sup_{k\in\mathbb N}\E[\|y_{t_k}^{\xi,\Delta}\|^2]\leq K(1+\|\xi\|^2+|b(\xi)|^2\Delta^2),$$
and by virtue of \cite[Lemma 6.14]{Villani}, it suffices  to prove that  the family $\{\mu_{t_k}^{\xi,\Delta}\}_{k=0}^{\infty}$ is a Cauchy sequence in the space $(\mathcal P, \mathbb{W}_1)$. 
In fact, for any $k,i\in\mathbb N$, applying \cite[Remark 6.5]{Villani} leads to
\begin{align*}
&\quad\mathbb{W}_1(\mu_{t_{k+i}}^{\xi,\Delta},\mu_{t_k}^{\xi,\Delta})\nn\\
&=\sup_{\|\Psi\|_{Lip}\vee\|\Psi\|_{\infty}\leq1}\Big|\int_{\mathcal C^d}\Psi(\phi)\mathrm d\mu_{t_{k+i}}^{\xi,\Delta}(\phi)-\int_{\mathcal C^d}\Psi(\phi)\mathrm d\mu_{t_k}^{\xi,\Delta}(\phi)\Big|,
\end{align*}
where $\Psi: \mathcal C^d\rightarrow\RR$  satisfies the Lipschitz condition and norms $\|\cdot\|_{Lip}$  and $\|\cdot\|_{\infty}$ are defined, respectively,  by
$$\|\Psi\|_{Lip}=\sup_{\substack{\phi_1,\phi_2\in \mathcal C^d\\
\phi_1\neq \phi_2}}\frac{|\Psi(\phi_1)-\Psi(\phi_2)|}{\|\phi_1-\phi_2\|}\;\text{ and }\;\|\Psi\|_{\infty}=\sup_{\phi\in \mathcal C^d}\frac{|\Psi(\phi)|}{\|\phi\|}.$$
By the definition of $\mu_{\cdot}^{\xi,\Delta}$ and the time-homogeneous Markov property of $\{y^{\xi,\Delta}_{t_{k}}\}_{k=0}^{\infty}$, we deduce
\begin{align*}
&\mathbb{W}_1(\mu_{t_{k+i}}^{\xi,\Delta},\mu_{t_k}^{\xi,\Delta})
=\sup_{\|\Psi\|_{Lip}\vee\|\Psi\|_{\infty}\leq1}|\E[\Psi(y^{\xi,\Delta}_{t_{k+i}})]-\E[\Psi(y^{\xi,\Delta}_{t_{k}})]|\nn\\
=&\,\sup_{\|\Psi\|_{Lip}\vee\|\Psi\|_{\infty}\leq1}\Big|\E\Big[\E\big[\Psi(y^{\xi,\Delta}_{t_{k+i}})|\mathcal F_{t_i}\big]\Big]-\E[\Psi(y^{\xi,\Delta}_{t_{k}})]\Big|\nn\\
=&\,\sup_{\|\Psi\|_{Lip}\vee\|\Psi\|_{\infty}\leq1}\Big|\E\Big[\E\big[\Psi(y^{\zeta,\Delta}_{t_{k}})\big]\big|_{\zeta=y^{\xi,\Delta}_{t_{i}}}\Big]-\E[\Psi(y^{\xi,\Delta}_{t_{k}})]\Big|\nn\\
=&\,\sup_{\|\Psi\|_{Lip}\vee\|\Psi\|_{\infty}\leq1}\Big|\E\Big[\E\big[\Psi(y^{\zeta,\Delta}_{t_{k}})-\Psi(y^{\xi,\Delta}_{t_{k}})\big]\big|_{\zeta=y^{\xi,\Delta}_{t_{i}}}\Big]\Big|.
\end{align*}
It follows from the Lipschitz property of $\Psi$ and Proposition \ref{p2.3}  that for any constant $M>\|\xi\|$,
\begin{align*}
&\mathbb{W}_1(\mu_{t_{k+i}}^{\xi,\Delta},\mu_{t_k}^{\xi,\Delta})
\leq\E\Big[\E\big[2\wedge\|y^{\zeta,\Delta}_{t_{k}}-y^{\xi,\Delta}_{t_{k}}\|\big]\big|_{\zeta=y^{\xi,\Delta}_{t_{i}}}\Big]\nn\\
\leq&\;K\E\Big[\big(|b(y^{\xi,\Delta}_{t_{i}})-b(\xi)|+\|y^{\xi,\Delta}_{t_{i}}-\xi\|\big)\textbf 1_{\{\|y^{\xi,\Delta}_{t_{i}}\|\leq M\}}\Big]e^{-\frac{\lambda_0 t_k}{2}}
+2\PP\{\|y^{\xi,\Delta}_{t_{i}}\|> M\}.
\end{align*}
According to the Chebyshev inequality and Proposition \ref{p2.2}, we have
\begin{align*}
\mathbb{W}_1(\mu_{t_{k+i}}^{\xi,\Delta},\mu_{t_k}^{\xi,\Delta})
\leq&\,K\sup_{\|\phi\|\leq M}\big(|b(\phi)|+\|\phi\|\big)e^{-\frac{\lambda_0 t_k}{2}}
+2\frac{\E[\|y^{\xi,\Delta}_{t_{i}}\|^2]}{M^2}\nn\\
\leq&\,K\sup_{\|\phi\|\leq M}\big(|b(\phi)|+\|\phi\|\big)e^{-\frac{\lambda_0 t_k}{2}}
+\frac{K(1+\|\xi\|^2+|b(\xi)|^2\Delta^2)}{M^2}.
\end{align*}
For any $\e>0$, choose a sufficiently large number $M>0$  such that
$$\frac{K(1+\|\xi\|^2+|b(\xi)|^2\Delta^2)}{M^2}<\frac{\e}{2}.$$
By virtue of the continuity of $b(\cdot)$, choose a sufficiently large number $k\in\mathbb N$ such that 
$$ K\sup_{\|\phi\|\leq M}\big(|b(\phi)|^2+\|\phi\|^2\big)e^{-\frac{\lambda_0 t_k}{2}}<\frac{\e}{2}.$$
Hence, the family $\{\mu_{t_k}^{\xi,\Delta}\}_{k=0}^{\infty}$ is a Cauchy sequence in  $(\mathcal P, \mathbb{W}_1)$. The proof is finished.
\end{proof}

Taking advantage of  Proposition \ref{p2.3} and Lemma \ref{lemm2.1}, we give the existence and uniqueness  of the numerical invariant measure of the $\theta$-EM functional solution.

\begin{thm}
Under Assumptions \ref{a1} and {\ref{a2}}, for $\Delta\in(0,1]$, the $\theta$-EM functional solution admits a unique invariant measure $\mu^{\Delta}$. Furthermore, for any $\Delta\in(0,1]$ the numerical invariant  measure $\mu^{\Delta}$ satisfies 
\begin{align}\label{4.3l+2}
\sup_{\Delta\in(0,1]}\mu^{\Delta}(\|\cdot\|^p):=\sup_{\Delta\in(0,1]}\int_{\mathcal C^d}\|\phi\|^{p}\mathrm d\mu^{\Delta}(\phi)\leq K.
\end{align}
\end{thm}
\begin{proof}
We divide the proof of the existence and uniqueness of the numerical invariant measure into two steps.

\underline{Step 1.} We present the existence of numerical invariant measure based on the Krylov--Bogoliubov theorem \cite[Theorem 3.1.1]{DaPrato1996}.
For any $\xi\in \mathcal C^d$, $\Delta\in(0,1]$, and $k\in\mathbb N$, 
it follows from Lemma \ref{lemm2.1} that   the family $\{\mu^{\xi,\Delta}_{t_{k}}\}_{k=0}^{\infty}$ is tight.
Hence we only need to prove that  $P_{t_k}^{\xi,\Delta}$ is Feller, i.e.,  to show that the mapping
$$\xi\rightarrow  P_{t_k}^{\xi,\Delta}(\Phi)=\E[\Phi(y^{\xi,\Delta}_{t_k})]$$ 
is continuous for any bounded continuous functional $\Phi: \mathcal C^d\rightarrow\RR$.
In fact, for $\xi,\eta\in\mathcal C^d,$ by Proposition \ref{p2.3}, we have  
\begin{align*}
y^{\xi,\Delta}_{t_k}\rightarrow y^{\eta,\Delta}_{t_k}\quad\mbox{in} ~L^2(\Omega; \mathcal C^d)
\end{align*}
as $\xi\rightarrow\eta$. 
By \cite[pp.256, Theorem 2]{Shiryaev}, we derive
\begin{align*}
y^{\xi,\Delta}_{t_k}\rightarrow y^{\eta,\Delta}_{t_k}\quad\mbox{in distribution}
\end{align*}
as $\xi\rightarrow\eta,$
which implies the Feller property.
Hence, for any $\xi\in \mathcal C^d$ and $\Delta\in(0,1]$, the $\theta$-EM functional solution admits an  invariant measure $\mu^{\Delta}$.

\underline{Step 2.} We shall   show that for any $\Delta\in(0,1]$, the numerical invariant measure is unique. 
Assume that there exist two numerical invariant measures $\mu^{\Delta}_1$ and $\mu^{\Delta}_2$. Then in view of the convex property of $\mathbb W_2$ and Proposition \ref{p2.3}, we obtain that for any $k\in\mathbb N$,
\begin{align}\label{th3.1.1}
&\mathbb W^2_2(\mu^{\Delta}_1,\mu^{\Delta}_2)\nn\\
\leq &\int_{\mathcal C^d\times \mathcal C^d}\mathbb W^2_2(\mu^{\xi,\Delta}_{t_k},\mu^{\eta,\Delta}_{t_k})\mathrm d\mu^{\Delta}_{1}(\xi)\mathrm d\mu^{\Delta}_{2}(\eta)\nn\\
\leq&\int_{\mathcal C^d\times \mathcal C^d}1\wedge\Big( K\big(|b(\xi)-b(\eta)|^2+\|\xi-\eta\|^2\big)e^{-\lambda_0 t_k}\Big)\mathrm d\mu^{\Delta}_{1}(\xi)\mathrm d\mu^{\Delta}_{2}(\eta).
\end{align}
It follows from the dominated convergence theorem  that 
$$\lim_{k\rightarrow \infty}\mathbb W_2(\mu^{\Delta}_1,\mu^{\Delta}_2)=0,$$
which finishes the proof of \underline{Step 2}.

In addition, by means of Proposition \ref{l4.3}, we have 
\begin{align*}
\int_{C}\|\phi\|^{p}\mathrm d\mu^{\Delta}(\eta)
=&\lim_{M\rightarrow +\infty}\lim_{k\rightarrow \infty}\E \big[M\wedge\|y^{0,\Delta}_{t_k}\|^{p}\big]\nn\\
\leq&\, \sup_{k\in \mathbb N}\E[\|y^{0,\Delta}_{t_k}\|^{p}]\leq  K.
\end{align*}
The proof is therefore finished. 
\end{proof}
Taking advantage of Theorem \ref{th4.1}, we estimate the error between the numerical invariant measure and the exact one.
\begin{thm}\label{th4.2}
Let Assumptions \ref{a1}, \ref{a2},  \ref{a4}--\ref{a7} hold. Then for any $\Delta\in(0,1]$, 
\begin{align*}
\mathbb W_2(\mu,\mu^{\Delta})\leq K \Delta^{\frac{1}{2}}.
\end{align*}
\end{thm}
\begin{proof}
For any $\Delta\in(0,1]$ and $k\in\mathbb N$, by virtue of Lemma \ref{l2.4} and Theorem  \ref{th4.1}, we obtain 
\begin{align}\label{th4.2.2}
\mathbb W_2(\mu,\mu^{\Delta})
&\leq\mathbb W_2(\mu,\mu^{\textbf{0}}_{t_k})+\mathbb W_2(\mu^{\textbf{0}}_{t_k},\mu^{\textbf{0},\Delta}_{t_k})+\mathbb W_2(\mu^{\textbf{0},\Delta}_{t_k},\mu^{\Delta})\nn\\
&\leq Ke^{-\frac{\lambda}{2} t_{k}}+\big(\E[\|x^{\textbf{0}}_{t_{k}}-y^{\textbf{0},\Delta}_{t_{k}}\|^2]\big)^{\frac{1}{2}}+\mathbb W_2(\mu^{\textbf{0},\Delta}_{t_k},\mu^{\Delta})\nn\\
&\leq Ke^{-\frac{\lambda}{2} t_{k}}+K\Delta^{\frac{1}{2}}+\mathbb W_2(\mu^{\textbf{0},\Delta}_{t_k},\mu^{\Delta}),
\end{align}
where the positive constant $\lambda$ satisfies \eqref{lamb}. 
Similar to the proof of \eqref{th3.1.1}, and  by Assumption \ref{a5},
 we have 
\begin{align*}
\mathbb W_2(\mu^{\textbf{0},\Delta}_{t_k},\mu^{\Delta})
\leq&\,\Big(\int_{\mathcal C^d}K\big(|b({\textbf{0}})-b(\eta)|^2+\|\eta\|^2\big)e^{-\lambda_0 t_k}\mathrm d\mu^{\Delta}(\eta)\Big)^{\frac{1}{2}}\nn\\
\leq&\, Ke^{-\frac{\lambda_0}{2} t_k}\Big(\int_{\mathcal C^d}\big(1+\|\eta\|^{2(\beta+1)}\big)\mathrm d\mu^{\Delta}(\eta)\Big)^{\frac{1}{2}}\nn\\
\leq &\,K e^{-\frac{\lambda_0}{2} t_k}.
\end{align*}
This, along with  \eqref{th4.2.2} implies that
\begin{align*}
\mathbb W_2(\mu,\mu^{\Delta})
\leq K e^{-\frac{\lambda}{2}t_k}+K\Delta^{\frac{1}{2}}+K e^{-\frac{\lambda_0}{2}t_k}.
\end{align*}
Letting $k\rightarrow \infty$, we obtain the desired argument.
\end{proof}

\section{Weak convergence rate}
In this section, we present   the 
longtime weak convergence analysis   of the $\theta$-EM method based on the  Malliavin calculus, and  show that  the  convergence rate of the numerical invariant measure is $1.$ The key in the analysis  is to present time-independent  
estimates  of the Malliavin derivatives and the G\^ateaux derivatives of both the exact solution and the numerical solution.
\subsection{Preliminaries} This subsection is devoted to giving  a brief introduction to the Malliavin calculus and presenting the decomposition of the  weak error for  the $\theta$-EM method. 
\subsubsection{Brief introduction to Malliavin calculus}
Let $H$ be the Hilbert space $L^2([0,+\infty);\mathbb R^m)$ endowed with the  inner product $\<g,h\>_H:=\int_0^{\infty}g(t)^{\top}h(t)\mathrm dt$ for $g,h\in H$. 
By identifying $W(t,\omega)$ with the value $\omega(t)$ at time $t$ of an element $\omega\in\mathcal C_0([0,+\infty);\mathbb R^m),$ we take $\Omega=\mathcal C_0([0,+\infty);\mathbb R^m)$ as the Wiener space and $\tilde{\mathbb P}$ as the Wiener measure. For $g=(g^1,\ldots,g^m)\in H,$ we set $$W(g):=\sum_{k=1}^m\int_0^{\infty}g^k(t)\mathrm dW^k(t).$$ Denote by $\mathcal S$ the class of smooth random variables such that $G\in\mathcal S$ has the form $G=f(W(g_1),\ldots,W(g_n)),$ where $f\in\mathcal C_{pol}^{\infty}(\mathbb R^n;\mathbb R),g_i\in H,i=1,\ldots,n.$ Here, $\mathcal C_{pol}^{\infty}(\mathbb R^n;\mathbb R)$ is the space of all real-valued smooth functions on $\mathbb R^n$ whose partial derivatives have at most polynomial growth.  
The Malliavin derivative of  a smooth random variable $G$ is an $H$-valued random variable given by $$DG=\sum_{i=1}^n\frac{\partial f}{\partial x_i}(W(g_1),\ldots,W(g_n))g_i,$$ which is also an $m$-dimensional stochastic process $DG=\{D_rG,r\ge 0\}$ with $$D_rG=\sum_{i=1}^n\partial _if(W(g_1),\ldots,W(g_n))g_i(r).$$
For any $p\ge 1,$ we denote the domain of $D$ in $L^p(\Omega)$ by $\mathbb D^{1,p},$ meaning that $\mathbb D^{1,p}$ is the closure of $\mathcal S$ with respect to the norm $\|G\|_{1,p}:=(\mathbb E[|G|^p+\|DG\|_H^p])^{\frac 1p}$. 
For $F\in\mathbb D^{1,2}(\mathbb R^d),$ the Malliavin covariance matrix is the  symmetric nonnegative random matrix defined by 
\begin{align*}
\gamma _F:=(\<D F^i,DF^j\>_H)_{1\leq i,j\leq d}.
\end{align*}

The iterated derivative $D^{\alpha}G$ is a random variable with values in $H^{\otimes \alpha}.$ Precisely, for $\alpha\in\mathbb N_+,$ $D^{\alpha}G=\{D_{r_1,\ldots,r_{\alpha}}G,\, r_i\ge 0,\,i=1,\ldots,\alpha\}$ is a measurable function on the product space $[0,+\infty)^{\alpha}\times \Omega.$ For any $p\ge 1$ and $\alpha\in\mathbb N_+,$ denote by $\mathbb D^{\alpha,p}$ the completion of $\mathcal S$ with respect  to the norm 
\begin{align*}
\|G\|_{\alpha,p}:=\Big(\mathbb E\Big[|G|^p+\sum_{j=1}^{\alpha}\|D^jG\|^p_{H^{\otimes j}}\Big]\Big)^{\frac1p}.
\end{align*}
For $G\in\mathbb D^{\alpha,p}$ and $g\in H,$ define the directional derivative by $D^{g}G:=\<DG,g\>_H$ and by induction, $D^{g_1,\ldots,g_j}G:=\<DD^{g_1,\ldots,g_{j-1}}G,g_j\>_{H},j=2,\ldots,\alpha,$ where $g_i\in H,\,i=1,\ldots,\alpha.$ 

Define $L^{\infty-}(\Omega):=\cap_{p\ge 1}L^p(\Omega),\; \mathbb D^{\alpha,\infty}:=\cap_{p\ge 1}\mathbb D^{\alpha,p},\; \mathbb D^{\infty}:=\cap_{p\ge 1}\cap_{\alpha\ge 1}\mathbb D^{\alpha,p}.$ 
Similarly, let $V$ be a real separable Hilbert space and define the space $\mathbb D^{\alpha,p}(V)$ as the completion of $V$-valued smooth random variables with respect to the norm
\begin{align*}
\|G\|_{\alpha,p,V}=\Big(\mathbb E\Big[\|G\|_V^p+\sum_{j=1}^{\alpha}\|D^jG\|^p_{H^{\otimes j}\otimes V}\Big]\Big)^{\frac1p}.
\end{align*}
In this case, the corresponding spaces are denoted by $L^{\infty-}(\Omega;V),\mathbb D^{\alpha,\infty}(V)$ and $\mathbb D^{\infty}(V),$ respectively. 

Let $D^*$ be the adjoint operator of $D.$
If an $H$-valued random variable $\phi\in L^2(\Omega;H)$ satisfies $|\mathbb E[\langle \phi,DX\rangle_H]|\leq K(\phi)\|X\|_{L^2(\Omega)}$ for all $X\in\mathbb D^{1,2},$ then $\phi\in\mathrm{Dom} (D^*)$ and $D^*(\phi)\in L^2(\Omega)$ is characterized by  
\begin{align}\label{IBPformula}
\mathbb E[\langle \phi,DX\rangle_H]=\mathbb E[X D^*(\phi)].
\end{align}
We refer to \cite[Definition 1.3.1]{Nualart} for more details.

We introduce the functional version of the chain rule for Malliavin derivatives, see e.g. \cite{Ma_functional} for details.

\begin{lemma}\label{Malliavin_SFDE}
Let $G:\mathcal C^d\to\mathbb R^d$ be a map with bounded continuous Fr\'echet derivatives of first and second order. Let $x_t(r)=x(t+r),r\in[-\tau,0],\,t\ge 0$ be a $d$-dimensional  continuous $\mathcal F_t$-adapted process such that $x_t(r)\in\mathbb D^{1,2}$. If $\mathbb E\Big[|G(x_t)|^2+\|\mathscr DG(x_t)(Dx_t)\|_H^2\Big]<\infty,$ then 
$G(x_t)\in\mathbb D^{1,2}$ and $$D_{t_1}[G(x_t)]=\mathscr  D G(x_t)(D_{t_1}x_t(\cdot)),$$ where $D_{t_1}x_t(\cdot) \in\mathcal C^d$ denotes the mapping $r\to D_{t_1}x_t(r)=D_{t_1}x(t+r)$ and $\mathscr  DG$ is the Fr\'echet derivative of $G.$ 
\end{lemma}

\subsubsection{Decomposition of the weak error} 
Without loss of generality, let $T$ be a multiple of $\tau$, i.e., there exists a number $N^{\Delta}\in\mathbb N_+$ such that $T=t_{N^\Delta}$. For any $Y(\cdot)\in\mathcal C([t_{i},t_{k}]; \RR^d)$ with $i,k\in\{-N,\ldots, N^{\Delta}\}$, by $Y^{Int}(\cdot)$ we denote the  linear interpolation with respect to $Y(t_i),\ldots, Y(t_k)$. Let  $\varphi(t;t_i,\eta)$ and   $\varphi_t(t_i,\eta)$ denote the solution  and functional solution of the SFDE  at time $t$ with initial value $\eta\in\mathcal C^d$ at $t_i$, $\Phi(t_k;t_i,\eta^{Int})$ and $\Phi_{t_k}(t_i,\eta^{Int})$ denote the solution and functional solution of the $\theta$-EM method at time $t_k$ with initial value $\eta^{Int}$ at $t_i$.
Let $f\in \mathcal C_b^3(\RR^d;\RR)$.
We have the following decomposition 
\begin{align}\label{decomp1}
&\quad \mathbb E[f(x(T))]-\mathbb E[f(y^{\xi,\Delta}(T))]=\mathbb E[f(\varphi(T;0,\xi))]-\mathbb E[f(\Phi(T;0,\Phi_0))]\nn\\
&=\mathbb E[f(\Phi(T;t_{N^{\Delta}},\varphi_{t_{N^{\Delta}}}^{Int}(0,\xi)))]-\mathbb E[f(\Phi(T;0,\Phi_0))]\nn\\
&=\mathbb E[f(\Phi(T;t_{N^{\Delta}},\varphi_{t_{N^{\Delta}}}^{Int}(0,\xi)))]-\mathbb E[f(\Phi(T;t_{N^{\Delta}},\varphi_{t_{N^{\Delta}}}^{Int}(0,\Phi_0)))]\nn\\
&\quad +\mathbb E[f(\Phi(T;t_{N^{\Delta}},\varphi_{t_{N^{\Delta}}}^{Int}(0,\Phi_0)))]-\mathbb E[f(\Phi(T;0,\Phi_0))].
\end{align}

We split the  term $\mathbb E[f(\Phi(T;t_{N^{\Delta}},\varphi_{t_{N^{\Delta}}}^{Int}(0,\Phi_0)))]-\mathbb E[f(\Phi(T;0,\Phi_0))]$ further  as
\begin{align}\label{int_all}
&\quad \mathbb E[f(\Phi(T;t_{N^{\Delta}},\varphi_{t_{N^{\Delta}}}^{Int}(0,\Phi_0)))]-\mathbb E[f(\Phi(T;0,\Phi_0))]\nn\\
&=\sum_{i=1}^{N^{\Delta}}\Big\{\mathbb E[f(\Phi(T;t_i,\varphi^{Int}_{t_i}(0,\Phi_0)))]-\mathbb E[f(\Phi(T;t_{i-1},\varphi^{Int}_{t_{i-1}}(0,\Phi_0)))]\Big\}\nn\\
&=\sum_{i=1}^{N^{\Delta}}\mathbb E\Big[\mathbb E\Big[f(\Phi(T;t_i,\varphi^{Int}_{t_i}(t_{i-1},\varphi_{t_{i-1}}(0,\Phi_0))))\nn\\
&\qquad \qquad -f(\Phi(T;t_i,\Phi_{t_i}(t_{i-1},\varphi^{Int}_{t_{i-1}}(0,\Phi_0))))\Big|\mathcal F_{t_i}\Big]\Big]\nn\\
&=\sum_{i=1}^{N^{\Delta}}\int_0^1\mathbb E\Big[\mathbb E\Big[f'(\Phi(T;t_i,Y^{\varsigma}_i))\mathcal D\Phi(T;t_i,Y^{\varsigma}_i)\nn\\
&\qquad\qquad \Big(\varphi^{Int}_{t_i}(t_{i-1},\varphi_{t_{i-1}}(0,\Phi_0))-\Phi_{t_i}(t_{i-1},\varphi^{Int}_{t_{i-1}}(0,\Phi_0))\Big)\Big|\mathcal F_{t_i}\Big]\Big]\mathrm d\varsigma\nn\\
&=\sum_{i=1}^{N^{\Delta}}\int_0^1\mathbb E\Big[f'(\Phi(T;t_i,Y^{\varsigma}_i))\mathcal D\Phi(T;t_i,Y^{\varsigma}_i)\nn\\
&\qquad \qquad \Big(\varphi^{Int}_{t_i}(t_{i-1},\varphi_{t_{i-1}}(0,\Phi_0))-\Phi_{t_i}(t_{i-1},\varphi^{Int}_{t_{i-1}}(0,\Phi_0))\Big)\Big]\mathrm d\varsigma,
\end{align}
where 
\begin{align}\label{Y_defini}
Y^{\varsigma}_i=\varsigma \varphi^{Int}_{t_i}(t_{i-1},\varphi_{t_{i-1}}(0,\Phi_0))+(1-\varsigma)\Phi_{t_i}(t_{i-1},\varphi^{Int}_{t_{i-1}}(0,\Phi_0)).
\end{align}
It follows from $\varphi^{Int}(t_j;0,\Phi_0)=\varphi(t_j;0,\Phi_0)$ that 
\begin{align}\label{int_1}
&\quad \varphi^{Int}_{t_i}(t_{i-1},\varphi_{t_{i-1}}(0,\Phi_0))-\Phi_{t_i}(t_{i-1},\varphi^{Int}_{t_{i-1}}(0,\Phi_0))\nn\\
&=\sum_{j=-N}^0I^j\Big(\varphi^{Int}_{t_i}(t_{i-1},\varphi_{t_{i-1}}(0,\Phi_0))(t_j)-\Phi_{t_i}(t_{i-1},\varphi^{Int}_{t_{i-1}}(0,\Phi_0))(t_j)\Big)\nn\\
&=\sum_{j=-N}^0I^j\Big(\varphi^{Int}(t_i+t_j;t_{i-1},\varphi_{t_{i-1}}(0,\Phi_0))-\Phi(t_i+t_j;t_{i-1},\varphi^{Int}_{t_{i-1}}(0,\Phi_0))\Big)\nn\\
&=\sum_{j=-N}^{-1}I^j\Big(\varphi(t_{i+j};0,\Phi_0)-\varphi^{Int}(t_{i+j};0,\Phi_0)\Big)\nn\\
&\quad +I^0\Big(\varphi(t_i;t_{i-1},\varphi_{t_{i-1}}(0,\Phi_0))-\Phi(t_i;t_{i-1},\varphi^{Int}_{t_{i-1}}(0,\Phi_0))\Big)\nn\\
&=I^0\Big(\varphi(t_i;t_{i-1},\varphi_{t_{i-1}}(0,\Phi_0))-\Phi(t_i;t_{i-1},\varphi^{Int}_{t_{i-1}}(0,\Phi_0))\Big),
\end{align}
where $I^j(s)=\frac{N}{\tau}(s-t_{j-1})\mathbf 1_{\Delta_{j-1}}(s)+\frac{N}{\tau}(t_{j+1}-s)\mathbf 1_{\Delta_j}(s)$ with $\Delta_j=[t_j,t_{j+1})$ for $j=-N+1,\ldots,-1,$ and $I^{-N}(s)=\frac{N}{\tau}(t_{-N+1}-s)\mathbf 1_{\Delta_{-N}}(s),$ and $I^0(s)=\frac{N}{\tau}(s-t_{-1})\mathbf 1_{\Delta_{-1}}(s)$ with $\Delta_{-N}=[-t_{-N}, t_{-N+1})$ and $\Delta_{-1}=[-t_{-1}, 0]$.

Note that for any $i\in\mathbb N_{+}$,
\begin{align}\label{int_2}
&\quad \varphi(t_i;t_{i-1},\varphi_{t_{i-1}}(0,\Phi_0))-\Phi(t_i;t_{i-1},\varphi^{Int}_{t_{i-1}}(0,\Phi_0))\nn\\
&=\int_{t_{i-1}}^{t_i}b(\varphi_r(t_{i-1},\varphi_{t_{i-1}}(0,\Phi_0)))\mathrm dr+\int_{t_{i-1}}^{t_i}\sigma(\varphi_r(t_{i-1},\varphi_{t_{i-1}}(0,\Phi_0)))\mathrm dW(r)\nn\\
&\quad -\int_{t_{i-1}}^{t_i}\big[(1-\theta)b(\varphi^{Int}_{t_{i-1}}(0,\Phi_0))+\theta b(\Phi_{t_i}(t_{i-1},\varphi^{Int}_{t_{i-1}}(0,\Phi_0)))\big]\mathrm dr \nn\\
&\quad -\int_{t_{i-1}}^{t_i}\sigma (\varphi^{Int}_{t_{i-1}}(0,\Phi_0))\mathrm dW(r)
=\mathcal I^i_b+\mathcal I^i_{b,\theta}+\mathcal I^i_{\sigma},
\end{align}
where 
\begin{align*}
&\mathcal I^i_b:=\int_{t_{i-1}}^{t_i}\big[b(\varphi_r(t_{i-1},\varphi_{t_{i-1}}(0,\Phi_0)))-b(\varphi^{Int}_{t_{i-1}}(0,\Phi_0))\big]\mathrm dr,\\
&\mathcal I^i_{b,\theta}:=\theta\int_{t_{i-1}}^{t_i}\big[b(\varphi^{Int}_{t_{i-1}}(0,\Phi_0))-b(\Phi_{t_i}(t_{i-1},\varphi^{Int}_{t_{i-1}}(0,\Phi_0)))\big]\mathrm dr,\\
&\mathcal I^i_{\sigma}:=\int_{t_{i-1}}^{t_i}\big[\sigma(\varphi_r(t_{i-1},\varphi_{t_{i-1}}(0,\Phi_0)))-\sigma (\varphi^{Int}_{t_{i-1}}(0,\Phi_0))\big]\mathrm dW(r).
\end{align*}
According to  \eqref{decomp1}--\eqref{int_2}, we derive 
\begin{align}\label{Wrate}
&\quad \mathbb E[f(x(T))]-\mathbb E[f(y^{\xi,\Delta}(T))]\nn\\
&=\mathbb E[f(\Phi(T;t_{N^{\Delta}},\varphi_{t_{N^{\Delta}}}^{Int}(0,\xi)))]-\mathbb E[f(\Phi(T;t_{N^{\Delta}},\varphi_{t_{N^{\Delta}}}^{Int}(0,\Phi_0)))]\nn\\
&\quad +\sum_{i=1}^{N^{\Delta}}\int_0^1\mathbb E\Big[f'(\Phi(T;t_i,Y^{\varsigma}_i))\mathcal D\Phi(T;t_i,Y^{\varsigma}_i)\Big(I^0(\mathcal I^i_{b}+\mathcal I^i_{b,\theta}+\mathcal I^i_{\sigma})\Big)\Big]\mathrm d\varsigma\nn\\
&=\mathbb E[f(\Phi(T;t_{N^{\Delta}},\varphi_{t_{N^{\Delta}}}^{Int}(0,\xi)))]-\mathbb E[f(\Phi(T;t_{N^{\Delta}},\varphi_{t_{N^{\Delta}}}^{Int}(0,\Phi_0)))]\nn\\
&\quad +\sum_{i=1}^{N^{\Delta}}\int_0^1\mathbb E\Big[\Big\<(\mathcal D\Phi(T;t_i,Y^{\varsigma}_i) I^0\mathrm {Id}_{d\times d})^*f'(\Phi(T;t_i,Y^{\varsigma}_i)),\mathcal I^i_{b}+\mathcal I^i_{b,\theta}+\mathcal I^i_{\sigma}\Big\>\Big]\mathrm d\varsigma\nn\\
&=\mathcal I_0+\mathcal I_{b}+\mathcal I_{b,\theta}+\mathcal I_{\sigma},
\end{align}
where 
\begin{align*}
&\mathcal I_0:=\mathbb E[f(\Phi(T;t_{N^{\Delta}},\varphi_{t_{N^{\Delta}}}^{Int}(0,\xi)))]-\mathbb E[f(\Phi(T;t_{N^{\Delta}},\varphi_{t_{N^{\Delta}}}^{Int}(0,\Phi_0)))],\\
&\mathcal I_b:=\sum_{i=1}^{N^{\Delta}}\int_0^1\mathbb E\Big[\Big\<(\mathcal D\Phi(T;t_i,Y^{\varsigma}_i) I^0\mathrm {Id}_{d\times d})^* f'(\Phi(T;t_i,Y^{\varsigma}_i)),\mathcal I^i_{b}\Big\>\Big]\mathrm d\varsigma,\\
&\mathcal I_{b,\theta}:=\sum_{i=1}^{N^{\Delta}}\int_0^1\mathbb E\Big[\Big\<(\mathcal D\Phi(T;t_i,Y^{\varsigma}_i) I^0\mathrm {Id}_{d\times d})^* f'(\Phi(T;t_i,Y^{\varsigma}_i)),\mathcal I^i_{b,\theta}\Big\>\Big]\mathrm d\varsigma,\\
&\mathcal I_{\sigma}:=\sum_{i=1}^{N^{\Delta}}\int_0^1\mathbb E\Big[\Big\<(\mathcal D\Phi(T;t_i,Y^{\varsigma}_i) I^0\mathrm {Id}_{d\times d})^* f'(\Phi(T;t_i,Y^{\varsigma}_i)),\mathcal I^i_{\sigma}\Big\>\Big]\mathrm d\varsigma.
\end{align*}

\subsection{Estimates of Malliavin derivatives and G\^ateaux derivatives}
In this subsection, we show some time-independent estimates of  Malliavin derivatives and G\^ateaux derivatives of both the exact solution and the numerical solution.
\begin{assp}\label{first_deri}
Assume that coefficients $b$ and $\sigma$ have continuous  G\^ateaux  derivatives. 
\end{assp}

For any $\phi,\phi_1\in\mathcal C^d,$ and $t\in\mathbb R$, it follows from Assumption \ref{first_deri} and the definition of the  G\^ateaux derivative that $b(\phi_1+t\phi)-b(\phi_1)=t\mathcal Db(\phi_1)\phi+o(t).$ Then 
based on Assumptions \ref{a1}--\ref{a2}, we have 
\begin{align}
&\<\mathcal Db(\phi_1)\phi,\phi(0)\>\leq -a_1|\phi(0)|^2+a_2\int_{-\tau}^0|\phi(r)|^2\mathrm d\nu_2(r),\label{DB}\\
&|\mathcal D\sigma(\phi_1)\phi|^2\leq L(|\phi(0)|^2+\int_{-\tau}^0|\phi(r)|^2\mathrm d\nu_1(r)),\label{DS}
\end{align}
where $\phi,\phi_1\in\mathcal C^d;$ see \cite[Eq. (2.4)]{dang2023} for a similar proof.

The following lemmas give some regularity estimates of both the exact solution and the numerical solution.
\begin{lemma}\label{lem_fir_d}
Let Assumptions \ref{a1}--\ref{a2} with $a_1-a_2-(2p-1)L>0$ hold for some $p\ge 1$. And let Assumption \ref{first_deri} hold. Then for any $t\geq 0$, 
\begin{align*}
&\mathbb E[|\mathcal Dx^{\tilde\xi}(t)\cdot \eta|^{2p}]\leq K\|\eta\|^{2p}e^{-\hat\lambda_1 pt},\\
&\mathbb E[\|\mathcal Dx^{\tilde\xi}_t\cdot\eta\|^{2p}]\leq K\|\eta\|^{2p}e^{-\hat\lambda_1 pt},
\end{align*}
where $\tilde\xi\in L^{2p}(\Omega;\mathcal C^d)$, $\eta\in \mathcal C^d$, and  constants $K,\hat\lambda_1>0$. 
\end{lemma}
\begin{proof}
For $\eta\in\mathcal C^d,$ $\mathcal Dx^{\tilde\xi}(t)\cdot\eta$ satisfies 
\begin{align*}
\mathrm d\mathcal Dx^{\tilde\xi}(t)\cdot\eta=\mathcal Db(x^{\tilde\xi}_t)\mathcal Dx^{\tilde\xi}_t\cdot\eta\mathrm dt+\mathcal D\sigma(x^{\tilde\xi}_t)\mathcal Dx^{\tilde\xi}_t\cdot\eta\mathrm dW(t),\quad t>0
\end{align*}
with the initial datum $\mathcal Dx^{\tilde\xi}(r)\cdot\eta=\eta(r)$ for $r\in[-\tau,0]$.
 Let $\hat\lambda_1\in(0,1)$ be determined later. Applying the It\^o formula, we obtain
\begin{align*}
&\quad \mathrm d(e^{\hat\lambda_1 pt}|\mathcal Dx^{\tilde\xi}(t)\cdot \eta|^{2p})\\
&=\hat\lambda_1 pe^{\hat\lambda_1 pt}|\mathcal Dx^{\tilde\xi}(t)\cdot \eta|^{2p}\mathrm dt+2pe^{\hat\lambda_1 pt}|\mathcal Dx^{\tilde\xi}(t)\cdot \eta|^{2p-2}\Big(\<\mathcal Dx^{\tilde\xi}(t)\cdot \eta,\mathcal Db(x^{\tilde\xi}_t)\mathcal Dx^{\tilde\xi}_t\cdot \eta\>\mathrm dt\\
&\quad +\<\mathcal Dx^{\tilde\xi}(t)\cdot \eta,\mathcal D\sigma (x^{\tilde\xi}_t)\mathcal Dx^{\tilde\xi}_t\cdot \eta\mathrm dW(t)\>\Big)\\
&\quad +2p(p-1)e^{\hat\lambda_1 pt}|\mathcal Dx^{\tilde\xi}(t)\cdot \eta|^{2p-4}|(\mathcal D\sigma (x^{\tilde\xi}_t)\mathcal Dx^{\tilde\xi}_t\cdot \eta)^{\top}\mathcal Dx^{\tilde\xi}(t)\cdot \eta|^2\mathrm dt\\
&\quad+pe^{\hat\lambda_1 pt}|\mathcal Dx^{\tilde\xi}(t)\cdot \eta|^{2p-2}|\mathcal D\sigma (x^{\tilde\xi}_t)\mathcal Dx^{\tilde\xi}_t\cdot \eta|^2\mathrm dt.
\end{align*} 
Combining \eqref{DB}--\eqref{DS}, we deduce 
\begin{align*}
&\quad e^{\hat\lambda_1 pt}|\mathcal Dx^{\tilde\xi}(t)\cdot \eta|^{2p}\nn\\
&\leq |\eta(0)|^{2p}\!+\!\!\int_0^t\hat\lambda_1 pe^{\hat\lambda_1 ps}|\mathcal Dx^{\tilde\xi}(s)\cdot \eta|^{2p}\mathrm ds\!+\!\int_0^t\!2pe^{\hat\lambda_1 ps}|\mathcal Dx^{\tilde\xi}(s)\cdot \eta|^{2p-2}\Big[\!\!-a_1|\mathcal Dx^{\tilde\xi}(s)\cdot \eta|^2\nn\\
&\quad+a_2\int_{-\tau}^0|\mathcal Dx^{\tilde\xi}(s+r)\cdot \eta|^2\mathrm d\nu_2(r)\Big]\mathrm ds+\int_0^t2pe^{\hat\lambda_1 ps}|\mathcal Dx^{\tilde\xi}(s)\cdot \eta|^{2p-2}\times\\
&\quad\big\<\mathcal Dx^{\tilde\xi}(s)\cdot \eta,\mathcal D\sigma (x^{\tilde\xi}_s)\mathcal Dx^{\tilde\xi}_s\cdot \eta\mathrm dW(s)\big\>+\int_0^tp(2p-1)e^{\hat\lambda_1 ps}|\mathcal Dx^{\tilde\xi}(s)\cdot\eta|^{2p-2}\times \\
&\quad L\Big[|\mathcal Dx^{\tilde\xi}(s)\cdot\eta|^2+\int_{-\tau}^0|\mathcal Dx^{\tilde\xi}(s+r)\cdot\eta|^2\mathrm d\nu_1(r)\Big]\mathrm ds.
\end{align*}
It follows from the H\"older inequality that
\begin{align*}
&\quad\int_0^te^{\hat\lambda_1 ps}|\mathcal Dx^{\tilde\xi}(s)\cdot \eta|^{2p-2}\int_{-\tau}^0|\mathcal Dx^{\tilde\xi}(s+r)\cdot\eta|^2\mathrm d\nu_2(r)\mathrm ds\\
&\leq \frac{p-1}{p}\int_0^te^{\hat\lambda_1 ps}|\mathcal Dx^{\tilde\xi}(s)\cdot \eta|^{2p}\mathrm ds+\frac{1}{p}\int_0^te^{\hat\lambda_1 ps}\int_{-\tau}^0|\mathcal Dx^{\tilde\xi}(s+r)\cdot\eta|^{2p}\mathrm d\nu_2(r)\mathrm ds\\
&= \frac{p-1}{p}\int_0^te^{\hat\lambda_1 ps}|\mathcal Dx^{\tilde\xi}(s)\cdot \eta|^{2p}\mathrm ds+\frac{1}{p}\int_{-\tau}^0\int_0^t e^{\hat\lambda_1 ps}|\mathcal Dx^{\tilde\xi}(s+r)\cdot\eta|^{2p}\mathrm ds\mathrm d\nu_2(r)\\
&\leq \frac{p-1}{p} \int_0^t e^{\hat\lambda_1 ps}|\mathcal Dx^{\tilde\xi}(s)\cdot\eta|^{2p}\mathrm ds+\frac1p e^{\hat\lambda_1 p\tau}\int_{-\tau}^t e^{\hat\lambda_1 pz}|\mathcal Dx^{\tilde\xi}(z)\cdot\eta|^{2p}\mathrm dz\\
&\leq \frac 1p\tau e^{\hat\lambda_1 p\tau}\|\eta\|^{2p}+e^{\hat\lambda_1 p\tau}\int_0^te^{\hat\lambda_1 ps}|\mathcal Dx^{\tilde\xi}(s)\cdot\eta|^{2p}\mathrm ds.
\end{align*}
We derive
\begin{align*}
\mathbb E[e^{\hat\lambda_1 pt}|\mathcal Dx^{\tilde\xi}(t)\cdot\eta|^{2p}]&\leq |\eta(0)|^{2p}+p\Big(\hat\lambda_1 -2a_1+2a_2e^{\hat\lambda_1 p\tau}+(2p-1)L(1+e^{\hat\lambda_1 p\tau})\Big)\times\\
&\int_0^te^{\hat\lambda_1 ps}\mathbb E[|\mathcal Dx^{\tilde\xi}(s)\cdot\eta|^{2p}]\mathrm ds+(2a_2+(2p-1)L)\tau e^{\hat\lambda_1 p\tau}\|\eta\|^{2p}.
\end{align*}
Take a sufficiently small number $\hat\lambda_1>0$  such that $\hat\lambda_1 -2a_1+2a_2e^{\hat\lambda_1 p\tau}+(2p-1)L+(2p-1)Le^{\hat\lambda_1 p\tau}< 0$. 
This gives $\mathbb E[e^{\hat\lambda_1 pt}|\mathcal Dx^{\tilde\xi}(t)\cdot\eta|^{2p}]\leq K\|\eta\|^{2p}.$ 
Moreover, applying the Burkholder--Davis--Gundy inequality,  we have
\begin{align*}
&\mathbb E\Big[\sup_{0\vee(t-\tau)\leq r\leq t}e^{\hat\lambda_1 pr}|\mathcal Dx^{\tilde\xi}(r)\cdot\eta|^{2p}\Big]\leq |\eta(0)|^{2p}+(2a_2+(2p-1)L)\tau e^{\hat\lambda_1 p\tau}\|\eta\|^{2p}\\
&\quad +\mathbb E\Big[\sup_{0\vee(t-\tau)\leq r\leq t}\int_0^r2pe^{\hat\lambda_1 ps}|\mathcal Dx^{\tilde\xi}(s)\cdot \eta|^{2p-2}\Big\<\mathcal Dx^{\tilde\xi}(s)\cdot\eta,\mathcal D\sigma(x^{\tilde\xi}_s)\mathcal Dx^{\tilde\xi}_s\cdot\eta\mathrm dW(s)\Big\>\Big]\\
&\leq K\|\eta\|^{2p}+\mathbb E\Big[\sup_{0\vee(t-\tau)\leq r\leq t}\int_{0\vee(t-\tau)}^r2pe^{\hat\lambda_1 ps}|\mathcal Dx^{\tilde\xi}(s)\cdot \eta|^{2p-2}\Big\<\mathcal Dx^{\tilde\xi}(s)\cdot\eta,\nn\\
&\quad\mathcal D\sigma(x^{\tilde\xi}_s)\mathcal Dx^{\tilde\xi}_s\cdot\eta\mathrm dW(s)\Big\>\Big]\\
&\leq  K\|\eta\|^{2p} +K\mathbb E\Big[\Big(\int_{0\vee(t-\tau)}^te^{2\hat\lambda_1 ps}|\mathcal Dx^{\tilde\xi}(s)\cdot\eta|^{4p-2}|\mathcal D\sigma(x^{\tilde\xi}_s)\mathcal Dx^{\tilde\xi}_s\cdot \eta|^2\mathrm ds\Big)^{\frac{1}{2}}\Big]\\
&\leq K\|\eta\|^{2p}+K\mathbb E\Big[\sup_{0\vee(t-\tau)\leq s\leq t}e^{\hat\lambda_1 (p-\frac12)s}|\mathcal Dx^{\tilde\xi}(s)\cdot \eta|^{2p-1}\times\nn\\
&\quad\Big(\int_{0\vee(t-\tau)}^te^{\hat\lambda_1 s}|\mathcal D\sigma(x^{\tilde\xi}_s)\mathcal Dx^{\tilde\xi}_s\cdot\eta|^2\mathrm ds\Big)^{\frac12}\Big].
\end{align*}
According to the Young inequality, \eqref{DS}, and the H\"older inequality, we obtain 
\begin{align*}
&\quad \mathbb E\Big[\sup_{0\vee(t-\tau)\leq r\leq t}e^{\hat\lambda_1 pr}|\mathcal Dx^{\tilde\xi}(r)\cdot\eta|^{2p}\Big]\\
&\leq K\|\eta\|^{2p}+\frac12 \mathbb E\Big[\sup_{0\vee(t-\tau)\leq s\leq t}e^{\hat\lambda_1 ps}|\mathcal Dx^{\tilde\xi}(s)\cdot\eta|^{2p}\Big]\\
&\quad +K\mathbb E\Big[\Big(\int_{0\vee(t-\tau)}^te^{\hat\lambda_1 s}\big(|\mathcal Dx^{\tilde\xi}(s)\cdot\eta|^2+\int_{-\tau}^0|\mathcal Dx^{\tilde\xi}(s+r)\cdot\eta|^2\mathrm d\nu_1(r)\big)\mathrm ds\Big)^p\Big]\\
&\leq K\|\eta\|^{2p}+\frac12 \mathbb E\Big[\sup_{0\vee(t-\tau)\leq s\leq t}e^{\hat\lambda_1 ps}|\mathcal Dx^{\tilde\xi}(s)\cdot\eta|^{2p}\Big]\nn\\
&\quad+\tau^{p-1}K\int_{0\vee(t-\tau)}^t\sup_{-\tau\leq r\leq s}\E\big[e^{\hat\lambda_1 pr}|\mathcal Dx^{\tilde\xi}(r)\cdot\eta|^{2p}\big]\mathrm ds.
\end{align*}
It follows from  $\mathbb E[e^{\hat\lambda_1 pt}|\mathcal Dx^{\tilde\xi}(t)\cdot\eta|^{2p}]\leq K\|\eta\|^{2p}$ that
\begin{align*}
e^{\hat\lambda_1 p(0\vee(t-\tau))}\mathbb E[\|\mathcal Dx^{\tilde\xi}_t\cdot\eta\|^{2p}]\leq K\|\eta\|^{2p}, 
\end{align*}
which gives
\begin{align*}
\mathbb E[\|\mathcal Dx^{\tilde\xi}_t\cdot\eta\|^{2p}]\leq K\|\eta\|e^{-\hat\lambda_1 pt}.
\end{align*}
The proof is finished. 
\end{proof}
\begin{lemma}\label{lem_fir_d2}
Let Assumptions \ref{a1}--\ref{a2} with $a_1-a_2-(2p-1)L>0$ hold for some $p\geq1$. And let Assumptions \ref {a4} and \ref{first_deri} hold. Then for any $t\geq 0$ and $u\ge-\tau$, 
\begin{align*}
\mathbb E[\|D_ux^{\tilde\xi}_t\|^{2p}]\leq Ke^{-\hat\lambda_1 p(t-u)}(1+\mathbb E[\|D_u\tilde\xi\|^{2p} ]),
\end{align*}
where $\tilde\xi\in L^{2p}(\Omega;\mathcal C^d)$ and            $D_u\tilde\xi\in L^{2p}(\Omega;\mathcal C^d\otimes \mathbb R^m)$, $K>0$, and $\hat\lambda_1>0$ is given in Lemma \ref{lem_fir_d}.
\end{lemma}
\begin{proof}
When $0\leq u\leq t,$ $D_ux^{\tilde\xi}(t)$ satisfies 
\begin{align*}
D_ux^{\tilde\xi}(t)&=\int_u^t\mathcal Db(x^{\tilde\xi}_s)D_ux^{\tilde\xi}_s\mathrm ds+\int_u^t\mathcal D\sigma(x^{\tilde\xi}_s)D_ux^{\tilde\xi}_s\mathrm dW(s)\\
&\quad +\sigma(x^{\tilde\xi}_u) \mathbf{1}_{[0,t]}(u)+D_u\tilde\xi(0).
\end{align*}
By the It\^o formula, we derive
\begin{align*}
&\mathrm d(e^{\hat\lambda_1 pt}|D_ux^{\tilde\xi}(t)|^{2p})\leq \hat\lambda_1 pe^{\hat\lambda_1 pt}|D_ux^{\tilde\xi}(t)|^{2p}\mathrm dt+2pe^{\hat\lambda_1 pt}|D_ux^{\tilde\xi}(t)|^{2p-2}\times\\
&\Big(\big\<D_ux^{\tilde\xi}(t),\mathcal Db(x^{\tilde\xi}_t)D_ux^{\tilde\xi}_t\big\>\mathrm dt+\big\<D_ux^{\tilde\xi}(t),\mathcal D\sigma(x^{\tilde\xi}_t)D_ux^{\tilde\xi}_t\mathrm dW(t)\big\>\Big)\\
&\quad+p(2p-1)e^{\hat\lambda_1 pt}|D_ux^{\tilde\xi}(t)|^{2p-2}|\mathcal D\sigma(x^{\tilde\xi}_t)D_ux^{\tilde\xi}_t|^2\mathrm dt.
\end{align*}
Similar to the proof of Lemma \ref{lem_fir_d}, we have
\begin{align*}
&\quad \mathbb E[e^{\hat\lambda_1 pt}|D_ux^{\tilde\xi}(t)|^{2p}]\leq  K\mathbb E[e^{\hat\lambda_1 pu}(|\sigma(x^{\tilde\xi}_u)|^{2p}+|D_u\tilde\xi(0)|^{2p})]\\
&\quad+p\big(\hat\lambda_1-2a_1+2a_2e^{\hat\lambda_1 p\tau}+(2p-1)L(1+e^{\hat\lambda_1 p\tau})\big)\int_u^te^{\hat\lambda_1 ps}\mathbb E[|D_ux^{\tilde\xi}(s)|^{2p}]\mathrm ds\\
&\quad+(2a_2+(2p-1)L)\tau e^{\hat\lambda_1 p\tau}\|D_u\tilde\xi\|^{2p}\\
&\leq K\mathbb E[e^{\hat\lambda_1 pu}(|\sigma(x^{\tilde\xi}_u)|^{2p}+|D_u\tilde\xi(0)|^{2p}]+(2a_2+(2p-1)L)\tau e^{\hat\lambda_1 p\tau}\|D_u\tilde\xi\|^{2p},
\end{align*}
where we used the facts $D_u x^{\tilde \xi}(s)=0$ for $u\geq s$ and 
\begin{align*}
&\quad\int_{u}^{t}\int_{-\tau}^{0}e^{\hat\lambda_1 ps}|D_ux^{\tilde \xi}_{s}(r)|^{2p}\mathrm d\nu_{i}(r)\mathrm ds=\int_{0}^{t}\int_{-\tau}^{0}e^{\hat\lambda_1 ps}|D_ux^{\tilde \xi}_{s}(r)|^{2p}\mathrm d\nu_{i}(r)\mathrm ds\nn\\
&\leq \tau  e^{\hat\lambda_1 p\tau}\|D_u\tilde\xi\|^{2p}
+e^{\hat\lambda_1 p\tau}\int_{u}^{t}e^{\hat\lambda_1 ps}|D_ux^{\tilde \xi}(s)|^{2p}\mathrm ds,\quad i=1,2.
\end{align*}
This, along with Assumption \ref{a1} and Lemma \ref{l4.1} implies that
\begin{align*}
\mathbb E[|D_ux^{\tilde\xi}(t)|^{2p}]\leq Ke^{-\hat\lambda_1 p(t-u)} (1+\mathbb E[\|D_u\tilde\xi\|^{2p}]).
\end{align*}
Moreover, 
\begin{align*}
&\quad\mathbb E\Big[\sup_{u\vee(t-\tau)\leq r\leq t}e^{\hat\lambda_1 pr}|D_ux^{\tilde\xi}(r)|^{2p}\Big]\nn\\
&\leq K e^{\hat\lambda_1 pu}\mathbb E[|\sigma (x^{\tilde\xi}_u)|^{2p}+|D_u\tilde\xi(0) |^{2p}\mathbf 1_{[-\tau,0]}(u)]+(2a_2+(2p-1)L)\tau e^{\hat\lambda_1 p\tau}\|D_u\tilde\xi\|^{2p}\\
&\quad +\mathbb E\Big[\sup_{u\vee(t-\tau)\leq r\leq t}\int_u^r2pe^{\hat\lambda_1 ps}\Big\<|D_ux^{\tilde\xi}(s)|^{2p-2}D_ux^{\tilde\xi}(s),\mathcal D\sigma(x^{\tilde\xi}_s)D_ux^{\tilde\xi}_s\mathrm dW(s)\Big\>\Big]\\
&\leq Ke^{\hat\lambda_1 pu}(1+\mathbb E[\|D_u\tilde\xi\|^{2p}])+K\mathbb E\Big[\Big(\int_{u\vee (t-\tau)}^te^{2\hat\lambda_1 ps}|D_ux^{\tilde\xi}(s)|^{4p-2}|\mathcal D\sigma(x^{\tilde\xi}_s)D_ux^{\tilde\xi}_s|^2\mathrm ds\Big)^{\frac12}\Big]\\
&\leq Ke^{\hat\lambda_1 pu}(1+\mathbb E[\|D_u\tilde\xi\|^{2p}])\\
&\quad +K\mathbb E\Big[\Big(\sup_{u\vee(t-\tau)\leq s\leq t}e^{\hat\lambda_1(p-\frac12)s}|D_ux^{\tilde\xi}(s)|^{2p-1}\Big)\Big(\int_{u\vee(t-\tau)}^te^{\hat\lambda_1 s}|\mathcal D\sigma(x^{\tilde\xi}_s)D_ux^{\tilde\xi}_s|^2\mathrm ds\Big)^{\frac12}\Big].
\end{align*}
Applying the Young inequality, we derive
\begin{align*}
&\quad\mathbb E\Big[\sup_{u\vee(t-\tau)\leq r\leq t}e^{\hat\lambda_1 pr}|D_ux^{\tilde\xi}(r)|^{2p}\Big]\nn\\
&\leq Ke^{\hat\lambda_1 pu}(1+\mathbb E[\|D_u\tilde\xi\|^{2p}])+\frac12 \mathbb E\Big[\sup_{u\vee(t-\tau)\leq s\leq t}e^{\hat\lambda_1 ps}|D_ux^{\tilde\xi}(s)|^{2p}\Big]\\
&\quad +K\mathbb E\Big[\Big(\int_{u\vee(t-\tau)}^te^{\hat\lambda_1 s}|\mathcal D\sigma(x^{\tilde\xi}_s)D_ux^{\tilde\xi}_s|^2\mathrm ds\Big)^{p}\Big]\\
&\leq Ke^{\hat\lambda_1 pu}(1+\mathbb E[\|D_u\tilde\xi\|^{2p}])+\frac12\mathbb E\Big[\sup_{u\vee(t-\tau)\leq s\leq t}e^{\hat\lambda_1 ps}|D_ux^{\tilde\xi}(s)|^{2p}\Big]\\ &\quad +\tau^{p-1}K\int_{u\vee(t-\tau)}^te^{\hat\lambda_1 ps}\big(\sup_{s-\tau\leq r\leq s}\E\big[|D_ux^{\tilde\xi}(r)|^{2p}\big]\big)\mathrm ds\\
&\leq Ke^{\hat\lambda_1 pu}(1+\mathbb E[\|D_u\tilde\xi\|^{2p}])+\tau^{p}Ke^{\hat\lambda_1 pu}(1+\mathbb E[\|D_u\tilde\xi\|^{2p}]).
\end{align*}
Hence, 
\begin{align*}
e^{\hat\lambda_1 p(t-\tau)}\mathbb E[\|D_ux^{\tilde\xi}_t\|^{2p}]\leq Ke^{\hat\lambda_1 pu}(1+\mathbb E[\|D_u\tilde\xi\|^{2p}]),
\end{align*}
which implies
\begin{align*}
\mathbb E[\|D_ux^{\tilde\xi}_t\|^{2p}]\leq Ke^{\hat\lambda_1 p(u-t)}(1+\mathbb E[\|D_u\tilde\xi\|^{2p}]).
\end{align*}
We complete the proof.
\end{proof}

\begin{lemma}\label{Dy_esti}
Let Assumptions \ref{a1}--\ref{a2},~\ref{a5}, and \ref{first_deri} hold. Then for $p\in\mathbb N_+$, there exists a number $a^*_1(L,a_2,\theta,p)>0$ such that for  $a_1>a_1^*$ and $\Delta\in(0,1]$, 
\begin{align*}
\mathbb E[\|\mathcal Dy^{\tilde\xi,\Delta}_{t_k}\cdot \eta\|^{2p}]\leq Ke^{-\hat{\lambda}_2p t_k}\|\eta\|^{2p}(1+\E[\|\tilde \xi\|^{2p\beta}]),
\end{align*}
where $\tilde\xi\in L^{2p\beta}(\Omega;\mathcal C^d)$, 
$\eta\in \mathcal C^d$, and $K,\hat{\lambda}_2>0$.
\end{lemma}
\begin{proof}For the case of $p=1,$ one can obtain the result by using the similar technique  to that in Lemma \ref{l4.2}.  We are in the position to show the proof of $p>1$. 

\underline{Step 1: Case of  $\Delta\in(0,\Delta^*]$ with some $\Delta^*\in(0,1].$}
 We have that for $k\in\mathbb N,$ 
\begin{align*}
&\mathcal Dz^{\tilde\xi,\Delta}(t_{k+1})\cdot\eta=\mathcal Dz^{\tilde\xi,\Delta}(t_k)\cdot \eta+\mathcal Db(y^{\tilde\xi,\Delta}_{t_k})\mathcal Dy^{\tilde\xi,\Delta}_{t_k}\cdot \eta\Delta+\mathcal D\sigma(y^{\tilde\xi,\Delta}_{t_k})\mathcal Dy^{\tilde\xi,\Delta}_{t_k}\cdot\eta\delta W_k,\\
&\mathcal Dz^{\tilde\xi,\Delta}(0)\cdot\eta=\eta(0)-\theta \Delta\mathcal Db(\tilde\xi)\eta,
\end{align*}
and 
\begin{align}\label{re_d_yZ}
\mathcal Dz^{\tilde\xi,\Delta}(t_{k})\cdot\eta=\mathcal Dy^{\tilde\xi,\Delta}(t_k)\cdot\eta-\theta \Delta \mathcal Db(y^{\tilde\xi,\Delta}_{t_k})\mathcal Dy^{\tilde\xi,\Delta}_{t_k}\cdot\eta.
\end{align}
Then
\begin{align*}
&\quad |\mathcal Dz^{\tilde\xi,\Delta}(t_{k+1})\cdot\eta|^2\\
&=|\mathcal Dz^{\tilde\xi,\Delta}(t_k)\cdot\eta|^2+|\mathcal Db(y^{\tilde\xi,\Delta}_{t_k})\mathcal Dy^{\tilde\xi,\Delta}_{t_k}\cdot\eta|^2\Delta^2+|\mathcal D\sigma (y^{\xi,\Delta}_{t_k})\mathcal Dy^{\tilde\xi,\Delta}_{t_k}\cdot\eta\delta W_k|^2\\
&\quad +\mathcal M_k+2\<\mathcal Dz^{\tilde\xi,\Delta}(t_k)\cdot\eta,\mathcal Db(y^{\tilde\xi,\Delta}_{t_k})\mathcal Dy^{\tilde\xi,\Delta}_{t_k}\cdot\eta\>\Delta,
\end{align*}
where $\mathcal M_k:=2\big\<\mathcal Dz^{\tilde\xi,\Delta}(t_k)\cdot\eta+\mathcal Db(y^{\tilde\xi,\Delta}_{t_k})\mathcal Dy^{\tilde\xi,\Delta}_{t_k}\cdot\eta\Delta,\mathcal D\sigma(y^{\tilde\xi,\Delta}_{t_k})\mathcal Dy^{\tilde\xi,\Delta}_{t_k}\cdot\eta\delta W_k\big\>$.
It follows from \eqref{DB} and \eqref{re_d_yZ} that
\begin{align*}
&|\mathcal Dz^{\tilde\xi,\Delta}(t_{k+1})\cdot\eta|^2=\frac{(1-\theta)^2}{\theta^2}|\mathcal Dz^{\tilde\xi,\Delta}(t_k)\cdot\eta|^2+\frac{1-2\theta}{\theta^2}|\mathcal Dy^{\tilde\xi,\Delta}(t_k)\cdot\eta|^2\\
&\quad +\frac{2(2\theta-1)}{\theta^2}\<\mathcal Dy^{\tilde\xi,\Delta}(t_k)\cdot\eta,\mathcal Dz^{\tilde\xi,\Delta}(t_k)\cdot\eta\>+2\<\mathcal Dy^{\tilde\xi,\Delta}(t_k)\cdot\eta,\mathcal Db(y^{\tilde\xi,\Delta}_{t_k})\mathcal Dy^{\tilde\xi,\Delta}_{t_k}\cdot\eta\>\Delta\\
&\quad +|\mathcal D\sigma(y^{\tilde\xi,\Delta}_{t_k})\mathcal Dy^{\tilde\xi,\Delta}_{t_k}\cdot\eta\delta W_k|^2+\mathcal M_k\\
&\leq A_{\theta,\Delta}|\mathcal Dz^{\tilde\xi,\Delta}(t_k)\cdot\eta|^2+\frac{2\theta-1}{\theta^2}\Delta|\mathcal Dy^{\tilde\xi,\Delta}(t_k)\cdot\eta|^2\\
&\quad+2\<\mathcal Dy^{\tilde\xi,\Delta}(t_k)\cdot\eta,\mathcal Db(y^{\tilde\xi,\Delta}_{t_k})\mathcal Dy^{\tilde\xi,\Delta}_{t_k}\cdot\eta\>\Delta +|\mathcal D\sigma(y^{\tilde\xi,\Delta}_{t_k})\mathcal Dy^{\tilde\xi,\Delta}_{t_k}\cdot\eta\delta W_k|^2+\mathcal M_k\\
&\leq A_{\theta,\Delta}|\mathcal Dz^{\tilde\xi,\Delta}(t_k)\cdot\eta|^2+\frac{2\theta-1}{\theta^2}\Delta|\mathcal Dy^{\tilde\xi,\Delta}(t_k)\cdot\eta|^2-2\Delta a_1|\mathcal Dy^{\tilde\xi,\Delta}(t_k)\cdot\eta|^2\\
&\quad +2a_2\Delta\int_{-\tau}^0|\mathcal Dy^{\tilde\xi,\Delta}_{t_k}\cdot\eta(r)|^2\mathrm d\nu_2(r)+|\mathcal D\sigma(y^{\tilde\xi,\Delta}_{t_k})\mathcal Dy^{\tilde\xi,\Delta}_{t_k}\cdot\eta\delta W_k|^2+\mathcal M_k,
\end{align*}
where $A_{\theta,\Delta}:=\frac{(1-\theta)^2}{\theta^2}+\frac{2\theta-1}{\theta^2(1+\Delta)}.$
Hence, 
\begin{align*}
\mathbb E[|\mathcal Dz^{\tilde\xi,\Delta}(t_{k+1})\cdot\eta|^{2p}]\leq A^p_{\theta,\Delta}\mathbb E[|\mathcal Dz^{\tilde\xi,\Delta}(t_k)\cdot\eta|^{2p}]+\sum_{i=1}^pC_p^i\mathbb E[\mathcal I_i],
\end{align*}
where 
\begin{align*}\mathcal I_i&=A^{p-i}_{\theta,\Delta}|\mathcal Dz^{\tilde\xi,\Delta}(t_k)\cdot\eta|^{2(p-i)}\Big(-\Delta(2a_1-\frac{2\theta-1}{\theta^2})|\mathcal Dy^{\tilde\xi,\Delta}(t_k)\cdot\eta|^2\\
&\quad +2a_2\Delta \int_{-\tau}^0|\mathcal Dy^{\tilde\xi,\Delta}_{t_k}(r)\cdot\eta|^2\mathrm d\nu_2(r) +|\mathcal D\sigma(y^{\tilde\xi,\Delta}_{t_k})\mathcal Dy^{\tilde\xi,\Delta}_{t_k}\cdot\eta\delta W_k|^2+\mathcal M_k\Big)^i.
\end{align*}
For the term $\mathcal I_1,$ by  \eqref{DS} and the property of the conditional expectation, we obtain
\begin{align}\label{obtain1}
\mathbb E[\mathcal I_1]&\leq \mathbb E\Big[A^{p-1}_{\theta,\Delta}|\mathcal Dz^{\tilde\xi,\Delta}(t_k)\cdot\eta|^{2(p-1)}\Big(-\Delta (2a_1-\frac{2\theta-1}{\theta^2}-L)|\mathcal Dy^{\tilde\xi,\Delta}(t_k)\cdot\eta|^2\nn\\
&\quad +2a_2\Delta \int_{-\tau}^0|\mathcal Dy^{\tilde\xi,\Delta}_{t_k}(r)\cdot\eta|^2\mathrm d\nu_2(r)+L\Delta \int_{-\tau}^0|\mathcal Dy^{\tilde\xi,\Delta}_{t_k}(r)\cdot\eta|^2\mathrm d\nu_1(r)\Big)\Big]\nn\\
&\leq -\Delta (2a_1-\frac{2\theta-1}{\theta^2}-L)A^{p-1}_{\theta,\Delta}\mathbb E[|\mathcal Dz^{\tilde\xi,\Delta}(t_k)\cdot\eta|^{2(p-1)}|\mathcal Dy^{\tilde\xi,\Delta}(t_k)\cdot\eta|^2]\nn\\
&\quad +\frac{L \Delta}{\epsilon^{p-1}_1} \int_{-\tau}^0 \mathbb E[|\mathcal Dy^{\tilde\xi,\Delta}_{t_k}(r)\cdot\eta|^{2p}]\mathrm d\nu_1(r)+\frac{2a_2 \Delta}{\epsilon^{p-1}_1} \int_{-\tau}^0 \mathbb E[|\mathcal Dy^{\tilde\xi,\Delta}_{t_k}(r)\cdot\eta|^{2p}]\mathrm d\nu_2(r)\nn\\
&\quad +(2a_2+L)\epsilon_1 \Delta A^p_{\theta,\Delta}\mathbb E[|\mathcal Dz^{\tilde\xi,\Delta}(t_k)\cdot\eta|^{2p}],
\end{align}
where in the last step we used the Young inequality. 
\eqref{DB} and the relation \eqref{re_d_yZ} imply  that
\begin{align}\label{re_d_yZ2}
|\mathcal Dy^{\tilde\xi,\Delta}(t_k)\cdot\eta|^2&\leq |\mathcal Dz^{\tilde\xi,\Delta}(t_k)\cdot\eta|^2+2\theta\Delta \<\mathcal Dy^{\tilde\xi,\Delta}(t_k)\cdot\eta,\mathcal Db(y^{\tilde\xi,\Delta}_{t_k})\mathcal Dy^{\tilde\xi,\Delta}_{t_k}\cdot\eta\>\nn\\
&\leq |\mathcal Dz^{\tilde\xi,\Delta}(t_k)\cdot\eta|^2+2\theta a_2\Delta \int_{-\tau}^0|\mathcal Dy^{\tilde\xi,\Delta}_{t_k}(r)\cdot\eta|^2\mathrm d\nu_2(r),
\end{align}
which leads to
\begin{align*}
&|\mathcal Dy^{\tilde\xi,\Delta}(t_k)\cdot\eta|^{2p}\leq |\mathcal Dz^{\tilde\xi,\Delta}(t_k)\cdot\eta|^{2(p-1)}|\mathcal Dy^{\tilde\xi,\Delta}(t_k)\cdot\eta|^2\\
&\quad+|\mathcal Dy^{\tilde\xi,\Delta}(t_k)\cdot\eta|^2\sum_{j=1}^{p-1}C_{p-1}^j|\mathcal Dz^{\tilde\xi,\Delta}(t_k)\cdot\eta|^{2(p-1-j)}\Big(2\theta a_2\Delta \int_{-\tau}^0|\mathcal Dy^{\tilde\xi,\Delta}_{t_k}(r)\cdot\eta|^2\mathrm d\nu_2(r)\Big)^j\\
&\leq (1+\epsilon_1)|\mathcal Dz^{\tilde\xi,\Delta}(t_k)\cdot\eta|^{2(p-1)}|\mathcal Dy^{\tilde\xi,\Delta}(t_k)\cdot\eta|^2+K(\epsilon_1)|\mathcal Dy^{\tilde\xi,\Delta}(t_k)\cdot\eta|^2\times\\
&\quad \Big(2\theta a_2\Delta\int_{-\tau}^0|\mathcal Dy^{\tilde\xi,\Delta}_{t_k}(r)\cdot\eta|^2\mathrm d\nu_2(r)\Big)^{p-1}.
\end{align*}
This implies  
\begin{align}\label{give1}
&-|\mathcal Dz^{\tilde\xi,\Delta}(t_k)\cdot\eta|^{2(p-1)}|\mathcal Dy^{\tilde\xi,\Delta}(t_k)\cdot\eta|^2\leq -(1+\epsilon_1)^{-1}|\mathcal Dy^{\tilde\xi,\Delta}(t_k)\cdot\eta|^{2p}\nn\\
&\quad+K(\epsilon_1)(1+\epsilon_1)^{-1}|\mathcal Dy^{\tilde\xi,\Delta}(t_k)\cdot\eta|^2\Big(2\theta a_2\Delta\int_{-\tau}^0|\mathcal Dy^{\tilde\xi,\Delta}_{t_k}(r)\cdot\eta|^2\mathrm d\nu_2(r)\Big)^{p-1}.
\end{align}
Inserting \eqref{give1} into \eqref{obtain1}, we obtain
\begin{align*}
&\quad\mathbb E[\mathcal I_1]\leq \Delta (2a_1-\frac{2\theta-1}{\theta^2}-L)A^{p-1}_{\theta,\Delta}\mathbb E\Big[-(1+\epsilon_1)^{-1}|\mathcal Dy^{\tilde\xi,\Delta}(t_k)\cdot\eta|^{2p}\nn\\
&\quad\ +K(\epsilon_1)(1+\epsilon_1)^{-1}|\mathcal Dy^{\tilde\xi,\Delta}(t_k)\cdot\eta|^2 \Big(2\theta a_2\Delta\int_{-\tau}^0|\mathcal Dy^{\tilde\xi,\Delta}_{t_k}(r)\cdot\eta|^2\mathrm d\nu_2(r)\Big)^{p-1}
\Big]\nn\\
&\quad\ +(2a_2+L)\epsilon_1 \Delta A^p_{\theta,\Delta}\mathbb E[|\mathcal Dz^{\tilde\xi,\Delta}(t_k)\cdot\eta|^{2p}]+\frac{2a_2}{\epsilon^{p-1}_1}\Delta \int_{-\tau}^0 \mathbb E[|\mathcal Dy^{\tilde\xi,\Delta}_{t_k}(r)\cdot\eta|^{2p}]\mathrm d\nu_2(r)\nn\\
& \quad+\frac{L}{\epsilon^{p-1}_1}\Delta \int_{-\tau}^0 \mathbb E[|\mathcal Dy^{\tilde\xi,\Delta}_{t_k}(r)\cdot\eta|^{2p}]\mathrm d\nu_1(r)\\
&\leq -\Delta (1+\epsilon_1)^{-1} (2a_1-\frac{2\theta-1}{\theta^2}-L)A^{p-1}_{\theta,\Delta}\mathbb E[|\mathcal Dy^{\tilde\xi,\Delta}(t_k)\cdot\eta|^{2p}]\\
&\quad+(2a_2+L)\epsilon_1\Delta A^p_{\theta,\Delta}\mathbb E[|\mathcal Dz^{\tilde\xi,\Delta}(t_k)\cdot\eta|^{2p}]+K(\epsilon_1)\Delta ^p\int_{-\tau}^0\mathbb E[|\mathcal Dy^{\tilde\xi,\Delta}_{t_k}\cdot\eta|^{2p}]\mathrm d\nu_2(r)\\
&\quad+K(\epsilon_1)\Delta ^p\mathbb E[|\mathcal Dy^{\tilde\xi,\Delta}(t_k)\cdot\eta|^{2p}]+\frac{2a_2}{\epsilon^{p-1}_1}\Delta \int_{-\tau}^0 \mathbb E[|\mathcal Dy^{\tilde\xi,\Delta}_{t_k}(r)\cdot\eta|^{2p}]\mathrm d\nu_2(r) \nn\\
&\quad+\frac{L}{\epsilon^{p-1}_1}\Delta \int_{-\tau}^0 \mathbb E[|\mathcal Dy^{\tilde\xi,\Delta}_{t_k}(r)\cdot\eta|^{2p}]\mathrm d\nu_1(r).
\end{align*}
For the term $\mathcal I_2$, we have
\begin{align*}
&\quad \mathbb E[\mathcal I_2]=\mathbb E\Big[A^{p-2}_{\theta,\Delta}|\mathcal Dz^{\tilde\xi,\Delta}(t_k)\cdot\eta|^{2(p-2)}\Big(-\Delta(2a_1-\frac{2\theta-1}{\theta^2})|\mathcal Dy^{\tilde\xi,\Delta}(t_k)\cdot\eta|^2\\
&\quad +2a_2\Delta \int_{-\tau}^0|\mathcal Dy^{\tilde\xi,\Delta}_{t_k}(r)\cdot\eta|^2\mathrm d\nu_2(r) +|\mathcal D\sigma(y^{\tilde\xi,\Delta}_{t_k})\mathcal Dy^{\tilde\xi,\Delta}_{t_k}\cdot\eta\delta W_k|^2\\
&\quad+2\Big\<\mathcal Dz^{\tilde\xi,\Delta}(t_k)\cdot\eta+\frac{1}{\theta}(\mathcal D y^{\tilde\xi,\Delta}(t_{k})\cdot \eta-\mathcal D z^{\tilde\xi,\Delta}(t_{k})\cdot \eta), \mathcal D\sigma(y^{\tilde\xi,\Delta}_{t_k})\mathcal Dy^{\tilde\xi,\Delta}_{t_k}\cdot\eta\delta W_k\Big\>\Big)^2\Big]\\
&\leq\E\Big[A^{p-2}_{\theta,\Delta}|\mathcal Dz^{\tilde\xi,\Delta}(t_k)\cdot\eta|^{2(p-2)}\Big(
8A_{\theta,\Delta}|\mathcal Dz^{\tilde\xi,\Delta}(t_k)\cdot\eta|^2|\mathcal D\sigma(y^{\tilde\xi,\Delta}_{t_k})\mathcal Dy^{\tilde\xi,\Delta}_{t_k}\cdot\eta|^2\Delta\nn\\
&\quad +\frac{8}{\theta^2}|\mathcal D y^{\tilde\xi,\Delta}(t_{k})\cdot \eta\mathcal D\sigma(y^{\tilde\xi,\Delta}_{t_k})\mathcal Dy^{\tilde\xi,\Delta}_{t_k}\cdot\eta|^2\Delta
+\big(-\Delta(2a_1-\frac{2\theta-1}{\theta^2})|\mathcal Dy^{\tilde\xi,\Delta}(t_k)\cdot\eta|^2\nn\\
&\quad +2a_2\Delta \int_{-\tau}^0|\mathcal Dy^{\tilde\xi,\Delta}_{t_k}(r)\cdot\eta|^2\mathrm d\nu_2(r) +|\mathcal D\sigma(y^{\tilde\xi,\Delta}_{t_k})\mathcal Dy^{\tilde\xi,\Delta}_{t_k}\cdot\eta\delta W_k|^2\big)^2\Big]\nn\\
&\leq 16\epsilon_1 \Delta A^{p}_{\theta,\Delta}\mathbb E[|\mathcal Dz^{\tilde\xi,\Delta}(t_k)\cdot\eta|^{2p}]
+K\Delta^2 A^{p}_{\theta,\Delta}\mathbb E[|\mathcal Dz^{\tilde\xi,\Delta}(t_k)\cdot\eta|^{2p}]\nn\\
&\quad+(\frac{8}{\epsilon^{p-1}_1}+\frac{4}{\epsilon_1^{\frac{p-2}{2}}}\frac{1}{\theta^p})L\Delta\mathbb E\Big[|\mathcal Dy^{\tilde\xi,\Delta}(t_k)\cdot\eta|^{2p} +\int_{-\tau}^0|\mathcal Dy^{\tilde\xi,\Delta}_{t_k}(r)\cdot\eta|^{2p}\mathrm d\nu_1(r)\Big]\\
&\quad+\frac{4}{\epsilon_1^{\frac{p-2}{2}}}\frac{1}{\theta^p}\Delta 
\mathbb E[|\mathcal Dy^{\tilde\xi,\Delta}(t_k)\cdot\eta|^{2p}]+K\Delta^2\mathbb E[|\mathcal Dy^{\tilde\xi,\Delta}(t_k)\cdot\eta|^{2p}]\\
&\quad +K\Delta^2\mathbb E\Big[\int_{-\tau}^0|\mathcal Dy^{\tilde\xi,\Delta}_{t_k}(r)\cdot\eta|^{2p}\mathrm d\nu_1(r)\Big]+K\Delta^2\mathbb E\Big[\int_{-\tau}^0|\mathcal Dy^{\tilde\xi,\Delta}_{t_k}(r)\cdot\eta|^{2p}\mathrm d\nu_2(r)\Big].
\end{align*}
Similarly, for the term $\mathcal I_i$ with $i\in\{3,\ldots,p\},$ we obtain that 
\begin{align*}
&\mathbb E[\mathcal I_i]\leq K\Delta^2A^p_{\theta,\Delta}\mathbb E[|\mathcal Dz^{\tilde\xi,\Delta}(t_k)\cdot\eta|^{2p}] +K\Delta^2\mathbb E[|\mathcal Dy^{\tilde\xi,\Delta}(t_k)\cdot\eta|^{2p}]\\
&\quad+K\Delta^2\mathbb E\Big[\int_{-\tau}^0|\mathcal Dy^{\tilde\xi,\Delta}_{t_k}(r)\cdot\eta|^{2p}\mathrm d\nu_1(r)\Big]+K\Delta^2\mathbb E\Big[\int_{-\tau}^0|\mathcal Dy^{\tilde\xi,\Delta}_{t_k}(r)\cdot\eta|^{2p}\mathrm d\nu_2(r)\Big].
\end{align*}

Therefore, we arrive at
\begin{align*}
&\mathbb E[|\mathcal Dz^{\tilde\xi,\Delta}(t_{k+1})\cdot\eta|^{2p}]\\
&\leq A^p_{\theta,\Delta}\mathbb E[|\mathcal Dz^{\tilde\xi,\Delta}(t_k)\cdot\eta|^{2p}]+p\mathbb E[\mathcal I_1]+\frac{p(p-1)}{2}\mathbb E[\mathcal I_2]+\sum_{i=3}^pC_p^i\mathbb E[\mathcal I_i]\\
&\leq (1+p(2a_2+L)\epsilon_1\Delta+\frac{16p(p-1)}{2}\epsilon_1\Delta +K\Delta^2)A^p_{\theta,\Delta}\mathbb E[|\mathcal Dz^{\tilde\xi,\Delta}(t_k)\cdot\eta|^{2p}]\\
&\quad -\Delta p(1+\epsilon_1)^{-1} (2a_1-\frac{2\theta-1}{\theta^2}-L)A^{p-1}_{\theta,\Delta}\mathbb E[|\mathcal Dy^{\tilde\xi,\Delta}(t_k)\cdot\eta|^{2p}]\\
&\quad+\frac{2pa_2}{\epsilon^{p-1}_1}\Delta \int_{-\tau}^0 \mathbb E[|\mathcal Dy^{\tilde\xi,\Delta}_{t_k}(r)\cdot\eta|^{2p}]\mathrm d\nu_2(r) +\frac{pL}{\epsilon^{p-1}_1}\Delta \int_{-\tau}^0 \mathbb E[|\mathcal Dy^{\tilde\xi,\Delta}_{t_k}(r)\cdot\eta|^{2p}]\mathrm d\nu_1(r)\\
&\quad+\frac{p(p-1)L}{2}(\frac{8}{\epsilon^{p-1}_1}+\frac{4}{\epsilon_1^{\frac{p-2}{2}}}\frac{1}{\theta^p})\Delta\mathbb E\Big[|\mathcal Dy^{\tilde\xi,\Delta}(t_k)\cdot\eta|^{2p}+\int_{-\tau}^0|\mathcal Dy^{\tilde\xi,\Delta}_{t_k}(r)\cdot\eta|^{2p}\mathrm d\nu_1(r)\Big]\\
&\quad+\frac{2p(p-1)}{\theta^p\epsilon_1^{\frac{p-2}{2}}}\Delta\mathbb E[|\mathcal Dy^{\tilde\xi,\Delta}(t_k)\cdot\eta|^{2p}]+(K\Delta^2+K(\epsilon_1)\Delta^p)\mathbb E[|\mathcal Dy^{\tilde\xi,\Delta}(t_k)\cdot\eta|^{2p}]\\
&\quad+K\Delta^2\mathbb E\Big[\int_{-\tau}^0|\mathcal Dy^{\tilde\xi,\Delta}_{t_k}(r)\cdot\eta|^{2p}\mathrm d\nu_1(r)\Big]\nn\\
&\quad+(K\Delta^2+K(\epsilon_1)\Delta^p)\mathbb E\Big[\int_{-\tau}^0|\mathcal Dy^{\tilde\xi,\Delta}_{t_k}(r)\cdot\eta|^{2p}\mathrm d\nu_2(r)\Big]\\
&\leq (1+K(\epsilon_1+\Delta)\Delta)A^{p}_{\theta,\Delta}\mathbb E[|\mathcal Dz^{\tilde\xi,\Delta}(t_k)\cdot\eta|^{2p}]+\Delta J_{1,k}+\Delta^2 J_{2,k},
\end{align*}
where
\begin{align*}
J_{1,k}&:=-p(1+\epsilon_1)^{-1} (2a_1-\frac{2\theta-1}{\theta^2}-L)A^{p-1}_{\theta,\Delta}\mathbb E[|\mathcal Dy^{\tilde\xi,\Delta}(t_k)\cdot\eta|^{2p}]\nn\\
&\quad+\frac{2pa_2}{\epsilon^{p-1}_1} \int_{-\tau}^0 \mathbb E[|\mathcal Dy^{\tilde\xi,\Delta}_{t_k}(r)\cdot\eta|^{2p}]\mathrm d\nu_2(r) +\frac{pL}{\epsilon^{p-1}_1}\int_{-\tau}^0 \mathbb E[|\mathcal Dy^{\tilde\xi,\Delta}_{t_k}(r)\cdot\eta|^{2p}]\mathrm d\nu_1(r)\\
&\quad+\frac{p(p-1)L}{2}(\frac{8}{\epsilon^{p-1}_1}+\frac{4}{\epsilon_1^{\frac{p-2}{2}}}\frac{1}{\theta^p})\mathbb E\Big[|\mathcal Dy^{\tilde\xi,\Delta}(t_k)\cdot\eta|^{2p}+\int_{-\tau}^0|\mathcal Dy^{\tilde\xi,\Delta}_{t_k}(r)\cdot\eta|^{2p}\mathrm d\nu_1(r)\Big]\\
&\quad+\frac{2p(p-1)}{\theta^p\epsilon_1^{\frac{p-2}{2}}}\mathbb E[|\mathcal Dy^{\tilde\xi,\Delta}(t_k)\cdot\eta|^{2p}],\nn\\
J_{2,k}&:=(K+K(\epsilon_1)\Delta^{p-2})\mathbb E[|\mathcal Dy^{\tilde\xi,\Delta}(t_k)\cdot\eta|^{2p}]+K\mathbb E\Big[\int_{-\tau}^0|\mathcal Dy^{\tilde\xi,\Delta}_{t_k}(r)\cdot\eta|^{2p}\mathrm d\nu_1(r)\Big]\nn\\
&\quad+(K+K(\epsilon_1)\Delta^{p-2})\mathbb E\Big[\int_{-\tau}^0|\mathcal Dy^{\tilde\xi,\Delta}_{t_k}(r)\cdot\eta|^{2p}\mathrm d\nu_2(r)\Big].
\end{align*}
Then
\begin{align}\label{DYexp1}
&\quad e^{\hat\lambda_2 pt_{k+1}}\mathbb E[|\mathcal Dz^{\tilde\xi,\Delta}(t_{k+1})\cdot\eta|^{2p}]\nn\\
&= \sum_{l=0}^{k}\Big(e^{\hat\lambda_2 pt_{l+1}}\mathbb E[|\mathcal Dz^{\tilde\xi,\Delta}(t_{l+1})\cdot\eta|^{2p}]-e^{\hat\lambda_2 pt_l}\mathbb E[|\mathcal Dz^{\tilde\xi,\Delta}(t_l)\cdot\eta|^{2p}]\Big)\nn\\
&\leq \mathbb E[|\mathcal Dz^{\tilde\xi,\Delta}(0)\cdot\eta|^{2p}]+\sum_{l=0}^k e^{\hat\lambda_2 pt_{l+1}}(\Delta J_{1,l}+\Delta^2 J_{2,l})\nn\\
&\quad+\Big((1+K(\epsilon_1+\Delta)\Delta)A^p_{\theta,\Delta}e^{\hat\lambda_2p\Delta}-1\Big)\sum_{l=0}^{k}e^{\hat\lambda_2 pt_l}\mathbb E[|\mathcal Dz^{\tilde\xi,\Delta}(t_l)\cdot\eta|^{2p}].
\end{align}
Similar to \eqref{5p2.1}, we deduce
\begin{align*}
\Delta\sum_{l=0}^ke^{\hat\lambda_2 pt_{l+1}}J_{1,l}&\leq -R(\epsilon_1, a_1)\Delta \sum_{l=0}^ke^{\hat\lambda_2 pt_{l+1}} \mathbb E[|\mathcal Dy^{\tilde\xi,\Delta}(t_l)\cdot\eta|^{2p}]\nn\\
&\quad+K(\epsilon_1)e^{\hat\lambda_2\tau}\tau\mathbb E[\|\mathcal Dy^{\tilde\xi,\Delta}_0\cdot\eta\|^{2p}]
\end{align*}
and 
\begin{align*}
\Delta^2\sum_{l=0}^ke^{\hat\lambda_2 pt_{l+1}}J_{2,l}&\leq \Delta ^2\big(K+K(\epsilon_1)\Delta^{p-2}\big)\sum_{l=0}^ke^{\hat\lambda_2 pt_{l+1}}\mathbb E[|\mathcal Dy^{\tilde\xi,\Delta}(t_l)\cdot\eta|^{2p}]\nn\\
&\quad+K(\epsilon)e^{\hat\lambda_2\tau}\tau\mathbb E[\|\mathcal Dy^{\tilde\xi,\Delta}_0\cdot\eta\|^{2p}],
\end{align*}
where 
\begin{align*}
R(\epsilon_1, a_1,\hat{\lambda}_2)&:=p(1+\epsilon_1)^{-1}(2a_1-\frac{2\theta-1}{\theta^2}-L)A^{p-1}_{\theta,\Delta}-\frac{2pa_2e^{\hat\lambda_2 \tau}}{\epsilon_1^{p-1}}-\frac{pLe^{\hat\lambda_2 \tau}}{\epsilon_1^{p-1}}\nn\\
&\quad-\frac{p(p-1)L(1+e^{\hat\lambda_2 \tau})}{2}(\frac{8}{\epsilon^{p-1}_1}+\frac{4}{\epsilon_1^{\frac{p-2}{2}}}\frac{1}{\theta^p})
-\frac{2p(p-1)}{\theta^p\epsilon_1^{\frac{p-2}{2}}}.
\end{align*}
Note that 
\begin{align*}
A_{\theta,\Delta}=1-\frac{(2\theta-1)\Delta}{\theta^2(1+\Delta)}\leq e^{-\frac{(2\theta-1)}{\theta^2(1+\Delta)}\Delta}
\end{align*}
implies  
\begin{align*}
(1+K(\epsilon_1+\Delta)\Delta)A^p_{\theta,\Delta}e^{\epsilon p\Delta}-1\leq e^{K(\epsilon_1+\Delta+\hat\lambda_2p-\frac{p(2\theta-1)}{\theta^2(1+\Delta)})\Delta}-1.
\end{align*}
Let $\epsilon_1\in(0,\frac{p(2\theta-1)}{2\theta^2})$ and denote
\begin{align*}
R^*(\epsilon_1, a_1,\hat\lambda_2)&:=p(1+\epsilon_1)^{-1}(2a_1-\frac{2\theta-1}{\theta^2}-L)\big(\frac{(1-\theta)^2}{\theta^2}+\frac{2\theta-1}{2\theta^2}\big)^{p-1}-\frac{2pa_2e^{\hat\lambda_2 \tau}}{\epsilon_1^{p-1}}\nn\\
&\quad-\frac{pLe^{\hat\lambda_2 \tau}}{\epsilon_1^{p-1}}
-\frac{p(p-1)L(1+e^{\hat\lambda_2 \tau})}{2}(\frac{8}{\epsilon^{p-1}_1}+\frac{4}{\epsilon_1^{\frac{p-2}{2}}}\frac{1}{\theta^p})
-\frac{2p(p-1)}{\theta^p\epsilon_1^{\frac{p-2}{2}}}.
\end{align*}
Take a sufficiently large number $a_1^*$ such that 
$R^*(\epsilon_1,a_1^*, 0)<0.$
Choose a sufficiently small number ${\hat{\hat \lambda}}_2 $ such that
$\epsilon_1+{\hat{\hat \lambda}}_2 p<\frac{p(2\theta-1)}{2\theta^2}$ and
$R^*(\epsilon_1,a_1^*, {\hat{\hat \lambda}}_2)<0$.
Furthermore, there exists a number $\Delta^{*}\in(0,1]$ such that for any $\Delta\in(0,\Delta^*]$,
$\epsilon_1+\Delta+{\hat{\hat \lambda}}_2 p-\frac{p(2\theta-1)}{2\theta^2}<0$ and $R(\epsilon_1, \hat a_1^*,\hat{\lambda}_2)+ \Delta(K+K(\epsilon_1)\Delta^{p-2})<0$.
Hence, combining \eqref{re_d_yZ} we have that for $a_1>a_1^*$, $\hat\lambda_2\in(0,{\hat{\hat \lambda}}_2]$, and $\Delta\in(0,\Delta^*]$,
\begin{align*}
 e^{\hat\lambda_2 pt_{k+1}}\mathbb E[|\mathcal Dz^{\tilde\xi,\Delta}(t_{k+1})\cdot\eta|^{2p}]&\leq \mathbb E[|\mathcal Dz^{\tilde\xi,\Delta}(0)\cdot\eta|^{2p}]+Ke^{\hat\lambda_2 \tau}\tau\mathbb E[\|\mathcal Dy^{\tilde\xi,\Delta}_0\cdot\eta\|^{2p}]\nn\\
&\leq K\|\eta\|^{2p}+K\E[\|\mathcal Db(\tilde\xi)\cdot \eta\|^{2p}].
\end{align*} 
According to \cref{a5,first_deri}, we have
\begin{align*}
\|\mathcal Db(\tilde\xi)\cdot \eta\|
\leq (1+\|\tilde \xi\|^{\beta})\|\eta\|,
\end{align*}
which implies
\begin{align*}
 e^{\hat\lambda_2 pt_{k+1}}\mathbb E[|\mathcal Dz^{\tilde\xi,\Delta}(t_{k+1})\cdot\eta|^{2p}]\leq 
 K\|\eta\|^{2p}(1+\E[\|\tilde \xi\|^{2p\beta}]).
\end{align*} 
It follow from \eqref{re_d_yZ} that 
\begin{align*}
 e^{\hat\lambda_2 pt_{k+1}}\mathbb E[|\mathcal Dy^{\tilde\xi,\Delta}(t_{k+1})\cdot\eta|^{2p}]\leq  K\|\eta\|^{2p}(1+\E[\|\tilde \xi\|^{2p\beta}]).
\end{align*}
\underline{Step 2: Case of $\Delta\in[\Delta^*,1]$.} Similarly, we have that for any $\epsilon_2\in(0,1)$,
\begin{align*}
&\quad\mathbb E[\mathcal I_1]
\leq -\Delta (1+\epsilon_2)^{-1} (2a_1-\frac{2\theta-1}{\theta^2}-L)A^{p-1}_{\theta,\Delta}\mathbb E[|\mathcal Dy^{\tilde\xi,\Delta}(t_k)\cdot\eta|^{2p}]+(2a_2+L)\times\nn\\
&\quad\epsilon_2A^p_{\theta,\Delta}\mathbb E[|\mathcal Dz^{\tilde\xi,\Delta}(t_k)\cdot\eta|^{2p}]
+K(\epsilon_2)(1+\epsilon_2)^{-1}\Delta ^p(2a_1-\frac{2\theta-1}{\theta^2}-L)A^{p-1}_{\theta,\Delta}\times\\
&\quad\mathbb E[|\mathcal Dy^{\tilde\xi,\Delta}(t_k)\cdot\eta|^{2p}] +\big(K(\epsilon_2)(1+\epsilon_2)^{-1}+\frac{2a_2}{\epsilon^{p-1}_2}\big)\Delta^p \int_{-\tau}^0 \mathbb E[|\mathcal Dy^{\tilde\xi,\Delta}_{t_k}(r)\cdot\eta|^{2p}]\mathrm d\nu_2(r) \\
&\quad+\frac{L}{\epsilon^{p-1}_2}\Delta^p \int_{-\tau}^0 \mathbb E[|\mathcal Dy^{\tilde\xi,\Delta}_{t_k}(r)\cdot\eta|^{2p}]\mathrm d\nu_1(r),
\end{align*}
\begin{align*}
& \mathbb E[\mathcal I_2]\leq 17\epsilon_2 A^{p}_{\theta,\Delta}\mathbb E[|\mathcal Dz^{\tilde\xi,\Delta}(t_k)\cdot\eta|^{2p}]+(\frac{8}{\epsilon^{p-1}_2}\Delta^p+\frac{4}{\epsilon_2^{\frac{p-2}{2}}}\frac{1}{\theta^p}\Delta^{\frac{p}{2}})L\times\\
&\quad\mathbb E\Big[|\mathcal Dy^{\tilde\xi,\Delta}(t_k)\cdot\eta|^{2p}+\int_{-\tau}^0|\mathcal Dy^{\tilde\xi,\Delta}_{t_k}(r)\cdot\eta|^{2p}\mathrm d\nu_1(r)\Big]\\
&\quad+\frac{4}{\epsilon_2^{\frac{p-2}{2}}}\frac{1}{\theta^p}\Delta^{\frac{p}{2}} 
|\mathcal Dy^{\tilde\xi,\Delta}(t_k)\cdot\eta|^{2p}+K(\epsilon_2)\Delta^p\mathbb E[|\mathcal Dy^{\tilde\xi,\Delta}(t_k)\cdot\eta|^{2p}]\\
&\quad+K(\epsilon_2)\Delta^p\mathbb E\Big[\int_{-\tau}^0|\mathcal Dy^{\tilde\xi,\Delta}_{t_k}(r)\cdot\eta|^{2p}\mathrm d\nu_1(r)+\int_{-\tau}^0|\mathcal Dy^{\tilde\xi,\Delta}_{t_k}(r)\cdot\eta|^{2p}\mathrm d\nu_2(r)\Big]
\end{align*}
and for $i\in\{3,\ldots,p\},$
\begin{align*}
&\mathbb E[\mathcal I_i]\leq K\epsilon_2A^p_{\theta,\Delta}\mathbb E[|\mathcal Dz^{\tilde\xi,\Delta}(t_k)\cdot\eta|^{2p}]+K(\epsilon_2)\Delta^{\frac p2}\mathbb E[|\mathcal Dy^{\tilde\xi,\Delta}(t_k)\cdot\eta|^{2p}]\\
&\quad+K(\epsilon_2)\Delta^{\frac p2}\mathbb E\Big[\int_{-\tau}^0|\mathcal Dy^{\tilde\xi,\Delta}_{t_k}(r)\cdot\eta|^{2p}\mathrm d\nu_1(r)+\int_{-\tau}^0|\mathcal Dy^{\tilde\xi,\Delta}_{t_k}(r)\cdot\eta|^{2p}\mathrm d\nu_2(r)\Big].
\end{align*}
Hence, 
\begin{align*}
&\quad e^{\hat\lambda_2 pt_{k+1}}\mathbb E[|\mathcal Dz^{\tilde\xi,\Delta}(t_{k+1})\cdot\eta|^{2p}]\leq \mathbb E[|\mathcal Dz^{\tilde\xi,\Delta}(0)\cdot\eta|^{2p}]+[(1+K\epsilon_2)A^p_{\theta,\Delta}e^{\hat\lambda_2p\Delta}-1]\times\nn\\
&\quad\sum_{l=0}^ke^{\hat\lambda_2 pt_l}\mathbb E[|\mathcal Dz^{\tilde\xi,\Delta}(t_{l})\cdot\eta|^{2p}]-\Big(p(1+\epsilon_2)^{-1}(2a_1-\frac{2\theta-1}{\theta^2}-L)A^{p-1}_{\theta,\Delta}-\big(\frac{2pa_2e^{\hat\lambda_2 \tau}}{\epsilon_2^{p-1}}\nn\\
&\quad+\frac{pLe^{\hat\lambda_2 \tau}}{\epsilon_2^{p-1}}\big)\Delta^{p-1}
-\frac{p(p-1)L(1+e^{\hat\lambda_2\tau})}{2}\big(\frac{8}{\epsilon^{p-1}_2}\Delta^{p-1}
+\frac{4}{\epsilon_2^{\frac{p-2}{2}}}\frac{1}{\theta^p}\Delta^{\frac{p-2}{2}}\big)
\\
&\quad
-K(\epsilon_2)\Delta^{\frac{p-2}{2}}\Big)\Delta\sum_{l=0}^ke^{\hat\lambda_2 pt_{l+1}}\mathbb E[|\mathcal Dy^{\tilde\xi,\Delta}(t_l)\cdot\eta|^{2p}]+Ke^{\hat\lambda_2\tau}\tau \mathbb E[\|\mathcal Dy^{\tilde\xi,\Delta}_0\cdot\eta\|^{2p}].
\end{align*}
Noting 
$$(1+K\epsilon_2)A^p_{\theta,\Delta}e^{\hat\lambda_2p\Delta}\leq e^{K\epsilon_2+\lambda_2 p\Delta-\frac{p(2\theta-1)\Delta^*}{2\theta^2}},$$
choose a sufficiently small number $\epsilon_2>0$ such that 
$ K\epsilon_2<\frac{p(2\theta-1)\Delta^*}{2\theta^2}.$
Letting $a_1^*$ be sufficiently large such that
\begin{align*}
&\quad p(1+\epsilon_2)^{-1}(2a_1^*-\frac{2\theta-1}{\theta^2}-L)\big(\frac{(1-\theta)^2}{\theta^2}+\frac{2\theta-1}{2\theta^2}\big)^{p-1}-\big(\frac{2pa_2}{\epsilon_2^{p-1}}+\frac{pL}{\epsilon_2^{p-1}}\big)(\Delta^*)^{p-1}\nn\\
&\quad-p(p-1)L(\frac{8}{\epsilon^{p-1}_2}(\Delta^*)^{p-1}
+\frac{4}{\epsilon_2^{\frac{p-2}{2}}}\frac{1}{\theta^p}(\Delta^*)^{\frac{p-2}{2}})-K(\epsilon_2)(\Delta^*)^{\frac{p-2}{2}}>0.
\end{align*}
Letting $\hat\lambda_2\in(0,{\hat{\hat\lambda}}_2]$ 
 be sufficiently small, we obtain that for  $a_1\geq a_1^*$ and $\Delta\in(\Delta^*,1]$,
\begin{align*}
e^{\hat\lambda_2 pt_{k+1}}\mathbb E[\|\mathcal Dz^{\tilde\xi,\Delta}_{t_{k}}\cdot\eta\|^{2p}]&\leq K(\|\mathcal Dz^{\tilde\xi,\Delta}_0\cdot\eta\|^{2p}+\|\mathcal Dy^{\tilde\xi,\Delta}_0\cdot\eta\|^{2p})\nn\\
&\leq K\|\eta\|^{2p}(1+\E[\|\tilde \xi\|^{2p\beta}]).
\end{align*}

Therefore, similar to the proof of Proposition \ref{l4.3}, the desired argument follows from \underline{Steps 1--2} and \eqref{re_d_yZ2}. The proof is completed.
\end{proof}

\begin{lemma}\label{Dy_esti2}
Let Assumptions \ref{a1}--\ref{a2}  and \ref{first_deri} hold. Then for $p\in\mathbb N_+$, $a_1>a_1^*$, and $\Delta\in(0,1]$, 
\begin{align*}
\mathbb E[\|D_uy^{\tilde\xi,\Delta}_{t_k}\|^{2p}]\leq Ke^{-\hat{\lambda}_2 p(t_k-u)}(1+\mathbb E[\|D_u\tilde\xi\|^{2p}]+\E[|\mathcal Db(\tilde \xi)D_{u}\tilde\xi|^{2p}]),
\end{align*}
where $\tilde\xi\in L^{2p}(\Omega;\mathcal C^d)$, $D_u\tilde\xi\in L^{2p}(\Omega;\mathcal C^d\otimes \mathbb R^m)$, $\mathcal Db(\tilde \xi)D_{u}\tilde\xi\in L^{2p}(\Omega; \RR^{d\times m})$, $\eta\in\mathcal C^d,$ $K>0$, and $a_1^*,  \,\hat{\lambda}_2>0$ are given in Lemma \ref{Dy_esti}.
\end{lemma}
\begin{proof}
The proof is similar to that of Lemma \ref{Dy_esti} and is omitted. 
\end{proof}

\begin{assp}\label{second_deri}
Assume that coefficients $b$ and $\sigma$ have continuous derivatives up to order $2$ satisfying  that for any $\phi,\phi_1,\phi_2\in\mathcal C^d,$
\begin{align*}
&|\mathcal D^2b(\phi)(\phi_1,\phi_2)|\leq K(1+\|\phi\|^{(\beta-1)\vee 0})\|\phi_1\|\|\phi_2\|,\\
&|\mathcal D^2\sigma(\phi)(\phi_1,\phi_2)|\leq K\|\phi_1\|\|\phi_2\|,
\end{align*}
where $K>0,$  the constant $\beta>0$ is given in Assumption \ref{a5}.
\end{assp}
\begin{lemma}\label{lem_sec_d1}
Let Assumptions \ref{a1}--\ref{a2} with $a_1-a_2-(2p-1)L>0$ hold for some $p\ge 1$. And let Assumptions \ref{a4}--\ref{a5} and  \ref{first_deri}--\ref{second_deri} hold. Then for any $t\geq0$ and $u\geq-\tau$,
\begin{align*}
&\mathbb E[|D_u\mathcal Dx^{\tilde\xi}(t)\cdot\eta|^{2p}]\leq Ke^{-\hat{\lambda}_1 pt}\|\eta\|^{2p} (1+\mathbb E[\|D_u\tilde\xi\|^{8p}]),\\
&\mathbb E[\|D_u\mathcal Dx^{\tilde\xi}_t\cdot\eta\|^{2p}]\leq Ke^{-\hat{\lambda}_1 pt}\|\eta\|^{2p}(1
+\mathbb E[\|D_u\tilde\xi\|^{8p}]),
\end{align*}
where $\tilde\xi\in L^{2p((2\beta-2)\vee1)}(\Omega;\mathcal C^d)$ and $D_u\tilde\xi\in L^{8p}(\Omega;\mathcal C^d\otimes \mathbb R^m)$, $\eta\in\mathcal C^d,$ $K>0$, and  $\hat \lambda_1>0$ is given in  Lemma \ref{lem_fir_d}.
\end{lemma}
\begin{proof}
For  $u\leq t,$ $D_u\mathcal Dx^{\tilde\xi}(t)\cdot \eta$ satisfies
\begin{align*}
D_u\mathcal Dx^{\tilde\xi}(t)\cdot\eta&=\int_u^t\Big(\mathcal D^2b(x^{\tilde\xi}_s)(D_ux^{\tilde\xi}_s,\mathcal Dx^{\tilde\xi}_s\cdot\eta)+\mathcal Db(x^{\tilde\xi}_s)D_u \mathcal D x^{\tilde\xi}_s\cdot\eta\Big)\mathrm ds\\
&\quad +\int_u^t\Big(\mathcal D^2\sigma(x^{\tilde\xi}_s)(D_ux^{\tilde\xi}_s,\mathcal Dx^{\tilde\xi}_s\cdot\eta)+\mathcal D\sigma(x^{\tilde\xi}_s)D_u\mathcal Dx^{\tilde\xi}_s\cdot\eta\Big)\mathrm dW(s)\\
&\quad +\mathcal D\sigma(x^{\tilde\xi}_u)\mathcal Dx^{\tilde\xi}_u\cdot\eta\mathbf 1_{[0,t]}(u).
\end{align*}
By the It\^o formula, we have 
\begin{align*}
&\quad\mathrm d(e^{\hat\lambda_1 pt}|D_u\mathcal Dx^{\tilde\xi}(t)\cdot\eta|^{2p})\leq \hat\lambda_1 pe^{\hat\lambda_1 pt}|D_u\mathcal Dx^{\tilde\xi}(t)\cdot\eta|^{2p}\mathrm dt+2pe^{\hat\lambda_1 pt}|D_u\mathcal Dx^{\tilde\xi}(t)\cdot\eta|^{2p-2}\times \\
&\quad\Big[\Big\<D_u\mathcal Dx^{\tilde\xi}(t)\cdot\eta,\mathcal D^2b(x^{\tilde\xi}_t)(D_ux^{\tilde\xi}_t,\mathcal Dx^{\tilde\xi}_t\cdot\eta)+\mathcal Db(x^{\tilde\xi}_t)D_u\mathcal Dx^{\tilde\xi}_t\cdot\eta\Big\>\mathrm dt\\
&\quad+\Big\<D_u\mathcal Dx^{\tilde\xi}(t)\cdot\eta,\Big(\mathcal D^2\sigma(x^{\tilde\xi}_t)(D_ux^{\tilde\xi}_t,\mathcal Dx^{\tilde\xi}_t\cdot\eta)+\mathcal D\sigma(x^{\tilde\xi}_t)D_u\mathcal Dx^{\tilde\xi}_t\cdot\eta\Big)\mathrm dW(t)\Big\>\Big]\\
&\quad+p(2p-1)e^{\hat\lambda_1 pt}|D_u\mathcal Dx^{\tilde\xi}(t)\cdot\eta|^{2p-2}|\mathcal D^2\sigma(x^{\tilde\xi}_t)(D_ux^{\tilde\xi}_t,\mathcal Dx^{\tilde\xi}_t\cdot\eta)\nn\\
&\quad+\mathcal D\sigma(x^{\tilde\xi}_t)D_u\mathcal Dx^{\tilde\xi}_t\cdot\eta|^2\mathrm dt.
\end{align*}
Taking expectation on both sides, and using  Assumptions \ref{first_deri} and  \ref{second_deri}, and the Young inequality, we obtain
\begin{align*}
&\quad\mathbb E[e^{\hat\lambda_1 pt}|D_u\mathcal Dx^{\tilde\xi}(t)\cdot\eta|^{2p}]\!\leq\! \mathbb E[e^{\hat\lambda_1 pu}|\mathcal D\sigma(x^{\tilde\xi}_u)\mathcal Dx^{\tilde\xi}_u\cdot\eta|^{2p}]\!\!+\!\!\int_u^t\!\!\mathbb E\Big[\hat\lambda_1 pe^{\hat\lambda_1 ps}|D_u\mathcal Dx^{\tilde\xi}(s)\cdot\eta|^{2p}\\
&\quad+2pe^{\hat\lambda_1 ps}|D_u\mathcal Dx^{\tilde\xi}(s)\cdot\eta|^{2p-2}\Big(-a_1|D_u\mathcal Dx^{\tilde\xi}(s)\cdot\eta|^2+a_2\int_{-\tau}^0|D_u\mathcal Dx^{\tilde\xi}(s+r)\cdot\eta|^2\mathrm d\nu_2(r)\Big)\\
&\quad+Ke^{\hat\lambda_1 ps}|D_u\mathcal Dx^{\tilde\xi}(s)\cdot\eta|^{2p-1}(1+\|x^{\xi}_s\|^{\beta-1})\|D_u x^{\xi}_s\|\|\mathcal Dx^{\tilde\xi}_s\cdot\eta\|\nn\\
&\quad+p(2p-1)e^{\hat\lambda_1 ps}|D_u\mathcal Dx^{\tilde\xi}(s)\cdot\eta|^{2p-2}\Big((1+\frac{1}{\gamma_1})|\mathcal D^2\sigma(x^{\tilde\xi}_s)(D_ux^{\tilde\xi}_s,\mathcal Dx^{\tilde\xi}_s\cdot\eta)|^2\nn\\
&\quad+(1+\gamma_1)L\big(|D_u\mathcal Dx^{\tilde\xi}(s)\cdot\eta|^2+\int_{-\tau}^0|D_u\mathcal Dx^{\tilde\xi}(s+r)\cdot\eta|^2\mathrm d\nu_1(r)\big)\Big)\Big]
\mathrm ds\\
&\leq \mathbb E[e^{\hat\lambda_1 pu}|\mathcal D\sigma(x^{\tilde\xi}_u)\mathcal Dx^{\tilde\xi}_u\cdot\eta|^{2p}]+p\big(\hat\lambda_1-2a_1+2a_2e^{\hat\lambda_1 p\tau}+(2p-1)(1+\gamma_1)L(1+e^{\hat\lambda_1 p\tau})\big)\times\nn\\
&\quad\int_u^te^{\hat\lambda_1 ps}\mathbb E[|D_u\mathcal Dx^{\tilde\xi}(s)\cdot\eta|^{2p}]\mathrm ds+K\int_u^te^{\hat\lambda_1 ps}\mathbb E[|D_u\mathcal Dx^{\tilde\xi}(s)\cdot\eta|^{2p-1}\times\nn\\
&\quad(1\!+\!\|x^{\xi}_s\|^{\beta-1})\|D_ux^{\tilde\xi}_s\|\|\mathcal Dx^{\tilde\xi}_s\cdot\eta\|]\mathrm ds\!+\!p(2p\!-\!1)(1\!+\!\frac{1}{\gamma_1})\int_u^te^{\hat\lambda_1 ps}\mathbb E[|D_u\mathcal Dx^{\tilde\xi}(s)\cdot\eta|^{2p-2}\nn\\
&\quad\times|\mathcal D^2\sigma(x^{\tilde\xi}_s)D_ux^{\tilde\xi}_s\mathcal Dx^{\tilde\xi}_s\cdot\eta|^2]\mathrm ds\\
&\leq \mathbb E[e^{\hat\lambda_1 pu}|\mathcal D\sigma(x^{\tilde\xi}_u)\mathcal Dx^{\tilde\xi}_u\cdot\eta|^{2p}]+p\big(\hat\lambda_1-2a_1+2a_2e^{\hat\lambda_1 p\tau}+(2p-1)(1+\gamma_1)L(1+e^{\hat\lambda_1 p\tau})\big)\times\nn\\
&\quad\int_u^te^{\hat\lambda_1 ps}\mathbb E[|D_u\mathcal Dx^{\tilde\xi}(s)\cdot\eta|^{2p}]\mathrm ds+\int_u^t\Big(p\gamma_2e^{\hat\lambda_1 ps}\mathbb E[|D_u\mathcal Dx^{\tilde\xi}(s)\cdot\eta|^{2p}]
+pK(\gamma_2)e^{\hat\lambda_1 ps}\mathbb E[(1\nn\\
&\quad+\|x^{\tilde\xi}_s\|^{2(\beta-1)p})\|D_ux^{\tilde\xi}_s\|^{2p}\|\mathcal Dx^{\tilde\xi}_s\cdot\eta\|^{2p}]\Big)\mathrm ds+\int_u^t p(2p-1)\gamma_{3}e^{\hat\lambda_1 ps}\mathbb E[|D_u\mathcal Dx^{\tilde\xi}(s)\cdot\eta|^{2p}]\\
&\quad+p(2p-1)K(\gamma_3)e^{\hat\lambda_1 ps}\mathbb E[\|D_ux^{\tilde\xi}_s\|^{2p}\|\mathcal Dx^{\tilde\xi}_s\cdot\eta\|^{2p}]\mathrm ds.
\end{align*}
By taking $\gamma_1,\gamma_2,\gamma_3$ sufficiently small such that 
\begin{align*}
\hat\lambda_1-2a_1+2a_2e^{\hat\lambda_1 p\tau}+(2p-1)(1+\gamma_1)L(1+e^{\hat\lambda_1 p\tau})+\gamma_2+(2p-1)\gamma_3\leq 0,
\end{align*}
  and combining Lemmas \ref{lem_fir_d}--\ref{lem_fir_d2},
  we arrive at
\begin{align*}
&\quad\mathbb E[e^{\hat\lambda_1 pt}|D_u\mathcal Dx^{\tilde\xi}(t)\cdot\eta|^{2p}]
\leq \mathbb E[e^{\hat\lambda_1 pu}|\mathcal D\sigma(x^{\tilde\xi}_u)\mathcal Dx^{\tilde\xi}_u\cdot\eta|^{2p}]+K\int_u^te^{\hat\lambda_1 ps}\mathbb E[(1\nn\\
&\quad+\|x^{\tilde\xi}_s\|^{2p((\beta-1)\vee0)})\|D_ux^{\tilde\xi}_s\|^{2p}\|\mathcal Dx^{\tilde\xi}_s\cdot\eta\|^{2p} +\|D_ux^{\tilde\xi}_s\|^{2p}\|\mathcal Dx^{\tilde\xi}_s\cdot\eta\|^{2p}]\mathrm ds\\
&\leq \mathbb E[e^{\hat\lambda_1 pu}|\mathcal D\sigma(x^{\tilde\xi}_u)\mathcal Dx^{\tilde\xi}_u\cdot\eta|^{2p}]+K\int_u^te^{\hat\lambda_1 ps}e^{-\hat{\lambda}_1 ps-\hat{\lambda}_1p(s-u)}\|\eta\|^{2p}\times\nn\\
&\quad(1+\mathbb E[\|D_u\tilde\xi\|^{8p}])\mathrm ds\\
&\leq Ke^{\hat\lambda_1 pu}\mathbb E[\|\mathcal Dx^{\tilde\xi}_u\cdot\eta\|^{2p}]+K\|\eta\|^{2p}(1+\mathbb E[\|D_u\tilde\xi\|^{8p}])\nn\\
&\leq K \|\eta\|^{2p}(1+\mathbb E[\|D_u\tilde\xi\|^{8p}]),
\end{align*}
which implies  
\begin{align*}
\mathbb E[|D_u\mathcal Dx^{\tilde\xi}(t)\cdot\eta|^{2p}]\leq Ke^{-\hat\lambda_1 pt}\|\eta\|^{2p}(1+\mathbb E[\|D_u\tilde\xi\|^{8p}]).
\end{align*}
Similar to the proof of Lemmas \ref{lem_fir_d}--\ref{lem_fir_d2}, by using the Burkholder--Davis--Gundy inequality, we finish the proof.
\end{proof}

\begin{lemma}\label{Ch3D2x}
Let Assumptions \ref{a1}--\ref{a2} with $a_1-a_2-(2p-1)L>0$ hold for some $p\ge1$. And let Assumptions \ref{a4}--\ref{a5} and  \ref{first_deri}--\ref{second_deri} hold. 
Then for $t\geq0$, $u\geq-\tau$, and $w\geq -\tau$,
\begin{align*}
&\mathbb E[|D_wD_ux^{\tilde\xi}(t)|^{2p}]\!\leq\!Ke^{-\hat{\lambda}_1p(t\!-\!u\vee w)}(1\!+\!\mathbb E[\|D_u\tilde\xi\|^{8p}]\!+\!\mathbb E[\|D_w\xi\|^{8p}]\!+\!\mathbb E[\|D_wD_u\tilde\xi\|^{2p}]),\\
&\mathbb E[\|D_wD_ux^{\tilde\xi}_t\|^{2p}]\!\leq\! Ke^{-\hat{\lambda}_1p(t\!-\!u\vee w)}(1\!+\!\mathbb E[\|D_u\tilde\xi\|^{8p}]\!+\!\mathbb E[\|D_w\xi\|^{8p}]\!+\!\mathbb E[\|D_wD_u\tilde\xi\|^{2p}]),
\end{align*}
where $\tilde\xi\in L^{2p((2\beta-2)\vee1)}(\Omega;\mathcal C^d)$,  $D_u\tilde\xi\in L^{8p}(\Omega;\mathcal C^d\otimes \mathbb R^m)$, $D_wD_u\tilde\xi\in L^{2p}(\Omega;\mathcal C^d)$, $K>0$, and $\hat \lambda_1$ is given by Lemma \ref{first_deri}.
\end{lemma}
\begin{proof}The proof is similar to that of Lemma \ref{lem_sec_d1}. We only give the sketch of the proof.  
$D_wD_ux^{\tilde\xi}(t)$ satisfies
\begin{align*}
&\quad D_wD_ux^{\tilde\xi}(t)=\int_{u\vee w}^t\Big(\mathcal D^2b(x^{\tilde\xi}_s)(D_wx^{\tilde\xi}_s,D_ux^{\tilde\xi}_s)+\mathcal Db(x^{\tilde\xi}_s)D_wD_ux^{\tilde\xi}_s\Big)\mathrm ds\nn\\
&\quad+\!\int_{u\vee w}^t\Big(\mathcal D^2\sigma(x^{\tilde\xi}_s)(D_wx^{\tilde\xi}_s,D_ux^{\tilde\xi}_s\big)\!+\!\mathcal D\sigma(x^{\tilde\xi}_s)D_wD_ux^{\tilde\xi}_s\Big)\mathrm dW(s)+\mathcal D\sigma(x^{\tilde\xi}_w)D_ux^{\tilde\xi}_w\mathbf 1_{[u,t]}(w)\nn\\
&\quad+\mathcal D\sigma(x^{\tilde\xi}_u)D_wx^{\tilde\xi}_u\mathbf 1_{[w,t]}(u)+D_wD_u\tilde\xi(0).
\end{align*}
 Applying the It\^o formula, \eqref{DB}--\eqref{DS}, and Assumption \ref{second_deri}, we deduce
\begin{align*}
&\quad \mathbb E[e^{\hat\lambda_1 pt}|D_wD_ux^{\tilde\xi}(t)|^{2p}]\leq \mathbb E\Big[e^{\hat\lambda_1 p(u\vee w)}|\mathcal D\sigma(x^{\tilde\xi}_w)D_ux^{\tilde\xi}_w\mathbf 1_{[u,u\vee w]}(w)\nn\\
&\quad+\mathcal D\sigma (x^{\xi}_u)D_wx^{\tilde\xi}_u\mathbf 1_{[w,u\vee w]}(u)
+D_wD_u\tilde\xi(0)|^{2p}\Big]\nn\\
&\quad+K\int_{u\vee w}^te^{\hat\lambda_1 ps}\mathbb E[(1+\|x^{\tilde\xi}_s\|^{2p((\beta-1)\vee0)})\|D_wx^{\tilde\xi}_s\|^{2p}\|D_ux^{\tilde\xi}_s\|^{2p}]\mathrm ds.
\end{align*}
It follows from Lemmas \ref{l4.1}  and \ref{lem_fir_d2} that
\begin{align*}
\mathbb E[|D_wD_ux^{\tilde\xi}(t)|^{2p}]\leq Ke^{-\hat\lambda_1 p(t-u\vee w)}(1\!+\!\mathbb E[\|D_u\tilde\xi\|^{8p}]\!+\!\mathbb E[\|D_w\xi\|^{8p}]\!+\!\mathbb E[\|D_wD_u\tilde\xi\|^{2p}]).
\end{align*}
Similarly, one can derive $$\mathbb E[\|D_wD_ux^{\tilde\xi}_t\|^{2p}]\leq Ke^{-\hat\lambda_1 p(t-u\vee w)}(1\!+\!\mathbb E[\|D_u\tilde\xi\|^{8p}]\!+\!\mathbb E[\|D_w\xi\|^{8p}]\!+\!\mathbb E[\|D_wD_u\tilde\xi\|^{2p}]).$$ The proof is completed. 
\end{proof}

In order to derive estimates of  high-order derivatives of the numerical solution, we give the following assumption on the third order derivative of coefficients $b$ and $\sigma.$
\begin{assp}\label{assp_third}
Assume that coefficients $b$ and $\sigma$ have continuous derivatives up to order $3$ satisfying that for any $\phi,\phi_1,\phi_2,\phi_3\in\mathcal C^d,$
\begin{align*}
&|\mathcal D^3b(\phi)(\phi_1,\phi_2,\phi_3)|\leq K(1+\|\phi\|^{(\beta-2)\vee0})\|\phi_1\|\|\phi_2\|\|\phi_3\|,\\
&|\mathcal D^3\sigma(\phi)(\phi_1,\phi_2,\phi_3)|\leq K\|\phi_1\|\|\phi_2\|\|\phi_3\|,
\end{align*}
where $K>0$, the constant $\beta>0$ is given in Assumption \ref{a5}.
\end{assp}

Similarly, we have the following estimates of the numerical solution. 

\begin{lemma}\label{lemma_high3}
Let Assumptions \ref{a1}--\ref{a2}, \ref{a4}--\ref{a5}, \ref{first_deri}--\ref{second_deri} hold. 
Then for $p\in\mathbb N_+$, there exists a number $a^{**}_1(L,a_2,\theta,p)>0$ such that for  $a_1>a_1^{**}$ and $\Delta\in(0,1]$, 
\begin{align*}
&\quad \mathbb E[\|D_u\mathcal Dy^{\tilde\xi,\Delta}_{t_k}\cdot\eta\|^{2p}]
\leq Ke^{-\!\hat{\lambda}_2t_k}\|\eta\|^{2p}\big(1+\mathbb E[\|D_u\tilde\xi\|^{16p}]\big),\\
&\quad \mathbb E[\|D_wD_uy^{\tilde\xi,\Delta}_{t_k}\|^{2p}]\leq Ke^{-\hat{\lambda}_2(t_k-u\vee w)}\big(1+\mathbb E[\|D_u\tilde\xi\|^{16p}]+\mathbb E[\|D_w\tilde\xi\|^{16p}]\nn\\
&\quad\quad+\mathbb E[\|D_wD_u\tilde\xi\|^{2p}]\big),
\\
&\quad \mathbb E[\|D_wD_u\mathcal Dy^{\tilde\xi,\Delta}_{t_{k}}\cdot\eta\|^{2p}]\leq Ke^{-\hat{\lambda}_2t_k}\|\eta\|^{2p}\big(1+\mathbb E[\|D_u\tilde\xi\|^{32p}]\nn\\
&\quad\quad+\mathbb E[\|D_w\tilde\xi\|^{32p}]+
\E[\|D_wD_u\tilde\xi\|^{4p}]\big),
\end{align*}
where $\tilde\xi\in L^{p'}(\Omega;\mathcal C^d)$  with some $p'>0$,  $\eta\in\mathcal C^d$,  $K>0$,  and $\hat \lambda_2>0$ is a sufficiently small number given in  Lemma \ref{Dy_esti}.
\end{lemma}

\subsection{Weak convergence rate}For the SFDE with the distributed delay argument, under the globally Lipschitz continuous coefficients condition, 
there have been some works devoted to 
the weak convergence of numerical methods on the finite time horizon.  For instance, authors in  \cite{Buchwar2005weak} investigate the weak convergence of the EM method by using the infinite-dimensional version of the Kolmogorov equation; Authors in \cite{CEK2006} obtain the weak convergence rate of the EM method by developing a dual approach. To deepen the study,
this subsection is devoted to presenting the weak convergence rate on the infinite time horizon of the $\theta$-EM method for the SFDE with superlinearly growing coefficients. Before that, we
propose the following assumption on the coefficients $b$ and $\sigma$. 
 
\begin{assp}\label{ass_nablab}
Assume that there exist constants $n_b,n_{\sigma}\in\mathbb N_+$ and  probability measures $\nu_3^{i},\nu_4^{j}$ on $[-\tau,0]$ for $i \in\{1,\ldots, n_b\}$, $j \in\{1,\ldots, n_{\sigma}\}$,  such that coefficients $b$ and $\sigma$ satisfy that for any $\phi_1,\phi_2\in\mathcal C^d,$
\begin{align*}
&\mathcal  Db(\phi_1)\phi_2=\sum_{i=1}^{n_b}\int_{-\tau}^0k_b^{i}(\phi_1(s))\phi_2(s)\mathrm d\nu^{i}_3(s),\\
&\mathcal  D\sigma (\phi_1)\phi_2=\sum_{j=1}^{n_{\sigma}}\int_{-\tau}^0k^{j}_{\sigma}(\phi_1(s))\phi_2(s)\mathrm d\nu_4^{j}(s),
\end{align*}
where $k_b^{i},k^{j}_{\sigma}:\mathcal C^d\to\mathcal C([-\tau,0];\mathbb R^{d\times d})$ satisfy  
\begin{align*}
&\Big\|\mathcal D k_{b}^{i}(\phi_1)\Big\|_{\mathcal L(\mathcal C^d;
\mathcal C([-\tau,0];\mathbb R^{d\times d})}
+\|k_{b}^{i}(\phi_1)\|\leq K(1+\|\phi_1\|^{\beta}),\\
&\Big\|\mathcal Dk^{j}_{\sigma}(\phi_1)\Big\|_{\mathcal L(\mathcal C^d;\mathcal C([-\tau,0];\mathbb R^{d\times d})} 
+\|k^{j}_{\sigma}(\phi_1)\|\leq K.
\end{align*} 
\end{assp}

Examples satisfying Assumption \ref{ass_nablab} are given as follows. 
\begin{expl}
Coefficients $b$ and $\sigma$ considered in \cite{Buchwar2005weak, CEK2006}  satisfy Assumption \ref{ass_nablab}, where the coefficients have the forms: $b(\int_{-\tau}^0 \phi(r)\mathrm d \nu_3(r))$ and $\sigma(\int_{-\tau}^0 \phi(r)\mathrm d \nu_4(r))$.  
In fact, it can be calculated  that for $\phi_1\in\mathcal C^d,$ 
 \begin{align*}
&\mathcal D b(\int_{-\tau}^0 \phi(r)\mathrm d \nu_3(r)) \phi_1=\int_{-\tau}^0\mathcal Db (\int_{-\tau}^0 \phi(r)\mathrm d \nu_3(r))\phi_1(s)\mathrm d\nu_3(s),\\
&\mathcal D\sigma(\int_{-\tau}^0 \phi(r)\mathrm d \nu_4(r)) \phi_1=\int_{-\tau}^0\mathcal D\sigma  (\int_{-\tau}^0 \phi(r)\mathrm d \nu_4(r))\phi_1(s)\mathrm d\nu_4(s). 
\end{align*}
\end{expl}

  \begin{expl}
Let $b(\phi)=\frac{1}{\tau}\int_{-\tau}^0\phi(s)\mathrm d s-|\phi(0)|^2\phi(0)-\phi(0)$,  $\phi\in\mathcal C^d$. Then it can be verified that Assumption \ref{ass_nablab} is satisfied: for $\phi_1\in\mathcal C^d,$ $$\mathcal Db(\phi)\phi_1=\frac{1}{\tau} \int_{-\tau}^0\phi_1(s)\mathrm ds-2 \phi(0)^{\top}\phi_1(0)\phi(0)-|\phi(0)|^2\phi_1(0)-\phi_1(0).$$
In this example, $n_b=2,$  $k_b^1(\phi)=\mathrm{Id}_{d\times d}, \,k^2_b(\phi)=-2\phi \phi^{\top}-\mathrm{Id}_{d\times d},\,\nu_3^1(\mathrm ds)=\frac{1}{\tau}\mathrm ds,\,\nu_3^2(\mathrm ds)=\delta_0(\mathrm ds).$
\end{expl}

\begin{thm}\label{thm_weakcon}
Let conditions in Lemma \ref{lemma_high3}, Assumption \ref{a7} with $\rho\geq 1$, and Assumption  \ref{ass_nablab} hold. Then for $f\in \mathcal C^3_b$,
\begin{align*}
|\mathbb E[f(x^{\xi}(T))]-\mathbb E[f(y^{\xi,\Delta}(T))]|\leq K\Delta,
\end{align*}
where $K$ is independent of $T$.
\end{thm}
\begin{proof} 
To simplify the proof, we take parameters $n_b,n_{\sigma}$ in Assumption \ref{ass_nablab} to be $n_b=n_{\sigma}=1$. The case of  $n_b>1, \, n_{\sigma}>1$ can be proved similarly. 

\textbf{Estimate of  term $\mathcal I_{0}.$}
Term $\mathcal I_0$ is estimated as 
\begin{align*}
&\quad \mathcal I_0
=\mathbb E[f(\varphi^{Int}(t_{N^{\Delta}};0,\xi)]-\mathbb E[f(\varphi^{Int}(t_{N^{\Delta}};0,\Phi_0))]\\
&=\int_0^1\mathbb E\Big[f'(\varphi^{Int}(t_{N^{\Delta}};0,\varsigma\xi+(1-\varsigma)\Phi_0))\mathcal D \varphi^{Int}(t_{N^{\Delta}};0,\varsigma\xi+(1-\varsigma)\Phi_0)(\xi-\Phi_0)\Big]\mathrm d\varsigma\\
&\leq K\int_0^1\|\mathcal D \varphi^{Int}(t_{N^{\Delta}};0,\varsigma\xi+(1-\varsigma)\Phi_0)(\xi-\Phi_0)\|_{0,2}\mathrm d\varsigma\\
&\leq K\Delta\int_0^1\|\mathcal D \varphi^{Int}(t_{N^{\Delta}};0,\varsigma\xi+(1-\varsigma)\Phi_0)\|_{0,2}\mathrm d\varsigma\\
&\leq K\Delta,
\end{align*}
where we used $\sup_{\phi\in\mathbb R^d}|f'(\phi)|\leq K$ and  Assumption \ref{a7}. 

\textbf{Estimate of  term $\mathcal I_{b}.$} By the Taylor formula, we have 
\begin{align*}
&\mathcal I_b=\sum_{i=1}^{N^{\Delta}}\int_0^1\mathbb E\Big[\Big\< (\mathcal D\Phi(T;t_i,Y^{\varsigma}_i) I^0\mathrm {Id}_{d\times d})^*f'(\Phi(T;t_i,Y^{\varsigma}_i)),\int_{t_{i-1}}^{t_i}\int_0^1\mathcal Db(Z^{\beta_1}_{i,r})\\
&\quad \big(\varphi_r(t_{i-1},\varphi_{t_{i-1}}(0,\Phi_0))-\varphi^{Int}_{t_{i-1}}(0,\Phi_0)\big)\mathrm d\beta_1\mathrm dr\Big\>\Big]\mathrm d\varsigma=:\mathcal I_{b,1}+\mathcal I_{b,2},
\end{align*}
where $Z^{\beta_1}_{i,r}=\beta_1\varphi_r(t_{i-1},\varphi_{t_{i-1}}(0,\Phi_0))+(1-\beta_1)\varphi^{Int}_{t_{i-1}}(0,\Phi_0),$ and 
\begin{align*}
&\mathcal I_{b,1}=\sum_{i=1}^{N^{\Delta}}\int_0^1\mathbb E\Big[\Big\< (\mathcal D\Phi(T;t_i,Y^{\varsigma}_i) I^0\mathrm {Id}_{d\times d})^*f'(\Phi(T;t_i,Y^{\varsigma}_i)),\int_{t_{i-1}}^{t_i}\int_0^1\mathcal Db(Z^{\beta_1}_{i,r})\\
&\quad \big(\varphi^{Int}_r(t_{i-1},\varphi_{t_{i-1}}(0,\Phi_0))-\varphi^{Int}_{t_{i-1}}(0,\Phi_0)\big)\mathrm d\beta_1\mathrm dr\Big\>\Big]\mathrm d\varsigma,\\
&\mathcal I_{b,2}=\sum_{i=1}^{N^{\Delta}}\int_0^1\mathbb E\Big[\Big\< (\mathcal D\Phi(T;t_i,Y^{\varsigma}_i) I^0\mathrm {Id}_{d\times d})^*f'(\Phi(T;t_i,Y^{\varsigma}_i)),\int_{t_{i-1}}^{t_i}\int_0^1\mathcal Db(Z^{\beta_1}_{i,r})\\
&\quad \big(\varphi_r(t_{i-1},\varphi_{t_{i-1}}(0,\Phi_0))-\varphi^{Int}_r(t_{i-1},\varphi_{t_{i-1}}(0,\Phi_0))\big)\mathrm d\beta_1\mathrm dr\Big\>\Big]\mathrm d\varsigma.
\end{align*}
\textit{Estimate of term $\mathcal I_{b,1}.$} 
We first decompose $\varphi^{Int}_r(t_{i-1},\varphi_{t_{i-1}}(0,\Phi_0))-\varphi^{Int}_{t_{i-1}}(0,\Phi_0)$ for $r\in[t_{i-1},t_i).$ For $s\in[t_j,t_{j+1}]\subset [-\tau,0],$ we have 
$t_{i-1}+s\in[t_{i+j-1},t_{i+j}),$ which implies
\begin{align}\label{varphi_int}
\varphi^{Int}_{t_{i-1}}(0,\Phi_0)(s)=\frac{t_{j+1}-s}{\Delta}\varphi(t_{i+j-1};0,\Phi_0)+\frac{s-t_j}{\Delta}\varphi(t_{i+j};0,\Phi_0).
\end{align}
We also have $r+s\in[t_{i+j-1},t_{i+j+1}),$ which leads to the following two cases.
 
\underline{Case 1: $r+s\in[t_{i+j-1},t_{i+j}).$} In this case,
\begin{align*}
&\quad\varphi^{Int}_{r}(t_{i-1},\varphi_{t_{i-1}}(0,\Phi_0))(s)\nn\\
&=\frac{t_{i+j}-(r+s)}{\Delta}\varphi(t_{i+j-1};0,\Phi_0)+\frac{(r+s)-t_{i+j-1}}{\Delta}\varphi(t_{i+j};0,\Phi_0).
\end{align*}
Hence,
\begin{align*}
&\quad\varphi^{Int}_{r}(t_{i-1},\varphi_{t_{i-1}}(0,\Phi_0))(s)-\varphi^{Int}_{t_{i-1}}(0,\Phi_0)(s)\nn\\
&=\frac{t_{i-1}-r}{\Delta}\varphi(t_{i+j-1};0,\Phi_0)+\frac{r-t_{i-1}}{\Delta}\varphi(t_{i+j};0,\Phi_0)\\
&=\frac{r-t_{i-1}}{\Delta}(\varphi(t_{i+j};0,\Phi_0)-\varphi(t_{i+j-1};0,\Phi_0)).
\end{align*}

\underline{Case 2: $r+s\in[t_{i+j},t_{i+j+1}).$} In this case,
\begin{align*}
&\quad \varphi^{Int}_r(t_{i-1},\varphi_{t_{i-1}}(0,\Phi_0))(s)\nn\\
&=\frac{t_{i+j+1}-(r+s)}{\Delta}\varphi(t_{i+j};0,\Phi_0)+\frac{r+s-t_{i+j}}{\Delta}\varphi(t_{i+j+1};0,\Phi_0)\\
&=\textbf 1_{\{i+j\geq1\}}\Big\{\frac{t_{i+j+1}-(r+s)}{\Delta}\Big[\varphi(t_{i+j-1};0,\Phi_0)+\int_{t_{i+j-1}}^{t_{i+j}}b(\varphi_u(0,\Phi_0))\mathrm du \nn\\
&\quad+\int_{t_{i+j-1}}^{t_{i+j}}\sigma(\varphi_u(0,\Phi_0))\mathrm dW(u)\Big]+\frac{r+s-t_{i+j}}{\Delta}\Big[\varphi(t_{i+j};0,\Phi_0)\nn\\
&\quad+\int_{t_{i+j}}^{t_{i+j+1}}b(\varphi_u(0,\Phi_0))\mathrm du+\int_{t_{i+j}}^{t_{i+j+1}}\sigma(\varphi_u(0,\Phi_0))\mathrm d W(u)\Big]\Big\}\nn\\
&\quad+\textbf 1_{\{i+j=0\}}\Big\{
\frac{t_{i+j+1}-(r+s)}{\Delta}\xi(0)+\frac{r+s-t_{i+j}}{\Delta}\Big[\xi(0)\nn\\
&\quad+\int_{0}^{t_{1}}b(\varphi_u(0,\Phi_0))\mathrm du+\int_{0}^{t_{1}}\sigma(\varphi_u(0,\Phi_0))\mathrm d W(u)\Big]\Big\}\nn\\
&\quad+\textbf 1_{\{i+j\leq-1\}}\Big\{\frac{t_{i+j+1}-(r+s)}{\Delta}\xi(t_{i+j})+\frac{r+s-t_{i+j}}{\Delta}\xi(t_{i+j+1})\Big\}.
\end{align*}
Then,
\begin{align*}
&\quad \varphi_r^{Int}(t_{i-1},\varphi_{t_{i-1}}(0,\Phi_0))(s)-\varphi^{Int}_{t_{i-1}}(0,\Phi_0)(s)\\
&=\textbf 1_{\{i+j\geq1\}}\Big\{
\frac{t_{j+1}-s}{\Delta}\Big[\int_{t_{i+j-1}}^{t_{i+j}}b(\varphi_u(0,\Phi_0))\mathrm du+\int_{t_{i+j-1}}^{t_{i+j}}\sigma(\varphi_u(0,\Phi_0))\mathrm dW(u)\Big]\nn\\
&\quad+\frac{r+s-t_{i+j}}{\Delta}\Big[\int_{t_{i+j}}^{t_{i+j+1}}b(\varphi_u(0,\Phi_0))\mathrm du+\int_{t_{i+j}}^{t_{i+j+1}}\sigma(\varphi_u(0,\Phi_0))\mathrm dW(u)\Big]\Big\}\nn\\
&\quad+\textbf 1_{\{i+j=0\}}\Big\{
\frac{t_{j+1}-s}{\Delta}\big(\xi(0)-\xi(-\Delta)\big)+\frac{r+s-t_{i+j}}{\Delta}\Big[\int_{0}^{t_{1}}b(\varphi_u(0,\Phi_0))\mathrm du\nn\\
&\quad+\!\int_{0}^{t_{1}}\sigma(\varphi_u(0,\Phi_0))\mathrm dW(u)\Big]\Big\}\!+\!\textbf 1_{\{i+j\leq-1\}}\Big\{
\frac{t_{j+1}-s}{\Delta}\big(\xi(t_{i+j})-\xi(t_{i+j-1})\big)\nn\\
&\quad+\frac{s+r-t_{i+j}}{\Delta}\big(\xi(t_{i+j+1})-\xi(t_{i+j})\big)
\Big\}.
\end{align*}

Combining \underline{Cases 1--2}, we deduce
\begin{align*}
&\quad \varphi^{Int}_r(t_{i-1},\varphi_{t_{i-1}}(0,\Phi_0))-\varphi^{Int}_{t_{i-1}}(0,\Phi_0)\\
&=\sum_{j=-N}^{-1}\bigg\{\textbf 1_{\{i+j\geq1\}}\Big\{\mathbf 1_{[t_{i+j}-r,t_{i+j+1}-r)}(\cdot)\!\Big[\frac{t_{j+1}-\cdot}{\Delta}\Big(\!\int_{t_{i+j-1}}^{t_{i+j}}\!b(\varphi_u(0,\Phi_0))\mathrm du\nn\\
&\quad+\int_{t_{i+j-1}}^{t_{i+j}}\sigma(\varphi_u(0,\Phi_0))\mathrm dW(u)\Big)+\frac{r+\cdot-t_{i+j}}{\Delta}\Big(\int_{t_{i+j}}^{t_{i+j+1}}b(\varphi_u(0,\Phi_0))\mathrm du\nn\\
&\quad+\int_{t_{i+j}}^{t_{i+j+1}}\sigma(\varphi_u(0,\Phi_0))\mathrm dW(u)\Big)\Big]+\mathbf 1_{[t_{i+j-1}-r,t_{i+j}-r)}(\cdot)
\frac{r-t_{i-1}}{\Delta}\Big(\int_{t_{i+j-1}}^{t_{i+j}}\!b(\varphi_u(0,\Phi_0))\mathrm du\nn\\
&\quad+\!\int_{t_{i+j-1}}^{t_{i+j}}\!\sigma(\varphi_u(0,\Phi_0))\mathrm dW(u)\Big)\!\Big\}+\textbf 1_{\{i+j=0\}}\Big\{
\mathbf 1_{[t_{i+j-1}-r,t_{i+j}-r)}(\cdot)\times\nn\\
&\quad\frac{r-t_{i-1}}{\Delta}\big(\xi(0)-\xi(-\Delta)\big)\!+\!\mathbf 1_{[t_{i+j}-r,t_{i+j+1}-r)}(\cdot)\Big[\frac{t_{j+1}-\cdot}{\Delta}\big(\xi(0)-\xi(-\Delta)\big)\nn\\
&\quad+\frac{r+\cdot-t_{i+j}}{\Delta}\Big(\int_{0}^{t_{1}}b(\varphi_u(0,\Phi_0))\mathrm du+\int_{0}^{t_{1}}\sigma(\varphi_u(0,\Phi_0))\mathrm dW(u)\Big)\Big]\Big\}\nn\\
&\quad+\textbf 1_{\{i+j\leq-1\}}\Big\{
\mathbf 1_{[t_{i+j-1}-r,t_{i+j}-r)}(\cdot)\frac{r-t_{i-1}}{\Delta}\big(\xi(t_{i+j})-\xi(t_{i+j-1})\big)\nn\\
&\quad+\mathbf 1_{[t_{i+j}-r,t_{i+j+1}-r)}(\cdot)
\Big(\frac{t_{j+1}-\cdot}{\Delta}\big(\xi(t_{i+j})-\xi(t_{i+j-1})\big)\nn\\
&\quad+\frac{\cdot+r-t_{i+j}}{\Delta}\big(\xi(t_{i+j+1})-\xi(t_{i+j})\big)\Big)
\Big\}
\bigg\}.
\end{align*}
Inserting the above equality into $\mathcal I_{b,1}$, we proceed to estimate sub-terms in $\mathcal I_{b,1}$ one-by-one.
It follows from Assumption \ref{ass_nablab} that
$|\mathcal D b(x)y|\leq K(1+\|x\|^{\beta})\|y\|$.
This, along with  Lemmas  \ref{l4.1} and \ref{Dy_esti},  and $f\in\mathcal C^3_b$ implies that  
\begin{align*}
&\quad\sum_{i=1}^{N^{\Delta}}\int_0^1\mathbb E\Big[\Big\< (\mathcal D\Phi(T;t_i,Y^{\varsigma}_i) I^0\mathrm {Id}_{d\times d})^*f'(\Phi(T;t_i,Y^{\varsigma}_i)),\int_{t_{i-1}}^{t_i}\int_0^1\mathcal Db(Z^{\beta_1}_{i,r})\\
&\quad \sum_{j=-N}^{-1}\mathbf 1_{\{i+j\geq1\}}\mathbf 1_{[t_{i+j}-r,t_{i+j+1}-r)}(\cdot)\frac{t_{j+1}-\cdot}{\Delta}\int_{t_{i+j-1}}^{t_{i+j}}b(\varphi_u(0,\Phi_0))\mathrm du\mathrm d\beta_1\mathrm dr\Big\>\Big]\mathrm d\varsigma\\
&\leq K\sum_{i=1}^{N^{\Delta}}\int_0^1\Big\|\mathcal D\Phi(T;t_i,Y^{\varsigma}_{i}) I^0\mathrm{Id}_{d\times d}\int_{t_{i-1}}^{t_i}\int_0^1 \mathcal Db(Z^{\beta_1}_{i,r})\sum_{j=-N}^{-1}\mathbf 1_{\{i+j\geq1\}}\nn\\
&\quad\mathbf 1_{[t_{i+j-r},t_{i+j+1}-r)}(\cdot)\frac{t_{j+1}-\cdot}{\Delta}\int_{t_{i+j-1}}^{t_{i+j}}b(\varphi_u(0,\Phi_0))\mathrm du\mathrm d\beta_1\mathrm dr\Big\|_{0,2}\mathrm d\varsigma\\
&\leq K\sum_{i=1}^{N^{\Delta}}\int_0^1 \Big(e^{-\hat\lambda_2(T-t_i)}\mathbb E\Big[\Big\| I^0\mathrm{Id}_{d\times d}\int_{t_{i-1}}^{t_i}\int_0^1 \mathcal Db(Z^{\beta_1}_{i,r})\sum_{j=-N}^{-1}\mathbf 1_{\{i+j\geq1\}}\nn\\
&\quad\mathbf 1_{[t_{i+j-r},t_{i+j+1}-r)}(\cdot)\frac{t_{j+1}-\cdot}{\Delta}\int_{t_{i+j-1}}^{t_{i+j}}b(\varphi_u(0,\Phi_0))\mathrm du\mathrm d\beta_1\mathrm dr\Big\|^2(1+\|Y^{\varsigma}_i\|^{2\beta})\Big]\Big)^{\frac 12}\mathrm d\varsigma\\
&\leq K\sum_{i=1}^{N^{\Delta}}\int_{0}^1e^{-\frac{\hat\lambda_2(T-t_i)}{2}}\Big(\E\Big[\int_{t_{i-1}}^{t_{i}}\int_{0}^1(1+\|Z_{i,r}^{\beta_1}\|^{\beta})
\int_{t_{i+j-1}}^{t_{i+j}}(1\nn\\
&\quad+\|\varphi_{u}(0,\Phi_{0})\|^{\beta+1})\mathrm du\mathrm d\beta_1\mathrm d r\Big]^{4}\Big(1+\sup_{i\geq0}\E[\|Y^{\varsigma}_{i}\|^{4\beta}]\Big)\Big) ^{\frac14}\mathrm d\varsigma \nn\\
&\leq K\Delta ^2\sum_{i=1}^{N^{\Delta}} e^{-\frac{\hat\lambda_2(T-t_i)}{2}}\big(1+\sup_{u\ge 0}\mathbb E[\|\varphi_u(0,\Phi_0)\|^{4(2\beta+1)}]\big)^{\frac14}\sup_{i\ge 0}\big(1+\mathbb E[\|Y^{\varsigma}_i\|^{4\beta}]\big)^{\frac14}\\
&\leq K\Delta,
\end{align*}
where we used the fact that $$\sup_{T>0}\Delta\sum_{i=1}^{N^{\Delta}}e^{-\frac{\hat\lambda_2(T-t_i)}{2}}=\sup_{N^{\Delta}> 0}\Delta \sum_{i=1}^{N^{\Delta}}e^{-\frac{\hat\lambda_2\Delta (N^{\Delta}-i)}{2}}<\infty.$$

For a term with the Brownian motion,  by  \eqref{IBPformula}, we have 
\begin{align*}
&\quad \mathcal I_{b,1}^1:=\sum_{i=1}^{N^{\Delta}}\int_0^1\mathbb E\Big[\Big\< (\mathcal D\Phi(T;t_i,Y^{\varsigma}_i) I^0\mathrm {Id}_{d\times d})^*f'(\Phi(T;t_i,Y^{\varsigma}_i)),\int_{t_{i-1}}^{t_i}\!\!\int_0^1\!\mathcal Db(Z^{\beta_1}_{i,r})\nn\\
&\quad\sum_{j=-N}^{-1}\mathbf 1_{\{i+j\geq1\}}
\mathbf 1_{[t_{i+j}-r,t_{i+j+1}-r)}(\cdot)\frac{t_{j+1}-\cdot}{\Delta}\int_{t_{i+j-1}}^{t_{i+j}}\sigma (\varphi_u(0,\Phi_0))\mathrm dW(u)\mathrm d\beta_1\mathrm dr\Big\>\Big]\mathrm d\varsigma\\
&= \sum_{i=1}^{N^{\Delta}}\!\!\int_0^1\!\!\int_{t_{i-1}}^{t_i}\!\!\int_0^1\!\!\mathbb E\Big[\Big\<(\mathcal D b(Z^{\beta_1}_{i,r})\sum_{j=-N}^{-1}\mathbf 1_{\{i+j\geq1\}}\mathbf 1_{[t_{i+j}-r,t_{i+j+1}-r})(\cdot)\frac{t_{j+1}-\cdot}{\Delta}\mathrm{Id}_{d\times d})^*\nn\\
&\quad(\mathcal D\Phi(T;t_i,Y^{\varsigma}_i) I^0\mathrm {Id}_{d\times d})^* f'(\Phi(T;t_i,Y^{\varsigma}_i)),\int_{t_{i+j-1}}^{t_{i+j}}\sigma(\varphi_u(0,\Phi_0))\mathrm dW(u)\Big\>\Big]\mathrm d{\beta_1}\mathrm dr\mathrm d\varsigma\\
&=\sum_{i=1}^{N^{\Delta}}\int_0^1\int_{t_{i-1}}^{t_i}\int_0^1\int_{t_{i+j-1}}^{t_{i+j}}\mathbb E\Big[\Big\<D_u\big[(\mathcal Db(Z^{\beta_1}_{i,r})\sum_{j=-N}^{-1}\mathbf 1_{\{i+j\geq1\}}\mathbf 1_{[t_{i+j}-r,t_{i+j+1}-r)}(\cdot)\frac{t_{j+1}-\cdot}{\Delta}\nn\\
&\quad\mathrm {Id}_{d\times d})^*(\mathcal D\Phi(T;t_i,Y^{\varsigma}_i)I^0\mathrm {Id}_{d\times d})^* f'(\Phi(T;t_i,Y^{\varsigma}_i))\big],\sigma(\varphi_u(0,\Phi_0))\Big\>\Big]\mathrm du \mathrm d\beta_1\mathrm dr\mathrm d\varsigma\\
&=\sum_{i=1}^{N^{\Delta}}\int_0^1\int_{t_{i-1}}^{t_i}\int_0^1\int_{t_{i+j-1}}^{t_{i+j}}\Big\{\mathbb E\Big[\Big\<f'(\Phi(T;t_i,Y^{\varsigma}_i)),\mathcal D\Phi(T;t_i,Y^{\varsigma}_i) I^0\mathrm {Id}_{d\times d}
\mathcal D^2b(Z^{\beta_1}_{i,r})\nn\\
&\quad\Big(D_uZ^{\beta_1}_{i,r},\sum_{j=-N}^{-1}\mathbf 1_{\{i+j\geq1\}}\mathbf1_{[t_{i+j}-r,t_{i+j+1}-r)}(\cdot)\frac{t_{j+1}-\cdot}{\Delta}\mathrm {Id}_{d\times d}\sigma (\varphi_u(0,\Phi_0))\Big)\Big\>\Big]\nn\\
&\quad+\mathbb E\Big[\Big\<f'(\Phi(T;t_i,Y^{\varsigma}_i)),D_u\mathcal D\Phi(T;t_i,Y^{\varsigma}_i)I^0\mathrm {Id}_{d\times d}\mathcal Db(Z^{\beta_1}_{i,r})\sum_{j=-N}^{-1}\mathbf 1_{\{i+j\geq1\}}\nn\\
&\quad\mathbf1_{[t_{i+j}-r,t_{i+j+1}-r)}(\cdot)\frac{t_{j+1}-\cdot}{\Delta}\mathrm {Id}_{d\times d}\sigma (\varphi_u(0,\Phi_0))\Big\>\Big]+\mathbb E\Big[\!\Big\<f''(\Phi(T;t_i,Y^{\varsigma}_i))\nn\\
&\quad D_u\Phi(T;t_i,Y^{\varsigma}_i),\mathcal D\Phi(T;t_i,Y^{\varsigma}_i)I^0\mathrm {Id}_{d\times d} \mathcal Db(Z^{\beta_1}_{i,r})\sum_{j=-N}^{-1}\mathbf 1_{\{i+j\geq1\}}\mathbf1_{[t_{i+j}-r,t_{i+j+1}-r)}(\cdot)\nn\\
&\quad\frac{t_{j+1}-\cdot}{\Delta}\mathrm {Id}_{d\times d}\sigma (\varphi_u(0,\Phi_0))\Big\>\Big]\Big\}\mathrm du\mathrm d\beta_1\mathrm dr\mathrm d\varsigma\\
&\leq\mathcal I_{b,1,1}^1+\mathcal I_{b,1,2}^1+\mathcal I_{b,1,3}^1,
\end{align*}
where
\begin{align*}
\mathcal I_{b,1,1}^1&:=K\sum_{i=1}^{N^{\Delta}}\int_0^1\int_{t_{i-1}}^{t_i}\int_0^1\!\int_{t_{i+j-1}}^{t_{i+j}}\Big\| \mathcal D\Phi(T;t_i,Y^{\varsigma}_i)I^0\mathrm {Id}_{d\times d} \mathcal D^2b\big(Z^{\beta_1}_{i,r})(D_uZ^{\beta_1}_{i,r},\nn\\
&\quad\quad\sum_{j=-N}^{-1}\mathbf1_{[t_{i+j}-r,t_{i+j+1}-r)}(\cdot)\frac{t_{j+1}-\cdot}{\Delta}\mathrm {Id}_{d\times d}\sigma (\varphi_u(0,\Phi_0))\big)\Big\|_{0,2}\mathrm du\mathrm d\beta_1\mathrm dr\mathrm d\varsigma,\nn\\
\mathcal I_{b,1,2}^1&:=K\sum_{i=1}^{N^{\Delta}}\int_0^1\int_{t_{i-1}}^{t_i}\int_0^1\!\int_{t_{i+j-1}}^{t_{i+j}}\Big\| D_u\mathcal D\Phi(T;t_i,Y^{\varsigma}_i)I^0\mathrm {Id}_{d\times d}\mathcal Db(Z^{\beta_1}_{i,r})\nn\\
&\quad\quad\sum_{j=-N}^{-1}\mathbf1_{[t_{i+j}-r,t_{i+j+1}-r)}(\cdot)\frac{t_{j+1}-\cdot}{\Delta} \mathrm {Id}_{d\times d}\sigma (\varphi_u(0,\Phi_0))\Big\|_{0,2}\mathrm du\mathrm d\beta_1\mathrm dr\mathrm d\varsigma,\nn\\
\mathcal I_{b,1,3}^1&:=K\sum_{i=1}^{N^{\Delta}}\int_0^1\int_{t_{i-1}}^{t_i}\int_0^1\!\int_{t_{i+j-1}}^{t_{i+j}}\|D_u\Phi(T;t_i,Y^{\varsigma}_i)\|_{0,2}\Big\|\mathcal D\Phi(T;t_i,Y^{\varsigma}_i)I^0\mathrm {Id}_{d\times d}\nn\\
&\quad\quad\mathcal Db(Z^{\beta_1}_{i,r})\sum_{j=-N}^{-1}\mathbf1_{[t_{i+j}-r,t_{i+j+1}-r)}(\cdot)\frac{t_{j+1}-\cdot}{\Delta}\mathrm {Id}_{d\times d}\sigma (\varphi_u(0,\Phi_0))\Big\|_{0,2}\mathrm du\mathrm d\beta_1\mathrm dr\mathrm d\varsigma.
\end{align*}
Making use of Lemmas  \ref{l4.1}, \ref{lem_fir_d2}, and \ref{Dy_esti}, and Assumption \ref{second_deri}, we have
\begin{align*}
&\quad \mathcal I_{b,1,1}^1\leq 
K\sum_{i=1}^{N^{\Delta}}\!\int_0^1\!\int_{t_{i-1}}^{t_i}\!\int_0^1\int_{t_{i+j-1}}^{t_{i+j}}\!e^{-\frac{\hat\lambda_2(T-t_i)}{2}}\Big(\mathbb E\Big[\big\|I^0\mathrm{Id}_{d\times d}\mathcal D^2b(Z^{\beta_1}_{i,r})\big(D_uZ^{\beta_1}_{i,r},\nn\\
&\quad\sum_{j=-N}^{-1}\mathbf 1_{[t_{i+j}-r,t_{i+j+1}-r)}(\cdot)\frac{t_{j+1}-\cdot}{\Delta}\mathrm {Id}_{d\times d}\sigma(\varphi_u(0,\Phi_0))\big)\big\|^2(1+\|Y^{\varsigma}_i\|^{2\beta})\Big]\Big)^{\frac12}\mathrm du\mathrm d\beta_1\mathrm dr\mathrm d\varsigma\\
&\leq K\Delta.
\end{align*}
For the term $\mathcal I_{b,1,2}^1,$ by virtue of Lemmas \ref{l4.1}, \ref{lem_fir_d2}, \ref{Dy_esti}, and \ref{lemma_high3}, we arrive at 
\begin{align*}
&\quad\mathcal I_{b,1,2}^1\leq K\sum_{i=1}^{N^{\Delta}}\int_0^1\int_{t_{i-1}}^{t_i}\int_0^1\int_{t_{i+j-1}}^{t_{i+j}}e^{-\frac{\hat{\lambda}_2(T-t_i)}{2}}\Big(\E\Big[\Big\|I^0\mathrm {Id}_{d\times d}\mathcal Db(Z^{\beta_1}_{i,r})\sum_{j=-N}^{-1}\nn\\
&\quad\mathbf 1_{[t_{i+j}-r,t_{i+j+1}-r)}(\cdot)\frac{t_{j+1}\!-\!\cdot\!}{\Delta}
\mathrm {Id}_{d\times d}\sigma(\varphi_u(0,\Phi_0))\Big\|^2(1+\|D_{u}Y^{\varsigma}_{i}\|^{16})\Big]\Big)^{\frac12}\mathrm du\mathrm d\beta_1\mathrm dr\mathrm d\varsigma\nn\\
&\leq
K\sum_{i=1}^{N^{\Delta}}\int_0^1\int_{t_{i-1}}^{t_i}\int_0^1\int_{t_{i+j-1}}^{t_{i+j}}e^{-\frac{\hat{\lambda}_2(T-t_i)}{2}}\Big(\E\big[(1+\|Z^{\beta_1}_{i,r}\|^{4\beta})|\sigma(\varphi_u(0,\Phi_0))|^4\big]\nn\\
&\quad+\big(1+\E\big[\|D_uY^{\varsigma}_i\|^{32}\big]\big)\Big)^{\frac{1}{4}}\mathrm du\mathrm d\beta_1\mathrm dr\mathrm d\varsigma\nn\\
&\leq K\Delta.
\end{align*}
Similar to the estimates of $\mathcal I^1_{b,1,1}$ and $\mathcal I^1_{b,1,2}$, for the term $\mathcal I_{b,1,3}^1,$ we obtain
$
\mathcal I^1_{b,1,3}
\leq K\Delta.
$

Other terms can be estimated similarly as above.  We can derive that $\mathcal I_{b,1}\leq K\Delta.$

\textit{Estimate of term $\mathcal I_{b,2}.$} 
We first decompose the term  $\varphi_r(t_{i-1},\varphi_{t_{i-1}}(0,\Phi_0))-\varphi^{Int}_r(t_{i-1},\varphi_{t_{i-1}}(0,\Phi_0))$ for $r\in[t_{i-1},t_i).$
For $s\in[t_j,t_{j+1}]\subset[-\tau,0],$ we have $r+s\in [t_{i+j-1},t_{i+j+1}),$ which leads to the following two cases.

\underline{Case 1: $r+s\in[t_{i+j-1},t_{i+j}).$} In this case, we have 
\begin{align*}
&\quad \varphi_r(t_{i-1},\varphi_{t_{i-1}}(0,\Phi_0))(s)-\varphi^{Int}_r(t_{i-1},\varphi_{t_{i-1}}(0,\Phi_0))(s)\\
&=\varphi(r+s; t_{i-1},\varphi_{t_{i-1}}(0,\Phi_0))-\Big[\frac{t_{i+j}-(r+s)}{\Delta}\varphi(t_{i+j-1};0,\Phi_0)\nn\\
&\quad+\frac{r+s-t_{i+j-1}}{\Delta}\varphi(t_{i+j};0,\Phi_0)\Big]\nn\\
&=\mathbf 1_{\{i+j\geq1\}}\Big\{\int_{t_{i+j-1}}^{r+s}b(\varphi_u(0,\Phi_0))\mathrm du\!+\!\int_{t_{i+j-1}}^{r+s}\sigma(\varphi_u(0,\Phi_0))\mathrm dW(u)\!-\!\frac{r+s-t_{i+j-1}}{\Delta}\times\\
&\quad \Big[\int_{t_{i+j-1}}^{t_{i+j}}b(\varphi_u(0,\Phi_0))\mathrm du+\int_{t_{i+j-1}}^{t_{i+j}}\sigma(\varphi_u(0,\Phi_0))\mathrm dW(u)\Big]\Big\}.
\end{align*}

\underline{Case 2: $r+s\in [t_{i+j},t_{i+j+1}).$} In this case,
\begin{align*}
&\quad \varphi_r(t_{i-1},\varphi_{t_{i-1}}(0,\Phi_0))(s)-\varphi^{Int}_r(t_{i-1},\varphi_{t_{i-1}}(0,\Phi_0))(s)\\
&=\mathbf 1_{\{i+j\geq0\}}\Big\{\int_{t_{i+j}}^{r+s}b(\varphi_u(0,\Phi_0))\mathrm du+\int_{t_{i+j}}^{r+s}\sigma(\varphi_u(0,\Phi_0))\mathrm dW(u)-\frac{r+s-t_{i+j}}{\Delta}\times\\
&\quad \Big[\int_{t_{i+j}}^{t_{i+j+1}}b(\varphi_u(0,\Phi_0))\mathrm du+\int_{t_{i+j}}^{t_{i+j+1}}\sigma(\varphi_u(0,\Phi_0))\mathrm dW(u)\Big]\Big\}.
\end{align*}

Combining \underline{Cases 1--2} gives
\begin{align*}
&\quad \varphi_r(t_{i-1},\varphi_{t_{i-1}}(0,\Phi_0))-\varphi^{Int}_r(t_{i-1},\varphi_{t_{i-1}}(0,\Phi_0))\\
&=\sum_{j=-N}^{-1}\Big\{\mathbf 1_{\{i+j\geq1\}}\mathbf 1_{[t_{i+j-1}-r,t_{i+j}-r)}(\cdot)\Big[\int_{t_{i+j-1}}^{r+\cdot}b(\varphi_u(0,\Phi_0))\mathrm du \nn\\
&\quad+\int_{t_{i+j-1}}^{r+\cdot}\sigma(\varphi_u(0,\Phi_0))\mathrm dW(u)-\frac{r+\cdot-t_{i+j-1}}{\Delta}\Big(\int_{t_{i+j-1}}^{t_{i+j}}\!b(\varphi_u(0,\Phi_0))\mathrm du \nn\\
&\quad+\int_{t_{i+j-1}}^{t_{i+j}}\sigma(\varphi_u(0,\Phi_0))\mathrm dW(u)\Big)\Big]\!+\!\mathbf 1_{\{i+j\geq0\}}\mathbf 1_{[t_{i+j}-r,t_{i+j+1}-r)}(\cdot)\times\nn\\
&\quad\Big[\int_{t_{i+j}}^{r+\cdot}b(\varphi_u(0,\Phi_0))\mathrm du
+\int_{t_{i+j}}^{r+\cdot}\sigma(\varphi_u(0,\Phi_0))\mathrm dW(u)\nn\\
&\quad-\frac{r+\cdot-t_{i+j}}{\Delta}\Big(\int_{t_{i+j}}^{t_{i+j+1}}\!b(\varphi_u(0,\Phi_0))\mathrm du+\int_{t_{i+j}}^{t_{i+j+1}}\!\sigma(\varphi_u(0,\Phi_0))\mathrm dW(u)\Big)\Big]\Big\}.
\end{align*}
Inserting the above equality into $\mathcal I_{b,2}$, we are in the position to estimate $\mathcal I_{b,2}.$ We only give the estimate  of the term with $$\mathbf 1_{\{i+j\geq1\}}\mathbf 1_{[t_{i+j}-r,t_{i+j+1}-r)}(\cdot)\int_{t_{i+j}}^{r+\cdot}\sigma(\varphi_u(0,\Phi_0))\mathrm dW(u),$$ since estimates of other terms are similar to  those in $\mathcal I_{b,1}.$
Based on  Assumption \ref{ass_nablab}, we have 
\begin{align*}
&\quad\mathcal I_{b,2,1}:=\sum_{i=1}^{N^{\Delta}}\!\int_0^1\!\mathbb E\Big[\Big\< (\mathcal D\Phi(T;t_i,Y^{\varsigma}_i) I^0\mathrm {Id}_{d\times d})^*f'(\Phi(T;t_i,Y^{\varsigma}_i)),\int_{t_{i-1}}^{t_i}\int_0^1\nn\\
&\quad\mathcal Db(Z^{\beta_1}_{i,r})\sum_{j=-N}^{-1}\mathbf 1_{\{i+j\geq1\}}\mathbf 1_{[t_{i+j}-r,t_{i+j+1}-r)}(\cdot)\int_{t_{i+j}}^{r+\cdot}\sigma(\varphi_u(0,\Phi_0))\mathrm dW(u) \mathrm d\beta_1\mathrm dr\Big\>\Big]\mathrm d\varsigma\\
&=\sum_{i=1}^{N^{\Delta}}\int_0^1\mathbb E\Big[\Big\< (\mathcal D\Phi(T;t_i,Y^{\varsigma}_i) I^0\mathrm {Id}_{d\times d})^*f'(\Phi(T;t_i,Y^{\varsigma}_i)),\int_{t_{i-1}}^{t_i}\int_0^1\int_{-\tau}^0 k_b^{1}(Z^{\beta_1}_{i,r}(s))\nn\\
&\quad\sum_{j=-N}^{-1}\mathbf 1_{\{i+j\geq1\}}\mathbf 1_{[t_{i+j}-r,t_{i+j+1}-r)}(s)\int_{t_{i+j}}^{r+s}\sigma(\varphi_u(0,\Phi_0))\mathrm dW(u) \mathrm d\nu^1_3(s)
 \mathrm d\beta_1\mathrm dr\Big\>\Big]\mathrm d\varsigma\\
 &=\sum_{i=1}^{N^{\Delta}}\!\sum_{j=-N}^{-1}\!\int_0^1\!\int_{t_{i-1}}^{t_i}\int_0^1\int_{-\tau\vee (t_{i+j}-r)}^{ (t_{i+j+1}-r)\wedge 0} \mathbb E\Big[\Big\<k_b^{1}(Z^{\beta_1}_{i,r}(s))^*\mathbf 1_{\{i+j\geq1\}}(\mathcal D\Phi(T;t_i,Y^{\varsigma}_i)I^0\nn\\
 &\quad\mathrm {Id}_{d\times d})^*f'(\Phi(T;t_i,Y^{\varsigma}_i)),\int_{t_{i+j}}^{r+s}\sigma(\varphi_u(0,\Phi_0))\mathrm dW(u)\Big\>\Big]\mathrm d\nu_3^1(s)\mathrm d\beta_1\mathrm dr\mathrm d\varsigma\\
 &=\sum_{i=1}^{N^{\Delta}}\sum_{j=-N}^{-1}\int_0^1\int_{t_{i-1}}^{t_i}\int_0^1\int_{-\tau\vee (t_{i+j}-r)}^{ (t_{i+j+1}-r)\wedge 0}\int_{t_{i+j}}^{r+s}\mathbb E\Big[\Big\<D_u[k_b^{1}(Z^{\beta_1}_{i,r}(s))^*\mathbf 1_{\{i+j\geq1\}}\nn\\
 &\quad(\mathcal D\Phi(T;t_i,Y^{\varsigma}_i)I^0\mathrm {Id}_{d\times d})^*f'(\Phi(T;t_i,Y^{\varsigma}_i))],\sigma(\varphi_u(0,\Phi_0))\Big\>\Big]\mathrm du\mathrm d\nu_3^1(s)\mathrm d\beta_1\mathrm dr\mathrm d\varsigma\\
 &=\sum_{i=1}^{N^{\Delta}}\sum_{j=-N}^{-1}\int_0^1\int_{t_{i-1}}^{t_i}\int_0^1\int_{-\tau\vee (t_{i+j}-r)}^{ (t_{i+j+1}-r)\wedge 0}\int_{t_{i+j}}^{r+s}\Big\{\mathbb E\Big[\Big\<f'(\Phi(T;t_i,Y^{\varsigma}_i)), \mathcal D\Phi(T;t_i,Y^{\varsigma}_i)\\
 &\quad I^0\mathrm {Id}_{d\times d} D_uk_b^{1}(Z^{\beta_1}_{i,r}(s))\mathbf 1_{\{i+j\geq1\}}\sigma(\varphi_u(0,\Phi_0))\Big\>\Big] +\mathbb E\Big[\Big\<f'(\Phi(T;t_i,Y^{\varsigma}_i)), \nn\\
 &\quad D_u\mathcal D\Phi(T;t_i,Y^{\varsigma}_i) I^0\mathrm {Id}_{d\times d}k_b^{1}(Z^{\beta_1}_{i,r}(s))\mathbf 1_{\{i+j\geq1\}}\sigma(\varphi_u(0,\Phi_0))\Big\>\Big]\nn\\
 &\quad+\mathbb E\Big[\Big\<f''(\Phi(T;t_i,Y^{\varsigma}_i))D_u \Phi(T;t_i,Y^{\varsigma}_i),
 \mathcal D\Phi(T;t_i,Y^{\varsigma}_i)I^0\mathrm {Id}_{d\times d}
 k_b^{1}(Z^{\beta_1}_{i,r}(s))\mathbf 1_{\{i+j\geq1\}}\nn\\
 &\quad\sigma(\varphi_u(0,\Phi_0))\Big\>\Big]\Big\}\mathrm du\mathrm d\nu_3^1(s)\mathrm d\beta_1\mathrm dr\mathrm d\varsigma\\
&\leq K\sum_{i=1}^{N^{\Delta}}\sum_{j=-N}^{-1}\int_0^1\int_{t_{i-1}}^{t_i}\int_0^1\int_{-\tau\vee (t_{i+j}-r)}^{ (t_{i+j+1}-r)\wedge 0}\int_{t_{i+j}}^{r+s}\Big\{\big\|\mathcal D\Phi(T;t_i,Y^{\varsigma}_i)I^0\mathrm {Id}_{d\times d} D_uk_b^{1}(Z^{\beta_1}_{i,r}(s))\nn\\
&\quad\sigma(\varphi_u(0,\Phi_0))\big\|_{0,2}+\big\|D_u\mathcal D\Phi(T;t_i,Y^{\varsigma}_i)I^0\mathrm {Id}_{d\times d}k_b^{1}(Z^{\beta_1}_{i,r}(s))\sigma(\varphi_u(0,\Phi_0))\big\|_{0,2}\\
&\quad +\|D_u \Phi(T;t_i,Y^{\varsigma}_i)\|_{0,2}\big\|
 \mathcal D\Phi(T;t_i,Y^{\varsigma}_i)I^0\mathrm {Id}_{d\times d}
 k_b^{1}(Z^{\beta_1}_{i,r}(s))\nn\\
 &\quad\sigma(\varphi_u(0,\Phi_0))\big\|_{0,2}\Big\}\mathrm du\mathrm d\nu_3^1(s)\mathrm d\beta_1\mathrm dr\mathrm d\varsigma\\
 &\leq K\Delta.
\end{align*}
Here we used the property 
$$\sup\limits_{u\in [t_{i+j},t_{i+j+1}]}\mathbb E[\sup\limits_{s\in[-\tau,0]}|D_uk_b^{1}(Z^{\beta_1}_{i,r}(s))|^{2p_0}]+\mathbb E[\sup\limits_{s\in[-\tau,0]}|k_b^{1}(Z^{\beta_1}_{i,r}(s))|^{2p_0}]\leq K$$ 
for some $p_0$, which is an  application  of Assumption \ref{ass_nablab}, and  \cref{l4.1,lem_fir_d2}.

Hence, the term $\mathcal I_{b,2}$ is estimated as  $\mathcal I_{b,2}\leq K\Delta.$

\textbf{Estimate of $\mathcal I_{b,\theta}.$} We rewrite $\mathcal I_{b,\theta}$ as 
\begin{align*}
\mathcal I_{b,\theta}&=\sum_{i=1}^{N^{\Delta}}\int_0^1\mathbb E\Big[\Big\<(\mathcal D\Phi(T;t_i,Y^{\varsigma}_i) I^0\mathrm {Id}_{d\times d})^* f'(\Phi(T;t_i,Y^{\varsigma}_i)),\theta \int_{t_{i-1}}^{t_i}\int_0^1\mathcal Db(Z^{\beta_2}_i)\\
&\quad (\varphi^{Int}_{t_{i-1}}(0,\Phi_0)-\Phi_{t_i}(t_{i-1},\varphi^{Int}_{t_{i-1}}(0,\Phi_0)))\mathrm d\beta_2\mathrm dr\Big\>\Big]\mathrm d\varsigma,
\end{align*}
where $Z^{\beta_2}_{i,r}=\beta_2\varphi_{t_{i-1}}^{Int}(0,\Phi_0)+(1-\beta_2)\Phi_{t_i}(t_{i-1},\varphi^{Int}_{t_{i-1}}(0,\Phi_0)).$

It follows from  definitions of $\varphi^{Int}_{t_i}$ and $\Phi_{t_i}$ that for any $i\in\mathbb N_+$,
\begin{align*}
&\quad \varphi^{Int}_{t_{i-1}}(0,\Phi_0)-\Phi_{t_i}(t_{i-1},\varphi^{Int}_{t_{i-1}}(0,\Phi_0))\\
&=\sum_{j=-N}^{-1}I^j[\varphi(t_{i+j-1};0,\Phi_0)-\varphi(t_{i+j};0,\Phi_0)]+I^0[\varphi(t_{i-1};0,\Phi_0)\nn\\
&\quad-\Phi(t_i;t_{i-1},\varphi_{t_{i-1}}^{Int}(0,\Phi_0))]\\
&=-\sum_{j=-N}^{-1}\mathbf 1_{\{i+j\geq1\}}I^j\Big(\int_{t_{i+j-1}}^{t_{i+j}}b(\varphi_u(0,\Phi_0))\mathrm du+\int_{t_{i+j-1}}^{t_{i+j}}\sigma(\varphi_u(0,\Phi_0))\mathrm dW(u)\Big)\\
&\quad+\sum_{j=-N}^{-1}\mathbf 1_{\{i+j\leq 0\}}I^j\big(\xi(t_{i+j-1})-\xi(t_{i+j})\big)-I^0\Big((1-\theta)b(\varphi^{Int}_{t_{i-1}}(0,\Phi_0))\Delta\nn\\
&\quad+\theta b(\Phi_{t_i}(t_{i-1},\varphi^{Int}_{t_{i-1}}(0,\Phi_0)))\Delta+\sigma(\varphi^{Int}_{t_{i-1}}(0,\Phi_0))\delta W_{i-1}\Big).
\end{align*}
Hence, similar to the estimate of $\mathcal I_b,$ we can derive that $\mathcal I_{b,\theta}\leq K\Delta$.

\textbf{Estimate of $\mathcal I_{\sigma}.$} We split $\mathcal I_{\sigma}$  as
\begin{align*}
\mathcal I_{\sigma}&=\sum_{i=1}^{N^{\Delta}}\int_0^1\mathbb E\Big[\Big\<(\mathcal D\Phi(T;t_i,Y^{\varsigma}_i) I^0\mathrm {Id}_{d\times d})^* f'(\Phi(T;t_i,Y^{\varsigma}_i)),\int_{t_{i-1}}^{t_i}\int_0^1\mathcal D\sigma(Z^{\beta_1}_{i,r})\\
&\quad (\varphi_r(t_{i-1},\varphi_{t_{i-1}}(0,\Phi_0))-\varphi^{Int}_{t_{i-1}}(0,\Phi_0))\mathrm d\beta_1\mathrm dW(r) \Big\>\Big]\mathrm d\varsigma=\mathcal I_{\sigma,1}+\mathcal I_{\sigma,2},
\end{align*}
where 
\begin{align*}
&\mathcal I_{\sigma,1}:=\sum_{i=1}^{N^{\Delta}}\int_0^1\mathbb E\Big[\Big\< (\mathcal D\Phi(T;t_i,Y^{\varsigma}_i) I^0\mathrm {Id}_{d\times d})^*f'(\Phi(T;t_i,Y^{\varsigma}_i)),\int_{t_{i-1}}^{t_i}\int_0^1\mathcal D \sigma(Z^{\beta_1}_{i,r})\\
&\quad \big(\varphi^{Int}_r(t_{i-1},\varphi_{t_{i-1}}(0,\Phi_0))-\varphi^{Int}_{t_{i-1}}(0,\Phi_0)\big)\mathrm d\beta_1\mathrm dW(r)\Big\>\Big]\mathrm d\varsigma,\\
&\mathcal I_{\sigma,2}:=\sum_{i=1}^{N^{\Delta}}\int_0^1\mathbb E\Big[\Big\< (\mathcal D\Phi(T;t_i,Y^{\varsigma}_i) I^0\mathrm {Id}_{d\times d})^*f'(\Phi(T;t_i,Y^{\varsigma}_i)),\int_{t_{i-1}}^{t_i}\int_0^1\mathcal D \sigma(Z^{\beta_1}_{i,r})\\
&\quad \big(\varphi_r(t_{i-1},\varphi_{t_{i-1}}(0,\Phi_0))-\varphi^{Int}_r(t_{i-1},\varphi_{t_{i-1}}(0,\Phi_0))\big)\mathrm d\beta_1\mathrm dW(r)\Big\>\Big]\mathrm d\varsigma.
\end{align*}
For the estimate of $\mathcal I_{\sigma,1},$ we only give  the estimate of the term 
\begin{align*}
&\quad\mathcal I_{\sigma,1,1}:=\sum_{i=1}^{N^{\Delta}}\int_0^1\mathbb E\Big[\Big\< (\mathcal D\Phi(T;t_i,Y^{\varsigma}_i) I^0\mathrm {Id}_{d\times d})^*f'(\Phi(T;t_i,Y^{\varsigma}_i)),\int_{t_{i-1}}^{t_i}\int_0^1\mathcal D \sigma(Z^{\beta_1}_{i,r})\\
&\quad\!\sum_{j=-N}^{-1}\!\mathbf 1_{\{i+j\geq1\}}\mathbf 1_{[t_{i+j}-r,t_{i+j+1}-r)}(\cdot)\frac{t_{j+1}-\cdot}{\Delta}\!\!\int_{t_{i+j-1}}^{t_{i+j}}\!\!\sigma(\varphi_u(0,\Phi_0))\mathrm dW(u)
\mathrm d\beta_1\mathrm dW(r)\Big\>\Big]\mathrm d\varsigma,
\end{align*}
since other terms can be estimated similarly.  By  \eqref{IBPformula}, we derive 
\begin{align*}
&\mathcal I_{\sigma,1,1}=\sum_{i=1}^{N^{\Delta}}\!\int_0^1\!\int_{t_{i-1}}^{t_i}\!\int_0^1\!\mathbb E\Big[\Big\<D_r[(\mathcal D\Phi(T;t_i,Y^{\varsigma}_i)I^0\mathrm{Id}_{d\times d})^* f'(\Phi(T;t_i,Y^{\varsigma}_i))],\mathcal D\sigma(Z^{\beta_1}_{i,r})\nn\\
&\quad\sum_{j=-N}^{-1}\mathbf 1_{\{i+j\geq1\}}\!\mathbf 1_{[t_{i+j}-r,t_{i+j+1}-r)}(\cdot)\frac{t_{j+1}-\cdot}{\Delta}\int_{t_{i+j-1}}^{t_{i+j}}\sigma(\varphi_u(0,\Phi_0))\mathrm dW(u)
\Big\>\Big]\mathrm d\beta_1\mathrm dr\mathrm d\varsigma\\
&=\sum_{i=1}^{N^{\Delta}}\!\int_0^1\!\int_{t_{i-1}}^{t_i}\!\int_0^1\int_{t_{i+j-1}}^{t_{i+j}}\mathbb E\Big[\Big\<D_u\big((\mathcal D\sigma(Z^{\beta_1}_{i,r})\sum_{j=-N}^{-1}\!\mathbf 1_{\{i+j\geq1\}}\!\mathbf 1_{[t_{i+j}-r,t_{i+j+1}-r)}(\cdot)\frac{t_{j+1}-\cdot}{\Delta}\nn\\
&\quad\mathrm {Id}_{d\times d})^* D_r[(\mathcal D\Phi(T;t_i,Y^{\varsigma}_i)I^0\mathrm{Id}_{d\times d})^*f'(\Phi(T;t_i,Y^{\varsigma}_i))]\big),\sigma(\varphi_u(0,\Phi_0))\Big\>\Big]\mathrm du\mathrm d\beta_1\mathrm dr\mathrm d\varsigma\\
&=\sum_{i=1}^{N^{\Delta}}\int_0^1\int_{t_{i-1}}^{t_i}\int_0^1\int_{t_{i+j-1}}^{t_{i+j}}\mathbb E\Big[\Big\<D_r[(\mathcal D\Phi(T;t_i,Y^{\varsigma}_i)I^0\mathrm{Id}_{d\times d})^*f'(\Phi(T;t_i,Y^{\varsigma}_i))], \mathcal D^2\sigma(Z^{\beta_1}_{i,r})\nn\\
&\quad\Big(D_u Z^{\beta_1}_{i,r},\sum_{j=-N}^{-1}\!\mathbf 1_{\{i+j\geq1\}}\!\mathbf 1_{[t_{i+j}-r,t_{i+j+1}-r)}(\cdot)\frac{t_{j+1}-\cdot}{\Delta}\mathrm {Id}_{d\times d}\sigma(\varphi_u(0,\Phi_0))\Big)
\Big\>\Big]\\
&\quad +\mathbb E\Big[\Big\<D_uD_r[(\mathcal D\Phi(T;t_i,Y^{\varsigma}_i)I^0\mathrm{Id}_{d\times d})^*f'(\Phi(T;t_i,Y^{\varsigma}_i))],\\
&\quad \mathcal D\sigma(Z^{\beta_1}_{i,r})\sum_{j=-N}^{-1}\!\mathbf 1_{\{i+j\geq1\}}\!\mathbf 1_{[t_{i+j}-r,t_{i+j+1}-r)}(\cdot)\frac{t_{j+1}-\cdot}{\Delta}\mathrm {Id}_{d\times d}\sigma(\varphi_u(0,\Phi_0))
\Big\>\Big]\mathrm du\mathrm d\beta_1\mathrm dr\mathrm d\varsigma\\
&\leq K\sum_{i=1}^{N^{\Delta}}\int_0^1\int_{t_{i-1}}^{t_i}\int_0^1\int_{t_{i+j-1}}^{t_{i+j}}\Big\{\Big\|D_r\mathcal D\Phi(T;t_i,Y^{\varsigma}_i)I^0\mathrm{Id}_{d\times d} \mathcal D^2\sigma(Z^{\beta_1}_{i,r})\Big(D_u Z^{\beta_1}_{i,r},\sum_{j=-N}^{-1}\!\mathbf 1_{\{i+j\geq1\}}\nn\\
&\quad\mathbf 1_{[t_{i+j}-r,t_{i+j+1}-r)}(\cdot)\frac{t_{j+1}-\cdot}{\Delta}\mathrm {Id}_{d\times d}\sigma(\varphi_u(0,\Phi_0))\Big)\Big\|_{0,2}+\|D_r\Phi(T;t_i,Y^{\varsigma}_i)\|_{0,2}\nn\\
&\quad\Big\| \mathcal D\Phi(T;t_i,Y^{\varsigma}_i)I^0\mathrm{Id}_{d\times d} \mathcal D^2\sigma(Z^{\beta_1}_{i,r})D_u Z^{\beta_1}_{i,r}\sum_{j=-N}^{-1}\!\mathbf 1_{\{i+j\geq1\}}\!\mathbf 1_{[t_{i+j}-r,t_{i+j+1}-r)}(\cdot)\frac{t_{j+1}-\cdot}{\Delta}\nn\\
&\quad\mathrm {Id}_{d\times d}\sigma(\varphi_u(0,\Phi_0))\Big\|_{0,2}+\Big\|D_uD_r\mathcal D\Phi(T;t_i,Y^{\varsigma}_i)I^0\mathrm{Id}_{d\times d} \mathcal D\sigma(Z^{\beta_1}_{i,r})\sum_{j=-N}^{-1}\!\mathbf 1_{\{i+j\geq1\}}\nn\\
&\quad\mathbf 1_{[t_{i+j}-r,t_{i+j+1}-r)}(\cdot)\frac{t_{j+1}-\cdot}{\Delta}\mathrm {Id}_{d\times d}\sigma(\varphi_u(0,\Phi_0))\Big\|_{0,2}+\|D_u\Phi(T;t_i,Y^{\varsigma}_i)\|_{0,2}\times\nn\\
&\quad\Big\| D_r\mathcal D\Phi(T;t_i,Y^{\varsigma}_i)I^0\mathrm{Id}_{d\times d}\mathcal D\sigma(Z^{\beta_1}_{i,r})\sum_{j=-N}^{-1}\!\mathbf 1_{\{i+j\geq1\}}\!\mathbf 1_{[t_{i+j}-r,t_{i+j+1}-r)}(\cdot)\frac{t_{j+1}-\cdot}{\Delta}\nn\\
&\quad\mathrm {Id}_{d\times d}\sigma(\varphi_u(0,\Phi_0))\Big\|_{0,2}+\|D_r\Phi(T;t_i,Y^{\varsigma}_i)\|_{0,2}\Big\| D_u\mathcal D\Phi(T;t_i,Y^{\varsigma}_i)I^0\mathrm{Id}_{d\times d} \mathcal D\sigma(Z^{\beta_1}_{i,r})\nn\\
&\quad\sum_{j=-N}^{-1}\!\mathbf 1_{\{i+j\geq1\}}\mathbf 1_{[t_{i+j}-r,t_{i+j+1}-r)}(\cdot)\frac{t_{j+1}-\cdot}{\Delta}\mathrm {Id}_{d\times d}\sigma(\varphi_u(0,\Phi_0))\Big\|_{0,2}\nn\\
&\quad+\|D_uD_r\Phi(T;t_i,Y^{\varsigma}_i)\|_{0,2}\Big\| \mathcal D\Phi(T;t_i,Y^{\varsigma}_i)I^0\mathrm{Id}_{d\times d} \mathcal D\sigma(Z^{\beta_1}_{i,r})\sum_{j=-N}^{-1}\!\mathbf 1_{\{i+j\geq1\}}\nn\\
&\quad\mathbf 1_{[t_{i+j}-r,t_{i+j+1}-r)}(\cdot)\frac{t_{j+1}-\cdot}{\Delta}\mathrm {Id}_{d\times d}\sigma(\varphi_u(0,\Phi_0))\Big\|_{0,2}\!\!+\!\!\|D_u\Phi(T;t_i,Y^{\varsigma}_i)D_r\Phi(T;t_i,Y^{\varsigma}_i)\|_{0,2}\nn\\
&\quad\times\Big\|\mathcal D\Phi(T;t_i,Y^{\varsigma}_i)I^0\mathrm{Id}_{d\times d}\mathcal D\sigma(Z^{\beta_1}_{i,r})\sum_{j=-N}^{-1}\!\mathbf 1_{\{i+j\geq1\}}\!\mathbf 1_{[t_{i+j}-r,t_{i+j+1}-r)}(\cdot)\frac{t_{j+1}-\cdot}{\Delta}\nn\\
&\quad\mathrm {Id}_{d\times d}\sigma(\varphi_u(0,\Phi_0))\Big\|_{0,2}\Big\}\mathrm du\mathrm d\beta_1\mathrm dr\mathrm d\varsigma
\leq K\Delta.
\end{align*}

For the estimate of $\mathcal I_{\sigma,2},$ we only give the estimate of the term 
\begin{align*}
&\quad \mathcal I_{\sigma,2,1}:=\sum_{i=1}^{N^{\Delta}}\int_0^1\mathbb E\Big[\Big\< (\mathcal D\Phi(T;t_i,Y^{\varsigma}_i) I^0\mathrm {Id}_{d\times d})^*f'(\Phi(T;t_i,Y^{\varsigma}_i)),\int_{t_{i-1}}^{t_i}\int_0^1\mathcal D \sigma(Z^{\beta_1}_{i,r})\nn\\
&\quad\sum_{j=-N}^{-1}\mathbf 1_{\{i+j\geq1\}} \mathbf 1_{[t_{i+j}-r,t_{i+j+1}-r)}(\cdot)\int_{t_{i+j}}^{r+\cdot}\sigma(\varphi_u(0,\Phi_0))\mathrm dW(u)
\mathrm d\beta_1\mathrm dW(r)\Big\>\Big]\mathrm d\varsigma.
\end{align*}
It follows  from Assumption \ref{ass_nablab} that
\begin{align*}
&\quad\mathcal I_{\sigma,2,1}=\sum_{i=1}^{N^{\Delta}}\int_0^1\mathbb E\Big[\Big\< (\mathcal D\Phi(T;t_i,Y^{\varsigma}_i) I^0\mathrm {Id}_{d\times d})^*f'(\Phi(T;t_i,Y^{\varsigma}_i)), \int_{t_{i-1}}^{t_i}\int_0^1\!\sum_{j=-N}^{-1}\mathbf 1_{\{i+j\geq1\}}\nn\\
&\quad\int_{-\tau}^0 \!k^{1}_{\sigma}(Z^{\beta_1}_{i,r}(s))\mathbf 1_{[t_{i+j}-r,t_{i+j+1}-r)}(s)\!\!\int_{t_{i+j}}^{r+s}\!\sigma(\varphi_u(0,\Phi_0))\mathrm dW(u)\mathrm d\nu_4^{1}(s)
\mathrm d\beta_1\mathrm dW(r)\Big\>\Big]\mathrm d\varsigma\\
&=\sum_{i=1}^{N^{\Delta}}\sum_{j=-N}^{-1}\mathbf 1_{\{i+j\geq1\}}\int_0^1\int_{t_{i-1}}^{t_i}\int_0^1\int_{(t_{i+j}-r)\vee -\tau}^{(t_{i+j+1}-r)\wedge 0} \mathbb E\Big[\Big\<D_r[(\mathcal D\Phi(T;t_i,Y^{\varsigma}_i)I^0\mathrm {Id}_{d\times d})^*\nn\\
&\quad f'(\Phi(T;t_i,Y^{\varsigma}_i))], k^{1}_{\sigma}(Z^{\beta_1}_{i,r}(s))\int_{t_{i+j}}^{r+s}\sigma(\varphi_u(0,\Phi_0))\mathrm dW(u)
\Big\>\Big]\mathrm d\nu_4^{1}(s)\mathrm d\beta_1\mathrm dr\mathrm d\varsigma\\
&=\sum_{i=1}^{N^{\Delta}}\sum_{j=-N}^{-1}\mathbf 1_{\{i+j\geq1\}}\int_0^1\int_{t_{i-1}}^{t_i}\int_0^1\int_{(t_{i+j}-r)\vee -\tau}^{(t_{i+j+1}-r)\wedge 0} \int_{t_{i+j}}^{r+s}\mathbb E\Big[\Big\<D_u\Big(k^{1}_{\sigma}(Z^{\beta_1}_{i,r}(s))^*\nn\\
&\quad D_r[(\mathcal D\Phi(T;t_i,Y^{\varsigma}_i)I^0\mathrm {Id}_{d\times d})^*f'(\Phi(T;t_i,Y^{\varsigma}_i))]\Big),\sigma(\varphi_u(0,\Phi_0))\Big\>\Big]\mathrm du
\mathrm d\nu_4^{1}(s)\mathrm d\beta_1\mathrm dr\mathrm d\varsigma\\
&=\sum_{i=1}^{N^{\Delta}}\sum_{j=-N}^{-1}\mathbf 1_{\{i+j\geq1\}}\int_0^1\int_{t_{i-1}}^{t_i}\int_0^1\int_{(t_{i+j}-r)\vee -\tau}^{(t_{i+j+1}-r)\wedge 0} \int_{t_{i+j}}^{r+s}\mathbb E\Big[\Big\<D_r[(\mathcal D\Phi(T;t_i,Y^{\varsigma}_i)\nn\\
&\quad I^0\mathrm {Id}_{d\times d})^* f'(\Phi(T;t_i,\!Y^{\varsigma}_i))],D_u k^{1}_{\sigma}(Z^{\beta_1}_{i,r}(s)) \sigma(\varphi_u(0,\Phi_0))\Big\>\Big]\!\!+\!\!\mathbb E\Big[\Big\<D_uD_r[(\mathcal D\Phi(T;t_i,\!Y^{\varsigma}_i)\nn\\
&\quad I^0\mathrm {Id}_{d\times d})^* f'(\Phi(T;t_i,Y^{\varsigma}_i))], k^{1}_{\sigma}(Z^{\beta_1}_{i,r}(s)) \sigma(\varphi_u(0,\Phi_0))\Big\>\Big]\mathrm du
\mathrm d\nu_4^{1}(s)\mathrm d\beta_1\mathrm dr\mathrm d\varsigma\\
&\leq K\sum_{i=1}^{N^{\Delta}}\sum_{j=-N}^{-1}\int_0^1\int_{t_{i-1}}^{t_i}\int_0^1\int_{(t_{i+j}-r)\vee -\tau}^{(t_{i+j+1}-r)\wedge 0} \int_{t_{i+j}}^{r+s} \Big\{\Big\|D_r\mathcal D\Phi(T;t_i,Y^{\varsigma}_i)I^0\mathrm {Id}_{d\times d}\nn\\
&\quad D_u k^{1}_{\sigma}(Z^{\beta_1}_{i,r}(s)) \sigma(\varphi_u(0,\Phi_0))\Big\|_{0,2}+\|D_r\Phi(T;t_i,Y^{\varsigma}_i)\|_{0,2}\Big\| \mathcal D\Phi(T;t_i,Y^{\varsigma}_i)I^0\mathrm{Id}_{d\times d}\nn\\
&\quad D_uk^{1}_{\sigma}
(Z^{\beta_1}_{i,r}(s))\sigma(\varphi_u(0,\Phi_0))\Big\|_{0,2}+\Big\|D_uD_r\mathcal D\Phi(T;t_i,Y^{\varsigma}_i)I^0\mathrm{Id}_{d\times d} k^{1}_{\sigma}
(Z^{\beta_1}_{i,r}(s))\nn\\ 
&\quad \sigma(\varphi_u(0,\Phi_0))\Big\|_{0,2}+\|D_u\Phi(T;t_i,Y^{\varsigma}_i)\|_{0,2}\Big\|D_r\mathcal D\Phi(T;t_i,Y^{\varsigma}_i)I^0
\mathrm{Id}_{d\times d} k^{1}_{\sigma}
(Z^{\beta_1}_{i,r}(s))\nn\\ &\quad \sigma(\varphi_u(0,\Phi_0))\Big\|_{0,2}+\|D_r\Phi(T;t_i,Y^{\varsigma}_i)\|_{0,2}\Big\| D_u\mathcal D\Phi(T;t_i,Y^{\varsigma}_i)I^0\mathrm{Id}_{d\times d} k^{1}_{\sigma}
(Z^{\beta_1}_{i,r}(s))\nn\\ &\quad\sigma(\varphi_u(0,\Phi_0))\Big\|_{0,2}+\|D_uD_r\Phi(T;t_i,Y^{\varsigma}_i)\|_{0,2}\Big\| \mathcal D\Phi(T;t_i,Y^{\varsigma}_i)I^0\mathrm{Id}_{d\times d} k^{1}_{\sigma}
(Z^{\beta_1}_{i,r}(s))\nn\\ &\quad\sigma(\varphi_u(0,\Phi_0))\Big\|_{0,2}+\|D_u\Phi(T;t_i,Y^{\varsigma}_i) D_r\Phi(T;t_i,Y^{\varsigma}_i)\|_{0,2}\Big\| \mathcal D\Phi(T;t_i,Y^{\varsigma}_i)I^0\mathrm{Id}_{d\times d}\nn\\
&\quad k^{1}_{\sigma}
(Z^{\beta_1}_{i,r}(s)) \sigma(\varphi_u(0,\Phi_0))\Big\|_{0,2}\Big\}
\mathrm du\mathrm d\nu_4^{1}(s)\mathrm d\beta_1\mathrm dr\mathrm d\varsigma
\leq K\Delta,
\end{align*} 
where we used the property  
$$\sup\limits_{u\in [t_{i+j},t_{i+j+1}]}\mathbb E[\sup\limits_{s\in[-\tau,0]}|D_uk^{1}_{\sigma}(Z^{\beta_1}_{i,r}(s))|^{2p_0}]+\mathbb E[\sup\limits_{s\in[-\tau,0]}|k^{1}_{\sigma}(Z^{\beta_1}_{i,r}(s))|^{2p_0}]\leq K$$ for some  $p_0$. 

Combining estimates of  terms $\mathcal I_0,\mathcal I_b,\mathcal I_{b,\theta}$, and $\mathcal I_{\sigma}$, we complete  the proof. 
\end{proof}

Based on Theorem \ref{thm_weakcon}, we can derive the convergence rate of the numerical invariant measure, which is stated in the following corollary. 

\begin{coro}
Under conditions  in Theorem \ref{thm_weakcon}, it holds that  for $f\in\mathcal  C^3_b,$
\begin{align*}
|\mu(f)-\mu^{\Delta}(f)|\leq K\Delta.
\end{align*}
\end{coro}
\begin{proof}
Let $\tilde \xi,\bar{\xi}$ be initial values   with distributions being $\mu,\mu^{\Delta}$, respectively. It follows from \eqref{l4.1.2} and \eqref{4.3l+2} that for any $p\geq 2$,  
$\mathbb E[\|\tilde \xi\|^{p}+\|\bar \xi\|^{p}]\leq K$.
Then for any $T>0$,
\begin{align*}
|\mu(f)-\mu^{\Delta}(f)|&=|\mathbb E[f(x^{\tilde \xi}(T))]-\mathbb E[f(y^{\Delta,\bar{\xi}}(T))]|\\
&\leq |\mathbb E[f(x^{\tilde \xi}(T))]-\mathbb E[f(x^{\bar{\xi}}(T))]|+|\mathbb E[f(x^{\bar{ \xi}}(T))]-\mathbb E[f(y^{\Delta,\bar{\xi}}(T))]|\\
&\leq K(\mathbb E[|x^{\tilde \xi}(T)-x^{\bar \xi}(T)|^2])^{\frac12}+K\Delta\\
&\leq Ke^{-KT}+K\Delta,
\end{align*}
which implies the result by letting $T\to\infty.$
\end{proof}

\section{Summary and outlook}
In this chapter, we investigate the numerical  approximation of  the invariant measure of the SFDE \eqref{FF} with the superlinearly growing drift coefficient. We first prove the existence and uniqueness of the invariant measure  of the $\theta$-EM functional solution. Then based on the Malliavin calculus, we analyze the longtime weak convergence of the SFDE \eqref{FF}, and further prove that the convergence rate of the numerical invariant measure is $1$. We remark that in our analysis, Assumption \ref{ass_nablab} is proposed for technical reason to pave a way for making use of the dual property of the Malliavin derivative operator (i.e., \eqref{IBPformula}, which is usually called the integration by parts formula). In fact, there are still some problems associated with these topics that are challenging and far from being well understood:
\begin{itemize}
\item[(\romannumeral1)] How to weaken or even eliminate Assumption \ref{ass_nablab} to include more general class of SFDEs? 
\item[(\romannumeral2)] 
Can we obtain the same weak convergence rate  if the diffusion coefficient  grows superlinearly?
\item[(\romannumeral3)] 
For the SFDE driven by the rough path, does the invariant measure of the numerical method exist  uniquely?
\item[(\romannumeral4)] If (\romannumeral3) holds, does the numerical invariant measure converge to the exact one?
\end{itemize}





    \cleardoublepage
%
%
%


\chapter{Approximation of density function}

In this chapter, we first study the existence of the density function for  the solution of the $\theta$-EM method with multiplicative noise. Then we investigate the convergence of the numerical density function. For the SFDE \eqref{FF} with the  multiplicative noise, the numerical density function is shown to converge to the exact one in $L^1(\mathbb R^d)$  based on the localization argument.   
For the case of the additive noise, the  convergence rate in the point-wise sense of the numerical density function is proved to be $1$ based on  the test-functional-independent weak convergence analysis.

\section{Existence of numerical density function}
In this section, we investigate the existence  of  density functions of  solutions for \eqref{FF} and the 
$\theta$-EM method based on the technique of the Malliavin calculus. 
The deterministic initial value $\xi$ is fixed in this section. 
We focus on the superlinealy growing coefficients case on the finite time horizon.
 Consider the following assumptions in this section.
\begin{assp}\label{repalceA2} 
Assume that there exists a constant $a_5>0$ such that for any $\phi_1,\phi_2\in\mathcal C^d$, 
\begin{align*}
\langle b(\phi_1)-b(\phi_2),\phi_1(0)-\phi_2(0)\rangle\leq a_5\Big(|\phi_1(0)-\phi_2(0)|^2+\int_{-\tau}^0|\phi_1(r)-\phi_2(r)|^2\mathrm d\nu_2(r)\Big).
\end{align*}
\end{assp}

\begin{assp}\label{Dbpoly1} 
Assume that the coefficient $b$ has continuous  G\^ateaux  derivative, and that there exists a constant $\tilde \beta\geq0$ such that  
\begin{align*}
&|\mathcal Db(\phi_1)\phi_2| \leq K(1+\|\phi_1\|^{\tilde\beta})\|\phi_2\|,
\end{align*}
where $\phi_1,\phi_2\in\mathcal C^d$, $K>0$.
\end{assp}
\begin{assp}\label{assp_den}
Assume that the coefficient $\sigma$ has continuous G\^ateaux derivative, and that  there exists some $\sigma_0>0$ such that
\begin{align*}
\inf_{\phi\in\mathcal C^d}\min_{u\in\mathbb R^d,|u|=1}u^{\top}\sigma(\phi)\sigma(\phi)^{\top}u\ge \sigma_0. 
\end{align*} 
\end{assp}
Under Assumptions \ref{a1} and \ref{repalceA2}, the existence and uniqueness of the global solution of \eqref{FF} can be obtained by using \cite[Theorem 2]{LMS11}. In addition, the functional solution of \eqref{FF}  has the following estimates. 
\begin{lemma}\label{local_Ma} Let  Assumptions \ref{a1}, \ref{repalceA2}--\ref{Dbpoly1} hold.  Assume that the coefficient $\sigma$ has continuous G\^atueax derivative. Then for any $p\ge 2$ and $T>0$,
\begin{align*}
\mathbb E[\sup_{t\in[0,T]}\|x^{\xi}_t\|^p]\leq K_{T},\quad\quad \sup_{r\in[0,T]}\mathbb E[\sup_{t\in[r,T]}\|D_rx^{\xi}_t\|^p]\leq K_{T}.
\end{align*}
\end{lemma}
\begin{proof}
 Similar to the proof of \cite[Theorem 2]{LMS11}, we can obtain the first inequality. For the second inequality, when $u\leq t,$ $D_ux^{\xi}(t)$ satisfies 
\begin{align*}
D_ux^{\xi}(t)=\int_u^t\mathcal Db(x^{\xi}_s)D_ux^{\xi}_s\mathrm ds+\int_u^t\mathcal D\sigma(x^{\xi}_s)D_ux^{\xi}_s\mathrm dW(s)+\sigma(x^{\xi}_u)\mathbf 1_{[0,t]}(u).
\end{align*}
By the It\^o formula, we derive
\begin{align}\label{boundxTD1}
&\mathrm d(|D_ux^{\xi}(t)|^{2p})\leq p|D_ux^{\xi}(t)|^{2p-2}\Big(2\<D_ux^{\xi}(t),\mathcal Db(x^{\xi}_t)D_ux^{\xi}_t\>\mathrm dt\nn\\
&\quad+(2p-1)|\mathcal D\sigma(x^{\xi}_t)D_ux^{\xi}_t|^2\mathrm dt
+2\<D_ux^{\xi}(t),\mathcal D\sigma(x^{\xi}_t)D_ux^{\xi}_t\mathrm dW(t)\>\Big).
\end{align}
It follows from  Assumptions \ref{a1} and \ref{repalceA2} that
\begin{align*}
&\quad |D_ux^{\xi}(t)|^{2p}\leq 
|\sigma(x^{\xi}_u)|^{2p} \mathbf 1_{[0,t]}(u)
+K\int_{u}^{t}\Big(|D_ux^{\xi}(s)|^{2p}+|D_ux^{\xi}(s)|^{2p-2}\times \nn\\
&\quad\big(\int_{-\tau}^{0}|D_ux^{\xi}_{s}(r)|^2\mathrm d\nu_1(r)+\int_{-\tau}^{0}|D_ux^{\xi}_{s}(r)|^2\mathrm d\nu_2(r)\big)\Big)\mathrm ds\nn\\
&\quad+
2p\int_{u}^{t}|D_ux^{\xi}(s)|^{2p-2}\big\<D_ux^{\xi}(s),\mathcal D\sigma(x^{\xi}_s)D_ux^{\xi}_s\mathrm dW(s)\big\>.
\end{align*}
Applying the Burkholder--Davis--Gundy inequality, we deduce 
\begin{align*}
&\quad \mathbb E\Big[\sup_{u\leq t\leq T}|D_ux^{\xi}(t)|^{2p}\Big]\nn\\
&\leq \E[|\sigma(x^{\xi}_u)|^{2p}]
+K\int_{u}^{T}\E\Big[|D_ux^{\xi}(s)|^{2p}+|D_ux^{\xi}(s)|^{2p-2}\big(\int_{-\tau}^{0}|D_ux^{\xi}_{s}(r)|^2\mathrm d\nu_1(r)\nn\\
&\quad+\int_{-\tau}^{0}|D_ux^{\xi}_{s}(r)|^2\mathrm d\nu_2(r)\big)\Big)\mathrm ds+K\E\Big[\Big(\int_{u}^{T}|D_ux^{\xi}(s)|^{4p-2}|\mathcal D\sigma(x^{\xi}_s)D_ux^{\xi}_s|^2\mathrm ds\Big)^{\frac{1}{2}}\Big].
\end{align*}
According to Assumption \ref{a1} and  the Young inequality, we arrive at
\begin{align*}
&\quad \mathbb E\Big[\sup_{u\leq t\leq T}|D_ux^{\xi}(t)|^{2p}\Big]\nn\\
&\leq K_{T}
+K\int_{u}^{T}\E\Big[|D_ux^{\xi}(s)|^{2p}+\int_{-\tau}^{0}|D_ux^{\xi}_{s}(r)|^{2p}\mathrm d\nu_1(r)+\!\int_{-\tau}^{0}|D_ux^{\xi}_{s}(r)|^{2p}\mathrm d\nu_2(r)\Big]\mathrm ds\!\nn\\
&\quad+ \frac{1}{2}\mathbb E\Big[\sup_{u\leq t\leq T}|D_ux^{\xi}(t)|^{2p}\Big]+
K\E\Big[\int_{u}^{T}|\mathcal D\sigma(x^{\xi}_s)D_ux^{\xi}_s|^{2p}\mathrm ds\Big]\nn\\
&\leq K_{T}
+K\int_{u}^{T}\E\Big[|D_ux^{\xi}(s)|^{2p}+\int_{-\tau}^{0}|D_ux^{\xi}_{s}(r)|^{2p}\mathrm d\nu_1(r)\nn\\
&\quad+\int_{-\tau}^{0}|D_ux^{\xi}_{s}(r)|^{2p}\mathrm d\nu_2(r)\Big]\mathrm ds+\frac{1}{2}\mathbb E\Big[\sup_{u\leq t\leq T}|D_ux^{\xi}(t)|^{2p}\Big]\nn\\
&\leq K_{T}+ K \int_{u}^{T}\E\Big[\sup_{u\leq r\leq s}|D_ux^{\xi}(r)|^{2p}\Big]\mathrm ds+\frac{1}{2}\mathbb E\Big[\sup_{u\leq t\leq T}|D_ux^{\xi}(t)|^{2p}\Big].
\end{align*}
Using the Gr\"onwall inequality, one has
$$\sup_{0\leq u\leq T}\mathbb E[\sup_{u\leq t\leq T}|D_ux^{\xi}(t)|^{2p}]\leq K_{T}.$$
The proof is completed.

\end{proof}

Based on Lemma \ref{local_Ma}  and $D_r x^{\xi}(t)=0$ for $r>t$, we derive that $x^{\xi}(t)\in \mathbb D^{1,p}$ for all $p\geq1$.
Hence, according to 
\cite[Theorem 2.1.2]{Nualart}, 
in order to obtain the existence of the density function of the solution of \eqref{FF}, it suffices to show that the Malliavin covariance matrix $\gamma^{E}(t):=\int_0^t D_rx^{\xi}(t)(D_rx^{\xi}(t))^{\top}\mathrm dr$  of $x^{\xi}(t)$ for $t\in[0,T]$ is a.s. invertible. To this end, by virtue of \cite[Section 3.3]{cui_CHeq},
we first  prove that for some $q>0,$ there exists a small number $\epsilon_0(q)$ such that for any $\epsilon<\epsilon_0(q),$ $\sup_{u\in\RR^d, |u|=1}\mathbb P(u^{\top}\gamma^{E}(t)u\leq \epsilon)\leq \epsilon^q,\;t\in[0,T].$ This is stated as follows.
\begin{prop}\label{multi_nonlip}
Under Assumptions \ref{a1}, \ref{repalceA2}--\ref{assp_den}, there exists some $\epsilon_0(q)$ such that for 
$\epsilon<\epsilon_0(q),$ 
it holds that
\begin{align*}
\sup_{u\in\RR^d, |u|=1}\mathbb P(u^{\top}\gamma^{E}(t)u\leq \epsilon)\leq \epsilon 
,\;t\in[0,T].
\end{align*}
\end{prop}
\begin{proof}
For any $r\leq t,$
\begin{align*}
D_rx^{\xi}(t)=\int_r^t\mathcal Db(x^{\xi}_s)D_rx^{\xi}_s\mathrm ds+\int_r^t\mathcal D\sigma(x^{\xi}_s)D_rx^{\xi}_s\mathrm dW(s)+\sigma(x^{\xi}_r)\mathbf 1_{[0,t]}(r). 
\end{align*}
 Fixing  $\epsilon\in(0,1)$ and letting  $\epsilon_1:=\frac{2\epsilon}{\sigma_0},$ we obtain   
  \begin{align*}
 & u^{\top}\gamma^{E}(t)u \ge \int_{t-\epsilon_1}^tu^{\top}D_rx^{\xi}(t)(D_rx^{\xi}(t))^{\top}u\mathrm dr\ge \int_{t-\epsilon_1}^tu^{\top}\sigma(x^{\xi}_r)(\sigma(x^{\xi}_r))^{\top}u\mathrm dr\\
  &\quad+2\int_{t-\epsilon_1}^tu^{\top} \Big(\int_r^t\mathcal Db(x^{\xi}_s)D_rx^{\xi}_s\mathrm ds+\int_r^t\mathcal D\sigma(x^{\xi}_s)D_rx^{\xi}_s\mathrm dW(s)\Big)(\sigma(x^{\xi}_r))^{\top}u\mathrm dr.
  \end{align*}
It follows from Assumption \ref{assp_den} that
\begin{align*}
\int_{t-\epsilon_1}^tu^{\top}\sigma(x^{\xi}_r)\sigma(x^{\xi}_r)^{\top}u\mathrm dr\ge \epsilon_1\sigma_0=2\epsilon.
\end{align*}
This, along with the Chebyshev inequality, the H\"older inequality, and the Bukerholder--Davis--Gundy inequality implies  that for  any $u\in\RR^d$ with $|u|=1,$
\begin{align*}
&\quad \mathbb P(u^{\top}\gamma^{E}(t)u\leq \epsilon)\\
&\leq \mathbb P\Big(2\Big|\!\int_{t-\epsilon_1}^t\!\!\!u^{\top}\! \Big(\!\!\int_r^t\!\!\mathcal Db(x^{\xi}_s)D_rx^{\xi}_s\mathrm ds\!+\!\int_r^t\!\!\mathcal  D\sigma(x^{\xi}_s)D_rx^{\xi}_s\mathrm dW(s)\Big)(\sigma(x^{\xi}_r) )^{\top}u\mathrm dr\Big|\ge \epsilon\Big)\\
&\leq 8\epsilon^{-2}\mathbb E\Big[\Big|\int_{t-\epsilon_1}^tu^{\top} \int_{r}^t\mathcal Db(x^{\xi}_s)D_rx^{\xi}_s\mathrm ds(\sigma(x^{\xi}_r))^{\top}u\mathrm dr\Big|^2\Big]\\
&\quad +8\epsilon^{-2}\mathbb E\Big[\Big|\int_{t-\epsilon_1}^tu^{\top} \int_r^t\mathcal D\sigma(x^{\xi}_s)D_rx^{\xi}_s\mathrm dW(s)(\sigma(x^{\xi}_r))^{\top}u\mathrm dr\Big|^2 \Big]\\
&\leq 8\epsilon^{-2}\epsilon_1\int_{t-\epsilon_1}^{t}\E\Big[\sup_{r\leq s\leq T}|\mathcal Db(x^{\xi}_s)D_rx^{\xi}_s|^2|\sigma(x^{\xi}_r)|^2\Big](t-r)^2\mathrm dr\nn\\
&\quad+8\epsilon^{-2}\epsilon_1\int_{t-\epsilon_1}^t \Big(\E\Big[|\int_r^t\mathcal D\sigma(x^{\xi}_s)D_rx^{\xi}_s\mathrm dW(s)|^4\Big]\Big)^{\frac12}\Big(\E[|\sigma(x^{\xi}_r)|^4]\Big)^{\frac{1}{2}}\mathrm dr\nn\\
&\leq K\epsilon^2\Big(\sup_{0\leq r\leq T}\mathbb E\Big[\sup_{r\leq s\leq T}|\mathcal Db(x^{\xi}_s)D_rx^{\xi}_s|^4\Big]\Big)^{\frac12}\Big(\sup_{0\leq r\leq T}\mathbb E[|\sigma(x^{\xi}_r)|^4]\Big)^{\frac12}\\
&\quad+K\epsilon\Big(\sup_{0\leq r\leq T}\mathbb E\Big[\sup_{r\leq s\leq T}|\mathcal D\sigma(x^{\xi}_s)D_rx^{\xi}_s|^4\Big]\Big)^{\frac12}\Big(\sup_{0\leq r\leq T}\mathbb E[|\sigma(x^{\xi}_r)|^4]\Big)^{\frac12}\\
&\leq \epsilon K_{T}, 
\end{align*}
where in the last inequality we used $$|\mathcal Db(x^{\xi}_s)D_rx^{\xi}_s|\leq K (1+\|x^{\xi}_s\|^{\tilde \beta})\|D_rx^{\xi}_s\|,\quad |\mathcal D\sigma(x^{\xi}_s)D_rx^{\xi}_s|\leq K\|D_rx^{\xi}_s\|,$$  Assumption \ref{a1}, and Lemma \ref{local_Ma}. 
The proof is finished.  
\end{proof}
The existence of the density function of the solution of \eqref{FF} is stated as follows. 
\begin{thm}\label{density_exact}
  Under Assumptions \ref{a1}, \ref{repalceA2}--\ref{assp_den}, for any $T>0$, the law of $x^{\xi}(T)$ admits a density function.
\end{thm}
\begin{proof}
Based on Proposition \ref{multi_nonlip},  we have
\begin{align*}
\sup_{u\in\RR^d,|u|=1}\mathbb P(u^{\top}\gamma^E(T)u=0) 
\leq \sup_{u\in\RR^d,|u|=1}\mathbb P(u^{\top}\gamma^E(T)u\leq \epsilon)\leq \epsilon\quad\forall \epsilon\in(0,1),
\end{align*}
which shows the a.s. invertibility of $\gamma^E(T)$ due to the arbitrary of $\epsilon$. This, together with $x^{\xi}(T)\in\mathbb D^{1,p}$ and \cite[Theorem 2.1.2]{Nualart}, completes the proof. 
\end{proof}

Now we are in the position to present the existence of the density function of the $\theta$-EM solution. We first show the estimates on the $\theta$-EM  functional solution.

\begin{lemma}\label{multi_nonlip_nu}
Let Assumptions \ref{a1}, \ref{repalceA2}, and \ref{Dbpoly1} hold.  Assume that the  coefficient $\sigma$ has continuous G\^atueax derivative.  Then for $\Delta\in(0, \frac{1}{4\theta a_5})$,  $p\geq2$, and $T>0$, 
\begin{align*}
\mathbb E\Big[\sup_{0\leq k\Delta\leq T}\|y^{\xi,\Delta}_{t_k}\|^p\Big]\leq K_{T},\quad\quad\sup_{0\leq u\leq T}\mathbb E\Big[\sup_{u\leq k\Delta\leq T}\|D_uy^{\xi,\Delta}_{t_k}\|^p\Big]\leq K_{T}.
\end{align*}
\end{lemma}
\begin{proof}
The proof of $\mathbb E[\sup_{0\leq k\Delta\leq T}\|y^{\xi,\Delta}_{t_k}\|^p]\leq K_{T}$ is similar to that of Proposition \ref{p2.1} and thus is  omitted. Below we prove  $\sup_{0\leq u\leq T}\mathbb E[\sup_{u\leq k\Delta\leq T}\|D_uy^{\xi,\Delta}_{t_k}\|^p]\leq K_{T}$.
For any $k\in\mathbb N$ with $0\leq k\Delta\leq T$ and $u\in[0,T]$, we have
\begin{align*}
&|D_u z^{\xi,\Delta}(t_{k+1})|^2=|D_u z^{\xi,\Delta}(t_k)|^2\!+\!|\mathcal Db(y^{\xi,\Delta}_{t_k})D_u y^{\xi,\Delta}_{t_k}|^2\Delta^2\!+\!|\mathcal D\sigma(y^{\xi,\Delta}_{t_k})D_u y^{\xi,\Delta}_{t_k}\delta W_k|^2\nn\\
&\quad+|\sigma(y^{\xi,\Delta}_{t_k})\mathbf 1_{[t_{k},t_{k+1}]}(u)|^2
+2\big\<D_u y^{\xi,\Delta}(t_k)-\theta\Delta\mathcal Db(y^{\xi,\Delta}_{t_k})D_u y^{\xi,\Delta}_{t_k}, \mathcal Db(y^{\xi,\Delta}_{t_k})D_u y^{\xi,\Delta}_{t_k}\Delta\>\nn\\
&\quad+2\big\<D_u z^{\xi,\Delta}(t_k)+\mathcal Db(y^{\xi,\Delta}_{t_k})D_u y^{\xi,\Delta}_{t_k}\Delta,\sigma(y^{\xi,\Delta}_{t_k})\mathbf 1_{[t_{k},t_{k+1}]}(u)\>
+\mathcal {\tilde M}_{k}\nn\\
&\leq|D_u z^{\xi,\Delta}(t_k)|^2+|\mathcal D\sigma(y^{\xi,\Delta}_{t_k})D_u y^{\xi,\Delta}_{t_k}\delta W_k|^2
+|\sigma(y^{\xi,\Delta}_{t_k})\mathbf 1_{[t_{k},t_{k+1}]}(u)|^2+2\big\<D_u y^{\xi,\Delta}(t_k),\nn\\
&\quad \mathcal Db(y^{\xi,\Delta}_{t_k})D_u y^{\xi,\Delta}_{t_k}\Delta\big\>
+\big\<\frac{2(\theta-1)}{\theta}D_u z^{\xi,\Delta}(t_k)+\frac{2}{\theta}D_u y^{\xi,\Delta}(t_k),\sigma(y^{\xi,\Delta}_{t_k})\mathbf 1_{[t_{k},t_{k+1}]}(u)\big\>,
\end{align*}
 where 
 \begin{align*}
 \mathcal {\tilde M}_{k}:=2\big\<D_u z^{\xi,\Delta}(t_k)\!\!+\!\!\mathcal Db(y^{\xi,\Delta}_{t_k})D_u y^{\xi,\Delta}_{t_k}\Delta\!\!+\!\!\sigma(y^{\xi,\Delta}_{t_k})\mathbf 1_{[t_{k},t_{k+1}]}(u),  \mathcal D\sigma(y^{\xi,\Delta}_{t_k})D_u y^{\xi,\Delta}_{t_k}\delta W_k\big\>.
 \end{align*}
 By Assumptions \ref{a1} and \ref{repalceA2}, one has that for any $\phi,\phi_1\in \mathcal C^d$, 
$$\<\phi(0),\mathcal Db(\phi_1) \phi\>\vee |\mathcal Db(\phi_1) \phi|^2\leq K\big(|\phi(0)|^2+\int_{-\tau}^{0}|\phi(r)|^2\mathrm d \nu_1(r)+\int_{-\tau}^{0}|\phi(r)|^2\mathrm d \nu_2(r)\big).$$ 
 This implies that for any $p\in\mathbb N_+$,
 \begin{align}\label{boundTD1}
 |D_u z^{\xi,\Delta}(t_{k+1})|^{2p}=&|D_u z^{\xi,\Delta}(t_k)|^{2p}+\sum_{l=1}^{p}C_{p}^{l}I_{k,l}
 =\sum_{i=0}^{k}\sum_{l=1}^{p}C_{p}^{l}I_{i,l},
 \end{align}
 where 
 \begin{align*}
 I_{i,l}&:=\,|D_u z^{\xi,\Delta}(t_i)|^{2(p-l)}\Big(|\mathcal D\sigma(y^{\xi,\Delta}_{t_i})D_u y^{\xi,\Delta}_{t_i}\delta W_i|^2
\!\!+\!\!|\sigma(y^{\xi,\Delta}_{t_i})\mathbf 1_{[t_{i},t_{i+1}]}(u)|^2\nn\\
&\quad+K\Delta \big(|D_{u}y^{\xi,\Delta}(t_i)|^{2}
+\int_{-\tau}^{0}|D_{u}y^{\xi,\Delta}_{t_i}(r)|^{2}\mathrm d \nu_2(r)\big)
+\big\<\frac{2(\theta-1)}{\theta}D_u z^{\xi,\Delta}(t_i)\nn\\
&\quad+\frac{2}{\theta}D_u y^{\xi,\Delta}(t_i),\sigma(y^{\xi,\Delta}_{t_i})\mathbf 1_{[t_{i},t_{i+1}]}(u)\big\>+\mathcal {\tilde M}_{i}\Big)^{l}.
 \end{align*}
For the term $I_{i,1}$, by  the Young inequality and the property of the conditional expectation,
we obtain
 \begin{align}\label{boundTD3}
&\E[I_{i,1}]\leq \E\Big[ |D_u z^{\xi,\Delta}(t_i)|^{2(p-1)} \Big(|\mathcal D\sigma(y^{\xi,\Delta}_{t_i})D_u y^{\xi,\Delta}_{t_i}|^2\Delta+|\sigma(y^{\xi,\Delta}_{t_i})\mathbf 1_{[t_{i},t_{i+1}]}(u)|^2\nn\\
&\quad+K\Delta \big(|D_{u}y^{\xi,\Delta}(t_i)|^{2}+\int_{-\tau}^{0}\!|D_{u}y^{\xi,\Delta}_{t_i}(r)|^{2}\mathrm d \nu_2(r)\big)\nn\\
&\quad+\big\<\frac{2(\theta\!-\!1)}{\theta}D_u z^{\xi,\Delta}(t_i)\!+\!\frac{2}{\theta}D_u y^{\xi,\Delta}(t_i),\sigma(y^{\xi,\Delta}_{t_i})\mathbf 1_{[t_{i},t_{i+1}]}(u)\big\>\Big)\Big]\nn\\
&\leq\big(K\Delta+\mathbf 1_{[t_{i},t_{i+1}]}(u)\big)\E\big[|D_u z^{\xi,\Delta}(t_i)|^{2p}+|D_{u}y^{\xi,\Delta}(t_i)|^{2p}\big]\nn\\
&\quad+\!K\Delta\E\Big[\!\int_{-\tau}^{0}\!|D_{u}y^{\xi,\Delta}_{t_i}(r)|^{2p}\mathrm d \nu_1(r)\!
+\!\!\int_{-\tau}^{0}\!|D_{u}y^{\xi,\Delta}_{t_i}(r)|^{2p}\mathrm d \nu_2(r)\!\Big]\!+\!K_{T}\mathbf 1_{[t_{i},t_{i+1}]}(u),
\end{align}
where in the last inequality we used Assumptions \ref{a1} and  $\E[\sup_{0\leq t_{i}\leq T}|y^{\xi,\Delta}(t_i)|^{2p}]\leq K_{T}$.
For the term $I_{i,2}$, we have
\begin{align}\label{boundTD4}
&\E[I_{i,2}]=\E\Big[|D_u z^{\xi,\Delta}(t_i)|^{2(p-2)}\Big(|\mathcal D\sigma(y^{\xi,\Delta}_{t_i})D_u y^{\xi,\Delta}_{t_i}\delta W_k|^2+|\sigma(y^{\xi,\Delta}_{t_i})\mathbf 1_{[t_{i},t_{i+1}]}(u)|^2\nn\\
&\quad +K\Delta \big(|D_{u}y^{\xi,\Delta}(t_i)|^{2}+\int_{-\tau}^{0}|D_{u}y^{\xi,\Delta}_{t_i}(r)|^{2}\mathrm d \nu_2(r)\big)+\big\<\frac{2(\theta-1)}{\theta}D_u z^{\xi,\Delta}(t_i)\nn\\
&\quad+\frac{2}{\theta}D_u y^{\xi,\Delta}(t_i),\sigma(y^{\xi,\Delta}_{t_i})\mathbf 1_{[t_{i},t_{i+1}]}(u)\big\>
\Big)^2+|D_u z^{\xi,\Delta}(t_i)|^{2(p-2)}\mathcal {\tilde M}_{i}^2\Big]\nn\\
&\leq\big(K\Delta+\mathbf 1_{[t_{i},t_{i+1}]}(u)\big)\E\big[|D_u z^{\xi,\Delta}(t_i)|^{2p}+|D_{u}y^{\xi,\Delta}(t_i)|^{2p}\big]\nn\\
&\quad+\!K\Delta \E\Big[\!\int_{-\tau}^{0}\!|D_{u}y^{\xi,\Delta}_{t_i}(r)|^{2p}\mathrm d \nu_1(r)
\!+\!\int_{-\tau}^{0}\!|D_{u}y^{\xi,\Delta}_{t_i}(r)|^{2p}\mathrm d \nu_2(r)\Big]\!+\!K_{T}\mathbf 1_{[t_{i},t_{i+1}]}(u).
 \end{align}
Similarly, we deduce that for any $l\in\{3,\ldots, p\}$,
\begin{align}\label{boundTD5}
&\quad\E[|I_{i,l}|]\leq\big(K\Delta+\mathbf 1_{[t_{i},t_{i+1}]}(u)\big)\E\big[|D_u z^{\xi,\Delta}(t_i)|^{2p}+|D_{u}y^{\xi,\Delta}(t_i)|^{2p}\big]\nn\\
&\quad+K\Delta \E\Big[\int_{-\tau}^{0}|D_{u}y^{\xi,\Delta}_{t_i}(r)|^{2p}\mathrm d \nu_1(r)
+\int_{-\tau}^{0}|D_{u}y^{\xi,\Delta}_{t_i}(r)|^{2p}\mathrm d \nu_2(r)\Big]+K_{T}\mathbf 1_{[t_{i},t_{i+1}]}(u).
\end{align}
Inserting \eqref{boundTD3}--\eqref{boundTD5} into \eqref{boundTD1}, we arrive at
\begin{align}\label{boundTD6}
&\quad \E[|D_u z^{\xi,\Delta}(t_{k+1})|^{2p}]\nn\\
&\leq K\sum_{i=0}^{k}\big(\Delta+\mathbf 1_{[t_{i},t_{i+1}]}(u)\big)\E\big[|D_u z^{\xi,\Delta}(t_i)|^{2p}+|D_{u}y^{\xi,\Delta}(t_i)|^{2p}\big]+K_{T}\sum_{i=0}^{k}\mathbf 1_{[t_{i},t_{i+1}]}(u)\nn\\
&\quad+K\Delta \sum_{i=0}^{k}\E\Big[\int_{-\tau}^{0}|D_{u}y^{\xi,\Delta}_{t_i}(r)|^{2p}\mathrm d \nu_1(r)\!+\!\int_{-\tau}^{0}|D_{u}y^{\xi,\Delta}_{t_i}(r)|^{2p}\mathrm d \nu_2(r)\Big]\nn\\
&\leq K\sum_{i=0}^{k}\big(\Delta+\mathbf 1_{[t_{i},t_{i+1}]}(u)\big)\E\big[|D_u z^{\xi,\Delta}(t_i)|^{2p}+|D_{u}y^{\xi,\Delta}(t_i)|^{2p}\big]+K_{T}.
 \end{align}
 Here, we used
\begin{align*}
\sum_{i=0}^{k}\int_{-\tau}^{0}|D_{u}y^{\xi,\Delta}_{t_i}(r)|^{2p}\mathrm d \nu_{j}(r)
\leq \sum_{i=0}^{k}|D_{u}y^{\xi,\Delta}(t_i)|^{2p},\quad j=1,2,
\end{align*}
whose proof is similar to \eqref{5p2.1}.
It follows from $$D_uz^{\xi,\Delta}(t_{i}) =D_uy^{\xi,\Delta}(t_{i})-\theta\Delta\mathcal D b(y^{\xi,\Delta}_{t_i})D_{u}y^{\xi,\Delta}_{t_i}$$ 
that
\begin{align*}
|D_uz^{\xi,\Delta}(t_{i})|^2 \geq&\, |D_uy^{\xi,\Delta}(t_{i})|^2-2\theta\Delta\big\<D_uy^{\xi,\Delta}(t_{i}),\mathcal D b(y^{\xi,\Delta}_{t_i})D_{u}y^{\xi,\Delta}_{t_i}\big\>\nn\\
\geq&\,|D_uy^{\xi,\Delta}(t_{i})|^2-2\theta\Delta a_5\big(|D_{u}y^{\xi,\Delta}(t_i)|^{2}
+\int_{-\tau}^{0}|D_{u}y^{\xi,\Delta}_{t_i}(r)|^{2}\mathrm d \nu_2(r)\big).
\end{align*}
Let $\Delta_1\in(0,\frac{1}{4\theta a_5})$.
Then for $\epsilon\in(0,\frac{1-2\theta a_5\Delta_1}{2\theta a_5\Delta_1}-1)$ and $\Delta\in(0,\Delta_1)$,
\begin{align}\label{Dyle+3}
&\quad(1-2\theta a_5\Delta)^p|D_uy^{\xi,\Delta}(t_{i})|^{2p}\nn\\
&\leq K(\epsilon)|D_uz^{\xi,\Delta}(t_{i})|^{2p}+(1+\epsilon)(2\theta a_5\Delta)^p\int_{-\tau}^{0}|D_{u}y^{\xi,\Delta}_{t_i}(r)|^{2p}\mathrm d \nu_2(r).
\end{align}
This implies that
\begin{align}\label{Dyle+2}
\sum_{i=0}^{k}|D_uy^{\xi,\Delta}(t_{i})|^{2p}&\leq\frac{K(\epsilon)}{(1-2\theta a_5\Delta_1)^p-(1+\epsilon)(2\theta a_5\Delta_1)^p}
\sum_{i=0}^{k}|D_uz^{\xi,\Delta}(t_{i})|^{2p}\nn\\
&\leq K(\epsilon) \sum_{i=0}^{k}|D_uz^{\xi,\Delta}(t_{i})|^{2p}.
\end{align}
and
\begin{align}\label{Dyle+1}
&\quad\sum_{i=0}^{k}\mathbf 1_{[t_{i},t_{i+1}]}(u)\E[|D_uy^{\xi,\Delta}(t_{i})|^{2p}]\nn\\
&\leq K\sum_{i=0}^{k}\mathbf 1_{[t_{i},t_{i+1}]}(u)\E[|D_uz^{\xi,\Delta}(t_{i})|^{2p}]+K\Delta \sum_{i=0}^{k}|D_uy^{\xi,\Delta}(t_{i})|^{2p}\nn\\
&\leq K(\epsilon)\sum_{i=0}^{k}\big(\Delta+\mathbf 1_{[t_{i},t_{i+1}]}(u)\big)\E[|D_u z^{\xi,\Delta}(t_i)|^{2p}].
\end{align}
Inserting \eqref{Dyle+2} and \eqref{Dyle+1} into \eqref{boundTD6} and applying discrete Gr\"onwall inequality, we derive
\begin{align}\label{boundTD7}
\sup_{u\leq t_{k}\leq T}\E[|D_u z^{\xi,\Delta}(t_{k})|^{2p}]
\leq \sup_{u\leq t_{k}\leq T}K_{T}e^{Kt_{k+1}+1}\leq K_{T}.
\end{align}
It follows from \eqref{Dyle+3} that
\begin{align}\label{Dyle+4}
\sup_{u\leq t_{k}\leq T}\E[|D_u y^{\xi,\Delta}(t_{k})|^{2p}]
\leq K_{T}.
\end{align}
In addition, by \eqref{boundTD1},  we obtain
\begin{align}\label{boundTD10}
&\quad\E\Big[\sup_{u\leq t_{k+1}\leq T}|D_u z^{\xi,\Delta}(t_{k+1})|^{2p}\Big]\nn\\
&\leq\E \Big[\sup_{u\leq t_{k}\leq T}\sum_{i=0}^{k}\big(C_{p}^{1}I_{i,1}+C_{p}^{2}I_{i,2}\big)\Big]+
\E\Big[\sum_{i=0}^{k}\sum_{l=3}^{p}C_{p}^{l}|I_{i,l}|\Big]\nn\\
&\leq \E\Big[\sup_{u\leq t_{k}\leq T}\sum_{i=0}^{k}\big(C_{p}^{1}I_{i,1}+C_{p}^{2}I_{i,2}\big)\Big]+
\sum_{i=0}^{k}\big(K\Delta\!+\!\mathbf 1_{[t_{i},t_{i+1}]}(u)\big)\times\nn\\
&\quad\E\big[|D_u z^{\xi,\Delta}(t_i)|^{2p}+|D_{u}y^{\xi,\Delta}(t_i)|^{2p}\big]+K_{T}\nn\\
&\leq \E \Big[\sup_{u\leq t_{k}\leq T}\sum_{i=0}^{k}\big(C_{p}^{1}I_{i,1}+C_{p}^{2}I_{i,2}\big)\Big]+K_T.
\end{align}
Utilizing  the Burkholder--Davis--Gundy inequality, \eqref{boundTD7}, and \eqref{Dyle+4}, we have
\begin{align}\label{boundTD8}
&\quad\E\Big[\sup_{u\leq t_{k}\leq T}\sum_{i=0}^{k}I_{i,1}\Big]
\leq \sum_{i=0}^{\lfloor T\rfloor/\Delta}\big(K\Delta+\mathbf 1_{[t_{i},t_{i+1}]}(u)\big)\E\Big[|D_u z^{\xi,\Delta}(t_i)|^{2p}+|D_{u}y^{\xi,\Delta}(t_i)|^{2p}\Big]\nn\\
&\quad+K_{T}\sum_{i=0}^{\lfloor T\rfloor/\Delta}\mathbf 1_{[t_{i},t_{i+1}]}(u)
+\E\big[\sup_{u\leq t_{k}\leq T}\sum_{i=0}^{k}|D_u z^{\xi,\Delta}(t_i)|^{2(p-1)}\mathcal {\tilde M}_{i}\big]\nn\\
&\leq K_{T}+K\E\Big[\Big(\sum_{i=\lfloor u\rfloor/\Delta}^{\lfloor T\rfloor/\Delta}|D_u z^{\xi,\Delta}(t_k)|^{4(p-1)}\big(|D_u z^{\xi,\Delta}(t_k)|+|D_u y^{\xi,\Delta}(t_k)|\nn\\
&\quad+|\sigma(y^{\xi,\Delta}_{t_k})|\mathbf 1_{[t_{k},t_{k+1}]}(u)\big)^2|\mathcal D\sigma(y^{\xi,\Delta}_{t_k})D_u y^{\xi,\Delta}_{t_k}|^2\Delta\Big)^{\frac12}\Big]\nn\\
&\leq K_{T}+\frac{1}{4C_{p}^1}\E\Big[\sup_{u\leq t_{k}\leq T}|D_u z^{\xi,\Delta}(t_k)|^{2p}\Big]+K\E\Big[\sum_{i=\lfloor u\rfloor/\Delta}^{\lfloor T\rfloor/\Delta}\Big(|\sigma(y^{\xi,\Delta}_{t_k})|^{2p}\mathbf 1_{[t_{k},t_{k+1}]}(u)\nn\\
&\quad+|\mathcal D\sigma(y^{\xi,\Delta}_{t_k})D_u y^{\xi,\Delta}_{t_k}|^{2p}\Delta^{\frac{p}{2}}+|D_u y^{\xi,\Delta}(t_k)|^{2p}\Delta^{\frac{p}{2}}\Big)\Big]\nn\\
&\leq K_{T}+\frac{1}{4C_{p}^1}\E\Big[\sup_{u\leq t_{k}\leq T}|D_u z^{\xi,\Delta}(t_k)|^{2p}\Big].
\end{align}
Similarly, 
\begin{align*}
\E\Big[\sup_{u\leq t_{k}\leq T}\sum_{i=0}^{k}I_{i,2}\Big]
\leq K_{T}+\frac{1}{4C_{p}^2}\E\Big[\sup_{u\leq t_{k}\leq T}|D_u z^{\xi,\Delta}(t_k)|^{2p}\Big].
 \end{align*}
This, along with \eqref{boundTD10} and  \eqref{boundTD8} gives that
$$\sup_{0\leq u\leq T}\E\Big[\sup_{u\leq t_{k+1}\leq T}|D_u z^{\xi,\Delta}(t_{k+1})|^{2p}\Big]
\leq K_{T}.$$
It follows from \eqref{Dyle+3} that
$$\sup_{0\leq u\leq T}\E\Big[\sup_{u\leq t_{k+1}\leq T}|D_u y^{\xi,\Delta}(t_{k+1})|^{2p}\Big]
\leq K_{T}.$$
The proof is competed.
\end{proof}
As a consequence, we obtain $y^{\xi,\Delta}(t_k)\in\mathbb D^{1,p}.$ Below we show the existence of the numerical density function. 

\begin{thm}\label{trun_num_den}
Under Assumptions \ref{a1}, \ref{repalceA2}--\ref{assp_den}, for $\Delta\in(0,\frac{1}{4\theta a_5})$ and  $k\in\mathbb N_{+},$ the law of $y^{\xi,\Delta}(t_k)$ admits a density function. 
\end{thm}
\begin{proof}
Simialr to the proof of Theorem \ref{density_exact}, it suffices to show that the Malliavin covariance matrix of $y^{\xi,\Delta}(t_k)$, defined by $\gamma_k:=\int_0^{t_{k}}D_ry^{\xi,\Delta}(t_k)(D_ry^{\xi,\Delta}(t_k))^{\top}\mathrm dr$, is a.s. invertible. 

Taking the Malliavin derivative on both sides of the $\theta$-EM solution, we derive
\begin{align*}
D_r y^{\xi,\Delta}(t_{k+1})&=D_ry^{\xi,\Delta}(t_k)+(1-\theta)\mathcal Db(y^{\xi,\Delta}_{t_k})D_ry^{\xi,\Delta}_{t_k}\Delta+\theta \mathcal Db(y^{\xi,\Delta}_{t_{k+1}})D_ry^{\xi,\Delta}_{t_{k+1}}\Delta\\
&\quad +\mathcal D\sigma(y^{\xi,\Delta}_{t_k})D_ry^{\xi,\Delta}_{t_k}\delta W_k+\sigma(y^{\xi,\Delta}_{t_k})D_r\delta W_k.
\end{align*}
Similar to  \eqref{int_1}, we have
$$y^{\xi,\Delta}_{t_{k+1}}=\sum_{j=-N}^{0}I^jy^{\xi,\Delta}_{t_{k+1}}(t_j),$$
where $I^{j}(\cdot)\in \mathcal C([-\tau,0];\mathbb R),\,j =-N,\ldots, 0$ are given in \eqref{int_1}. 
This implies
\begin{align*}
D_ry^{\xi,\Delta}_{t_{k+1}}=D_r\Big(\sum_{j=-N}^{-1}I^jy^{\xi,\Delta}_{t_{k+1}}(t_j)\Big)+D_ry^{\xi,\Delta}(t_{k+1})I^0.
\end{align*}
Hence, 
\begin{align*}
&\quad D_ry^{\xi,\Delta}(t_{k+1})-\theta \mathcal Db(y^{\xi,\Delta}_{t_{k+1}}) D_ry^{\xi,\Delta}(t_{k+1})I^0\Delta\nn\\
&=D_ry^{\xi,\Delta}(t_k)+(1-\theta)\mathcal Db(y^{\xi,\Delta}_{t_k})D_ry^{\xi,\Delta}_{t_k}\Delta+\sigma (y^{\xi,\Delta}_{t_k})D_r\delta W_k\nn\\
&\quad+\theta \mathcal Db(y^{\xi,\Delta}_{t_{k+1}})\Big(\sum_{j=-N}^{-1}I^jD_ry^{\xi,\Delta}_{t_{k+1}}(s_j)\Big)\Delta+\mathcal D\sigma(y^{\xi,\Delta}_{t_k})D_ry^{\xi,\Delta}_{t_k}\delta W_k\\
&=D_ry^{\xi,\Delta}(t_k)\!+\!(1\!-\!\theta)\mathcal Db(y^{\xi,\Delta}_{t_k})D_r\Big(\!\!\sum_{j=-N}^0I^jy^{\xi,\Delta}_{t_k}(t_j)\!\Big)\Delta\!+\!\theta \mathcal Db(y^{\xi,\Delta}_{t_{k+1}})I^{-1}D_ry^{\xi,\Delta}(t_k)\Delta\\
&\quad+\theta \mathcal Db(y^{\xi,\Delta}_{t_{k+1}})\sum_{j=-N}^{-2}I^jD_ry^{\xi,\Delta}_{t_{k+1}}(t_j)\Delta+\mathcal D\sigma(y^{\xi,\Delta}_{t_k})D_r\Big(\sum_{j=-N}^0I^jy^{\xi,\Delta}_{t_k}(t_j)\Big)\delta W_k\nn\\
&\quad+\sigma(y^{\xi,\Delta}_{t_k})D_r\delta W_k.
\end{align*}
It is straightforward to see that
\begin{align}\label{expre_deri}
&\Big(\mathrm {Id}_{d\times d}-\theta \mathcal Db(y^{\xi,\Delta}_{t_{k+1}})I^0\Delta\Big)D_ry^{\xi,\Delta}(t_{k+1})\nn\\
=&\sum_{j=-N}^0A_{j,k}D_ry^{\xi,\Delta}(t_{k+j})+ \sigma (y^{\xi,\Delta}_{t_k})\mathbf 1_{[t_k,t_{k+1}]}(r),
\end{align}
where 
\begin{align*}
A_{0,k}:&=\mathrm {Id}_{d\times d}+(1-\theta)\mathcal Db(y^{\xi,\Delta}_{t_k}) I^0\Delta+\theta \mathcal Db(y^{\xi,\Delta}_{t_{k+1}})I^{-1}\Delta +\mathcal D\sigma(y^{\xi,\Delta}_{t_k})I^0\delta W_k,\\
A_{-N,k}:&=(1-\theta)\mathcal Db(y^{\xi,\Delta}_{t_k})I^{-N}\Delta +\mathcal D\sigma(y^{\xi,\Delta}_{t_k})I^{-N}\delta W_k,\\
A_{j,k}:&=(1-\theta)\mathcal Db(y^{\xi,\Delta}_{t_k})I^{j}\Delta+\theta \mathcal Db(y^{\xi,\Delta}_{t_{k+1}})I^{j-1}\Delta +\mathcal D\sigma(y^{\xi,\Delta}_{t_k})I^{j}\delta W_k
\end{align*}
for $j=-1,\ldots,-N+1$.
It follows from Assumption \ref{repalceA2} that for any  $0\neq u\in\RR^d$,
\begin{align*}
u^{\top} \mathcal Db(y^{\xi,\Delta}_{t_{k+1}})I^0  u \leq 2a_5|u|^2.
\end{align*}
Then for $\Delta\in(0,\frac{1}{4\theta a_5})$,
\begin{align*}
u^{\top}\Big(\mathrm {Id}_{d\times d}-\theta \mathcal Db(y^{\xi,\Delta}_{t_{k+1}})I^0\Delta\Big)u\geq (1-2\theta a_5\Delta)|u|^2>0,
\end{align*}
which implies that
$\mathrm {Id}_{d\times d}-\theta \mathcal Db(y^{\xi,\Delta}_{t_{k+1}})I^0\Delta$ is invertible.
Denoting $A_{1,k}:=\big(\mathrm {Id}_{d\times d}-\theta \mathcal Db(y^{\xi,\Delta}_{t_{k+1}})I^0\Delta\big)^{-1}$, we derive
$\|A_{1,k}\|_{\mathcal L(\RR^d; \RR^{d})}\in(0, \frac{1}{1-2\theta a_5\Delta})$.
According to \eqref{expre_deri}, we arrive at 
\begin{align}\label{recur_gamma}
&\quad\gamma_{k+1}=\int_0^{t_{k+1}}D_ry^{\xi,\Delta}(t_{k+1})(D_ry^{\xi,\Delta}(t_{k+1}))^{\top}\mathrm dr\nn\\
=&\int_0^{t_{k}}A_{1,k}\sum_{j=-N}^0A_{j,k}D_ry^{\xi,\Delta}(t_{k+j})\Big(\sum_{j=-N}^0A_{j,k}D_ry^{\xi,\Delta}(t_{k+j})\Big)^{\top}(A_{1,k})^{\top}\mathrm dr\nn\\
&\quad +\int_{t_k}^{t_{k+1}}A_{1,k}\sigma(y^{\xi,\Delta}_{t_k})\sigma(y^{\xi,\Delta}_{t_k})^{\top}(A_{1,k})^{\top}\mathrm dr,
\end{align}
where we used $D_ry^{\xi,\Delta}(t_k)=0$ for $r> t_{k}$.
Since $\sigma(x)\sigma(x)^{\top}$ is positive definite, we deduce
\begin{align*}
u^{\top}\gamma_{k+1}u&\ge \Delta u^{\top}A_{1,k}\sigma(y^{\xi,\Delta}_{t_k})\sigma(y^{\xi,\Delta}_{t_k})^{\top}(A_{1,k})^{\top}u>0\quad a.s.
\end{align*}
Moreover, utilizing again the invertibility of $\sigma\sigma^{\top},$ we have that $$\gamma_1=\int_0^{t_1}A_{1,0}\sigma(\xi)\sigma(\xi)^{\top}(A_{1,0})^{\top}\mathrm dr$$ is also a.s. invertible. This finishes the proof. 
\end{proof}

\section{Convergence of numerical density function}
In this section, we investigate the convergence of the numerical density function of the $\theta$-EM method on the finite and infinite time horizons, respectively.
We first give the convergence of the numerical density function in $L^1(\RR^d)$ on the finite time horizon for the superlinearly growing drift  coefficient case. Then we show that the longtime convergence rate of the numerical density function is 1 for the linear growing drift coefficient   case.
\subsection{Convergence on the finite time horizon}
\begin{assp}\label{Dbpoly2} 
Assume that coefficients $b$ and $\sigma$ have continuous G\^ateaux  derivatives up to order 2  satisfying 
\begin{align*}
&|\mathcal D^2b(\phi_1)(\phi_2,\phi_3)| \leq K(1+\|\phi_1\|^{(\tilde\beta-1)\vee 0})\|\phi_2\|\|\phi_3\|,\\
&|\mathcal D^2\sigma(\phi_1)(\phi_2,\phi_3)| \leq K\|\phi_2\|\|\phi_3\|,
\end{align*}
where $\phi_1,\phi_2, \phi_3\in\mathcal C^d$, $K>0$, and  $\tilde \beta$ is given in Assumption \ref{Dbpoly1}.
\end{assp}
\begin{thm}\label{den_exist}
Let Assumptions \ref{a1}, \ref{a7}, \ref{repalceA2}--\ref{Dbpoly2} hold. 
Then for any $T>0$,
\begin{align*}
\lim_{\Delta\to0}\sup_{0\leq k\Delta\leq T}\int_{\mathbb R^d}|\mathfrak p^{\Delta}(t_{k},x)-\mathfrak p(t_{k},x)|\mathrm dx=0,
\end{align*}
where $\mathfrak p(t_k,\cdot),~\mathfrak p^{\Delta}(t_k,\cdot)$ are density functions of $x^{\xi}(t_k)$ and $y^{\xi,\Delta}(t_k)$, respectively.
\end{thm}
\begin{proof} The proof is based on the localization argument and is divided into two steps.

\underline{Step 1.} Introduce the smooth cut-off functional $\Theta_{R}:\mathcal C^d\to[0,1]$ with continuous derivatives and compact support. For $R>\|\xi\|$, let $\Theta_R(x)=1$ for $\|x\|\leq R$ and $\Theta_R(x)=0$ for $\|x\|>R+1.$ Then we consider the truncated version of \eqref{FF} on $[0,T]$,
\begin{align}\label{tr_FF}
&\mathrm dx^{\xi,R}(t)=b^R(x^{\xi,R}_t)\mathrm dt+\sigma^R(x^{\xi,R}_t)\mathrm dW(t),\quad t\in(0,T],\\
&x^{\xi,R}(t)=\xi,\quad t\in[-\tau,0],\nn
\end{align}
where $b^R=\Theta_Rb$ and $\sigma^R=\Theta_R\sigma$ are globally Lipschitz continuous for each  $R$. 
Moreover, the coefficients $b^R$ and $\sigma^R$ satisfy
\begin{align*}
|\mathcal Db^{R}(\phi_1)\phi_2|^2\vee|\mathcal D\sigma^{R}(\phi_1)\phi_2|^2\leq K_{R}\Big(|\phi_2(0)|^2\!\!+\!\!\int_{-\tau}^0\!|\phi_2(r)|^2\mathrm d\nu_1(r)\!\!+\!\!\int_{-\tau}^0\!|\phi_2(r)|^2\mathrm d\nu_2(r)\Big),
\end{align*}
where $\phi_1,\phi_2\in\mathcal C^d$. 
In addition,  the numerical solution, the numerical functional solution, and the auxiliary process of the $\theta$-EM method for \eqref{tr_FF}
are denoted by $\{y^{\xi,\Delta,R}(t_{k})\}_{k\geq-N}$, $\{y^{\xi,\Delta,R}_{t_{k}}\}_{k\geq0}$, and $\{z^{\xi,\Delta,R}(t_{k})\}_{k\geq-N}$, respectively. We need to estimate the error between $x^{\xi,R}(t)$ and $y^{\xi,\Delta, R}(t)$ in $\|\cdot\|_{1,2}$.
Similar to the proof of Theorem \ref{th4.1}, using Lemmas \ref{local_Ma} and \ref{multi_nonlip_nu}, we deduce from Assumptions \ref{a1}, \ref{a5},  \ref{repalceA2}, and \ref{Dbpoly1} that 
\begin{align}\label{convT}
\E\Big[\sup_{0\leq t_{k}\leq T}|x^{\xi}(t_k)-y^{\xi,\Delta}(t_k)|^4\Big]\leq K_{T}\Delta^2.
\end{align}
This implies
 \begin{align}\label{convTR}
\E\Big[\sup_{0\leq t_{k}\leq T}|x^{\xi, R}(t_k)-y^{\xi,\Delta, R}(t_k)|^4\Big]\leq K_{T, R}\Delta^2.
\end{align}
Moreover, we have
\begin{align}\label{convTD1}
&\quad \mathbb E[|D_{u}x^{\xi,R}(t)-D_{u}z^{\xi,\Delta, R}(t)|^2]\nn\\
&\leq K(t-u)\E\Big[\int_u^t\big|\mathcal Db^{R}(x^{\xi,R}_s)D_ux^{\xi,R}_s-\mathcal Db^{R}(y^{\xi,\Delta, R}_{\lfloor s \rfloor})D_uy^{\xi,\Delta, R}_{\lfloor s \rfloor}|^2\mathrm ds\Big]\nn\\
&\quad+ K\E\Big[\int_u^t\big|\mathcal D\sigma^{R}(x^{\xi,R}_s)D_ux^{\xi,R}_s-\mathcal D\sigma^{R}(y^{\xi,\Delta, R}_{\lfloor s \rfloor})D_uy^{\xi,\Delta, R}_{\lfloor s \rfloor}|^2\mathrm ds\Big]\nn\\
&\quad +\E\Big[|\sigma^{R}(x^{\xi,R}_u)\mathbf 1_{[0,t]}(u)-\sigma^{R}(y^{\xi,\Delta, R}_{\lfloor u \rfloor})\mathbf 1_{[0,t]}(u)|^2\Big].
\end{align}
It follows from the Taylor formula that
\begin{align*}
&\quad\Big|\mathcal Db^{R}(x^{\xi,R}_s)D_ux^{\xi,R}_s-\mathcal Db^{R}(y^{\xi,\Delta, R}_{\lfloor s \rfloor})D_uy^{\xi,\Delta, R}_{\lfloor s \rfloor}\Big|\nn\\
&\leq \Big|\!\!\int_{0}^{1}\mathcal D^2b^{R}\big(\varsigma x^{\xi,R}_{\lfloor s \rfloor}+(1-\varsigma)y^{\xi,\Delta, R}_s\big)(x^{\xi,R}_s-y^{\xi,\Delta, R}_{\lfloor s \rfloor},D_ux^{\xi,R}_s)\mathrm d\varsigma \Big|\nn\\
&\quad+\Big|\mathcal Db^{R}(y^{\xi,\Delta, R}_{\lfloor s \rfloor})(D_ux^{\xi,R}_s-D_uy^{\xi,\Delta, R}_{\lfloor s \rfloor})\Big|\nn\\
&\leq K_{R}\|x^{\xi,R}_s-y^{\xi,\Delta, R}_{\lfloor s \rfloor}\|\|D_u x^{\xi,R}_s\|+K_{R}\Big(|D_ux^{\xi,R}(s)-D_uy^{\xi,\Delta, R}(\lfloor s \rfloor)|^2\nn\\
&\quad+\int_{-\tau}^0\!|D_ux^{\xi,R}_{s}(r)\!-\!D_uy^{\xi,\Delta, R}_{\lfloor s \rfloor}(r)|^2\mathrm d\nu_2(r)\Big)^{\frac12}.
\end{align*}
where in the last inequality we  used  the boundedness of  operators $\mathcal Db^{R}(\cdot)$ and  $\mathcal D^2b^{R}(\cdot)$ on their compact support $\{x\in\mathcal C^d: \|x\|\leq R+1 \}$.
Similar to the proof of Lemma \ref{l4.4}, we obtain 
\begin{align*}
\sup_{u\leq t\leq T}\E[\|D_uy^{\xi,\Delta, R}_{\lfloor t \rfloor}-D_uz^{\xi,\Delta, R}_{t}\|^2]\leq K_{T,R}\Delta.
\end{align*}
Then
\begin{align}\label{convTD2}
&\quad\E\Big[\big|\mathcal Db^{R}(x^{\xi,R}_s)D_ux^{\xi,R}_s-\mathcal Db^{R}(y^{\xi,\Delta, R}_{\lfloor s \rfloor})D_uy^{\xi,\Delta, R}_{\lfloor s \rfloor}\big|^2\Big]\nn\\
&\leq K_{R}\E\big[\|x^{\xi,R}_s-y^{\xi,\Delta, R}_{\lfloor s \rfloor}\|^2\|D_u x^{\xi,R}_s\|^2\big]\!+\!K_{R}\E\Big[|D_ux^{\xi,R}(s)-D_uz^{\xi,\Delta, R}(s)|^2\nn\\
&\quad+\int_{-\tau}^0\!|D_ux^{\xi,R}_{s}(r)-D_uz^{\xi,\Delta, R}_{s}(r)|^2\mathrm d\nu_2(r)\Big]+K_{T,R}\Delta.
\end{align}
Similarly, we deduce
\begin{align}\label{convTD3}
&\quad\Big|\mathcal D\sigma^{R}(x^{\xi,R}_s)D_ux^{\xi,R}_s-\mathcal D\sigma^{R}(y^{\xi,\Delta, R}_{\lfloor s \rfloor})D_uy^{\xi,\Delta, R}_{\lfloor s \rfloor}\Big|\nn\\
&\leq K_{R}\E\big[\|x^{\xi,R}_s-y^{\xi,\Delta, R}_{\lfloor s \rfloor}\|^2\|D_u x^{\xi,R}_s\|^2\big]+K_{R}\E\Big[|D_ux^{\xi,R}(s)-D_uz^{\xi,\Delta, R}(s)|^2\nn\\
&\quad+\int_{-\tau}^0\!|D_ux^{\xi,R}_{s}(r)-D_uz^{\xi,\Delta, R}_{s}(r)|^2\mathrm d\nu_2(r)\Big]+K_{T,R}\Delta.
\end{align}
Inserting \eqref{convTD2} and \eqref{convTD3} into \eqref{convTD1}, and then using Assumption \ref{a1}, Lemmas \ref{local_Ma} and  \ref{multi_nonlip_nu},  and \eqref{convTR}, we obtain that for $\Delta\in(0,\frac{1}{4\theta a_5})$,
\begin{align*}
&\quad\E\Big[|D_{u}x^{\xi,R}(t)-D_{u}z^{\xi,\Delta, R}(t)|^2\Big]\nn\\
&\leq K_{T,R}\Delta
\!+\!K_{T,R}\int_{u}^{T}\big(\E[\|x^{\xi,R}_{\lfloor s \rfloor}-y^{\xi,\Delta, R}_{\lfloor s \rfloor}\|^4]\big)^{\frac12}\mathrm ds\!+\!K_{T,R}\int_{u}^{T}\big(\E[\|x^{\xi,R}_{\lfloor s \rfloor}\!-\!x^{\xi, R}_{s}\|^4]\big)^{\frac12}\mathrm ds\nn\\
&\quad+K_{T,R}\E\Big[\int_{u}^{T}|D_ux^{\xi,R}(s)-D_uz^{\xi,\Delta, R}(s)|^2\mathrm ds\Big]+K\sup_{0\leq s\leq T}\E[\|x^{\xi,R}_{\lfloor s \rfloor}-x^{\xi, R}_{s}\|^2]\nn\\
&\leq K_{T,R}\Delta+K_{T,R}\int_{u}^{T}\E\big[|D_ux^{\xi,R}(s)-D_u z^{\xi,\Delta, R}(s)|^2\big]\mathrm ds.
\end{align*}
Applying the Gr\"onwall inequality, we arrive at 
\begin{align*}
\E\Big[|D_{u}x^{\xi,R}(t)-D_{u}z^{\xi,\Delta, R}(t)|^2\Big]\leq K_{T,R}\Delta,
\end{align*}
which implies
\begin{align}\label{convDT}
\int_{0}^{T}\E\Big[|D_{u}x^{\xi,R}(t_{k})-D_{u}y^{\xi,\Delta, R}(t_{k})|^2\Big]\mathrm du\leq K_{T,R}\Delta.
\end{align}
It follows from \eqref{convTR} and \eqref{convDT} that
 $$\sup_{0\leq k\Delta\leq T}\|y^{\xi,\Delta,R}(t_{k})-x^{\xi,R}(t_{k})\|_{1,2}\leq K_{T,R}\Delta^{\frac12}.$$
Furthermore,  similar to \cite[Lemma A.5]{hong2023density},  combining Proposition \ref{multi_nonlip}, we obtain that for any $\varrho\in(0,1)$,
$\mathbb E[(\int_{0}^{T}|D_u x^{\xi,R}(t)|^2\mathrm du)^{-\varrho}]\leq K_{T, R, \varrho}.$ 
Similar to the proof of Lemma \ref{local_Ma}, it follows from Assumptions \ref{a1}, \ref{repalceA2}, \ref{Dbpoly1}, and \ref{Dbpoly2} that $x^{\xi}_t\in\mathbb D^{2,4}(\mathcal C^d)$ for all $t\in[0,T].$   
Then, by \cite[Lemma A.1]{hong2023density}, we arrive at 
  \begin{align}\label{TVR}
  &\quad\sup_{0\leq k\Delta\leq T}\int_{\mathbb R^d}|\mathfrak p^{\Delta,R}(t_{k},x)-\mathfrak p^{R}(t_{k},x)|\mathrm dx=\sup_{0\leq k\Delta\leq T}\mathrm d_{TV}(y^{\Delta,\xi,R}(t_k),x^{\xi,R}(t_k))\nn\\
  &\leq  K_{T,R} \sup_{0\leq k\Delta\leq T}\|y^{\Delta,\xi,R}(t_k)-x^{\xi,R}(t_k)\|_{1,2}^{\frac{2\rho}{2\rho+2}}\nn\\
  &\leq  K_{T,R}\Delta^{\frac{\rho}{2\rho+2}},
  \end{align}
where $\mathrm {d}_{TV}(X,Y)$ denotes the total variation distance between two $\mathbb R^d$-valued random variables $X$ and $Y$, and  $\mathfrak p^{R}(t_k,\cdot),~\mathfrak p^{\Delta,R}(t_k,\cdot)$ are density functions of  $x^{\xi,R}(t_k)$ and $y^{\xi,\Delta,R}(t_k)$, respectively.

 \underline{Step 2.} 
Denoting two sequences of events by
\begin{align*}
\Omega_R:=\{\omega\in\Omega:\sup_{t\in[0,T]}|x^{\xi}(t)|\leq R\},\\
\Omega_{R,y}:=\{\omega\in\Omega:\sup_{t_k\in[0,T]}|y^{\xi,\Delta}(t_k)|\leq R\},
\end{align*}
we have $\lim_{R\to\infty}\mathbb P(\Omega_R)=\mathbb P(\Omega)=\lim_{R\to\infty}\mathbb P(\Omega_{R,y})=1.$ 
According to \cite[Section 6]{hong2023density} and \cite[(6.3)]{hong2023density}, 
we derive 
\begin{align*}
&\quad\int_{\mathbb R^d}|\mathfrak p^{\Delta}(t_{k},x)-\mathfrak p(t_{k},x)|\mathrm dx=\mathrm d_{TV}(y^{\Delta,\xi}(t_k),x^{\xi}(t_k))\nn\\
&\leq 4\PP(\Omega_{R-1}^c)+2\PP(\sup_{0\leq t_k\leq T}|x^{\xi}(t_k)-y^{\xi,\Delta}(t_{k})|\geq 1)+\mathrm d_{TV}(y^{\Delta,\xi,R}(t_k),x^{\xi,R}(t_k))\nn\\
&\leq 4\frac{\E[\sup_{0\leq t\leq T}|x^{\xi}(t)|^2]}{(R-1)^2}+2\E [\sup_{0\leq t_k\leq T}|x^{\xi}(t_k)-y^{\xi,\Delta}(t_{k})|^2]
+ K_{T,R}\Delta^{\frac{\rho}{2\rho+2}}.
\end{align*}
Combining   \eqref{convT}, \eqref{TVR}, and Lemma \ref{local_Ma}, and letting $\Delta\rightarrow 0$, $R\rightarrow \infty$, we obtain the desired argument.
\end{proof}

\begin{rem}
It follows from the proof of Theorem \ref{den_exist} that when the coefficients $b$ and $\sigma$ are globally Lipschitz continuous, the convergence rate of the numerical density function is almost  $1/4$ in $L^{1}(\RR^d)$.
\end{rem}

\subsection{Convergence on the infinite time horizon}
In this subsection, we focus on the additive noise case and study the longtime convergence of the numerical density function based on the test-functional-independent weak convergence of the $\theta$-EM method.

\begin{assp}\label{addnoise}
Assume that the noise of \eqref{FF} is of additive type,  i.e., there exists some matrix $\tilde\sigma$ such that for any $\phi\in\mathcal C^d$,
$\sigma(\phi)\equiv\tilde\sigma$.
\end{assp}

\begin{assp}\label{high5} 
Assume that the coefficient  $b$ has continuous and bounded derivatives up to order  $4$ satisfying $$\sup_{\phi\in\mathcal C^d}\|\mathcal D^{k}b(\phi)\|_{\mathcal L\big((\mathcal C^{d})^{\otimes k} ;\mathbb R^{d}\big)}\leq K$$ for $k\in\mathbb N_+$ with $k\leq 4.$ 
In addition, assume that $b$ satisfies the following form
\begin{align}\label{linearb}
|b(\phi)|^2\leq K\big(1+|\phi(0)|^2+\int_{-\tau}^0|\phi(r)|^2\mathrm d\nu_2(r)\big),
\end{align}
where $\phi\in\mathcal C^d.$ 
\end{assp}
The longtime convergence rate of the numerical density function is stated as follows.  

\begin{thm}\label{dens_con}
Let \cref{a2} with $2a_1-
\frac{2\theta-1}{\theta^2}-2a_2>0$,  \cref{ass_nablab} with $L_b:=\sup_{i\in\{1,\ldots, n_b\}}\sup_{\phi\in\mathcal C^d}\|k_b^{i}(\phi)\|<\infty$, 
\cref{addnoise,high5} hold.  In addition, assume that $b$ satisfies 
$$\sup_{i\in\{1, \cdots, n_b\}}\sup_{l\in\{1,2,3\}}\sup_{\phi\in\mathcal C^d} \|D^lk_b^{i}(\phi)\|_{\mathcal L((\mathcal C^d)^{\otimes l}; \mathcal C([-\tau,0]; \RR^{d\times d}))}\leq K.$$
Then there exists a constant $\tilde \Delta\in(0,1]$ such that for any $\Delta\in(0,\tilde \Delta]$,
\begin{align*}
\sup_{T\ge T_0,z\in\mathbb R^d}|\mathfrak p(T,z)-\mathfrak p^{\Delta}(T,z)|\leq K\Delta,
\end{align*}
where $T_0=\frac{\ln(\frac32)}{2L_bn_b(\theta+2)}$.
\end{thm}

\begin{rem}
Theorem \ref{dens_con}, together with the Scheff\'e lemma, indicates that 
$$\sup_{T\ge T_0}\int_{\mathbb R^d}|\mathfrak p(T,z)-\mathfrak p^{\Delta}(T,z)|\mathrm dz\to0 \text{ as } \Delta \to0;$$ see also \cite[Section 3.1]{Sheng_density_SPDE} for details. 
\end{rem}

According to \cite{KAKA08, Sheng_density_SPDE},
we have the following approximation of density functions for
both the exact solution   and the numerical solution: for fixed $T>0$ and $z\in\mathbb R^d,$
\begin{align*}
&\mathfrak p(T,z)=\lim_{n\to\infty}\int_{\mathbb R^d}g_{n^{-1}}(y-z)\mathfrak p(T,y)\mathrm dy=\lim_{n\to\infty}\mathbb E[g_{n^{-1}}(x^{\xi}(T)-z)],\\
&\mathfrak{p}^{\Delta}(T,z)=\lim_{n\to\infty}\int_{\mathbb R^d}g_{n^{-1}}(y-z)\mathfrak p^{\Delta}(T,y)\mathrm dy=\lim_{n\to\infty}\mathbb E[g_{n^{-1}}(y^{\xi,\Delta}(T)-z)],
\end{align*}
where $g_{\zeta}$ denotes the Gaussian density function with mean $0$ and covariance matrix $\zeta\mathrm {Id}_{d\times d}.$
This gives that
\begin{align}\label{dens_error}
|\mathfrak p(T,z)-\mathfrak p^{\Delta}(T,z)|=\lim_{n\to\infty}|\mathbb E[g_{n^{-1}}(x^{\xi}(T)-z)]-\mathbb E[g_{n^{-1}}(y^{\xi,\Delta}(T)-z)]|.
\end{align}

Note that the function $\{g_{n^{-1}}(\cdot-z)\}_{n\ge 1,z\in\mathbb R^d}$ belongs to
\begin{align*}
\mathscr C:=\Big\{f:\mathbb R^d\to\mathbb R\,\Big|&\,f\in\mathcal C_{pol}^{\infty}(\mathbb R^d;\mathbb R),\exists F:\mathbb R^d\to\mathbb R\text{ with }0\leq F\leq 1\text{ such that  }\\
&\,F(x_1,\ldots,x_d)=\int_{-\infty}^{x_1}\cdots\int_{-\infty}^{x_d}f(y_1,\ldots,y_d)\mathrm dy_d\cdots \mathrm dy_1\Big\}.
\end{align*}
To obtain the convergence rate of the numerical density function, it suffices to estimate the  error $|\mathbb E[f(x^{\xi}(T))]-\mathbb E[f(y^{\xi,\Delta}(T))]|$
for any $f\in\mathscr C$.  

We begin with giving some  estimates of the exact solution and numerical solution as well as  their high-order derivatives.
\begin{lemma}
Let \cref{a2,addnoise},  and \eqref{linearb} hold. Then for any $p\in\mathbb N_+$ with $p\geq2$,
\begin{align}\label{Lunix}
\sup_{t\geq0}\E[\|x_t^{\xi}\|^{2p}]\leq K(1+\|\xi\|^{2p}).
\end{align}
In addition, there exists a constant $\tilde \Delta_p\in(0,1]$ such that for any $\Delta\in(0,\tilde \Delta_p]$,
\begin{align}\label{Luniy}
\sup_{k\in \mathbb N}\E[\|y_{t_k}^{\xi,\Delta}\|^{2p}]\leq K(1+\|\xi\|^{2p}).
\end{align}
\end{lemma}
\begin{proof}
The proof of \eqref{Lunix} is similar to that of \cite[Theorem 3]{LMS11} and is omitted.
Below we show the proof of \eqref{Luniy}. 
By virtue of \eqref{ST},  we have that for any $k\in\mathbb N$,
\begin{align*}
&\quad|z^{\xi,\Delta}(t_{k+1})|^{2}
=|y^{\xi,\Delta}(t_{k})+(1-\theta)b(y^{\xi,\Delta}_{t_{k}})\Delta+\tilde\sigma\delta W_k|^2\nn\\
&=|y^{\xi,\Delta}(t_{k})|^2+(1-\theta)^2|b(y^{\xi,\Delta}_{t_{k}})|^2\Delta^2+|\tilde\sigma\delta W_k|^2
+2(1-\theta)\Delta\<y^{\xi,\Delta}(t_{k}),b(y^{\xi,\Delta}_{t_{k}})\>\nn\\
&\quad+2\<y^{\xi,\Delta}(t_{k})+(1-\theta)b(y^{\xi,\Delta}_{t_{k}})\Delta, \tilde\sigma\delta W_k\>.
\end{align*}
Then for any $p\in\mathbb N_+$ with $p\geq2$,
\begin{align}\label{AB1}
\E[|z^{\xi,\Delta}(t_{k+1})|^{2p}]
= \E[|y^{\xi,\Delta}(t_{k})|^{2p}]
+\sum_{i=1}^{p}C_{p}^{i}\E[I_{i}],
\end{align}
where 
\begin{align*}
I_i&:=|y^{\xi,\Delta}(t_{k})|^{2(p-i)}\Big((1-\theta)^2|b(y^{\xi,\Delta}_{t_{k}})|^2\Delta^2+|\tilde\sigma\delta W_k|^2
+2(1-\theta)\Delta\<y^{\xi,\Delta}(t_{k}),b(y^{\xi,\Delta}_{t_{k}})\>\nn\\
&\quad+2\<y^{\xi,\Delta}(t_{k})+(1-\theta)b(y^{\xi,\Delta}_{t_{k}})\Delta, \tilde\sigma\delta W_k\>\Big)^{i}.
\end{align*}
For the term $I_1$, it follows from the property of the conditional expectation, \cref{a2}, and \eqref{linearb} that for  $\varepsilon_1\in(0,1)$,
\begin{align}\label{AB2}
&\quad\E[I_1]\leq \E\Big[|y^{\xi,\Delta}(t_{k})|^{2(p-1)}\Big((1-\theta)^2|b(y^{\xi,\Delta}_{t_{k}})|^2\Delta^2+|\tilde \sigma|^2\Delta+K\Delta\nn\\
&\quad-2a_1(1-\theta)\Delta|y^{\xi,\Delta}(t_{k})|^2+2a_2(1-\theta)\Delta\int_{-\tau}^0|y^{\xi,\Delta}_{t_{k}}(r)|^2\mathrm d\nu_2(r)\Big)\Big]\nn\\
&\leq \E\Big[|y^{\xi,\Delta}(t_{k})|^{2(p-1)}\big(K\Delta+(1-\theta)^2|b(y^{\xi,\Delta}_{t_{k}})|^2\Delta^2\big)\Big]
-2(a_1-\frac{a_2(p-1)}{p})\times\nn\\
&\quad(1-\theta)\Delta\E[|y^{\xi,\Delta}(t_{k})|^{2p}]+\frac{2a_2}{p}(1-\theta)\Delta\E\Big[\int_{-\tau}^0|y^{\xi,\Delta}_{t_{k}}(r)|^{2p}\mathrm d\nu_2(r)\Big]\nn\\
&\leq K(\varepsilon_1)\Delta\!+\!K\Delta^2\Big(\E[|y^{\xi,\Delta}(t_{k})|^{2p}]\!+\!\E\Big[\int_{-\tau}^0|y^{\xi,\Delta}_{t_{k}}(r)|^{2p}\mathrm d\nu_2(r)\Big]\Big)
\!-\!2(a_1\!-\!\frac{a_2(p-1)}{p}\nn\\
&\quad-\varepsilon_1)(1-\theta)\Delta\E[|y^{\xi,\Delta}(t_{k})|^{2p}]+\frac{2a_2}{p}(1-\theta)\Delta\E\Big[\int_{-\tau}^0|y^{\xi,\Delta}_{t_{k}}(r)|^{2p}\mathrm d\nu_2(r)\Big],
\end{align}
where we used the Young inequality.
For the term $I_2$, using \eqref{linearb} and the Young inequality again,  we have
\begin{align}\label{AB3}
&\quad\E[I_2]\leq \E\Big[|y^{\xi,\Delta}(t_{k})|^{2(p-2)}\Big((1\!-\!\theta)^2|b(y^{\xi,\Delta}_{t_{k}})|^2\Delta^2\!+\!|\tilde \sigma|^2\Delta\!+\!2(1-\theta)\Delta\<y^{\xi,\Delta}(t_{k}),\nn\\
&\quad b(y^{\xi,\Delta}_{t_{k}})\>\Big)^2
+|y^{\xi,\Delta}(t_{k})|^{2(p-2)}\Big(2\<y^{\xi,\Delta}(t_{k})+(1-\theta)b(y^{\xi,\Delta}_{t_{k}})\Delta, \tilde\sigma\delta W_k\>\Big)^2\Big]\nn\\
&\leq K\Delta^2\Big(1+\E[|y^{\xi,\Delta}(t_{k})|^{2p}]+\E\Big[\int_{-\tau}^0|y^{\xi,\Delta}_{t_{k}}(r)|^{2p}\mathrm d\nu_2(r)\Big]\Big)\nn\\
&\quad+K\Big(\Delta\E[|y^{\xi,\Delta}(t_{k})|^{2(p-1)}]+\Delta^2\E\big[|y^{\xi,\Delta}(t_{k})|^{2(p-2)}|b(y^{\xi,\Delta}_{t_{k}})|^2\big]\Big)\nn\\
&\leq K(\varepsilon_1)\Delta+K\Delta^2\Big(\E[|y^{\xi,\Delta}(t_{k})|^{2p}]+\E\Big[\int_{-\tau}^0|y^{\xi,\Delta}_{t_{k}}(r)|^{2p}\mathrm d\nu_2(r)\Big]\Big)\nn\\
&\quad+\varepsilon_1\Delta\E[|y^{\xi,\Delta}(t_{k})|^{2p}].
\end{align}
Similarly, we derive 
\begin{align}\label{AB4}
\E[I_i]&\leq K(\varepsilon_1)\Delta+K\Delta^2\Big(\E[|y^{\xi,\Delta}(t_{k})|^{2p}]+\E\Big[\int_{-\tau}^0|y^{\xi,\Delta}_{t_{k}}(r)|^{2p}\mathrm d\nu_2(r)\Big]\Big)\nn\\
&\quad+\varepsilon_1\Delta\E[|y^{\xi,\Delta}(t_{k})|^{2p}],
\quad i\in\mathbb N ~\mbox{with}~3\leq i\leq p.
\end{align}
Inserting \eqref{AB2}--\eqref{AB4} into \eqref{AB1}, one has
\begin{align*}
&\quad\E[|z^{\xi,\Delta}(t_{k+1})|^{2p}]\leq \E[|y^{\xi,\Delta}(t_{k})|^{2p}]
+K(\varepsilon_1)\Delta
-\Big(2p(1-\theta)\big(a_1-\frac{a_2(p-1)}{p}\big)\nn\\
&\quad-K\Delta-K\varepsilon_1\Big)\Delta\E[|y^{\xi,\Delta}(t_{k})|^{2p}]\nn\\
&\quad+\Big(2a_2(1-\theta)+K\Delta\Big)\Delta\E\Big[\int_{-\tau}^0|y^{\xi,\Delta}_{t_{k}}(r)|^{2p}\mathrm d\nu_2(r)\Big].
\end{align*}
Then for $\epsilon\in(0,1)$,
\begin{align}\label{AB7}
&\quad e^{\epsilon t_{k+1}}\E[|z^{\xi,\Delta}(t_{k+1})|^{2p}]=\sum_{i=0}^{k}\Big(e^{\epsilon t_{i+1}}\E[|z^{\xi,\Delta}(t_{i+1})|^{2p}]-e^{\epsilon t_{i}}\E[|z^{\xi,\Delta}(t_{i})|^{2p}]\Big)\nn\\
&\leq \sum_{i=0}^{k}e^{\epsilon t_{i+1}}\E[|y^{\xi,\Delta}(t_{i})|^{2p}]
-\sum_{i=0}^{k}e^{\epsilon t_{i}}\E[|z^{\xi,\Delta}(t_{i})|^{2p}]
+K(\varepsilon_1)\Delta\sum_{i=0}^{k}e^{\epsilon t_{i+1}}\nn\\
&\quad-\Big(2p(1\!-\!\theta)\big(a_1\!-\!\frac{a_2(p-1)}{p}\big)-K\Delta-K\varepsilon_1\Big)\Delta\sum_{i=0}^{k}e^{\epsilon t_{i+1}}\E[|y^{\xi,\Delta}(t_{i})|^{2p}]\nn\\
&\quad+\Big(2a_2(1-\theta)+K\Delta\Big)\Delta\sum_{i=0}^{k}e^{\epsilon t_{i+1}}\E\Big[\int_{-\tau}^0|y^{\xi,\Delta}_{t_{i}}(r)|^{2p}\mathrm d\nu_2(r)\Big]\nn\\
&\leq \sum_{i=0}^{k}e^{\epsilon t_{i+1}}\E[|y^{\xi,\Delta}(t_{i})|^{2p}]
-\sum_{i=0}^{k}e^{\epsilon t_{i}}\E[|z^{\xi,\Delta}(t_{i})|^{2p}]
+K(\varepsilon_1, \epsilon)e^{\epsilon t_{k+1}}+K\|\xi\|^{2p}\nn\\
&\quad-\Big(2p(1\!-\!\theta)\big(a_1\!-\!a_2e^{\epsilon\tau}\big)-K\Delta-K\varepsilon_1\Big)\Delta\sum_{i=0}^{k}e^{\epsilon t_{i+1}}\E[|y^{\xi,\Delta}(t_{i})|^{2p}].
\end{align}
According to $z^{\xi,\Delta}(t_{k})=y^{\xi,\Delta}(t_{k})-\theta\Delta b(y^{\xi,\Delta}_{t_{k}})$ and \cref{a2}, we obtain
\begin{align*}
|z^{\xi,\Delta}(t_{k})|^{2}\geq -K\Delta+(1+2a_1\theta\Delta)|y^{\xi,\Delta}(t_{k})|^2-2a_2\theta\Delta\int_{-\tau}^0|y^{\xi,\Delta}{t_{k}}(r)|^{2}\mathrm d\nu_2(r).
\end{align*}
This implies
\begin{align}\label{AB6}
&\quad(1+2a_1\theta\Delta)^{p}\sum_{i=0}^ke^{\epsilon t_i}|y^{\xi,\Delta}(t_{i})|^{2p}\nn\\
&\leq (1\!+\!\varepsilon_1)\sum_{i=0}^ke^{\epsilon t_i}\Big(|z^{\xi,\Delta}(t_{i})|^{2}\!+\!2a_2\theta\Delta\int_{-\tau}^0|y^{\xi,\Delta}_{t_{i}}(r)|^{2}\mathrm d\nu_2(r)\Big)^{p}\!+\!K(\varepsilon_1)\Delta^{p}\sum_{i=0}^k e^{\epsilon t_i}\nn\\
&\leq(1+\varepsilon_1)\sum_{i=0}^k\sum_{l=0}^{p}
\frac{C_{p}^{l}(2a_2\theta\Delta)^{l}(p-l)}{p}e^{\epsilon t_i}|z^{\xi,\Delta}(t_{i})|^{2p}\nn\\
&\quad+
(1+\varepsilon_1)\sum_{i=0}^k\sum_{l=0}^{p}
\frac{C_{p}^{l}(2a_2\theta\Delta)^{l}l}{p}e^{\epsilon t_i}\int_{-\tau}^0|y^{\xi,\Delta}_{t_{i}}(r)|^{2p}\mathrm d\nu_2(r)+K(\varepsilon_1)\Delta e^{\epsilon t_{k}}\nn\\
&\leq(1+\varepsilon_1)\sum_{i=0}^k\sum_{l=0}^{p}
\frac{C_{p}^{l}(2a_2\theta\Delta)^{l}(p-l)}{p}e^{\epsilon t_i}|z^{\xi,\Delta}(t_{i})|^{2p}\nn\\
&\quad+
(1+\varepsilon_1)e^{\epsilon\tau}\sum_{i=0}^k\sum_{l=0}^{p}
\frac{C_{p}^{l}(2a_2\theta\Delta)^{l}l}{p}e^{\epsilon t_i}|y^{\xi,\Delta}(t_i)|^{2p}+K\|\xi\|^{2p}\Delta+K(\varepsilon_1)\Delta e^{\epsilon t_k}.
\end{align}
Owing to $a_1>a_2$, choose sufficiently small numbers $\varepsilon_1, \epsilon>0$ such that
$a_1>a_2(1+\varepsilon_1)e^{\epsilon \tau}.$
Combining \eqref{AB6}, we arrive at
\begin{align*}
&\quad\sum_{i=0}^k\sum_{l=0}^{p}
\frac{C_{p}^{l}(2a_1\theta\Delta)^{l}(p-l)}{p}e^{\epsilon t_i}|y^{\xi,\Delta}(t_{i})|^{2p}\nn\\
&\leq(1+\varepsilon_1)\sum_{i=0}^k\sum_{l=0}^{p}
\frac{C_{p}^{l}(2a_2\theta\Delta)^{l}(p-l)}{p}e^{\epsilon t_i}|z^{\xi,\Delta}(t_{i})|^{2p}\nn\\
&+K\|\xi\|^{2p}+K(\varepsilon_1)\Delta e^{\epsilon t_k}.
\end{align*}
Plugging above inequality into \eqref{AB7}, we deduce
\begin{align*}
&\quad e^{\epsilon t_{k+1}}\E[|z^{\xi,\Delta}(t_{k+1})|^{2p}]
\leq \Big(e^{\epsilon \Delta}-\frac{\sum_{l=0}^{p}C_{p}^{l}(2a_1\theta\Delta)^{l}(p-l)}{\sum_{l=0}^{p}
C_{p}^{l}(2a_2(1+\varepsilon_1)\theta\Delta)^{l}(p-l)}\Big)\times\nn\\
&\quad\sum_{i=0}^ke^{\epsilon t_{i}}\E[|y^{\xi,\Delta}(t_{i})|^{2p}]+K(\varepsilon_1, \epsilon)e^{\epsilon t_{k+1}}+K\|\xi\|^{2p}-\Big(2p(1\!-\!\theta)\big(a_1\!-\!a_2e^{\epsilon\tau}\big)\nn\\
&\quad-K\Delta-K\varepsilon_1\Big)\Delta\sum_{i=0}^{k}e^{\epsilon t_{i+1}}\E|y^{\xi,\Delta}(t_{i})|^{2p}].
\end{align*}
Since
\begin{align*}
&\quad\frac{\sum_{l=0}^{p}C_{p}^{l}(2a_1\theta\Delta)^{l}(p-l)}{\sum_{l=0}^{p}
C_{p}^{l}(2a_2(1+\varepsilon_1)\theta\Delta)^{l}(p-l)}\nn\\
&\geq 1+ \frac{\sum_{l=0}^{p}C_{p}^{l}(2(a_1-a_2(1+\varepsilon_1))\theta\Delta)^{l}(p-l)}{\sum_{l=0}^{p}
C_{p}^{l}(2a_2(1+\varepsilon_1)\theta\Delta)^{l}(p-l)}\nn\\
&\geq1+ \frac{C_{p}^{0}p}{\sum_{l=0}^{p}
C_{p}^{l}(4a_2\theta)^{l}p}=1+ \frac{1}{(4a_2\theta+1)^{p}},
\end{align*}
it is straightforward to see that
\begin{align*}
&\quad e^{\epsilon t_{k+1}}\E[|z^{\xi,\Delta}(t_{k+1})|^{2p}]
\leq \Big(e^{\epsilon \Delta}-1-\frac{1}{(4a_2\theta+1)^{p}}\Big)\sum_{i=0}^ke^{\epsilon t_{i}}\E[|y^{\xi,\Delta}(t_{i})|^{2p}]\nn\\
&\quad+K(\varepsilon_1, \epsilon)e^{\epsilon t_{k+1}}+K\|\xi\|^{2p}-\Big(2p(1-\theta)\big(a_1-a_2e^{\epsilon\tau}\big)-K\Delta-K\varepsilon_1\Big)\Delta\times\nn\\
&\quad\sum_{i=0}^{k}e^{\epsilon t_{i+1}}\E[|y^{\xi,\Delta}(t_{i})|^{2p}].
\end{align*}
There exists a constant $\tilde\Delta_p\in(0,1]$ such that 
$2p(1-\theta)(a_1-a_2)-K\tilde\Delta_p>0.$
Let $\varepsilon_1,\epsilon >0$ be sufficiently small such that
\begin{align*}
&e^{\epsilon \Delta}-1-\frac{1}{(4a_2\theta+1)^{p}}<0,\\
&2p(1-\theta)\big(a_1-a_2e^{\epsilon\tau}\big)-K\tilde\Delta_p-K\varepsilon_1>0.
\end{align*}
Thus for any $\Delta\in(0,\tilde \Delta_p]$,
\begin{align*}
e^{\epsilon t_{k+1}}\E[|z^{\xi,\Delta}(t_{k+1})|^{2p}]
\leq K(\varepsilon_1, \epsilon)e^{\epsilon t_{k+1}}+K\|\xi\|^{2p},
\end{align*}
which yields 
$\sup_{k\geq-N}\E[|z^{\xi,\Delta}(t_{k})|^{2p}]\leq K(1+\|\xi\|^{2p})$.
This, along with $z^{\xi,\Delta}(t_{k})=y^{\xi,\Delta}(t_{k})-\theta\Delta b(y^{\xi,\Delta}_{t_{k}})$, implies that
$\sup_{k\geq-N}\E[|y^{\xi,\Delta}(t_{k})|^{2p}]\leq K(1+\|\xi\|^{2p}).$
Furthermore, similar to \underline{Step 2} in the proof of \cref{p2.2}, we conclude that
$\sup_{k\in\mathbb N}\E[\|y^{\xi,\Delta}_{t_{k}}\|^{2p}]\leq K(1+\|\xi\|^{2p}).$
The proof is completed.
\end{proof}

\begin{lemma}\label{Ch4highDx}
Under \cref{a2,addnoise,high5}, there exists a constant $\hat\lambda_3>0$ such that for any $p\in\mathbb N_+$ with 
$p\geq2$,
\begin{align*}
&\quad\E[\|\mathcal D x^{\tilde\xi}_t\cdot\eta\|^{2p}]\leq K e^{-\hat\lambda_3 pt}\|\eta\|^{2p},\; t>0,\\
&\quad \E[\| D_{r_1} x^{\tilde\xi}_t\|^{2p}]\leq K e^{-\hat\lambda_3 p(t-r_1)}(1+\E[\|D_{r_1}\tilde \xi\|^{2p}]),\;t>r_1,\\
&\quad \E[\|D_{r_1}\mathcal D x^{\tilde\xi}_t\cdot\eta\|^{2p}]\leq K e^{-\hat\lambda_3 pt}\|\eta\|^{2p}\big(1+\E[\|D_{r_1}\tilde \xi\|^{4p}]\big),\;t>0,\\
&\quad \E[\|D_{r_1, r_2}x^{\tilde\xi}_t\|^{2p}]\leq
K e^{-\hat\lambda_3 p(t-r_1\vee r_2)}\big(1\!+\!\E[\|D_{r_1, r_2}\tilde\xi\|^{2p}]\!+\!\E[\|D_{r_1}\tilde \xi\|^{4p}]\!\\
&\quad \quad\quad\quad\quad\quad\quad\quad+\!\E[\|D_{r_2}\tilde \xi\|^{4p}]\big),\; t>r_1\vee r_2,\\
&\quad \E[\|D_{r_1, r_2}\mathcal D x^{\tilde\xi}_t\cdot\eta\|^{2p}]\leq
 K e^{-\hat\lambda_3 pt}\|\eta\|^{2p}\big(1+
 \E[\|D_{r_1,r_2}\tilde \xi\|^{4p}]+\E[\|D_{r_1}\tilde \xi\|^{8p}]\nn\\
 &\quad\quad\quad\quad\quad\quad\quad\quad\quad\quad+\E[\|D_{r_2}\tilde \xi\|^{8p}]\big),\;t>0,,
\end{align*}
where $\tilde \xi\in \mathbb D^{2,8p}(\mathcal C^d)$.
\end{lemma}
\begin{proof}
The proofs are similar to those of \cref{lem_fir_d,lem_fir_d2,lem_sec_d1,Ch3D2x}, and hence  are omitted.
\end{proof}
\begin{lemma}\label{Ch4highDy}
Let Assumption \ref{a2} with $2a_1-
\frac{2\theta-1}{\theta^2}-2a_2>0$, Assumptions \ref{addnoise} and \ref{high5} hold. Then there exists a constant $\hat\lambda_4>0$ such that for any $p\in\mathbb N_+$ and $\Delta\in(0,\tilde \Delta_p]$, 
\begin{align*}
&\E[\|\mathcal D y^{\tilde\xi,\Delta}_{t_k}\cdot\eta\|^{2p}]\leq Ke^{-\hat\lambda_4p t_k}\|\eta\|^{2p},\;t_k>0,\\
&\E[\|D_{r_1} y^{\tilde\xi,\Delta}_{t_k}\|^{2p}]\leq Ke^{-\hat\lambda_4p (t_k-r_1)}(1+\E[\|D_{r_1}\tilde \xi\|^{2p}]),\;t_k>r_1,\\
&\E[\|D_{r_1}\mathcal D y^{\tilde\xi,\Delta}_{t_k}\cdot\eta\|^{2p}]\leq Ke^{-\hat\lambda_4p t_k}\|\eta\|^{2p}(1+\E[\|D_{r_1}\tilde \xi\|^{4p}]),\; t_k>0,\\
&\E[\|D_{r_1, r_2}y^{\tilde\xi,\Delta}_{t_k}\|^{2p}]\leq Ke^{-\hat\lambda_4p (t-r_1\vee r_2)}\big(1\!+\!\E[\|D_{r_1, r_2}\tilde\xi\|^{2p}]\!+\!\E[\|D_{r_1}\tilde \xi\|^{4p}]\!\\
&\quad \quad\quad\quad\quad\quad\quad\quad\quad+\!\E[\|D_{r_2}\tilde \xi\|^{4p}]\big),\;t_k>r_1\vee r_2,\\
&\E[\|D_{r_1, r_2}\mathcal D y^{\tilde\xi,\Delta}_{t_k}\cdot\eta\|^{2p}]\leq Ke^{-\hat\lambda_4p t_k}\|\eta\|^{2p}\big(1+
 \E[\|D_{r_1,r_2}\tilde \xi\|^{4p}]+\E[\|D_{r_1}\tilde \xi\|^{8p}]\nn\\
 &\quad\quad\quad\quad\quad\quad\quad\quad\quad\quad+\E[\|D_{r_2}\tilde \xi\|^{8p}]\big),\; t_k>0,\\
&\E[\|D_{r_1, r_2,r_3}y^{\tilde\xi,\Delta}_{t_k}\|^{2p}]\leq Ke^{-\hat\lambda_4p (t_k-r_1\vee r_2\vee r_3)}\big(1+\E[\|D_{r_1,r_2,r_3}\tilde \xi\|^{2p}]+\E[\|D_{r_1,r_2}\tilde \xi\|^{4p}]\nn\\
&\quad\quad+\E[\|D_{r_1,r_3}\tilde \xi\|^{4p}]+\E[\|D_{r_2,r_3}\tilde \xi\|^{4p}]+\E[\|D_{r_1}\tilde \xi\|^{8p}]+\E[\|D_{r_2}\tilde \xi\|^{8p}]\\
&\quad \quad +\E[\|D_{r_3}\tilde \xi\|^{8p}]\big),\;t_k>r_1\vee r_2\vee r_3,\\
&\E[\|D_{r_1, r_2, r_3}\mathcal D y^{\tilde\xi,\Delta}_{t_k}\cdot\eta\|^{2p}]\leq Ke^{-\hat\lambda_4p t_k}\|\eta\|^{2p}\Big(1+\E[\|D_{r_1,r_2,r_3}\tilde \xi\|^{4p}]\nn\\
&\quad\quad+\sum_{\substack{i_1,i_2\in\{1,2,3\}\\i_1<i_2}}
 \E[\|D_{r_{i_1},r_{i_2}}\tilde \xi\|^{8p}]+\sum_{i\in\{1,\ldots, 3\}}\E[\|D_{r_{i}}\tilde \xi\|^{16p}]\Big),\;t_k>0,\\
&\E[\|D_{r_1 r_2,r_3,r_4}y^{\tilde\xi,\Delta}_{t_k}\|^{2p}]\leq Ke^{-\hat\lambda_4p (t_k-r_1\vee r_2\vee r_3\vee r_4)}\Big(1+\E[\|D_{r_1,r_2,r_3,r_4}\tilde \xi\|^{2p}]\nn\\
&\quad\quad+\sum_{\substack{i_1,i_2,i_3
\in
\{1,\ldots,4\}\\ i_1<i_2<i_3}}\E[\|D_{r_{i_1},r_{i_2},r_{i_3}}\tilde \xi\|^{4p}]+
\sum_{\substack{i_1,i_2\in\{1,\ldots,4\}\\i_1<i_2}}\E[\|D_{r_{i_1},r_{i_2}}\tilde \xi\|^{8p}]\nn\\
&\quad\quad+\sum_{i\in\{1,\ldots,4\}}\E[\|D_{r_{i}}\tilde \xi\|^{16p}]\Big),\; t_k>r_1\vee\cdots\vee r_4,
\end{align*}
where $\tilde \xi\in \mathbb D^{4,p''}(\mathcal C^d)$ with some $p''>0$.
\end{lemma}
\begin{proof}
We only show the estimate of $\mathcal D y^{\tilde \xi,\Delta}_{t_{k}}\cdot\eta$, since estimates of other terms  are similar.
Similar to the proof of Lemma \ref{Dy_esti}, by Assumptions \ref{a2} and \ref{addnoise}, we derive 
\begin{align*}
&\quad|\mathcal D z^{\tilde \xi,\Delta}(t_{k+1})\cdot\eta|^2=|\mathcal D z^{\tilde \xi,\Delta}(t_{k})\cdot\eta|^2+|\mathcal D b(y^{\tilde \xi,\Delta}_{t_k})\mathcal D y^{\tilde \xi,\Delta}_{t_k}\cdot\eta|^2\Delta^2\nn\\
&\quad+2\<\mathcal D y^{\tilde \xi,\Delta}(t_{k})\cdot\eta-\theta\Delta\mathcal D b(y^{\tilde \xi,\Delta}_{t_k})\mathcal D y^{\tilde \xi,\Delta}_{t_k}\cdot\eta,\mathcal D b(y^{\tilde \xi,\Delta}_{t_k})\mathcal D y^{\tilde \xi,\Delta}_{t_k}\cdot\eta\>\Delta\nn\\
&\leq A_{\theta,\Delta}|\mathcal D z^{\tilde \xi,\Delta}(t_{k})\cdot\eta|^2
-\big(2a_1-
\frac{2\theta-1}{\theta^2}\big)\Delta|\mathcal D y^{\tilde \xi,\Delta}(t_{k})\cdot\eta|^2\nn\\
&\quad+2a_2\Delta\int_{-\tau}^{0}|\mathcal D y^{\tilde \xi,\Delta}(t_{k}+r)\cdot\eta|^2\mathrm d\nu_2(r),
\end{align*}
where $A_{\theta,\Delta}=\frac{(1-\theta)^2}{\theta^2}+\frac{2\theta-1}{\theta^2(1+\Delta)}$.
For $\hat\lambda_4>0$, we arrive at
\begin{align*}
&\quad e ^{\hat\lambda_4 t_{k+1}}|\mathcal D z^{\tilde \xi,\Delta}(t_{k+1})|^2
\leq \|\eta \|^2+(A_{\theta,\Delta}e^{\hat\lambda_4\Delta}-1)\sum_{i=0}^{k}e ^{\hat\lambda_4 t_{i}}|\mathcal D z^{\tilde \xi,\Delta}(t_{k})\cdot\eta|^2\nn\\
&\quad-\big(2a_1-
\frac{2\theta-1}{\theta^2}\big)\Delta\sum_{i=0}^{k}e ^{\hat\lambda_4 t_{i+1}}|\mathcal D y^{\tilde \xi,\Delta}(t_{i})\cdot\eta|^2\nn\\
&\quad+2a_2\Delta\sum_{i=0}^{k}e ^{\hat\lambda_4 t_{i+1}}\int_{-\tau}^{0}|\mathcal D y^{\tilde \xi,\Delta}(t_{i}+r)\cdot\eta|^2\mathrm d\nu_2(r)\nn\\
&\leq K\|\eta\|^2+(A_{\theta,\Delta}e^{\hat\lambda_4\Delta}-1)\sum_{i=0}^{k}e ^{\hat\lambda_4 t_{i}}|\mathcal D z^{\tilde \xi,\Delta}(t_{k})\cdot\eta|^2\nn\\
&\quad-\big(2a_1-
\frac{2\theta-1}{\theta^2}-2a_2e^{\hat\lambda_4\tau}\big)\Delta\sum_{i=0}^{k}e ^{\hat\lambda_4 t_{i+1}}|\mathcal D y^{\tilde \xi,\Delta}(t_{i})\cdot\eta|^2.
\end{align*}
Note that $A_{\theta,\Delta}=1-\frac{(2\theta-1)\Delta}{\theta^2(1+\Delta)}\leq e^{-\frac{(2\theta-1)\Delta}{2\theta^2}}$. Choose a sufficiently small number $\hat\lambda_4>0$ such that
$A_{\theta,\Delta}e^{\hat\lambda_4\Delta}<1$ and $2a_1-
\frac{2\theta-1}{\theta^2}-2a_2e^{\hat\lambda_4\tau}>0$. Then
\begin{align}\label{LDy+1}
e^{\hat\lambda_4 t_{k+1}}|\mathcal D z^{\tilde \xi,\Delta}(t_{k+1})\cdot\eta|^2\leq K\|\eta\|^2, \quad \mbox{a.s}.
\end{align}
According to $\mathcal D z^{\tilde \xi,\Delta}(t_{k})\cdot\eta=\mathcal D y^{\tilde \xi,\Delta}(t_{k})\cdot\eta-\theta\Delta\mathcal D b(y^{\tilde \xi,\Delta}_{t_k})\mathcal D y^{\tilde \xi,\Delta}_{t_k}\cdot\eta$ and Assumption \ref{a2}, we obtain
\begin{align*}
&\quad (1+2a_1\theta\Delta)\sup_{0\leq i\leq k}e^{\hat\lambda_4 t_{i}}|\mathcal D y^{\tilde \xi,\Delta}(t_{i})\cdot\eta|^2\nn\\
&\leq\sup_{0\leq i\leq k}e^{\hat\lambda_4 t_{i}}|\mathcal D z^{\tilde \xi,\Delta}(t_{i})\cdot\eta|^2
+2a_2\theta\Delta\sup_{0\leq i\leq k}e^{\hat\lambda_4 t_{i}}\int_{-\tau}^{0}|\mathcal D y^{\tilde \xi,\Delta}_{t_{i}}(r)|^2\mathrm d\nu_2(r)\nn\\
&\leq K\|\eta\|^2+ \sup_{0\leq i\leq k}e^{\hat\lambda_4 t_{i}}|\mathcal D z^{\tilde \xi,\Delta}(t_{i})\cdot\eta|^2+2a_2\theta e^{\hat\lambda_4 \tau}\Delta\sup_{0\leq i\leq k}e^{\hat\lambda_4 t_{i}}|\mathcal D y^{\tilde \xi,\Delta}(t_{i})\cdot\eta|^2.
\end{align*}
This, along with $a_2 e^{\hat \lambda_4\tau}<a_1$ and \eqref{LDy+1}, implies that
$$\sup_{0\leq i\leq k}e^{\hat\lambda_4 t_{i}}|\mathcal D y^{\tilde \xi,\Delta}(t_{i})\cdot\eta|^2\leq K\|\eta\|^2, \quad \mbox{a.s}.$$
The desired argument follows by taking expectations.

\end{proof}

\begin{lemma}\label{lower_matrix}
Let Assumption \ref{ass_nablab}  with $L_b=\sup_{i\in\{1,\ldots, n_b\}}\sup_{\phi\in\mathcal C^d}\|k_b^{i}(\phi)\|<\infty$ and 
 Assumption \ref{addnoise} hold.   Then for any $u\in\mathbb R^d$ with $|u|=1$ and $\Delta\in(0,\frac{1}{2L_bn_b\theta}]$,
\begin{align}\label{lower_1}
|u^{\top}\gamma_{\Phi(T;t_j,Y^{\varsigma}_j)}u|\ge \frac{1}{4}\tilde \sigma(T\wedge T_0),
\end{align}
where $T_0=\frac{\ln(\frac32)}{2L_bn_b(\theta+2)}$ is given in Theorem \ref{dens_con}, 
 and $Y^{\varsigma}_j$ is given in \eqref{Y_defini}.
In particular, we have $(\gamma_{\Phi(T;t_i,Y^{\varsigma}_i)})^{-1}\in L^{\infty-}(\Omega).$ 
\end{lemma}
\begin{proof}  We first prove \eqref{lower_1} for the case of  $d=1$. 
For $r\in(t_k,t_{k+1}]$ with $j\leq k\leq i-1$, we have
\begin{align*}
D_r\Phi(t_i;t_j,Y^{\varsigma}_j)&=\sum_{n=k}^{i-1}\Big[(1-\theta)\mathcal Db(\Phi_{t_n}(t_j,Y^{\varsigma}_j))D_r\Phi_{t_n}(t_j,Y^{\varsigma}_j)\Delta\nn\\
&\quad+\theta \mathcal Db(\Phi_{t_{n+1}}(t_j,Y^{\varsigma}_j))D_r\Phi_{t_{n+1}}(t_j,Y^{\varsigma}_j)\Delta\Big]+\tilde \sigma.
\end{align*}
This, together with $D_r\Phi_{t_n}(t_j,Y^{\varsigma}_j)=\sum_{l=-N}^0I^lD_r\Phi(t_{n+l};t_j,Y^{\varsigma}_j)$ and Assumption \ref{ass_nablab}, implies that
\begin{align}\label{rela1}
&\quad D_r\Phi(t_i;t_j,Y^{\varsigma}_j)\nn\\
&=\sum_{n=k}^{i-1}\Big[(1-\theta)\sum_{\ell=1}^{n_b}\sum_{l=-N}^0\int_{-\tau}^0k_b^{\ell }(\Phi_{t_n}(t_j,Y^{\varsigma}_j)(s))I^l(s)\mathrm d\nu_3^{\ell}(s)D_r\Phi(t_{n+l};t_j,Y^{\varsigma}_j)\Delta\nn\\
&\quad+\theta \sum_{\ell=1}^{n_b}\sum_{l=-N}^0\int_{-\tau}^0k_b^{\ell}(\Phi_{t_{n+1}}(t_j,Y^{\varsigma}_j)(s))I^l(s)\mathrm d\nu_3^{\ell}(s)D_r\Phi(t_{n+l+1};t_j,Y^{\varsigma}_j)\Delta \Big]+\tilde \sigma.
\end{align}

\underline{Step 1.} To derive a lower bound of $D_r\Phi(t_i;t_j,Y^{\varsigma}_j)$, we need to present  a discrete comparison principle. Define a two-parameter nonnegative sequence $\{A_i^k\}_{0\leq k,i\leq N^{\Delta}}$ as follows: when $i\leq k$, define  $A^{k}_{i}=0$; when $0\leq k\leq i-1,$ define 
\begin{align*}
A^k_i=L_b\Delta \sum_{n=k}^{i-1}\sum_{\ell=1}^{n_b}\sum_{l=-N}^0\nu_3^{\ell}(\Delta_{l-1}+\Delta_{l})\big((1-\theta)A^k_{n+l}+\theta A^k_{n+l+1}\big)+\tilde \sigma,
\end{align*}
where we let $\nu_3^{\ell}(\Delta_{-N-1})=\nu_3^{\ell}(\Delta_0)=0$ and $L_b=\sup_{\ell\in\{1,\ldots, n_b\}}\sup_{\phi\in\mathcal C^d}|k_b^{\ell}(\phi)|$. 
It follows from the  definition that when $i_1-k_1=i_2-k_2>0,$ $A^{k_1}_{i_1}=A^{k_2}_{i_2}=:\mathcal A_{i_1-k_1}$. 
Then
\begin{align*}
\mathcal A_{i-k}&=A^k_i=L_b\Delta \sum_{\ell=1}^{n_b}\sum_{l=-N}^0\nu_3^{\ell}(\Delta_{l-1}+\Delta_l)\big((1-\theta)A^k_{i-1+l}+\theta A^k_{i+l}\big)\\
&\quad+L_b\Delta\sum_{n=k}^{i-2}\sum_{\ell=1}^{n_b}\sum_{l=-N}^0\nu_3^{\ell}(\Delta_{l-1}\!+\!\Delta_l)\big((1\theta)A^k_{n+l}+\theta A^k_{n+1+l}\big)+\tilde \sigma\\
&=L_b\Delta \sum_{\ell=1}^{n_b}\sum_{l=-N}^0\nu_3^{\ell}(\Delta_{l-1}+\Delta_l)\big((1-\theta)A^k_{i-1+l}+\theta A^k_{i+l}\big)+\mathcal A_{i-k-1}.
\end{align*}
This gives that for any $\Delta\in(0,\frac{1}{L_b\theta n_b})$,
\begin{align*}
\mathcal A_{i-k}
&=\big(1\!-\!L_b\Delta \theta\sum_{\ell=1}^{n_b}\nu_3^{\ell}(\Delta_{-1})\big)^{-1}\Big\{\mathcal A_{i-k-1}\!+\!L_b\Delta \Big[(1\!-\!\theta)\sum_{\ell=1}^{n_b}\sum_{l=-N}^0\nu_3^{\ell}(\Delta_{l-1}\!+\!\Delta_l)\times\nn\\
&\quad\mathcal A_{i-1-k+l}+\theta\sum_{\ell=1}^{n_b}\sum_{l=-N+1}^0\nu_3^{\ell}(\Delta_{l-2}+\Delta_{l-1})\mathcal A_{i-k+l-1}\Big]\Big\}.
\end{align*}
By the relation $(1-L_b\Delta\theta n_b)^{-1}=1+\frac{L_b\theta\Delta n_b}{1-L_b\Delta\theta n_b}$ and by iteration, we obtain
\begin{align*}
\mathcal A_{i-k}
&\leq(1+\frac{L_b\theta\Delta n_b}{1-L_b\Delta\theta n_b})\mathcal A_{i-k-1}+(1-L_b\Delta \theta n_b)^{-1}L_b\Delta \Big[(1-\theta)\sum_{\ell=1}^{n_b}\sum_{l=-N}^0\nn\\
&\quad\nu_3^{\ell}(\Delta_{l-1}+\Delta_l)\times\mathcal A_{i-1-k+l}+\theta\sum_{\ell=1}^{n_b}\sum_{l=-N+1}^0\nu_3^{\ell}(\Delta_{l-2}+\Delta_{l-1})\mathcal A_{i-k+l-1}\Big]\nn\\
&\leq \mathcal A_1+\sum_{n=1}^{i-k-1}\frac{L_b\theta\Delta n_b}{1-L_b\Delta\theta n_b}\mathcal A_n+(1-L_b\Delta\theta n_b)^{-1}L_b\Delta\sum_{\ell=1}^{n_b}\sum_{n=1}^{i-k-1}\Big[(1-\theta)\times\nn\\
&\quad\sum_{l=-N}^0\nu_3^{\ell}(\Delta_{l-1}+\Delta_l)\mathcal A_{n+l} +\theta\sum_{l=-N+1}^{0}\nu_3^{\ell}(\Delta_{l-2}+\Delta_{l-1})\mathcal A_{n+l}\Big].
\end{align*}
Noting that
\begin{align*}
\sum_{n=1}^{i-k-1}\sum_{l=-N}^0\nu_3^{\ell}(\Delta_{l-1}+\Delta_l)\mathcal A_{n+l}&=\sum_{l=-N}^0\sum_{n=1}^{i-k-1}\nu_3^{\ell}(\Delta_{l-1}+\Delta_l)\mathcal A_{n+l}\\
&\leq \sum_{l=-N}^0\sum_{m=-N}^{i-k-1}\nu_3^{\ell}(\Delta_{l-1}+\Delta_l)\mathcal A_{m}\leq 2\sum_{n=1}^{i-k-1}\mathcal A_n
\end{align*}
and $\mathcal A_1=(1-L_b\sum_{\ell=1}^{n_b}\nu_3^{\ell}(\Delta_{-1})\Delta\theta )^{-1}\tilde \sigma\leq (1-L_b\Delta\theta n_b)^{-1}\tilde \sigma,$
we arrive at
\begin{align*}
\mathcal A_{i-k}\leq (1-L_b\Delta\theta n_b)^{-1}\tilde \sigma+\frac{L_bn_b(\theta+2)}{1-L_b\Delta\theta n_b}\sum_{n=1}^{i-k-1}\mathcal A_n\Delta.
\end{align*}
Combining the discrete Gr\"onwall inequality, we deduce 
\begin{align}\label{boun1}
\mathcal A_{i-k}&\leq (1-L_b\Delta\theta n_b)^{-1}\tilde \sigma \exp\Big\{{\frac{L_bn_b(\theta+2)\Delta (i-k-1)}{1-L_b\Delta\theta n_b}}\Big\}\nn\\
&=:K_0\exp\Big\{{\frac{L_bn_b(\theta+2)\Delta (i-k-1)}{1-L_b\Delta\theta n_b}}\Big\}.
\end{align}

 We claim that \begin{align}\label{claim1}
|D_r\Phi(t_i;t_j,Y^{\varsigma}_j)|\leq A^k_i,\text{ for any } r\in(t_k,t_{k+1}],\; j\leq k\leq i-1.
\end{align} 
We prove the claim by the induction argument on $i-k$. In fact, when $j\leq i_1\leq k_1$ and $r\in(t_{k_1},t_{k_1+1}],$ we have $D_r\Phi(t_{i_1};t_j,Y^{\varsigma}_j)=0= A^{k_1}_{i_1}$. This gives $|D_r\Phi(t_{i_1};t_j,Y^{\varsigma}_j)|\leq A^{k_1}_{i_1}$ for $i_1-k_1\leq 0.$ Suppose that 
\eqref{claim1} holds for integers $k,i$ satisfying $0\leq i-k\leq i-k'-1$. Then we show \eqref{claim1} holds for $i-k=i-k'.$ 
By  \eqref{rela1} and  the induction assumption $|D_r\Phi(t_i;t_j,Y^{\varsigma}_j)|\leq A^{k'}_n,\;n\leq i-1$, we have 
\begin{align*}
&\quad|D_r\Phi(t_i;t_j,Y^{\varsigma}_j)|\nn\\
&\leq \sum_{n=k'}^{i-2}\sum_{\ell=1}^{n_b}\sum_{l=-N}^0L_b\nu_3^{\ell}(\Delta_{l-1}+\Delta_l)\big((1-\theta)A^{k'}_{n+l}\Delta+\theta A^{k'}_{n+l+1}\Delta\big)+\tilde \sigma\nn\\
&\quad+(1-\theta)\sum_{\ell=1}^{n_b}\sum_{l=-N}^0L_b\nu_3^{\ell}(\Delta_{l-1}+\Delta_l)A^{k'}_{i-1+l}\Delta
+\theta \sum_{\ell=1}^{n_b}\sum_{l=-N}^{-1} L_b\nu_3^{\ell}(\Delta_{l-1}+\Delta_l)A^{k'}_{i+l}\Delta\nn\\
&\quad+L_b\theta\Delta\sum_{\ell=1}^{n_b}\nu_3^{\ell}(\Delta_{-1})|D_r\Phi(t_i;t_j,Y^{\varsigma}_j)|,
\end{align*}
which yields 
\begin{align*}
&\quad|D_r\Phi(t_i;t_j,Y^{\varsigma}_j)|\leq (1-L_b\Delta\theta\sum_{\ell=1}^{n_b}\nu_3^{\ell}(\Delta_{-1}))^{-1}\Big[A^{k'}_{i-1}+L_b\Delta (1-\theta)\sum_{\ell=1}^{n_b}\sum_{l=-N}^0\nu_3^{\ell}(\Delta_{l-1}\nn\\
&\quad+\Delta_l)A^{k'}_{i-1+l}+L_b\Delta \theta\sum_{\ell=1}^{n_b}\sum_{l=-N}^{-1}\nu_3^{\ell}(\Delta_{l-1}+\Delta_l)A^{k'}_{i+l}\Big]=A^{k'}_i,\quad r\in(t_{k'},t_{k'+1}].
\end{align*}
This finishes the proof of the claim \eqref{claim1}. 

\underline{Step 2.} 
From \eqref{rela1} and 
\eqref{boun1}, we obtain
\begin{align*}
&\quad|D_r\Phi(t_i;t_j,Y^{\varsigma}_j)|\nn\\
&\ge \tilde \sigma-L_b\Delta \Big[\sum_{n=k}^{i-1}\sum_{\ell=1}^{n_b}\sum_{l=-N}^0\Big((1-\theta)\nu_3^{\ell}(\Delta_{l-1}+\Delta_l)A^k_{n+l}+\theta\nu_3^{\ell}(\Delta_{l-1}+\Delta_l)A^k_{n+l+1}\Big)\Big]\\
&\ge \tilde \sigma-L_b\Delta \Big[\sum_{n=k}^{i-1}\sum_{\ell=1}^{n_b}\sum_{l=-N}^0\Big((1-\theta)\nu_3^{\ell}(\Delta_{l-1}+\Delta_l)K_0e^{-\frac{L_bn_b(\theta+2)\Delta}{1-L_b\Delta\theta n_b}}+\theta\nu_3^{\ell}(\Delta_{l-1}+\Delta_l)K_0\Big)\\
&\quad\times \exp\{\frac{L_bn_b(\theta+2)\Delta(n+l-k)}{1-L_b\Delta\theta n_b}\}\Big]\\
&\ge \tilde \sigma-2L_bn_b\Delta \Big((1-\theta)K_0e^{-\frac{L_bn_b(\theta+2)\Delta}{1-L_b\Delta\theta n_b}}+\theta K_0\Big)\frac{\exp\{\frac{L_bn_b(\theta+2)\Delta(i-k)}{1-L_b\Delta\theta n_b}\}-1}{\exp\{\frac{L_bn_b(\theta+2)\Delta}{1-L_b\Delta\theta n_b}\}-1}\\
&\ge \tilde \sigma-\frac{2}{\theta+2}\Big(\exp\Big\{\frac{L_bn_b(\theta+2)\Delta(i-k)}{1-L_b\Delta\theta n_b}\Big\}-1\Big)\tilde \sigma \\
&\ge \tilde \sigma-\Big(\exp\Big\{\frac{L_bn_b(\theta+2)\Delta(i-k)}{1-L_b\Delta\theta n_b}\Big\}-1\Big)\tilde \sigma,
\end{align*}
where we used $e^x\ge x+1,x\ge 0.$
Then, 
\begin{align*}
|D_r\Phi(t_i;t_j,Y^{\varsigma}_j)|\ge \frac12\tilde \sigma,
\end{align*}
when  $\Delta\leq \frac{1}{2L_b\theta n_b}$ and 
  $i-k\leq \frac{\ln(\frac{3}{2})}{2L_bn_b(\theta+2)\Delta}$. 
Thus, 
\begin{align*}
\gamma_{\Phi(T;t_j,Y^{\varsigma}_j)}:&=\int_{0}^TD_r\Phi(T;t_j,Y^{\varsigma}_j)(D_r\Phi(T;t_j,Y^{\varsigma}_j))^{\top}\mathrm dr\\
&\ge \int_{(T-T_0)\vee 0}^T|D_r\Phi(T;t_j,Y^{\varsigma}_j)|^2\mathrm dr\ge \frac14\tilde{\sigma}^2(T\wedge T_0),
\end{align*}
where $T_0:=\frac{\ln(\frac32)}{2L_bn_b(\theta+2)}$.
This finishes the proof of \eqref{lower_1} for  $d=1.$

For the case of $d\ge 2,$  by replacing $D_r\Phi$  with $u^{\top}D_r\Phi$ in the above argument, one can  obtain \eqref{lower_1}. 

Moreover, it follows from \eqref{lower_1}  that  $$\lambda_{\min}(\gamma_{\Phi(T;t_j,Y^{\varsigma}_j)})=\min_{u\in\mathbb R^d,|u|=1}u^{\top}\gamma_{\Phi(T;t_j,Y^{\varsigma}_j)} u\ge \frac 14\tilde \sigma(T\wedge T_0),$$ which implies 
$$|(\gamma_{\Phi(T;t_j,Y^{\varsigma}_j)})^{-1}|\leq K|\det(\gamma_{\Phi(T;t_j,Y^{\varsigma}_j)})^{-1}|\leq K[\frac 14\tilde \sigma(T\wedge T_0)]^{-d}.$$ Thus the proof is completed. 
\end{proof}

As a consequence, 
one can obtain  the following Malliavin integration by parts formula.
\begin{lemma}
Let $\alpha\in \{1,\ldots,d\}^4$ be the multi-index, 
$f\in\mathscr C,$ and $G_1\in\mathbb D^{\alpha+1,\infty}.$ Then under conditions in \cref{Ch4highDy,lower_matrix},  for $\Delta\in(0,\frac{1}{2L_bn_b\theta}]$, $i\in\mathbb N$ with $i\Delta\leq \lfloor T-\Delta\rfloor,$ there exists a random variable $H_{|\alpha|+1}(\Phi(T;t_i,Y^{\varsigma}_i),G_1)$ such that
\begin{align}\label{IBP1}
\mathbb E[\mathcal \partial_{\alpha}f(\Phi(T;t_i,Y^{\varsigma}_i))G_1]=\mathbb E[F(\Phi(T;t_i,Y^{\varsigma}_i))H_{|\alpha|+1}(\Phi(T;t_i,Y^{\varsigma}_i),G_1)].
\end{align}
Moreover, 
\begin{align}\label{IBP2}
\sup_{T\ge T_0}|\mathbb E[\mathcal \partial_{\alpha}f(\Phi(T;t_i,Y^{\varsigma}_i))G_1]|\leq K\|G_1\|_{|\alpha|+1,2},
\end{align}
where $T_0$ is given in Theorem \ref{dens_con},  and $Y^{\varsigma}_i$ is given in \eqref{Y_defini}. 
\end{lemma}
\begin{proof}
According to the definition of $\mathscr C,$
and applying \cite[Proposition 2.1.4, (2.29)--(2.32)]{Nualart} give \eqref{IBP1}. Moreover, for $q>q_1\ge 1,$ there exist constants $\eta_1,\eta_2>0$ and integers $n_1,n_2>0$ such that
\begin{align*}
&\quad\|H_{|\alpha|+1}(\Phi(T;t_i,Y^{\varsigma}_i),G_1)\|_{q_1}\nn\\
&\leq K(q_1,q)\|\gamma_{\Phi(T;t_i,Y^{\varsigma}_i)}^{-1}\|^{n_1}_{0,\eta_1}\|D\Phi(T;t_i,Y^{\varsigma}_i)\|^{n_2}_{|\alpha|+1,\eta_2}\|G_1\|_{|\alpha|+1,q}\\
&\leq K\Big[\frac 14\tilde \sigma(T\wedge T_0)\Big]^{-dn_1}\|G_1\|_{|\alpha|+1,q},
\end{align*}
where in the last inequality we used \cref{Ch4highDy,lower_matrix}.
Taking $q_1=1,q=2$ and supremum for $T\ge T_0$, we finishes the proof.
\end{proof}

Now we give the test-functional-independent weak convergence rate of the $\theta$-EM method. 

\begin{thm}\label{thm_weak2}
Let conditions in Theorem \ref{dens_con} hold. Then for  $\Delta\in(0,\tilde \Delta]$,
\begin{align*}
\sup_{f\in \mathscr C}\sup_{T\ge T_0}\big|\mathbb E[f(x^{\xi}(T))]-\mathbb E[f(y^{\xi,\Delta}(T))]\big|\leq K\Delta,
\end{align*}
where $T_0$ is given in \cref{dens_con}. 
\end{thm}
\begin{proof} In order to obtain the test-functional-independent weak convergence rate, we need to reestimate   the sub-terms in  Theorem \ref{thm_weakcon} one-by-one by means of the Malliavin integration by parts formula \eqref{IBP1} and the inequality \eqref{IBP2}. 

Recalling \eqref{Wrate}, for the term 
$$\mathcal I_{0}=\mathbb E[f(\Phi(T;t_{N^{\Delta}},\varphi_{t_{N^{\Delta}}}^{Int}(0,\xi)))]-\mathbb E[f(\Phi(T;t_{N^{\Delta}},\varphi_{t_{N^{\Delta}}}^{Int}(0,\Phi_0)))],
$$
it follows from \eqref{IBP2} that
\begin{align*}
&\quad \mathbb E[f(\Phi(T;t_{N^{\Delta}},\varphi_{t_{N^{\Delta}}}^{Int}(0,\xi)))]-\mathbb E[f(\Phi(T;t_{N^{\Delta}},\varphi_{t_{N^{\Delta}}}^{Int}(0,\Phi_0)))]\\
&=\mathbb E[f(\varphi^{Int}(t_{N^{\Delta}};0,\xi)]-\mathbb E[f(\varphi^{Int}(t_{N^{\Delta}};0,\Phi_0))]\\
&=\int_0^1\mathbb E\Big[f'(\varphi^{Int}(t_{N^{\Delta}};0,\varsigma\xi+(1\!-\!\varsigma)\Phi_0))\mathcal D \varphi^{Int}(t_{N^{\Delta}};0,\varsigma\xi\!+\!(1\!-\!\varsigma)\Phi_0)(\xi\!-\!\Phi_0)\Big]\mathrm d\varsigma\\
&\leq K\int_0^1\|\mathcal D \varphi^{Int}(t_{N^{\Delta}};0,\varsigma\xi+(1-\varsigma)\Phi_0)(\xi-\Phi_0)\|_{2,2}\mathrm d\varsigma\\
&\leq K\Delta\int_0^1\|\mathcal D \varphi^{Int}(t_{N^{\Delta}};0,\varsigma\xi+(1-\varsigma)\Phi_0)\|_{2,2}\mathrm d\varsigma\\
&\leq K\Delta,
\end{align*}
where $K$ is independent of $T$, and we used  Assumption \ref{a7} with $\rho\geq1$.

\textbf{Reestimate of  term $\mathcal I_{b}.$} 
Recall 
\begin{align*}
&\mathcal I_{b,1}=\sum_{i=1}^{N^{\Delta}}\int_0^1\mathbb E\Big[\Big\< (\mathcal D\Phi(T;t_i,Y^{\varsigma}_i) I^0\mathrm {Id}_{d\times d})^*f'(\Phi(T;t_i,Y^{\varsigma}_i)),\int_{t_{i-1}}^{t_i}\int_0^1\mathcal Db(Z^{\beta_1}_{i,r})\\
&\quad \big(\varphi^{Int}_r(t_{i-1},\varphi_{t_{i-1}}(0,\Phi_0))-\varphi^{Int}_{t_{i-1}}(0,\Phi_0)\big)\mathrm d\beta_1\mathrm dr\Big\>\Big]\mathrm d\varsigma,\\
&\mathcal I_{b,2}=\sum_{i=1}^{N^{\Delta}}\int_0^1\mathbb E\Big[\Big\< (\mathcal D\Phi(T;t_i,Y^{\varsigma}_i) I^0\mathrm {Id}_{d\times d})^*f'(\Phi(T;t_i,Y^{\varsigma}_i)),\int_{t_{i-1}}^{t_i}\int_0^1\mathcal Db(Z^{\beta_1}_{i,r})\\
&\quad \big(\varphi_r(t_{i-1},\varphi_{t_{i-1}}(0,\Phi_0))-\varphi^{Int}_r(t_{i-1},\varphi_{t_{i-1}}(0,\Phi_0))\big)\mathrm d\beta_1\mathrm dr\Big\>\Big]\mathrm d\varsigma,
\end{align*}
and $Z^{\beta_1}_{i,r}=\beta_1\varphi_r(t_{i-1},\varphi_{t_{i-1}}(0,\Phi_0))+(1-\beta_1)\varphi^{Int}_{t_{i-1}}(0,\Phi_0).$

\textit{Reestimate of term $\mathcal I_{b,1}.$} 
We only reestimate the sub-term  
\begin{align*}
&\sum_{i=1}^{N^{\Delta}}\int_0^1\mathbb E\Big[\Big\< (\mathcal D\Phi(T;t_i,Y^{\varsigma}_i) I^0\mathrm {Id}_{d\times d})^*f'(\Phi(T;t_i,Y^{\varsigma}_i)),\int_{t_{i-1}}^{t_i}\int_0^1\mathcal Db(Z^{\beta_1}_{i,r})\nn\\
&\quad\sum_{j=-N}^{-1}\mathbf 1_{\{i+j\geq1\}} \mathbf 1_{[t_{i+j}-r,t_{i+j+1}-r)}(\cdot)\frac{t_{j+1}-\cdot}{\Delta}\int_{t_{i+j-1}}^{t_{i+j}}b(\varphi_u(0,\Phi_0))\mathrm du\mathrm d\beta_1\mathrm dr\Big\>\Big]\mathrm d\varsigma
\end{align*}
and
the sub-term with the Brownian motion
\begin{align*}
&\quad\mathcal I_{b,1}^1=\sum_{i=1}^{N^{\Delta}}\int_0^1\mathbb E\Big[\Big\< (\mathcal D\Phi(T;t_i,Y^{\varsigma}_i) I^0\mathrm {Id}_{d\times d})^*f'(\Phi(T;t_i,Y^{\varsigma}_i)),\int_{t_{i-1}}^{t_i}\!\int_0^1\!\mathcal Db(Z^{\beta_1}_{i,r})\nn\\
&\quad\sum_{j=-N}^{-1}\mathbf 1_{\{i+j\geq1\}}\mathbf 1_{[t_{i+j}-r,t_{i+j+1}-r)}(\cdot)\frac{t_{j+1}-\cdot}{\Delta}\!\!\int_{t_{i+j-1}}^{t_{i+j}}\!\tilde\sigma\mathrm dW(u)\mathrm d\beta_1\mathrm dr\Big\>\Big]\mathrm d\varsigma,
\end{align*}
since other sub-terms in $\mathcal I_{b,1}$ can be estimated similarly. 
It follows from \eqref{IBP1} and \eqref{IBP2}  that  
\begin{align}\label{Ch4b11}
&\quad\sum_{i=1}^{N^{\Delta}}\int_0^1\mathbb E\Big[\Big\< (\mathcal D\Phi(T;t_i,Y^{\varsigma}_i) I^0\mathrm {Id}_{d\times d})^*f'(\Phi(T;t_i,Y^{\varsigma}_i)),\int_{t_{i-1}}^{t_i}\int_0^1\mathcal Db(Z^{\beta_1}_{i,r})\nn\\
&\quad\sum_{j=-N}^{-1}\mathbf 1_{\{i+j\geq1\}}
\mathbf 1_{[t_{i+j}-r,t_{i+j+1}-r)}(\cdot)\frac{t_{j+1}-\cdot}{\Delta}\int_{t_{i+j-1}}^{t_{i+j}}b(\varphi_u(0,\Phi_0))\mathrm du\mathrm d\beta_1\mathrm dr\Big\>\Big]\mathrm d\varsigma\nn\\
&\leq K\sum_{i=1}^{N^{\Delta}}\int_0^1\Big\|\mathcal D\Phi(T;t_i,Y^{\varsigma}_{i}) I^0\mathrm{Id}_{d\times d}\int_{t_{i-1}}^{t_i}\int_0^1 \mathcal Db(Z^{\beta_1}_{i,r})\sum_{j=-N}^{-1}\mathbf 1_{\{i+j\geq1\}}\nn\\
&\quad\mathbf 1_{[t_{i+j-r},t_{i+j+1}-r)}(\cdot)\frac{t_{j+1}-\cdot}{\Delta}\int_{t_{i+j-1}}^{t_{i+j}}b(\varphi_u(0,\Phi_0))\mathrm du\mathrm d\beta_1\mathrm dr\Big\|_{2,2}\mathrm d\varsigma.
\end{align}
According to \cref{Ch4highDy}, $\sup_{\phi_1\in\mathcal C^d}|\mathcal D b(\phi_1)\phi_2|\leq K\|\phi_2\|$, \eqref{linearb}, and  \eqref{Lunix}, we have
\begin{align}\label{esti02}
&\quad\Big\|\mathcal D\Phi(T;t_i,Y^{\varsigma}_{i}) I^0\mathrm{Id}_{d\times d}\int_{t_{i-1}}^{t_i}\int_0^1 \mathcal Db(Z^{\beta_1}_{i,r})\sum_{j=-N}^{-1}\mathbf 1_{\{i+j\geq1\}}\nn\\
&\quad\mathbf 1_{[t_{i+j-r},t_{i+j+1}-r)}(\cdot)\frac{t_{j+1}-\cdot}{\Delta}\int_{t_{i+j-1}}^{t_{i+j}}b(\varphi_u(0,\Phi_0))\mathrm du\mathrm d\beta_1\mathrm dr\Big\|_{0,2}\nn\\
&\leq K\Big(e^{-\hat \lambda_4(T-t_i)}\E\Big[\Big\|I^0\mathrm{Id}_{d\times d}\int_{t_{i-1}}^{t_i}\int_0^1 \mathcal Db(Z^{\beta_1}_{i,r})\sum_{j=-N}^{-1}\mathbf 1_{\{i+j\geq1\}}\nn\\
&\quad\mathbf 1_{[t_{i+j-r},t_{i+j+1}-r)}(\cdot)\frac{t_{j+1}-\cdot}{\Delta}\int_{t_{i+j-1}}^{t_{i+j}}b(\varphi_u(0,\Phi_0))\mathrm du\mathrm d\beta_1\mathrm dr\Big\|^2\Big]\Big)^{\frac12}\nn\\
&\leq K e^{-\frac{\hat \lambda_4(T-t_i)}{2}}\Delta^2.
\end{align}
In addition, using \cref{Ch4highDy} and Assumption \ref{high5}, we derive
\begin{align}\label{esti12}
&\quad \Big(\int_{t_i}^T\mathbb E\Big[\Big\|D_{r_1}\Big[\mathcal D\Phi(T;t_i,Y^{\varsigma}_{i}) I^0\mathrm{Id}_{d\times d}\int_{t_{i-1}}^{t_i}\int_0^1 \mathcal Db(Z^{\beta_1}_{i,r})\sum_{j=-N}^{-1}\mathbf 1_{\{i+j\geq1\}}\nn\\
&\quad\mathbf 1_{[t_{i+j-r},t_{i+j+1}-r)}(\cdot)\frac{t_{j+1}-\cdot}{\Delta}\int_{t_{i+j-1}}^{t_{i+j}}b(\varphi_u(0,\Phi_0))\mathrm du\mathrm d\beta_1\mathrm dr\Big]\Big\|^2\Big]\mathrm dr_1\Big)^{\frac12}\nn\\
&\leq K\Big\{\int_{t_i}^{T}\Big(e^{-\hat \lambda_4(T-t_i)}\E\Big[\Big\|I^0\mathrm{Id}_{d\times d}\int_{t_{i-1}}^{t_i}\int_0^1\mathcal Db(Z^{\beta_1}_{i,r})\sum_{j=-N}^{-1}\mathbf 1_{\{i+j\geq1\}}\nn\\
&\quad\mathbf 1_{[t_{i+j-r},t_{i+j+1}-r)}(\cdot)\frac{t_{j+1}-\cdot}{\Delta}\int_{t_{i+j-1}}^{t_{i+j}} b(\varphi_u(0,\Phi_0))\mathrm du\mathrm d\beta_1\mathrm dr\Big\|^2
\big(1+\|D_{r_1}Y^{\varsigma}_i\|^{4}\big)\Big]\nn\\
&\quad+e^{-\hat \lambda_4(T-t_i)}\E\Big[\Big\|I^0\mathrm{Id}_{d\times d}\int_{t_{i-1}}^{t_i}\int_0^1\mathcal D^2b(Z^{\beta_1}_{i,r})\Big(D_{r_1}Z^{\beta_1}_{i,r},\sum_{j=-N}^{-1}\mathbf 1_{\{i+j\geq1\}}\nn\\
&\quad\mathbf 1_{[t_{i+j-r},t_{i+j+1}-r)}(\cdot)\frac{t_{j+1}-\cdot}{\Delta}\int_{t_{i+j-1}}^{t_{i+j}}b(\varphi_u(0,\Phi_0))\mathrm du\Big)\mathrm d\beta_1\mathrm dr\Big\|^2\Big]\nn\\
&\quad+e^{-\hat \lambda_4(T-t_i)}\E\Big[\Big\|I^0\mathrm{Id}_{d\times d}\int_{t_{i-1}}^{t_i}\int_0^1 \mathcal Db(Z^{\beta_1}_{i,r})\sum_{j=-N}^{-1}\mathbf 1_{\{i+j\geq1\}}\mathbf 1_{[t_{i+j-r},t_{i+j+1}-r)}(\cdot)\nn\\
&\quad\frac{t_{j+1}-\cdot}{\Delta}\int_{t_{i+j-1}}^{t_{i+j}}\mathcal Db(\varphi_u(0,\Phi_0))D_{r_1}\varphi_u(0,\Phi_0)\mathrm du\mathrm d\beta_1\mathrm dr\Big\|^2\Big]\Big)
\mathrm dr_1\Big\}^{\frac12}\nn\\
&\leq K\Big\{\int_{t_i}^{T}\Big(\Delta^{4}e^{-\hat \lambda_4(T-t_i)} (1+\sup_{u\geq0}\E[\|\varphi_u(0,\Phi_0)\|^4])^{\frac12}\big(1+\E\big[\|D_{r_1}Y^{\varsigma}_i\|^{8}\big]\big)^{\frac12}\nn\\
&\quad+\Delta^3
e^{-\hat \lambda_4(T-t_i)}\int_{t_{i-1}}^{t_i}\int_0^1\E[\|D_{r_1}Z^{\beta_1}_{i,r}\|^4]^{\frac{1}{2}}\mathrm d \beta_1\mathrm d r(1+\sup_{u\geq0}\E[\|\varphi_u(0,\Phi_0)\|^4])^{\frac{1}{2}}\nn\\
&\quad+\Delta^3 e^{-\hat \lambda_4(T-t_i)}\int_{t_{i+j-1}}^{t_{i+j}}\E[\|D_{r_1}\varphi_u(0,\Phi_0)\|^2]\mathrm du\Big)\mathrm dr_1\Big\}^{\frac{1}{2}}\nn\\
&\leq K e^{-\frac{\hat \lambda_4(T-t_i)}{2}}(T-t_i)^{\frac12}\Delta^2.
\end{align}
Similarly,
\begin{align}\label{esti22}
&\Big\{\int_{t_i}^{T}\int_{t_i}^{T}\mathbb E\Big[ \Big\|D_{r_1,r_2}\Big(\mathcal D\Phi(T;t_i,Y^{\varsigma}_{i}) I^0\mathrm{Id}_{d\times d}\int_{t_{i-1}}^{t_i}\int_0^1 \mathcal Db(Z^{\beta_1}_{i,r})\sum_{j=-N}^{-1}\mathbf 1_{\{i+j\geq1\}}\nn\\
&\quad\mathbf 1_{[t_{i+j-r},t_{i+j+1}-r)}(\cdot)\frac{t_{j+1}-\cdot}{\Delta}\int_{t_{i+j-1}}^{t_{i+j}}b(\varphi_u(0,\Phi_0))\mathrm du\mathrm d\beta_1\mathrm dr\Big)\Big\|^2\Big]\mathrm dr_1\mathrm dr_2\Big\}^{\frac12}\nn\\
&\leq K e^{-\frac{\hat \lambda_4(T-t_i)}{2}}(T-t_i)\Delta^2.
\end{align}
Inserting \eqref{esti02}--\eqref{esti22} into \eqref{Ch4b11}, one has
\begin{align*}
&\quad\sum_{i=1}^{N^{\Delta}}\mathbb E\Big[\Big\< (\mathcal D\Phi(T;t_i,Y^{\varsigma}_i) I^0\mathrm {Id}_{d\times d})^*f'(\Phi(T;t_i,Y^{\varsigma}_i)),\int_{t_{i-1}}^{t_i}\int_0^1\mathcal Db(Z^{\beta_1}_{i,r})\sum_{j=-N}^{-1}\nn\\
&\quad\mathbf 1_{\{i+j\geq1\}}
\mathbf 1_{[t_{i+j}-r,t_{i+j+1}-r)}(\cdot)\frac{t_{j+1}-\cdot}{\Delta}\int_{t_{i+j-1}}^{t_{i+j}}b(\varphi_u(0,\Phi_0))\mathrm du\mathrm d\beta_1\mathrm dr\Big\>\Big]\mathrm d\varsigma\\
 &\leq K\sum_{i=1}^{N^{\Delta}}e^{-\frac{\hat \lambda_4(T-t_i)}{2}}\Delta^{2}\big(1+(T-t_i)^{\frac12}+(T-t_i)\big)
\nn\\
&\leq K\Delta,
\end{align*}
 where we used the fact $$\sup_{T>0}\Delta\sum_{i=1}^{N^{\Delta}}e^{-\frac{\hat\lambda_4}{2}(T-t_i)}(T-t_i)^{l}\leq \sup_{T\geq0}\int_{0}^{T}e^{-\frac{\hat\lambda_4\Delta}{2}} e^{-\frac{\hat\lambda_4}{2}(T-s)}(T-s)^l\mathrm ds<\infty \quad\forall~l\in\mathbb N.$$

For the sub-term with the Brownian motion,  we have 
\begin{align*}
&\quad\mathcal I_{b,1}^1
=\sum_{i=1}^{N^{\Delta}}\int_0^1\int_{t_{i-1}}^{t_i}\int_0^1\int_{t_{i+j-1}}^{t_{i+j}}\Big\{\mathbb E\Big[\Big\<f'(\Phi(T;t_i,Y^{\varsigma}_i)),\mathcal D\Phi(T;t_i,Y^{\varsigma}_i)I^0\mathrm {Id}_{d\times d}\\
&\quad \mathcal D^2b(Z^{\beta_1}_{i,r})\Big(D_uZ^{\beta_1}_{i,r},\sum_{j=-N}^{-1}\mathbf 1_{\{i+j\geq1\}}\mathbf1_{[t_{i+j}-r,t_{i+j+1}-r)}(\cdot)\frac{t_{j+1}-\cdot}{\Delta}\mathrm {Id}_{d\times d}\tilde\sigma\Big)\Big\>\Big]\nn\\
&\quad+\mathbb E\Big[\Big\<f'(\Phi(T;t_i,Y^{\varsigma}_i)),D_u\mathcal D\Phi(T;t_i,Y^{\varsigma}_i)I^0\mathrm {Id}_{d\times d}\mathcal Db(Z^{\beta_1}_{i,r})\sum_{j=-N}^{-1}\mathbf 1_{\{i+j\geq1\}}\nn\\
&\quad\mathbf1_{[t_{i+j}-r,t_{i+j+1}-r)}(\cdot)\frac{t_{j+1}-\cdot}{\Delta}\mathrm {Id}_{d\times d}\tilde\sigma\Big\>\Big]+\mathbb E\Big[\Big\<f''(\Phi(T;t_i,Y^{\varsigma}_i))\nn\\
&\quad D_u\Phi(T;t_i,Y^{\varsigma}_i),\mathcal D\Phi(T;t_i,Y^{\varsigma}_i) I^0\mathrm {Id}_{d\times d}\mathcal Db(Z^{\beta_1}_{i,r})\sum_{j=-N}^{-1}\mathbf 1_{\{i+j\geq1\}}\mathbf1_{[t_{i+j}-r,t_{i+j+1}-r)}(\cdot)\nn\\
&\quad\frac{t_{j+1}-\cdot}{\Delta}\mathrm {Id}_{d\times d}\tilde\sigma\Big\>\Big]\Big\}\mathrm du\mathrm d\beta_1\mathrm dr\mathrm d\varsigma.
\end{align*}
This, together with \eqref{IBP1} and \eqref{IBP2}, implies that
\begin{align*}
\mathcal I^1_{b,1}&\leq K\sum_{i=1}^{N^{\Delta}}\int_0^1\int_{t_{i-1}}^{t_i}\int_0^1\int_{t_{i+j-1}}^{t_{i+j}}\Big\{\Big\| \mathcal D\Phi(T;t_i,Y^{\varsigma}_i)I^0\mathrm {Id}_{d\times d} \mathcal D^2b(Z^{\beta_1}_{i,r})\Big(D_uZ^{\beta_1}_{i,r},\nn\\
&\quad\sum_{j=-N}^{-1}\mathbf 1_{\{i+j\geq1\}}\mathbf1_{[t_{i+j}-r,t_{i+j+1}-r)}(\cdot)
\frac{t_{j+1}-\cdot}{\Delta}\mathrm {Id}_{d\times d}\tilde\sigma\Big)\Big\|_{2,2}\nn\\
&\quad+\Big\| D_u\mathcal D\Phi(T;t_i,Y^{\varsigma}_i)I^0\mathrm {Id}_{d\times d}\mathcal Db(Z^{\beta_1}_{i,r})\sum_{j=-N}^{-1}\mathbf 1_{\{i+j\geq1\}}\mathbf1_{[t_{i+j}-r,t_{i+j+1}-r)}(\cdot)\nn\\
&\quad\frac{t_{j+1}-\cdot}{\Delta}\mathrm {Id}_{d\times d}\tilde\sigma\Big\|_{2,2}+\Big\|\Big\<D_u\Phi(T;t_i,Y^{\varsigma}_i),\mathcal D\Phi(T;t_i,Y^{\varsigma}_i)I^0\mathrm {Id}_{d\times d}\nn\\
&\quad\mathcal Db(Z^{\beta_1}_{i,r})\sum_{j=-N}^{-1}\mathbf 1_{\{i+j\geq1\}}\mathbf1_{[t_{i+j}-r,t_{i+j+1}-r)}(\cdot)\frac{t_{j+1}-\cdot}{\Delta}\mathrm {Id}_{d\times d}\nn\\
&\quad\tilde\sigma\Big\>\Big\|_{3,2}\Big\}\mathrm du\mathrm d\beta_1\mathrm dr\mathrm d\varsigma=:\widetilde{\mathcal I}^1_{b,1,1}+\widetilde{\mathcal I}^1_{b,1,2}+\widetilde{\mathcal I}^1_{b,1,3}.
\end{align*}
By Assumption \ref{high5}, \cref{Ch4highDx,Ch4highDy}, and \eqref{Lunix}, the term $\widetilde{\mathcal I}^1_{b,1,1}$ is  estimated as 
\begin{align*}
&\quad\widetilde{\mathcal I}^1_{b,1,1}\leq K\sum_{i=1}^{N^{\Delta}}\int_0^1\int_{t_{i-1}}^{t_i}\int_0^1\int_{t_{i+j-1}}^{t_{i+j}}\Big\{\Big(e^{-\hat\lambda_4(T-t_i)}\mathbb E\Big[\Big\|I^0\mathrm{Id}_{d\times d}\mathcal D^2b(Z^{\beta_1}_{i,r})\Big(D_uZ^{\beta_1}_{i,r},\nn\\
&\quad\sum_{j=-N}^{-1}\mathbf 1_{\{i+j\geq1\}}\mathbf 1_{[t_{i+j}-r,t_{i+j+1}-r)}(\cdot)\frac{t_{j+1}-\cdot}{\Delta}\mathrm {Id}_{d\times d}\tilde \sigma\Big)\Big\|^2\Big]\Big)^{\frac12}\\
&\quad+\!\Big(\int_{t_{i}}^{T}\mathbb E\Big[\Big\|D_{r_1}\Big(
 \mathcal D\Phi(T;t_i,Y^{\varsigma}_i)I^0\mathrm {Id}_{d\times d} \mathcal D^2b(Z^{\beta_1}_{i,r})\Big(D_uZ^{\beta_1}_{i,r},\nn\\
&\quad\sum_{j=-N}^{-1}\mathbf 1_{\{i+j\geq1\}}\mathbf1_{[t_{i+j}-r,t_{i+j+1}-r)}(\cdot)
\frac{t_{j+1}-\cdot}{\Delta}\mathrm {Id}_{d\times d}\tilde\sigma\Big)\Big\|^2\Big]\mathrm dr_1\Big)^{\frac12}\\
&\quad +\Big(\int_{t_{i}}^{T}\int_{t_{i}}^{T}\mathbb E\Big[\Big\|D_{r_1,r_2} \Big(\mathcal D\Phi(T;t_i,Y^{\varsigma}_i)I^0\mathrm {Id}_{d\times d} \mathcal D^2b(Z^{\beta_1}_{i,r})\Big(D_uZ^{\beta_1}_{i,r},\sum_{j=-N}^{-1}\mathbf 1_{\{i+j\geq1\}}\nn\\
&\quad\mathbf1_{[t_{i+j}-r,t_{i+j+1}-r)}(\cdot)
\frac{t_{j+1}-\cdot}{\Delta}\mathrm {Id}_{d\times d}\tilde\sigma\Big)\Big\|^2\Big]\mathrm d r_1\mathrm dr_2\Big)^{\frac12}\Big\}\mathrm du\mathrm d\beta_1\mathrm d r\mathrm d\varsigma\\
&\leq K\Delta.
\end{align*}
For the term $\widetilde{\mathcal I}^1_{b,1,2},$ we arrive at
\begin{align*}
&\quad\widetilde{\mathcal I}^1_{b,1,2}\leq K\sum_{i=1}^{N^{\Delta}}\int_0^1\int_{t_{i-1}}^{t_i}\int_0^1\int_{t_{i+j-1}}^{t_{i+j}}\Big\{\Big(e^{-\hat\lambda_4(T-t_i)}\mathbb E\Big[\Big\|I^0\mathrm {Id}_{d\times d}\mathcal Db(Z^{\beta_1}_{i,r})\sum_{j=-N}^{-1}\nn\\
&\quad\mathbf 1_{\{i+j\geq1\}}\mathbf 1_{[t_{i+j}-r,t_{i+j+1}-r)}(\cdot)\frac{t_{j+1}-\cdot}{\Delta}\mathrm {Id}_{d\times d}\tilde \sigma\Big\|^2\big(1+\|D_{u}Y^{\varsigma}_i\|^4\big)\Big]\Big)^{\frac12}\\
&\quad +\Big(\int_{t_{i}}^{T}\mathbb E\Big[\Big\|D_{r_1}\Big( D_u\mathcal D\Phi(T;t_i,Y^{\varsigma}_i)I^0\mathrm {Id}_{d\times d}\mathcal Db(Z^{\beta_1}_{i,r})\sum_{j=-N}^{-1}\mathbf 1_{\{i+j\geq1\}}\mathbf1_{[t_{i+j}-r,t_{i+j+1}-r)}(\cdot)\nn\\
&\quad\frac{t_{j+1}-\cdot}{\Delta}\mathrm {Id}_{d\times d}\tilde\sigma\Big)\Big\|^2\Big]\mathrm dr_1\Big)^{\frac12} +\Big(\!\int_{t_{i}}^{T}\!\int_{t_{i}}^{T}\mathbb E\Big[\Big\|\!D_{r_1,r_2}\Big( D_u\mathcal D\Phi(T;t_i,Y^{\varsigma}_i)I^0\mathrm {Id}_{d\times d}\nn\\
&\quad\mathcal Db(Z^{\beta_1}_{i,r})\sum_{j=-N}^{-1}\mathbf 1_{\{i+j\geq1\}}\mathbf1_{[t_{i+j}-r,t_{i+j+1}-r)}(\cdot)\frac{t_{j+1}-\cdot}{\Delta}\mathrm {Id}_{d\times d}\tilde\sigma\Big)\Big\|^2\Big]\nn\\
&\quad\mathrm dr_1\mathrm dr_2\Big)^{\frac12}\Big\}\mathrm du\mathrm d\beta_1\mathrm dr\mathrm d\varsigma\leq K\Delta.
\end{align*}
Similar to estimates of $\widetilde{\mathcal I}^1_{b,1,1}$ and $\widetilde{\mathcal I}_{b,1,2}^1$, for the term $\widetilde{\mathcal I}_{b,1,3}^1,$ we obtain
\begin{align*}
\widetilde{\mathcal I}_{b,1,3}^1\leq& K\sum_{i=1}^{N^{\Delta}}\int_0^1\int_{t_{i-1}}^{t_i}\int_0^1\int_{t_{i+j-1}}^{t_{i+j}} \Big\|D_u\Phi(T;t_i,Y^{\varsigma}_i)\Big\|_{3,4}\Big\|\mathcal D\Phi(T;t_i,Y^{\varsigma}_i)I^0\mathrm {Id}_{d\times d}\\
&\quad \mathcal Db(Z^{\beta_1}_{i,r})\sum_{j=-N}^{-1}\mathbf 1_{\{i+j\geq1\}}\mathbf1_{[t_{i+j}-r,t_{i+j+1}-r)}(\cdot)\frac{t_{j+1}-\cdot}{\Delta}\nn\\
&\quad\mathrm {Id}_{d\times d}\tilde\sigma\Big\|_{3,4}\mathrm du\mathrm d\beta_1\mathrm dr\mathrm d\varsigma\leq K\Delta.
\end{align*}

Other terms can be estimated similarly as above.  Therefore, we have $\mathcal I_{b,1}\leq K\Delta.$

\textit{Reestimate of term $\mathcal I_{b,2}.$}
We still only estimate the sub-term with $$\sum_{j=-N}^{-1}\mathbf 1_{\{i+j\geq1\}}\mathbf 1_{[t_{i+j}-r,t_{i+j+1}-r)}(\cdot)\int_{t_{i+j}}^{r+\cdot}\tilde \sigma\mathrm dW(u),$$ since other sub-terms in $\mathcal I_{b,2}$ can be estimated similarly to  those in $\mathcal I_{b,1}.$
Recall that 
\begin{align*}
&\quad\mathcal I_{b,2,1}=\sum_{i=1}^{N^{\Delta}}\sum_{j=-N}^{-1}\mathbf 1_{\{i+j\geq1\}}\int_0^1\mathbb E\Big[\Big\< (\mathcal D\Phi(T;t_i,Y^{\varsigma}_i) I^0\mathrm {Id}_{d\times d})^*f'(\Phi(T;t_i,Y^{\varsigma}_i)),\nn\\
&\quad\int_{t_{i-1}}^{t_i}\int_0^1\mathcal Db(Z^{\beta_1}_{i,r})\mathbf 1_{[t_{i+j}-r,t_{i+j+1}-r)}(\cdot)\int_{t_{i+j}}^{r+\cdot}\tilde \sigma\mathrm dW(u) \mathrm d\beta_1\mathrm dr\Big\>\Big]\mathrm d\varsigma\\
 &=\sum_{i=1}^{N^{\Delta}}\sum_{j=-N}^{-1}\mathbf 1_{\{i+j\geq1\}}\int_0^1\int_{t_{i-1}}^{t_i}\int_0^1\int_{-\tau\vee (t_{i+j}-r)}^{ (t_{i+j+1}-r)\wedge 0}\int_{t_{i+j}}^{r+s}\Big\{\mathbb E\Big[\Big\<f'(\Phi(T;t_i,Y^{\varsigma}_i)),\nn\\
 &\quad\mathcal D\Phi(T;t_i,Y^{\varsigma}_i) I^0\mathrm {Id}_{d\times d} D_uk_b^{1}(Z^{\beta_1}_{i,r}(s))\tilde \sigma\Big\>\Big]\mathrm du\mathrm d\nu_3^{1}(s)\mathrm d\beta_1\mathrm dr\mathrm d\varsigma\\
 &\quad+\mathbb E\Big[\Big\<f'(\Phi(T;t_i,Y^{\varsigma}_i)), D_u\mathcal D\Phi(T;t_i,Y^{\varsigma}_i)
 I^0\mathrm {Id}_{d\times d}k_b^{1}(Z^{\beta_1}_{i,r}(s))\tilde \sigma\Big\>\Big]\\
 &\quad+\mathbb E\Big[\Big\<f''(\Phi(T;t_i,Y^{\varsigma}_i))D_u \Phi(T;t_i,Y^{\varsigma}_i), \mathcal D\Phi(T;t_i,Y^{\varsigma}_i)I^0\mathrm {Id}_{d\times d}
 k_b^{1}(Z^{\beta_1}_{i,r}(s))\nn\\
 &\quad\tilde \sigma\Big\>\Big]\Big\}\mathrm du\mathrm d\nu_3^{1}(s)\mathrm d\beta_1\mathrm dr\mathrm d\varsigma.
 \end{align*}
Combining \eqref{Lunix}, \eqref{IBP1}, \eqref{IBP2}, and Lemma \ref{Ch4highDy}, we deduce 
 \begin{align*}
&\quad\mathcal I_{b,2,1}\leq K\sum_{i=1}^{N^{\Delta}}\sum_{j=-N}^{-1}\mathbf 1_{\{i+j\geq1\}}\int_0^1\!\int_{t_{i-1}}^{t_i}\!\int_0^1\!\int_{-\tau\vee (t_{i+j}-r)}^{ (t_{i+j+1}-r)\wedge 0}\!\int_{t_{i+j}}^{r+s}\!\Big\{\Big\|\mathcal D\Phi(T;t_i,Y^{\varsigma}_i)\nn\\
&\quad I^0\mathrm {Id}_{d\times d} D_uk_b^{1}(Z^{\beta_1}_{i,r}(s))\tilde \sigma\Big\|_{2,2}+\Big\|D_u\mathcal D\Phi(T;t_i,Y^{\varsigma}_i)I^0\mathrm {Id}_{d\times d}k_b^{1}(Z^{\beta_1}_{i,r}(s))\nn\\
&\quad\tilde \sigma\Big\|_{2,2}+\big\|\big\<D_u\Phi(T;t_i,\!Y^{\varsigma}_i),
 \mathcal D\Phi(T;t_i,\!Y^{\varsigma}_i)I^0\mathrm {Id}_{d\times d}
 k_b^{1}(Z^{\beta_1}_{i,r}(s))\nn\\
 &\quad\tilde \sigma\big\>\big\|_{3,2}\Big\}\mathrm du\mathrm d\nu_3^{1}(s)\mathrm d\beta_1\mathrm dr\mathrm d\varsigma\leq K\Delta,
\end{align*}
where we used the assumption 
$$\sup_{l\in\{1,2,3\}}\sup_{\phi\in\mathcal C^d}\Big(\|D^lk_b^{1}(\phi)\|_{\mathcal L((\mathcal C^d)^{\otimes l}; \mathcal C([-\tau,0]; \RR^{d\times d}))}+\|k_b^{1}(\phi)\|\Big)\leq K.$$

Hence, we have  $\mathcal I_{b,2}\leq K\Delta.$

\textbf{Reestimate of $\mathcal I_{b,\theta}.$} 
Similar to the proof of $\mathcal I_b,$ we can derive that $\mathcal I_{b,\theta}\leq K\Delta.$

\textbf{Reestimate of $\mathcal I_{\sigma}.$} 
When the noise of \eqref{FF} is additive, the term $\mathcal I_{\sigma}\equiv0$.

Combining the above terms, we complete the proof.
\end{proof}

\begin{proof}[Proof of Theorem \ref{dens_con}]
Based on the combination of Theorem \ref{thm_weak2} and \eqref{dens_error}, we complete the proof. 
\end{proof}

\section{Summary and outlook}

This chapter presents the numerical approximation of the density function of the SFDE. We first prove the existence of the density function for the $\theta$-EM  solution. For the SFDE with the multiplicative noise, we show the convergence of the numerical density function by the localization argument. Then by establishing the test-functional-independent weak convergence, we show that for the  additive noise case, the numerical density function  converges to the exact one with the convergence rate $1$. 
 There are some problems that remain to be solved, for example,
\begin{itemize}
\item[(\romannumeral1)]  can we obtain the  convergence rate with order $1$ of the numerical density function for the multiplicative noise case? 
\item[(\romannumeral2)] For the SFDE driven by the rough path, does the density function of the exact solution exist, and further is it smooth?
\item[(\romannumeral3)] If (\romannumeral2) holds, can we obtain the approximation result for the underlying density function?
\end{itemize}

These problems are challenging. We would like to mention that authors in \cite{Friz} prove the existence of the density function for the rough stochastic differential equation under the H\"ormander condition. But there is no work for the case of the SFDE. 
We leave them as open problems and attempt to solve them in the future.

    \cleardoublepage
    
%
%
%


\chapter{Large deviation principle of numerical solution}
In this chapter, we focus on the $\theta$-EM method of the SFDE \eqref{FF} with small noise and investigate the LDP of the underlying  numerical solution. 
By the weak convergence approach, 
 we deduce the Freidlin--Wentzell LDP of the $\theta$-EM solution on the infinite time horizon under the suplerlinearly growing drift coefficient condition.  
Based on the estimates of negative moments of the determinant of the corresponding Malliavin covariance matrix, 
we prove  the existence and smoothness of the numerical density function for the $\theta$-EM method of \eqref{FF} with small noise. Combining the technique of the Malliavin calculus, we present the  logarithmic estimates of the numerical density function and reveal the relation between the logarithmic limit and the rate function of the LDP. 

\section{The Freidlin--Wentzell large deviation  of numerical solution}
In this section, we consider the SFDE \eqref{FF} with small noise on the infinite time horizon, which is written as follows
\begin{align}\label{FF_small}
\begin{cases}
\mathrm dx^{\xi,\epsilon}(t)=b(x^{\xi,\epsilon}_t)\mathrm dt+\sqrt{\epsilon}\sigma(x^{\xi,\epsilon}_t)\mathrm dW(t),\quad t>0,\\
x^{\xi,\epsilon}_0=\xi\in\mathcal C^d,
\end{cases}
\end{align}
where $\epsilon\in(0,1)$.
The aim is to study the  Freidlin--Wentzell  large deviation  of the $\theta$-EM method of \eqref{FF_small} by means of the weak convergence approach. The $\theta$-EM method of \eqref{FF_small} reads as 
\begin{align}
\label{EM_epsilon}
\begin{cases}
y^{\Delta,\epsilon}(t_{k+1})=y^{\Delta,\epsilon}(t_k)+(1-\theta)b(y^{\Delta,\epsilon}_{t_k})\Delta+\theta b(y^{\Delta,\epsilon}_{t_{k+1}})\Delta+\sqrt{\epsilon} \sigma(y^{\Delta,\epsilon}_{t_k})\delta W_k,\;k\in\mathbb N,\\
y^{\Delta,\epsilon}(t_k)=\xi(t_k),\quad k=-N,\ldots,0.
\end{cases}
\end{align}

We denote 
$$L^2_{\Delta}([0,+\infty);\mathbb R^m):=\Big\{v: [0,+\infty)\rightarrow \RR^{m} \text{ is measurable}\; \Big|\int_0^{+\infty}|v(\lfloor s\rfloor)|^2\mathrm ds<\infty\Big\}.$$  For any $M\in(0,+\infty)$, 
set $$\mathcal S_M:=\Big\{v\in L^2_{\Delta}([0,+\infty);\mathbb R^m)\Big|\int_0^{+\infty}|v(\lfloor s\rfloor)|^2\mathrm ds\leq M\Big\}$$ and $$\mathcal P_M:=\{v:\Omega\times [0,+\infty)\to\mathbb R^m|v\text{ is }\mathcal F_t\text{-measurable and }v\in\mathcal S_M \text{ a.s.}\}.$$ It follows from 
\cite[Chapter \uppercase\expandafter{\romannumeral3}, Section 28, Theorem 1']{weaktopo}
that $\mathcal S_M$ is a Polish space endowed with the weak topology. 
Introduce the space $\mathcal C_{\xi}:=\mathcal C_{\xi}([-\tau,+\infty);\mathbb R^d)$ consisting of continuous functions $u: [-\tau,\infty)\to\mathbb R^d$ with $u(r)=\xi(r),\;r\in[-\tau,0],$  
which is a polish space endowed with the norm 
\begin{align*}
\|u\|_{\mathcal C_{\xi}}:=\sum_{k=1}^{\infty}e^{-\mathfrak {a} t_k}\big(\sup_{t\in[-\tau,t_k]}|u(t)|\big)\Delta
\end{align*}
for some constant $\mathfrak a>0$.

Denote $\{y^{\Delta,\epsilon}(\cdot)\}$ by the linear interpolation  with respect to $y^{\Delta,\epsilon}(-t_N),\ldots, y^{\Delta,\epsilon}(t_k),\ldots$ We will prove in \cref{estimate_y2} that $\{y^{\Delta,\epsilon}(\cdot)\}$ belongs to $\mathcal C_{\xi}$.  


In order to present the LDP for the solution of \eqref{EM_epsilon}, we consider the following stochastic
controlled equation
\begin{align}\label{control}
\begin{cases}
y^{\Delta,v^{\epsilon}}(t_{k+1})=y^{\Delta,v^{\epsilon}}(t_k)+(1-\theta)b(y^{\Delta,v^{\epsilon}}_{t_k})\Delta+\theta b(y^{\Delta,v^{\epsilon}}_{t_{k+1}})\Delta\\
\quad\quad\quad\quad\quad\quad+\sigma(y^{\Delta,v^{\epsilon}}_{t_k})\big(\sqrt{\epsilon}\delta W_k+v^{\epsilon}(t_k)\Delta\big),\;k\in\mathbb N,\\
y^{\Delta,v^{\epsilon}}(t_k)=\xi(t_k),\quad k=-N,\ldots,0,
\end{cases}
\end{align}
and the skeleton equation
\begin{align}\label{skeleton}
\begin{cases}
w^{\Delta,v}(t_{k+1})=w^{\Delta,v}(t_k)\!+\!(1\!-\!\theta)b(w^{\Delta,v}_{t_k})\Delta\!+\!\theta b(w^{\Delta,v}_{t_{k+1}})\Delta\!+\!\sigma(w^{\Delta,v}_{t_k})v(t_k)\Delta,\;k\in\mathbb N,\\
w^{\Delta,v}(t_k)=\xi(t_k),\quad k=-N,\ldots,0,
\end{cases}
\end{align}
where $v^{\epsilon},v\in L^2_{\Delta}([0,+\infty);\mathbb R^m)$, and 
$y^{\Delta, v^{\epsilon}}_{t_k}$ and $w^{\Delta, v}_{t_k}$ are defined by : for $r\in[t_j,t_{j+1}]$, $j\in\{-N,\ldots,-1\}$,
\begin{align*}
y^{\Delta, v^{\epsilon}}_{t_k}(r)&:=\frac{t_{j+1}-r}{\Delta}y^{\Delta, v^{\epsilon}}(t_{k+j})+\frac{r-t_j}{\Delta}y^{\Delta, v^{\epsilon}}(t_{k+j+1}),\nn\\
w^{\Delta, v}_{t_k}(r)&:=\frac{t_{j+1}-r}{\Delta}w^{\Delta, v}(t_{k+j})+\frac{r-t_j}{\Delta}w^{\Delta, v}(t_{k+j+1}).
\end{align*}
 The continuous versions for solutions  of \eqref{control} and \eqref{skeleton} are defined by the linear interpolations of the solutions, respectively. 
Define measurable maps  $\mathcal G^{\Delta,\epsilon},\mathcal G^{\Delta}:\mathcal C([0,+\infty);\mathbb R^m)\to\mathcal C_{\xi}([-\tau,+\infty);\mathbb R^d)$ respectively, by  $$\mathcal G^{\Delta,\epsilon}\Big(\sqrt{\epsilon}W+\int _0^{\cdot}v^{\epsilon}(\lfloor s\rfloor)\mathrm ds\Big)=y^{\Delta,v^{\epsilon}},\quad  \mathcal G^{\Delta}\Big(\int_0^{\cdot}v(\lfloor s\rfloor)\mathrm ds\Big)=w^{\Delta,v}.$$ 
Introduce the auxiliary process
\begin{align}\label{auxi}
\begin{cases}
z^{\Delta,v}(t_k)=\xi(t_k),\quad k=-N,\ldots,-1,\\
z^{\Delta,v}(t_k)=w^{\Delta,v}(t_k)-\theta b(w^{\Delta,v}_{t_k})\Delta,\quad k= 0,\\
z^{\Delta,v}(t_k)=z^{\Delta,v}(t_{k-1})+b(w^{\Delta,v}_{t_{k-1}})\Delta+\sigma(w^{\Delta,v}_{t_{k-1}})v(t_{k-1})\Delta,\quad k\in\mathbb N_+
\end{cases}
\end{align}
and the continuous version $\{z^{\Delta,v}(t)\}_{t\geq0}$ as follows: for $t\in [t_k,t_{k+1})$ with $k\in\mathbb N,$ 
\begin{align*}
z^{\Delta,v}(t)=z^{\Delta,v}(t_k)+b(w^{\Delta,v}_{t_k})(t-t_k)+\sigma(w^{\Delta,v}_{t_k})v(t_k)(t-t_k),
\end{align*}
for $t\in[-\tau,0]$, $z^{\Delta,v}(t)=\xi(t)$.
\begin{thm}\label{main_LDP}
Let  \cref{a1,a2,a4,a5} hold. Then for any $\Delta\in(0,\frac{1}{4\theta a_2}]$, the family of random variables $\{y^{\Delta,\epsilon}(\cdot)\}_{\epsilon\in(0,1)}$ satisfies  the LDP in $\mathcal C_{\xi}([-\tau,+\infty);\mathbb R^d)$, i.e.,
\begin{itemize}
\item[(\romannumeral1)] for each closed subset $F$ of $\mathcal C_{\xi}([-\tau,+\infty);\mathbb R^d)$, 
\begin{align*}
\limsup_{\epsilon\to0}\epsilon\log\mathbb P\big(y^{\Delta,\epsilon}(\cdot)\in F\big)\leq-\inf_{x\in F}I(x);
\end{align*}
\item[(\romannumeral2)] for each open  subset $O$ of $\mathcal C_{\xi}([-\tau,+\infty);\mathbb R^d)$, 
\begin{align*}
\liminf_{\epsilon\to0}\epsilon\log\mathbb P\big(y^{\Delta,\epsilon}(\cdot)\in O\big)\ge-\inf_{x\in O}I(x),
\end{align*}
\end{itemize}
 where the good rate function is given  by
\begin{align*}
I(f)=\inf_{\{v\in L^2_{\Delta}([0,+\infty);\mathbb R^m):f=\mathcal G^{\Delta}(\int_0^{\cdot}v(\lfloor s\rfloor)\mathrm ds)\}}\frac12\int_0^{+\infty}|v(\lfloor s\rfloor)|^2\mathrm ds.
\end{align*}
\end{thm}

According to  the contraction principle (see e.g. \cite[Theorem 4.2.1]{Dembo}), we obtain the following LDP for the family $\{y^{\Delta,\epsilon}(t)\}_{\epsilon\in(0,1)}$ for each $t>0$. 
\begin{coro}\label{coro_LDP1}
Under  conditions in Theorem \ref{main_LDP},
for any $t>0,$ the family of random variables $\{y^{\Delta,\epsilon}(t)\}_{\epsilon\in(0,1)}$ satisfies the LDP in $\mathbb R^d$ with the good rate function given   by
\begin{align*}
\tilde I(z)=\inf\{I(f):f\in\mathcal C_{\xi}([-\tau,\infty);\mathbb R^d),z=f(t)\},
\end{align*}
where $I(\cdot)$ is given in Theorem \ref{main_LDP}.
\end{coro}

The proof of Theorem \ref{main_LDP} is  based on the equivalence of the LDP and the Laplace principle. According to   \cite[Condition 2.1, Theorem 2.10]{ZhengZH}, it suffices to prove the compactness of solutions of the  skeleton equation and the stochastic controlled equation (i.e., Propositions \ref{propKM} and \ref{prop_cond}).  
We first give the \textit{a priori} estimates of $w^{\Delta, v}$ and $z^{\Delta, v}$.

\begin{lemma}\label{bound_control}
Let $M>0$ and $v\in  \mathcal S_M$. Under \cref{a1,a2,a5},  it holds that for any $\Delta\in(0,\frac{1}{4\theta a_2}]$,
$$\sup_{k\ge 0}|w^{\Delta,v}(t_k)|+\sup_{k\ge 0}|z^{\Delta,v}(t_k)|\leq(1+\|\xi\|^{\beta+1}) K(M).$$
\end{lemma}
\begin{proof} 
By the definition of $z^{\Delta,v}$, we have 
\begin{align*}
&\quad\frac12\mathrm d|z^{\Delta,v}(t)|^2=\langle b(w^{\Delta,v}_{\lfloor t\rfloor}),z^{\Delta,v}(t)\rangle\mathrm dt+\langle \sigma (w^{\Delta,v}_{\lfloor t\rfloor})v(\lfloor t\rfloor),z^{\Delta,v}(t)\rangle\mathrm dt\\
&=\Big\langle b(w^{\Delta,v}_{\lfloor t\rfloor}),
w^{\Delta,v}(\lfloor t\rfloor)-\theta b(w^{\Delta,v}_{\lfloor t\rfloor})\Delta
+b(w^{\Delta,v}_{\lfloor t\rfloor})(t-\lfloor t\rfloor)+\sigma(w^{\Delta,v}_{\lfloor t\rfloor})v(\lfloor t\rfloor)(t-\lfloor t\rfloor)
\Big\rangle\mathrm dt\nn\\
&\quad+\Big\langle \sigma (w^{\Delta,v}_{\lfloor t\rfloor})v(\lfloor t\rfloor),z^{\Delta,v}(\lfloor t\rfloor)+b(w^{\Delta,v}_{\lfloor t\rfloor})(t-\lfloor t\rfloor)+\sigma(w^{\Delta,v}_{\lfloor t\rfloor})v(\lfloor t\rfloor)(t-\lfloor t\rfloor)\Big\rangle\mathrm dt\\
&= \Big[(t-\lfloor t\rfloor-\theta\Delta)|b(w^{\Delta,v}_{\lfloor t\rfloor})|^2+ \big\langle b(w^{\Delta,v}_{\lfloor t\rfloor}),w^{\Delta,v}({\lfloor t\rfloor})\big\rangle +2\big\langle b(w^{\Delta,v}_{\lfloor t\rfloor}),\sigma (w^{\Delta,v}_{\lfloor t\rfloor})\nn\\
&\quad v(\lfloor t\rfloor)(t-\lfloor t\rfloor)\big\rangle+|\sigma(w^{\Delta,v}_{\lfloor t\rfloor})v(\lfloor t\rfloor)|^2(t-\lfloor t\rfloor)+\langle \sigma(w^{\Delta,v}_{\lfloor t\rfloor})v(\lfloor t\rfloor),z^{\Delta,v}(\lfloor t\rfloor)\rangle\Big]\mathrm dt.
\end{align*}
This implies
\begin{align*}
&\quad |z^{\Delta,v}(t_{k+1})|^2=|z^{\Delta,v}(t_k)|^2\!+\!2\int_{t_k}^{t_{k+1}}|b(w^{\Delta,v}_{t_k})|^2(t\!-\!t_k-\theta \Delta)\mathrm dt\!+\!2\int_{t_k}^{t_{k+1}}\big\langle b(w^{\Delta,v}_{t_k}),\nn\\
&\quad w^{\Delta,v}(t_k)\big\rangle \mathrm dt+\int_{t_k}^{t_{k+1}}\big\langle 4b(w^{\Delta,v}_{t_k})(t-t_k)+2z^{\Delta,v}(t_k),\sigma(w^{\Delta,v}_{t_k})v(t_k)\big\rangle\mathrm dt\\
&\quad +2\int_{t_k}^{t_{k+1}}|\sigma(w^{\Delta,v}_{t_k})v(t_k)|^2(t-t_k)\mathrm dt\\
&= |z^{\Delta,v}(t_k)|^2+\frac{1-2\theta}{\theta^2}|w^{\Delta,v}(t_k)-z^{\Delta,v}(t_k)|^2+2\int_{t_k}^{t_{k+1}}\big\langle b(w^{\Delta,v}_{t_k}),w^{\Delta,v}(t_k)\big\rangle \mathrm dt\\
&\quad +\int_{t_k}^{t_{k+1}}2\big\langle \frac{1}{\theta}(w^{\Delta,v}({t_k})-z^{\Delta,v}(t_k))+z^{\Delta,v}(t_k),\sigma(w^{\Delta,v}_{t_k})v(t_k)\big\rangle \mathrm dt\nn\\
&\quad+\Delta \int_{t_k}^{t_{k+1}}|\sigma(w^{\Delta, v}_{t_k})v(t_k)|^2\mathrm dt\\
&\leq \Big(\frac{(1-\theta)^2}{\theta^2}+\frac{2\theta-1}{\theta^2(1+\tilde c\Delta)}\Big)|z^{\Delta,v}(t_k)|^2+\Big(\frac{2\theta-1}{\theta^2}(1+\tilde c\Delta)+\frac{1-2\theta}{\theta^2}\Big)|w^{\Delta,v}(t_k)|^2\\
&\quad +2\int_{t_k}^{t_{k+1}}\big\langle b(w^{\Delta,v}_{t_k}),w^{\Delta,v}(t_k)\big\rangle\mathrm dt+\Delta \int_{t_k}^{t_{k+1}}|\sigma(w^{\Delta,v}_{t_k})v(t_k)|^2\mathrm dt\\
&\quad +\int_{t_k}^{t_{k+1}}2\big\langle \frac{1}{\theta}(w^{\Delta,v}({t_k})-z^{\Delta,v}(t_k))+z^{\Delta,v}(t_k),\sigma(w^{\Delta,v}_{t_k})v(t_k)\big\rangle \mathrm dt,
\end{align*}
where in the second equality  we used  $z^{\Delta,v}(\lfloor t\rfloor)=w^{\Delta,v}(\lfloor t\rfloor)-\theta b(w^{\Delta,v}_{\lfloor t\rfloor})\Delta$ and $\int_{t_k}^{t_{k+1}}(t-t_k)\mathrm dt=\frac{1}{2}\Delta^{2}$, and in the last step we used the inequality $2\langle a,b\rangle\leq (1+\tilde c\Delta)|a|^2+\frac{1}{1+\tilde c\Delta}|b|^2$ for $\tilde c>0$. 
Therefore, by iterating, we arrive at that for $\iota\in(0,1)$,
\begin{align*}
&\quad e^{\iota  t_{k+1}}|z^{\Delta,v}(t_{k+1})|^2=|z^{\Delta,v}(0)|^2+\sum_{i=0}^k\Big(e^{\iota t_{i+1}}|z^{\Delta,v}(t_{i+1})|^2-e^{\iota t_i}|z^{\Delta,v}(t_i)|^2\Big)\\
&\leq |z^{\Delta,v}(0)|^2+\Big(\frac{(1-\theta)^2}{\theta^2}+\frac{2\theta-1}{\theta^2(1+\tilde c\Delta)}-e^{-\iota \Delta}\Big)\sum_{i=0}^ke^{\iota t_{i+1}}|z^{\Delta,v}(t_{i})|^2\\
&\quad +\Delta\sum_{i=0}^ke^{\iota t_{i+1}}\Big[\frac{2\theta-1}{\theta^2}\tilde c|w^{\Delta,v}(t_i)|^2+2 \langle b(w^{\Delta,v}_{t_i}),w^{\Delta,v}(t_i)\rangle +\Delta |\sigma (w^{\Delta,v}_{t_i})|^2|v(t_i)|^2\\
&\quad +2\Big \langle \frac{1}{\theta}(w^{\Delta,v}(t_i)-z^{\Delta,v}(t_i))+z^{\Delta,v}(t_i),\sigma (w^{\Delta,v}_{t_i})v(t_i)\Big \rangle \Big].
\end{align*}
On account of \eqref{varep}, take sufficiently small numbers  $\iota,\tilde c>0$ independent of $\Delta$ such that $\frac{(1-\theta)^2}{\theta^2}+\frac{2\theta-1}{\theta^2(1+\tilde c\Delta)}-e^{-\iota \Delta}\leq 0.$ 
Then, utilizing  the H\"older inequality, \eqref{F3r1}, and \cref{a5}, we see that  
\begin{align*}
&\quad e^{\iota t_{k+1}}|z^{\Delta,v}(t_{k+1})|^2\leq |z^{\Delta,v}(0)|^2+\int_0^{t_{k+1}}e^{\iota \lceil t\rceil }\Big[\frac{2\theta-1}{\theta^2}\tilde c|w^{\Delta,v}(\lfloor t\rfloor)|^2+2\langle b(w^{\Delta,v}_{\lfloor t\rfloor}),\nn\\
&\quad w^{\Delta,v}(\lfloor t\rfloor)\rangle 
+\Delta|\sigma (w^{\Delta,v}_{\lfloor t\rfloor})|^2 |v(\lfloor t\rfloor)|^2+
\Big|\frac{1}{\theta}w^{\Delta,v}(\lfloor t\rfloor)+\frac{\theta-1}{\theta}z^{\Delta,v}(\lfloor t\rfloor)\Big|^2|v(\lfloor t\rfloor)|^2\nn\\
&\quad+|\sigma(w^{\Delta,v}_{\lfloor t\rfloor})|^2 \Big]\mathrm dt\\
&\leq |\xi(0)-\theta\Delta b(\xi)|^2+\int_0^{t_{k+1}}e^{\iota \lceil t \rceil}\Big[K+(\frac{2\theta-1}{\theta^2}\tilde c-a_1-a_2)|w^{\Delta,v}(\lfloor t\rfloor)|^2\\
&\quad +\frac{a_1-a_2+L}{2}\int_{-\tau}^0|w^{\Delta,v}_{\lfloor t\rfloor}(r)|^2\mathrm d\nu_1(r)+2a_2\int_{-\tau}^0|w^{\Delta,v}_{\lfloor t\rfloor}(r)|^2\mathrm d\nu_2(r)\Big]\mathrm dt\\
&\quad +K\int_0^{t_{k+1}}e^{\iota \lceil t\rceil}\Big(|w^{\Delta,v}(\lfloor t\rfloor)|^2+|z^{\Delta,v}(\lfloor t\rfloor)|^2+\int_{-\tau}^0|w^{\Delta,v}_{\lfloor t\rfloor}(r)|^2\mathrm d\nu_1(r)\Big)|v(\lfloor t\rfloor)|^2\mathrm dt\\
&\leq (1+\|\xi\|^{2(\beta+1)})+\!\int_0^{t_{k+1}}\!e^{\iota \lceil t\rceil}\Big[(\frac{2\theta\!-\!1}{\theta^2}\tilde c\!-\!a_1\!-\!a_2)\!+\!(\frac{a_1\!-\!a_2\!+\!L}{2}\!+\!2a_2)e^{\iota \tau}\Big]\times\nn\\
&\quad|w^{\Delta,v}(\lfloor t\rfloor)|^2\Big]\mathrm dt+Ke^{\iota t_{k+1}}+(\frac{a_1-a_2+L}{2}+2a_2)e^{\iota \tau}\tau\|\xi\|^2\\
&\quad +K\int_0^{t_{k+1}}\Big(\sup_{r\in[-\tau,t]}e^{\iota \lfloor r\rfloor}|w^{\Delta,v}(\lfloor r\rfloor)|^2+\sup_{r\in[0,t]}e^{\iota \lfloor r\rfloor}|z^{\Delta,v}(\lfloor r\rfloor)|^2\Big)|v(\lfloor t\rfloor)|^2\mathrm dt.
\end{align*}
Here in the last step we used 
\begin{align*}
&\quad\int_{0}^{t_{k+1}}\!\!\!e^{\iota\lceil t\rceil}\!\!\int_{-\tau}^{0}\!|w^{\Delta,v}_{\lfloor t\rfloor}(r)|^2\mathrm d\nu_{i}(r)\mathrm dt\nn\\
&\leq e^{\iota \tau}\tau\|\xi\| e^{\iota \tau}\!+\!e^{\iota \tau}\int_{0}^{t_{k+1}}\!\! e^{\iota\lceil t\rceil}|w^{\Delta,v}(\lfloor t\rfloor)|^2\mathrm dt\quad\mbox{for}~ i=1,2,
\end{align*}
whose proof is similar to that of \eqref{5p2.1}. Analogously  to the proof of Proposition \ref{p2.2},  parameters $\tilde c$ and $\iota$ can  further be  taken  such that $\frac{2\theta-1}{\theta^2}\tilde c-a_1-a_2+(\frac{a_1-a_2+L}{2}+2a_2)e^{\iota \tau}\leq 0.$ Consequently, 
\begin{align}\label{extimate_z1}
e^{\iota t_{k+1}}|z^{\Delta,v}(t_{k+1})|^2&\leq K(1+\|\xi\|^{2(\beta+1)}+e^{\iota t_{k+1}})+K\int_0^{t_{k+1}}\Big(\sup_{r\in[-\tau,t]}e^{\iota \lfloor r\rfloor}|w^{\Delta,v}(\lfloor r\rfloor)|^2\nn\\
&\quad+\sup_{r\in[0,t]}e^{\iota \lfloor r\rfloor}|z^{\Delta,v}(\lfloor r\rfloor)|^2\Big)|v(\lfloor t\rfloor)|^2\mathrm dt.
\end{align}
It follows from \eqref{F1r1} and \eqref{auxi} that
\begin{align*}
&\quad |z^{\Delta,v}(t_{k+1})|^2\ge |w^{\Delta,v}(t_{k+1})|^2-2\theta\Delta\langle w^{\Delta,v}(t_{k+1}),b(w^{\Delta,v}_{t_{k+1}})\rangle\\
&\ge |w^{\Delta,v}(t_{k+1})|^2\!-\!\theta \Delta\Big(K-2a_1|w^{\Delta,v}(t_{k+1})|^2\!+\!2a_2\int_{-\tau}^0|w^{\Delta,v}_{t_{k+1}}(r)|^2\mathrm d\nu_2(r)\Big)\\
&\geq |w^{\Delta,v}(t_{k+1})|^2-K\theta \Delta-2\theta a_2\Delta\int_{-\tau}^0|w^{\Delta,v}_{t_{k+1}}(r)|^2\mathrm d\nu_2(r),
\end{align*}
which implies
\begin{align*}
&\quad e^{\iota t_{k+1}} |z^{\Delta,v}(t_{k+1})|^2\\
&\ge e^{\iota t_{k+1}} |w^{\Delta,v}(t_{k+1})|^2-K\theta \Delta e^{\iota t_{k+1}}-2\theta a_2e^{\iota\tau}\Delta\sup_{t\in[-\tau,t_{k+1}]} e^{\iota \lfloor t\rfloor} |w^{\Delta,v}(\lfloor t\rfloor)|^2.
\end{align*}
Then we arrive at
\begin{align*}
e^{\iota t_{k+1}} |w^{\Delta,v}(t_{k+1})|^2&\leq K\theta \Delta e^{\iota t_{k+1}}+2\theta a_2e^{\iota\tau}\Delta\sup_{t\in[-\tau,t_{k+1}]}e^{\iota \lfloor t\rfloor} |w^{\Delta,v}(\lfloor t\rfloor)|^2 \nn\\
&\quad+K(1+\|\xi\|^{2(\beta+1)}+e^{\iota t_{k+1}})+K\int_0^{t_{k+1}}\Big(\sup_{r\in[-\tau,t]}e^{\iota \lfloor r\rfloor}|w^{\Delta,v}(\lfloor r\rfloor)|^2\nn\\
&\quad+\sup_{r\in[0,t]}e^{\iota \lfloor r\rfloor}|z^{\Delta,v}(\lfloor r\rfloor)|^2\Big)|v(\lfloor t\rfloor)|^2\mathrm dt.
\end{align*}
Choose a sufficiently small number  $\iota$  such that $e^{\iota \tau}\leq \frac{3}{2}$.
Then for $\Delta\in(0,\frac{1}{4\theta a_2}],$
\begin{align*}
&\quad\frac14\sup_{t\in[-\tau,t_{k+1}]}e^{\iota \lfloor t\rfloor} |w^{\Delta,v}(\lfloor t\rfloor)|^2\nn\\
&\leq K(1+\|\xi\|^{2(\beta+1)}+e^{\iota t_{k+1}})+K\int_0^{t_{k+1}}\Big(\sup_{r\in[-\tau,t]}e^{\iota \lfloor r\rfloor}|w^{\Delta,v}(\lfloor r\rfloor)|^2\nn\\
&\quad+\sup_{r\in[0,t]}e^{\iota \lfloor r\rfloor}|z^{\Delta,v}(\lfloor r\rfloor)|^2\Big)|v(\lfloor t\rfloor)|^2\mathrm dt, 
\end{align*}
which, together with \eqref{extimate_z1} gives that
\begin{align*}
&\quad \sup_{t\in[-\tau,t_{k+1}]}e^{\iota \lfloor t\rfloor} |w^{\Delta,v}(\lfloor t\rfloor)|^2+\sup_{t\in[0,t_{k+1}]}e^{\iota \lfloor t\rfloor} |z^{\Delta,v}(\lfloor t\rfloor)|^2\\&\leq K(1+\|\xi\|^{2(\beta+1)}+e^{\iota t_{k+1}})+K\int_0^{t_{k+1}}\Big(\sup_{r\in[-\tau,t]}e^{\iota \lfloor r\rfloor}|w^{\Delta,v}(\lfloor r\rfloor)|^2\nn\\
&\quad+\sup_{r\in[0,t]}e^{\iota \lfloor r\rfloor}|z^{\Delta,v}(\lfloor r\rfloor)|^2\Big)|v(\lfloor t\rfloor)|^2\mathrm dt.
\end{align*}
Applying the Gr\"onwall inequality and using $v\in\mathcal S_{M}$, we deduce  
\begin{align*}
&\sup_{t\in[-\tau,t_{k+1}]}e^{\iota \lfloor t\rfloor} |w^{\Delta,v}(\lfloor t\rfloor)|^2+\sup_{t\in[0,t_{k+1}]}e^{\iota \lfloor t\rfloor} |z^{\Delta,v}(\lfloor t\rfloor)|^2\nn\\
\leq&\, K(1+\|\xi\|^{2(\beta+1)}+e^{\iota t_{k+1}})e^{K\int_{0}^{t_{k+1}}|v(\lfloor t\rfloor)|^2\mathrm dt}
\leq K(1+\|\xi\|^{2(\beta+1)}+e^{\iota t_{k+1}})e^{KM}.
\end{align*}
This implies that $\sup_{k\ge 0}|w^{\Delta,v}(t_k)|+\sup_{k\ge 0}|z^{\Delta,v}(t_k)|\leq (1+\|\xi\|^{\beta+1})K(M)$.
The proof is finished. 
\end{proof}

\begin{prop}\label{propKM}
Let $M>0.$ Under Assumptions \ref{a1}, \ref{a2}, and \ref{a5}, for any 
 $\Delta\in(0,\frac{1}{4\theta a_2}]$, $\mathcal{K}_{M}:=\{\mathcal G^{\Delta}(\int_0^{\cdot}v(\lfloor s\rfloor)\mathrm ds)|v\in\mathcal S_M\}$ is compact in $\mathcal C_{\xi}([-\tau,+\infty);\mathbb R^d).$
\end{prop}
\begin{proof} 
Since $\mathcal S_M$ is compact with respect to the weak topology, it suffices to show that the mapping $\mathcal G^{\Delta}\circ F$ is continuous, where
$F: L^2_{\Delta}([0,+\infty);\mathbb R^m)\rightarrow \mathcal C([0,\infty);\RR^m)$ is defined by
$F(v)(\cdot):=\int_0^{\cdot}v(\lfloor s\rfloor)\mathrm ds$. Namely, letting  $v^{\epsilon}\to v$ in $\mathcal S_M$, we shall prove $w^{\Delta, v^{\epsilon}}\to w^{\Delta, v}$ in $\mathcal C_{\xi}([-\tau,+\infty);\mathbb R^d).$
First, we show that for any $k\in\mathbb N_+,$ $\sup_{t\in[0,t_k]}|w^{\Delta, v^{\epsilon}}(\lfloor t\rfloor)-w^{\Delta, v}(\lfloor t\rfloor)|\to0$. 
By \eqref{auxi} and the continuous version of $z^{\Delta, v}(\cdot),$ we have 
\begin{align*}
&\quad \frac12\mathrm d|z^{\Delta, v^{\epsilon}}(t)-z^{\Delta, v}(t)|^2\\
&=\langle b(w^{\Delta, v^{\epsilon}}_{\lfloor t\rfloor})-b(w^{\Delta, v}_{\lfloor t\rfloor}),z^{\Delta, v^{\epsilon}}(t)-z^{\Delta, v}(t)\rangle\mathrm dt\nn\\
&\quad+\langle \sigma(w^{\Delta, v^{\epsilon}}_{\lfloor t\rfloor})v^{\epsilon}(\lfloor t\rfloor)-\sigma(w^{\Delta, v}_{\lfloor t\rfloor})v(\lfloor t\rfloor),z^{\Delta, v^{\epsilon}}(t)-z^{\Delta, v}(t)\rangle\mathrm dt\\
&=\langle b(w^{\Delta, v^{\epsilon}}_{\lfloor t\rfloor})-b(w^{\Delta, v}_{\lfloor t\rfloor}),b(w^{\Delta, v^{\epsilon}}_{\lfloor t\rfloor})-b(w^{\Delta, v}_{\lfloor t\rfloor})\rangle (t-\lfloor t\rfloor)\mathrm dt\\
&\quad +\langle b(w^{\Delta, v^{\epsilon}}_{\lfloor t\rfloor})-b(w^{\Delta, v}_{\lfloor t\rfloor}),\sigma(w^{\Delta, v^{\epsilon}}_{\lfloor t\rfloor})v^{\epsilon}(\lfloor t\rfloor)-\sigma(w^{\Delta, v}_{\lfloor t\rfloor})v(\lfloor t\rfloor)\rangle(t-\lfloor t\rfloor)\mathrm dt\\
&\quad +\langle b(w^{\Delta, v^{\epsilon}}_{\lfloor t\rfloor})-b(w^{\Delta, v}_{\lfloor t\rfloor}),w^{\Delta, v^{\epsilon}}(\lfloor t\rfloor)-w^{\Delta, v}(\lfloor t\rfloor)\rangle\mathrm dt\\
&\quad -\theta\Delta  \langle b(w^{\Delta, v^{\epsilon}}_{\lfloor t\rfloor})-b(w^{\Delta, v}_{\lfloor t\rfloor}),b(w^{\Delta, v^{\epsilon}}_{\lfloor t\rfloor})-b(w^{\Delta, v}_{\lfloor t\rfloor})\rangle\mathrm dt\\
&\quad +\langle \sigma(w^{\Delta, v^{\epsilon}}_{\lfloor t\rfloor})v^{\epsilon}(\lfloor t\rfloor)-\sigma(w^{\Delta, v}_{\lfloor t\rfloor})v(\lfloor t\rfloor),z^{\Delta, v^{\epsilon}}(t)-z^{\Delta, v}(t)\rangle\mathrm dt. 
\end{align*}
This implies 
\begin{align*}
&\quad |z^{\Delta, v^{\epsilon}}(t_{k+1})-z^{\Delta, v}(t_{k+1})|^2\nn\\
&= |z^{\Delta, v^{\epsilon}}(t_k)-z^{\Delta, v}(t_k)|^2+(1-2\theta)\Delta^2|b(w^{\Delta, v^{\epsilon}}_{t_k})-b(w^{\Delta, v}_{t_k})|^2\\
&\quad +\int_{t_k}^{t_{k+1}}2\langle b(w^{\Delta, v^{\epsilon}}_{\lfloor t\rfloor})-b(w^{\Delta, v}_{\lfloor t\rfloor}),w^{\Delta, v^{\epsilon}}(\lfloor t\rfloor)-w^{\Delta, v}(\lfloor t\rfloor)\rangle\mathrm dt\\
&\quad +\int_{t_k}^{t_{k+1}}2\Big\langle (\sigma (w^{\Delta, v^{\epsilon}}_{\lfloor t\rfloor})-\sigma(w^{\Delta, v}_{\lfloor t\rfloor}))v^{\epsilon}(\lfloor t\rfloor),z^{\Delta, v^{\epsilon}}(t)-z^{\Delta, v}(t)\nn\\
&\quad+\frac12 \Delta \big(b(w^{\Delta, v^{\epsilon}}_{\lfloor t\rfloor})-b(w^{\Delta, v}_{\lfloor t\rfloor})\big)\Big\rangle\mathrm dt\\
&\quad +\int_{t_k}^{t_{k+1}}2\Big\langle \sigma(w^{\Delta, v}_{\lfloor t\rfloor})(v^{\epsilon}(\lfloor t\rfloor)-v(\lfloor t\rfloor)),z^{\Delta, v^{\epsilon}}(t)-z^{\Delta, v}(t)\nn\\
&\quad+\frac12 \Delta \big(b(w^{\Delta, v^{\epsilon}}_{\lfloor t\rfloor})-b(w^{\Delta, v}_{\lfloor t\rfloor})\big)\Big\rangle\mathrm dt,
\end{align*}
where we used
\begin{align}\label{rebzw}
b(w^{\Delta, v^{\epsilon}}_{\lfloor t\rfloor})-b(w^{\Delta, v}_{\lfloor t\rfloor})=\frac{1}{\theta\Delta}\Big(w^{\Delta, v^{\epsilon}}(\lfloor t\rfloor)-z^{\Delta, v^{\epsilon}}(\lfloor t\rfloor)-\big(w^{\Delta, v^{\epsilon}}(\lfloor t\rfloor)-z^{\Delta, v}(\lfloor t\rfloor)\big)\Big).
\end{align}
Taking into account   $1-2\theta\leq0$ and 
\begin{align*}
&\quad(1-2\theta)\Delta^2|b(w^{\Delta, v^{\epsilon}}_{t_k})-b(w^{\Delta, v}_{t_k})|^2\nn\\
&\leq\frac{(1-2\theta)}{\theta^2}(1-\frac{1}{(1+c_{\lambda}\Delta)})|z^{\Delta, v^{\epsilon}}(t_k)-z^{\Delta, v}(t_k)|^2\nn\\
&\quad+\frac{(1\!-\!2\theta)}{\theta^2}(1\!-\!(1\!+\!c_{\lambda}\Delta))|w^{\Delta, v^{\epsilon}}(t_k)\!-\!w^{\Delta, v^{\epsilon}}(t_k)|^2,
\end{align*}
we obtain
\begin{align}\label{com1}
&\quad |z^{\Delta, v^{\epsilon}}(t_{k+1})-z^{\Delta, v}(t_{k+1})|^2\notag\nn\\
&\leq \big(\frac{(1\!-\!\theta)^2}{\theta^2}\!+\!\frac{2\theta\!-\!1}{\theta^2(1\!+\!c_{\lambda}\Delta)}\big)|z^{\Delta, v^{\epsilon}}(t_{k})\!\!-\!\!z^{\Delta, v}(t_{k})|^2\!+\!c_{\lambda}\Delta|w^{\Delta, v^{\epsilon}}(t_k)\!\!-\!\!w^{\Delta, v}(t_k)|^2\nn\\
&\quad +\int_{t_k}^{t_{k+1}}2\langle b(w^{\Delta, v^{\epsilon}}_{\lfloor t\rfloor})-b(w^{\Delta, v}_{\lfloor t\rfloor}),w^{\Delta, v^{\epsilon}}(\lfloor t\rfloor)-w^{\Delta, v}(\lfloor t\rfloor)\rangle\mathrm dt\nn\\
&\quad +\int_{t_k}^{t_{k+1}}2\Big\langle (\sigma (w^{\Delta, v^{\epsilon}}_{\lfloor t\rfloor})-\sigma(w^{\Delta, v}_{\lfloor t\rfloor}))v^{\epsilon}(\lfloor t\rfloor),z^{\Delta, v^{\epsilon}}(t)-z^{\Delta, v}(t)\nn\\
&\quad+\frac12 \Delta \big(b(w^{\Delta, v^{\epsilon}}_{\lfloor t\rfloor})-b(w^{\Delta, v}_{\lfloor t\rfloor})\big)\Big\rangle\mathrm dt
+\int_{t_k}^{t_{k+1}}2\Big\langle \sigma(w^{\Delta, v}_{\lfloor t\rfloor})(v^{\epsilon}(\lfloor t\rfloor)-v(\lfloor t\rfloor)),\nn\\
&\quad z^{\Delta, v^{\epsilon}}(t)-z^{\Delta, v}(t)+\frac12 \Delta \big(b(w^{\Delta, v^{\epsilon}}_{\lfloor t\rfloor})-b(w^{\Delta, v}_{\lfloor t\rfloor})\big)\Big\rangle\mathrm dt,
\end{align}
where the constant $c_{\lambda}$ is defined by \eqref{lamb}. Utilizing  \eqref{auxi} and the continuous version of $z^{\Delta, v}(\cdot)$ again,  and applying the Young inequality, one has
\begin{align}\label{com2}
&\quad\int_{t_k}^{t_{k+1}}\!\!\!2\Big\langle (\sigma (w^{\Delta, v^{\epsilon}}_{\lfloor t\rfloor})\!\!-\!\!\sigma(w^{\Delta, v}_{\lfloor t\rfloor}))v^{\epsilon}(\lfloor t\rfloor),z^{\Delta, v^{\epsilon}}(t)\!\!-\!\!z^{\Delta, v}(t)\!\!+\!\!\frac12 \Delta \big(b(w^{\Delta, v^{\epsilon}}_{\lfloor t\rfloor})\!\!-\!\!b(w^{\Delta, v}_{\lfloor t\rfloor})\big)\Big\rangle\mathrm dt\nn\\
&=\!\int_{t_k}^{t_{k+1}}\!\!2\Big\langle (\sigma (w^{\Delta, v^{\epsilon}}_{\lfloor t\rfloor})\!\!-\!\!\sigma(w^{\Delta, v}_{\lfloor t\rfloor}))v^{\epsilon}(\lfloor t\rfloor),z^{\Delta, v^{\epsilon}}(\lfloor t\rfloor)\!\!-\!\!z^{\Delta, v}(\lfloor t\rfloor)\nn\\
&\quad+ \frac{\Delta}{2} \big(b(w^{\Delta, v^{\epsilon}}_{\lfloor t\rfloor})\!-\!b(w^{\Delta, v}_{\lfloor t\rfloor})\big)\Big\rangle\mathrm dt
+\int_{t_k}^{t_{k+1}}|\big(\sigma (w^{\Delta, v^{\epsilon}}_{\lfloor t\rfloor})-\sigma(w^{\Delta, v}_{\lfloor t\rfloor})\big)v^{\epsilon}(\lfloor t\rfloor)|^2\Delta\mathrm dt
\nn\\
&\quad+\int_{t_k}^{t_{k+1}}\big\langle (\sigma (w^{\Delta, v^{\epsilon}}_{\lfloor t\rfloor})-\sigma(w^{\Delta, v}_{\lfloor t\rfloor}))v^{\epsilon}(\lfloor t\rfloor),\sigma(w^{\Delta, v}_{\lfloor t\rfloor}))(v^{\epsilon}(\lfloor t\rfloor)-v(\lfloor t\rfloor))\big\>\Delta\mathrm dt\nn\\
&\leq\int_{t_k}^{t_{k+1}}|\sigma (w^{\Delta, v^{\epsilon}}_{\lfloor t\rfloor})-\sigma(w^{\Delta, v}_{\lfloor t\rfloor})|^2\mathrm dt+K \int_{t_k}^{t_{k+1}}|v^{\epsilon}(\lfloor t\rfloor)|^2\big(|z^{\Delta, v^{\epsilon}}(\lfloor t\rfloor)-z^{\Delta, v}(\lfloor t\rfloor)|^2\nn\\
&\quad+|w^{\Delta, v^{\epsilon}}(\lfloor t\rfloor)-w^{\Delta, v}(\lfloor t\rfloor)|^2\big)\mathrm dt
+\int_{t_k}^{t_{k+1}}|\big(\sigma (w^{\Delta, v^{\epsilon}}_{\lfloor t\rfloor})-\sigma(w^{\Delta, v}_{\lfloor t\rfloor})\big)v^{\epsilon}(\lfloor t\rfloor)|^2\Delta\mathrm dt\nn\\
&\quad+\int_{t_k}^{t_{k+1}}\big\langle (\sigma (w^{\Delta, v^{\epsilon}}_{\lfloor t\rfloor})-\sigma(w^{\Delta, v}_{\lfloor t\rfloor}))v^{\epsilon}(\lfloor t\rfloor),\sigma(w^{\Delta, v}_{\lfloor t\rfloor}))(v^{\epsilon}(\lfloor t\rfloor)-v(\lfloor t\rfloor))\big\>\Delta\mathrm dt,
\end{align}
where we used $\int_{t_{k}}^{t_{k+1}}(t-\lfloor t\rfloor)\mathrm dt=\frac{1}{2}\Delta^2$.
Inserting \eqref{com2}  into \eqref{com1}, and then by iterating,  we obtain that for $\gamma\in(0,1)$,
\begin{align*}
&\quad e^{\gamma t_{k+1}}|z^{\Delta, v^{\epsilon}}(t_{k+1})-z^{\Delta, v}(t_{k+1})|^2\nn\\
&=\sum_{i=0}^k\Big(e^{\gamma t_{i+1}}|z^{\Delta, v^{\epsilon}}(t_{i+1})-z^{\Delta, v}(t_{i+1})|^2-e^{\gamma t_i}|z^{\Delta, v^{\epsilon}}(t_{i})-z^{\Delta, v}(t_{i})|^2\Big)\\
&\leq \sum_{i=0}^{k}e^{\gamma t_{i+1}}\Big[\big(\frac{(1-\theta)^2}{\theta^2}+\frac{2\theta-1}{\theta^2(1+c_{\lambda}\Delta)}-e^{-\gamma\Delta}\big)|z^{\Delta, v^{\epsilon}}(t_{i})-z^{\Delta, v}(t_{i})|^2\nn\\
&\quad+c_{\lambda}\Delta|w^{\Delta, v^{\epsilon}}(t_i)-w^{\Delta, v}(t_i)|^2\Big]+\int_{0}^{t_{k+1}}e^{\gamma \lceil t\rceil}\Big[2\langle b(w^{\Delta, v^{\epsilon}}_{\lfloor t\rfloor})-b(w^{\Delta, v}_{\lfloor t\rfloor}),w^{\Delta, v^{\epsilon}}(\lfloor t\rfloor)\nn\\
&\quad-w^{\Delta, v}(\lfloor t\rfloor)\rangle
+|\sigma (w^{\Delta, v^{\epsilon}}_{\lfloor t\rfloor})-\sigma(w^{\Delta, v}_{\lfloor t\rfloor})|^2\Big]\mathrm dt+K \int_{0}^{t_{k+1}}e^{\gamma \lceil t\rceil}|v^{\epsilon}(\lfloor t\rfloor)|^2\times\nn\\
&\quad\Big[|z^{\Delta, v^{\epsilon}}(\lfloor t\rfloor)\!-\!z^{\Delta, v}(\lfloor t\rfloor)|^2\!+\!|w^{\Delta, v^{\epsilon}}(\lfloor t\rfloor)\!-\!w^{\Delta, v}(\lfloor t\rfloor)|^2\Big]\mathrm dt+\int_{0}^{t_{k+1}}e^{\gamma \lceil t\rceil}\big|\sigma (w^{\Delta, v^{\epsilon}}_{\lfloor t\rfloor})\nn\\
&\quad-\sigma(w^{\Delta, v}_{\lfloor t\rfloor})\big|^2|v^{\epsilon}(\lfloor t\rfloor)|^2\Delta\mathrm dt+J(t_{k+1}),
\end{align*}
where $J$ is defined by 
\begin{align*}
J(s):=&\int_{0}^{s}2e^{\gamma \lceil t\rceil}\Big\langle \sigma(w^{\Delta, v}_{\lfloor t\rfloor})(v^{\epsilon}(\lfloor t\rfloor)-v(\lfloor t\rfloor)),\frac{1}{2}\Delta(\sigma (w^{\Delta, v^{\epsilon}}_{\lfloor t\rfloor})-\sigma(w^{\Delta, v}_{\lfloor t\rfloor}))v^{\epsilon}(\lfloor t\rfloor)\nn\\
&\quad+z^{\Delta, v^{\epsilon}}(t)-z^{\Delta, v}(t)+\frac12 \Delta \big(b(w^{\Delta, v^{\epsilon}}_{\lfloor t\rfloor})-b(w^{\Delta, v}_{\lfloor t\rfloor})\big)\Big\rangle\mathrm dt.
\end{align*}
It follows Assumptions \ref{a1} and \ref{a2} that 
\begin{align*}
&\quad e^{\gamma t_{k+1}}|z^{\Delta, v^{\epsilon}}(t_{k+1})-z^{\Delta, v}(t_{k+1})|^2\nn\\
&\leq \sum_{i=0}^{k}e^{\gamma t_{i+1}}\big(\frac{(1-\theta)^2}{\theta^2}+\frac{2\theta-1}{\theta^2(1+c_{\lambda}\Delta)}-e^{-\gamma\Delta}\big)|z^{\Delta, v^{\epsilon}}(t_{i})-z^{\Delta, v}(t_{i})|^2\nn\\
&\quad+\sum_{i=0}^{k}-\big(2a_1-L-c_{\lambda}-(L+2a_2)e^{\gamma\tau}\big)\Delta e^{\gamma t_{i+1}}|w^{\Delta, v^{\epsilon}}(t_i)-w^{\Delta, v}(t_i)|^2\Big]\\
&\quad+\!K \!\int_{0}^{t_{k+1}}\!e^{\gamma \lceil t\rceil}|v^{\epsilon}(\lfloor t\rfloor)|^2\Big[|z^{\Delta, v^{\epsilon}}(\lfloor t\rfloor)\!-\!z^{\Delta, v}(\lfloor t\rfloor)|^2\!+\!|w^{\Delta, v^{\epsilon}}(\lfloor t\rfloor)\!-\!w^{\Delta, v}(\lfloor t\rfloor)|^2\Big]\mathrm dt\nn\\
&\quad+2Le^{\gamma\Delta}\int_{0}^{t_{k+1}}\sup_{s\in[0,t]}e^{\gamma \lfloor s\rfloor}|w^{\Delta, v^{\epsilon}}(\lfloor s\rfloor)-w^{\Delta, v}(\lfloor s\rfloor)|^2|v^{\epsilon}(\lfloor t\rfloor)|^2\Delta\mathrm dt+J(t_{k+1}).
\end{align*}
Similar to the proof of Proposition \ref{p2.3}, one can take a sufficiently small number $\gamma$   independent of $\Delta$ such that $\frac{(1-\theta)^2}{\theta^2}+\frac{2\theta-1}{\theta^2(1+c_{\lambda}\Delta)}-e^{-\gamma\Delta}\leq0$ and $2a_1-L-c_{\lambda}-(L+2a_2)e^{\gamma\tau}\geq0$.
Then, we see that  
\begin{align}\label{com3}
&\quad e^{\gamma t_{k+1}}|z^{\Delta, v^{\epsilon}}(t_{k+1})-z^{\Delta, v}(t_{k+1})|^2\nn\\
&\leq K \int_{0}^{t_{k+1}}|v^{\epsilon}(\lfloor t\rfloor)|^2\sup_{s\in[0,t]}\Big[e^{\gamma \lfloor s\rfloor}|z^{\Delta, v^{\epsilon}}(\lfloor s\rfloor)-z^{\Delta, v}(\lfloor s\rfloor)|^2\nn\\
&\quad+e^{\gamma \lfloor s\rfloor}|w^{\Delta, v^{\epsilon}}(\lfloor s\rfloor)-w^{\Delta, v}(\lfloor t\rfloor)|^2\Big]\mathrm dt
+J(t_{k+1}).
\end{align}
Denote $E^{\epsilon}(t)=\int_0^t\sigma (w^{\Delta, v}_{\lfloor s\rfloor})(v^{\epsilon}(\lfloor s\rfloor)-v(\lfloor s\rfloor))\mathrm ds $. Applying the integration by parts formula, $v^{\epsilon}(\cdot)\in\mathcal S_{M}$, Assumptions \ref{a1} and \ref{a5}, and Lemma \ref{bound_control}, we arrive at
\begin{align}\label{com5}
&\quad J(t_{k+1})
=2e^{\gamma \lceil t\rceil}\Big\langle E^{\epsilon}(t),\frac{1}{2}\Delta(\sigma (w^{\Delta, v^{\epsilon}}_{\lfloor t\rfloor})-\sigma(w^{\Delta, v}_{\lfloor t\rfloor}))v^{\epsilon}(\lfloor t\rfloor)+z^{\Delta, v^{\epsilon}}(t)-z^{\Delta, v}(t)\nn\\
&\quad+\frac12 \Delta \big(b(w^{\Delta, v^{\epsilon}}_{\lfloor t\rfloor})\!-\!b(w^{\Delta, v}_{\lfloor t\rfloor})\big)\Big\rangle\Big|_{t=0}^{t=t_{k+1}}\!-\!2\int_0^{t_{k+1}}e^{\gamma \lceil t\rceil} \Big\langle E^{\epsilon}(t),(z^{\Delta, v^{\epsilon}}(t)\!-\!z^{\Delta, v}(t))'\Big\rangle\mathrm dt\nn\\
&\leq e^{\gamma t_{k+2}}|E^{\epsilon}(t_{k+1})||\sigma (w^{\Delta, v^{\epsilon}}_{t_{k+1}})-\sigma(w^{\Delta, v}_{t_{k+1}})|\int_{t_{k+1}}^{t_{k+2}}|v^{\epsilon}(\lfloor t\rfloor)|\mathrm dt+Ke^{\gamma t_{k+1}}|E^{\epsilon}(t_{k+1})|^2\nn\\
&\quad+\frac14\sup_{t\in[0,t_{k+1}]}\!e^{\gamma \lfloor t\rfloor}|z^{\Delta, v^{\epsilon}}(\lfloor t\rfloor)\!-\!z^{\Delta, v}(\lfloor t\rfloor)|^2\!+\!\frac14 \sup_{t\in[0,t_{k+1}]}e^{\gamma \lfloor t\rfloor}|w^{\Delta, v^{\epsilon}}(\lfloor t\rfloor)\!-\!w^{\Delta, v}(\lfloor t\rfloor)|^2\nn\\
&\quad -2\int_0^{t_{k+1}}e^{\gamma \lceil t\rceil}\Big\<E^{\epsilon}(t),b(w^{\Delta, v^{\epsilon}}_{\lfloor t\rfloor})-b(w^{\Delta, v}_{\lfloor t\rfloor}) +\sigma(w^{\Delta, v^{\epsilon}}_{\lfloor t\rfloor})v^{\epsilon}(\lfloor t\rfloor)-\sigma(w^{\Delta, v}_{\lfloor t\rfloor})v(\lfloor t\rfloor)\Big\>\mathrm dt\nn\\
&\leq K(M)e^{\gamma t_{k+1}}\sup_{k\ge -N}(1+|w^{\Delta, v^{\epsilon}}(t_k)|+|w^{\Delta, v}(t_k)|)|E^{\epsilon}(t_{k+1})|+Ke^{\gamma t_{k+1}}|E^{\epsilon}(t_{k+1})|^2\nn\\
&\quad+\frac14\sup_{t\in[0,t_{k+1}]}\!e^{\gamma \lfloor t\rfloor}|z^{\Delta, v^{\epsilon}}(\lfloor t\rfloor)\!-\!z^{\Delta, v}(\lfloor t\rfloor)|^2\!+\!\frac14 \sup_{t\in[0,t_{k+1}]}\!e^{\gamma \lfloor t\rfloor}|w^{\Delta, v^{\epsilon}}(\lfloor t\rfloor)\!-\!w^{\Delta, v}(\lfloor t\rfloor)|^2\nn\\
&\quad +K\int_0^{t_{k+1}}e^{\gamma \lceil t\rceil}|E^{\epsilon}(t)|\Big[1+\sup_{k\ge -N}(|w^{\Delta, v^{\epsilon}}(t_k)|^{\beta+1}+|w^{\Delta, v}(t_k)|^{\beta+1})\nn\\
&\quad+(1+\sup_{k\ge -N}|w^{\Delta, v^{\epsilon}}(t_k)|)|v^{\epsilon}(\lfloor t\rfloor)|
+(1+\sup_{k\ge -N}|w^{\Delta, v}(t_k)|)|v (\lfloor t\rfloor)|\Big]\mathrm dt\nn\\
&\leq K(M)e^{\gamma t_{k+1}}|E^{\epsilon}(t_{k+1})|\!+\!K(M)e^{\gamma t_{k+1}}\!\sup_{t\in[0,t_{k+1}]}|E^{\epsilon}(t)|^2\!+\!\frac14\sup_{t\in[0,t_{k+1}]}e^{\gamma \lfloor t\rfloor}|z^{\Delta, v^{\epsilon}}(\lfloor t\rfloor)\nn\\
&\quad-z^{\Delta, v}(\lfloor t\rfloor)|^2
+\frac14 \sup_{t\in[0,t_{k+1}]}e^{\gamma \lfloor t\rfloor}|w^{\Delta, v^{\epsilon}}(\lfloor t\rfloor)-w^{\Delta, v}(\lfloor t\rfloor)|^2, 
\end{align}
where we used the Young  inequality and the H\"older inequality. 
Inserting \eqref{com5} into \eqref{com3} and together with 
 \begin{align*}
 &\quad e^{\gamma t_{k+1}}|z^{\Delta, v^{\epsilon}}(t_{k+1})-z^{\Delta, v}(t_{k+1})|^2\\
 &\ge e^{\gamma t_{k+1}}|w^{\Delta, v^{\epsilon}}(t_{k+1})\!-\!w^{\Delta, v}(t_{k+1})|^2\!-\!2\theta e^{\gamma t_{k+1}} \Delta\<w^{\Delta, v^{\epsilon}}(t_{k+1})\!-\!w^{\Delta, v}(t_{k+1}),\nn\\
 &\quad b(w^{\Delta, v^{\epsilon}}_{t_{k+1}})\!-\!b(w^{\Delta, v}_{t_{k+1}})\>\\
 &\ge e^{\gamma t_{k+1}}|w^{\Delta, v^{\epsilon}}(t_{k+1})\!-\!w^{\Delta, v}(t_{k+1})|^2\!-\!2\theta e^{\gamma t_{k+1}}\Delta \Big(\!-\!a_1|w^{\Delta, v^{\epsilon}}(t_{k+1})\!-\!w^{\Delta, v}(t_{k+1})|^2\\
 &\quad +a_2\int_{-\tau}^0|w^{\Delta, v^{\epsilon}}_{t_{k+1}}(r)-w^{\Delta, v}_{t_{k+1}}(r)|^2\mathrm d\nu_2(r)\Big)\nn\\
 &\ge e^{\gamma t_{k+1}} |w^{\Delta, v^{\epsilon}}(t_{k+1})-w^{\Delta, v}(t_{k+1})|^2\nn\\
 &\quad-2\theta a_2e^{\gamma(\tau+\Delta)}\Delta\sup_{t\in[0,t_{k+1}]}e^{\gamma \lfloor t\rfloor}|w^{\Delta, v^{\epsilon}}(\lfloor t\rfloor)\!-\!w^{\Delta, v}(\lfloor t\rfloor)|^2, 
 \end{align*}
 we obtain that for $\Delta\in(0,\frac{1}{4a_2\theta}]$,
 \begin{align*}
 &\quad e^{\gamma t_{k+1}}|w^{\Delta, v^{\epsilon}}(t_{k+1})-w^{\Delta, v}(t_{k+1})|^2\\
 &\leq K \int_{0}^{t_{k+1}}|v^{\epsilon}(\lfloor t\rfloor)|^2\sup_{s\in[0,t]}\Big[e^{\gamma \lfloor s\rfloor}|z^{\Delta, v^{\epsilon}}(\lfloor s\rfloor)-z^{\Delta, v}(\lfloor s\rfloor)|^2+e^{\gamma \lfloor s\rfloor}|w^{\Delta, v^{\epsilon}}(\lfloor s\rfloor)\nn\\
 &\quad-w^{\Delta, v}(\lfloor s\rfloor)|^2\Big]\mathrm dt+\frac{1}{4}\sup_{t\in(0,t_{k+1]}}e^{\gamma\lfloor t\rfloor}|z^{\Delta, v^{\epsilon}}(\lfloor t\rfloor)-z^{\Delta, v}(\lfloor t\rfloor)|^2\mathrm dt\nn\\
 &\quad+\frac{7}{8}\sup_{t\in(0,t_{k+1]}}e^{\gamma\lfloor t\rfloor}|w^{\Delta, v^{\epsilon}}(\lfloor t\rfloor)-w^{\Delta, v}(\lfloor t\rfloor)|^2\mathrm dt\nn\\
&\quad+K(M) e^{\gamma t_{k+1}}\sup_{t\in[0,t_{k+1}]}|E^{\epsilon}(t)|+K(M) e^{\gamma t_{k+1}}\sup_{t\in[0,t_{k+1}]}|E^{\epsilon}(t)|^2,
 \end{align*} 
 where we chose $\gamma$ sufficiently small such that $e^{\gamma(\tau+1)}\leq \frac{5}{4}$.
It follows from the Gr\"onwall inequality and $v^{\epsilon}(\cdot)\in\mathcal S_{M}$  that 
 \begin{align}\label{leadsto1}
&\quad\sup_{t\in[0,t_{k+1}]}e^{\gamma \lfloor t\rfloor}|z^{\Delta, v^{\epsilon}}(\lfloor t\rfloor)\!-\!z^{\Delta, v}(\lfloor t\rfloor)|^2\!+\!\sup_{t\in[0,t_{k+1}]}e^{\gamma \lfloor t\rfloor}|w^{\Delta, v^{\epsilon}}(\lfloor t\rfloor)\!-\!w^{\Delta, v}(\lfloor t\rfloor)|^2 \nn\\
&\leq K(M)e^{\gamma t_{k+1}}\sup_{t\in[0,t_{k+1}]}|E^{\epsilon}(t)|+K(M)e^{\gamma t_{k+1}}\sup_{t\in[0,t_{k+1}]}|E^{\epsilon}(t)|^2.
 \end{align}
By virtue of
$$|\int_{r_1}^{r_2}\sigma(w^{\Delta, v}_{\lfloor s\rfloor})(v^{\epsilon}(\lfloor s\rfloor)-v(\lfloor s\rfloor))\mathrm ds|\leq K(M)\sqrt{|r_2-r_1|},$$ we deduce the equicontinuity of $\{E^{\epsilon}(t), t\in[0, +\infty)\}_{\epsilon\in(0,1)}$. According to \cref{a1}, \cref{bound_control}, $v^{\epsilon}(\cdot), v(\cdot)\in\mathcal S_M$, and the H\"older inequality, one has
$$\sup_{\epsilon\in(0,1)}\sup_{t\in[0,+\infty)}|\int_0^{+\infty}\sigma(w^{\Delta, v}_{\lfloor s\rfloor})(v^{\epsilon}(\lfloor s\rfloor)-v(\lfloor s\rfloor))\mathrm ds|\leq K(M)(1+\|\xi\|^{\beta+1}) ,$$ which implies the uniform boundedness of $\{E^{\epsilon}(\cdot)\}_{\epsilon\in(0,1)}$. Hence, according to the Ascoli--Arzela theorem, we obtain  that $\{E^{\epsilon}(\cdot)\}_{\epsilon\in(0,1)}$ is compact in $\mathcal C([0,t_{k+1}];\mathbb R^d).$ 
Note that 
\begin{align*}
\sigma(w^{\Delta, v}_{\lfloor \cdot\rfloor})(v^{\epsilon}(\lfloor \cdot\rfloor)-v(\lfloor \cdot\rfloor))\to0 \quad \mbox{as} \quad\epsilon\rightarrow 0
\end{align*} with respect to the weak topology of $\mathcal S_{M'}$ for $M'=M K(M)(1+\|\xi\|^{2(\beta+1)})$.
By \cite[Chapter \uppercase\expandafter{\romannumeral6}, Proposition 3.3]{FunAnal}, we conclude  that
$\sup_{t\in[0,t_{k+1})}|E^{\epsilon}(t)|\to0$. 
This, together with \eqref{leadsto1} implies that for each $k\in\mathbb N_+,$
\begin{align*}
\sup_{t\in[0,t_{k+1}]}|w^{\Delta, v^{\epsilon}}(\lfloor t\rfloor)-w^{\Delta, v}(\lfloor t\rfloor )|^2\to0\text{ as }\epsilon\to0.
\end{align*}
Thus, by the monotone convergence theorem, we arrive at
\begin{align*}
\lim_{\epsilon\to0}\|w^{\Delta, v^{\epsilon}}-w^{\Delta, v}\|_{\mathcal C_{\xi}}&=\lim_{\epsilon\to0}\sum_{k=1}^{\infty}e^{-\mathfrak at_k}\big(\sup_{t\in[0,t_k]}|w^{\Delta, v^{\epsilon}}(\lfloor t\rfloor)-w^{\Delta, v}(\lfloor t\rfloor)|\big)\Delta\\
&=\sum_{k=1}^{\infty}e^{-\mathfrak at_k}\big(\lim_{\epsilon\to0}\sup_{t\in[0,t_k]}|w^{\Delta, v^{\epsilon}}(\lfloor t\rfloor)-w^{\Delta, v}(\lfloor t\rfloor)|\big)\Delta=0.
\end{align*}
This finishes the proof.
\end{proof}

In order to obtain the boundedness of the moments of $y^{\Delta, v^{\epsilon}},$ we 
introduce the auxiliary process 
\begin{align}\label{auxi2}
\begin{cases}
Z^{\Delta, v^{\epsilon}}(t_k)=\xi(t_k),\quad k=-N,\ldots,-1,\\
Z^{\Delta, v^{\epsilon}}(t_k)=y^{\Delta, v^{\epsilon}}(t_k)-\theta b(y^{\Delta, v^{\epsilon}}_{t_k})\Delta,\quad k= 0,\\
Z^{\Delta, v^{\epsilon}}(t_k)=Z^{\Delta, v^{\epsilon}}(t_{k-1})+b(y^{\Delta, v^{\epsilon}}_{t_{k-1}})\Delta+\sqrt{\epsilon}\sigma(y^{\Delta, v^{\epsilon}}_{t_{k-1}})\delta W_{k-1}\\
\quad\quad\quad\quad\quad\;\;+\sigma(y^{\Delta, v^{\epsilon}}_{t_{k-1}})v^{\epsilon}(t_{k-1})\Delta,\quad k\in\mathbb N_+.
\end{cases}
\end{align}
Define  the continuous version $Z^{\Delta, v^{\epsilon}}(\cdot)$ as follows:
for $t\in[t_{k}, t_{k+1})$ with $k\in\mathbb N$,
\begin{align*}
Z^{\Delta, v^{\epsilon}}(t)&:=Z^{\Delta, v^{\epsilon}}(t_{k})+b(y^{\Delta, v^{\epsilon}}_{t_{k}})(t-t_k)+\sqrt{\epsilon}\sigma(y^{\Delta, v^{\epsilon}}_{t_{k}})( W(t)-W(t_{k}))\nn\\
&\quad+\sigma(y^{\Delta, v^{\epsilon}}_{t_{k}})v^{\epsilon}(t_{k})(t-t_k),
\end{align*}
and for $t\in[-\tau,0]$, $Z^{\Delta, v^{\epsilon}}(t):=\xi(t)$.

\begin{lemma}\label{estimate_y2}
Let $M>0$ and  $\{v^{\epsilon}\}_{\epsilon\in(0,1)}\subset \mathcal P_M.$ Under \cref{a1,a2,a4,a5}, for any $\Delta\in(0,1]$ and  $p\in\mathbb N_+,$ we have 
\begin{align*}
&\quad\mathbb E\Big[\sup_{t\in[0,t_{k+1}]}|Z^{\Delta,v^{\epsilon}}(\lfloor t\rfloor)|^{2p}\Big]+\mathbb E\Big[\sup_{t\in[0,t_{k+1}]}|y^{\Delta,v^{\epsilon}}(\lfloor t\rfloor)|^{2p}\Big]
\\&\leq  K(M)(1+\|\xi\|^{2p(\beta+1)})e^{\sqrt{\epsilon} K(M)t_{k+1}}.
\end{align*} 
As a consequence, for sufficiently small $\epsilon>0$, 
$Z^{\Delta,v^{\epsilon}}(\cdot)$ and $y^{\Delta,v^{\epsilon}}(\cdot)$ belong to  $L^{2p}(\Omega; \mathcal C_{\xi})$.

\end{lemma}

\begin{proof}
\underline{Step 1.} We show that for $\Delta\in(0,\Delta^*]$ with some sufficiently small $\Delta^*\in(0,1),$
 \begin{align*}
&\quad\mathbb E\Big[\sup_{t\in[0,t_{k+1}]}|Z^{\Delta,v^{\epsilon}}(\lfloor t\rfloor)|^{2p}\Big]+\mathbb E\Big[\sup_{t\in[0,t_{k+1}]}|y^{\Delta,v^{\epsilon}}(\lfloor t\rfloor)|^{2p}\Big]
\\&\leq  K(M)(1+\|\xi\|^{2p(\beta+1)})e^{\sqrt{\epsilon} K(M)t_{k+1}}.
\end{align*} 
It follows from the definition of $Z^{\Delta,v^{\epsilon}}$ and the Young inequality that 
 \begin{align*}
 &\quad |Z^{\Delta,v^{\epsilon}}(t_{k+1})|^2
 =|Z^{\Delta,v^{\epsilon}}(t_{k})|^2+|b(y^{\Delta,v^{\epsilon}}_{t_k})|^2\Delta^2+\epsilon |\sigma(y^{\Delta,v^{\epsilon}}_{t_k})\delta W_k|^2\nn\\
 &\quad+2\<y^{\Delta,v^{\epsilon}}(t_k)-\theta\Delta b(y^{\Delta,v^{\epsilon}}_{t_k}),b(y^{\Delta,v^{\epsilon}}_{t_k})\Delta\>+2\<Z^{\Delta,v^{\epsilon}}(t_k),\sqrt{\epsilon}\sigma(y^{\Delta,v^{\epsilon}}_{t_k})\delta W_k\>\nn\\
 &\quad+2\sqrt{\epsilon}\<b(y^{\Delta,v^{\epsilon}}_{t_k})\Delta,\sigma(y^{\Delta,v^{\epsilon}}_{t_k})\delta W_k\>+|\sigma(y^{\Delta,v^{\epsilon}}_{t_k})v^{\epsilon}(t_k)|^2\Delta^2\\
 &\quad+2\<Z^{\Delta,v^{\epsilon}}(t_k)+b(y^{\Delta,v^{\epsilon}}_{t_k})\Delta+\sqrt{\epsilon}\sigma(y^{\Delta,v^{\epsilon}}_{t_k})\delta W_k,\sigma(y^{\Delta,v^{\epsilon}}_{t_k})v^{\epsilon}(t_k)\Delta\>\\
 &\leq \Big(\frac{(1-\theta)^2}{\theta^2}+\frac{2\theta-1}{\theta^2(1+\Delta)}\Big)|Z^{\Delta,v^{\epsilon}}(t_k)|^2+\Big(\frac{1-2\theta}{\theta^2}+\frac{2\theta-1}{\theta^2}(1+\Delta)\Big)|y^{\Delta,v^{\epsilon}}(t_k)|^2\\
 &\quad+2\Delta \<y^{\Delta,v^{\epsilon}}(t_k),b(y^{\Delta,v^{\epsilon}}_{t_k})\>+\Delta |\sigma(y^{\Delta,v^{\epsilon}}_{t_k})|^2+\Delta |v^{\epsilon}(t_k)|^2\Big(|Z^{\Delta,v^{\epsilon}}(t_k)+b(y^{\Delta,v^{\epsilon}}_{t_k})\Delta|^2\nn\\
 &\quad+|\sigma(y^{\Delta,v^{\epsilon}}_{t_k})|^2\Delta \Big)+\epsilon |\sigma (y^{\Delta,v^{\epsilon}}_{t_k})\delta W_k|^2+2\widetilde M_k, 
 \end{align*}
 where $\widetilde M_k:=\<\sqrt{\epsilon}\sigma (y^{\Delta,v^{\epsilon}}_{t_k})\delta W_k,\sigma (y^{\Delta,v^{\epsilon}}_{t_k})v^{\epsilon}(t_k)\Delta +Z^{\Delta,v^{\epsilon}}(t_k)+b(y^{\Delta,v^{\epsilon}}_{t_k})\Delta \>.$
Recalling the notation $A_{\theta,\Delta}:=\frac{(1-\theta)^2}{\theta^2}+\frac{2\theta-1}{\theta^2(1+\Delta)}$ and using \cref{a1,a4}, we have
\begin{align}\label{derive_eq1}
&\quad|Z^{\Delta,v^{\epsilon}}(t_{k+1})|^2\!\leq \!A_{\theta,\Delta}|Z^{\Delta,v^{\epsilon}}(t_{k})|^2+\frac{2\theta-1}{\theta^2}\Delta |y^{\Delta,v^{\epsilon}}(t_k)|^2+K\Delta\nn\\
&\quad-2a_3\Delta |y^{\Delta,v^{\epsilon}}(t_k)|^{2+\ell}+2L\Delta|y^{\Delta,v^{\epsilon}}(t_k)|^2+2L\Delta \int_{-\tau}^0|y_{t_k}^{\Delta,v^{\epsilon}}(r)|^2\mathrm d\nu_1(r)\nn\\
&\quad+2a_4\Delta \int_{-\tau}^0|y_{t_k}^{\Delta,v^{\epsilon}}(r)|^2\mathrm d\nu_2(r)+\Delta |v^{\epsilon}(t_k)|^2\Big(|Z^{\Delta,v^{\epsilon}}(t_k)+b(y^{\Delta,v^{\epsilon}}_{t_k})\Delta|^2\nn\\
&\quad+|\sigma(y^{\Delta,v^{\epsilon}}_{t_k})|^2\Delta\Big)+\epsilon |\sigma(y_{t_k}^{\Delta,v^{\epsilon}})\delta W_k|^2+2\widetilde{M}_k\nn\\
&\leq A_{\theta,\Delta}|Z^{\Delta,v^{\epsilon}}(t_{k})|^2+K\Delta-a_3\Delta |y^{\Delta,v^{\epsilon}}(t_k)|^{2+l}-R\Delta |y^{\Delta,v^{\epsilon}}(t_k)|^2\nn\\
&\quad+\!2L\Delta \int_{-\tau}^0|y_{t_k}^{\Delta,v^{\epsilon}}(r)|^2\mathrm d\nu_1(r)\!+\!2a_4\Delta \int_{-\tau}^0|y_{t_k}^{\Delta,v^{\epsilon}}(r)|^2\mathrm d\nu_2(r)\!+\!\epsilon |\sigma(y_{t_k}^{\Delta,v^{\epsilon}})\delta W_k|^2\nn\\
&\quad+\Delta |v^{\epsilon}(t_k)|^2\Big(|Z^{\Delta,v^{\epsilon}}(t_k)+b(y^{\Delta,v^{\epsilon}}_{t_k})\Delta|^2+|\sigma(y^{\Delta,v^{\epsilon}}_{t_k})|^2\Delta\Big)+2\widetilde{M}_k,
\end{align}
where $R$ is positive constant with $R>2L.$
From this, we obtain that for $p\in\mathbb N$ with $p>1$,
\begin{align*}
(|Z^{\Delta,v^{\epsilon}}(t_{k+1})|^{2}+a_3\Delta|y^{\Delta,v^{\epsilon}}(t_k)|^{2+\ell})^p\leq A^p_{\theta,\Delta}|Z^{\Delta,v^{\epsilon}}(t_k)|^{2p}+\sum_{i=1}^pC_p^i\mathcal I_i^{\epsilon},
\end{align*}
where
\begin{align*}
\mathcal I_i^{\epsilon}&=A^{p-i}_{\theta,\Delta}|Z^{\Delta,v^{\epsilon}}(t_k)|^{2(p-i)}\Big[K\Delta\!-\!R\Delta |y^{\Delta,v^{\epsilon}}(t_k)|^2\!+\!2L\Delta \int_{-\tau}^0|y_{t_k}^{\Delta,v^{\epsilon}}(r)|^2\mathrm d\nu_1(r)\\
&\quad+2a_4\Delta \int_{-\tau}^0|y_{t_k}^{\Delta,v^{\epsilon}}(r)|^2\mathrm d\nu_2(r)+\epsilon |\sigma(y_{t_k}^{\Delta,v^{\epsilon}})\delta W_k|^2\\
&\quad+\Delta |v^{\epsilon}(t_k)|^2\Big(|Z^{\Delta,v^{\epsilon}}(t_k)+b(y^{\Delta,v^{\epsilon}}_{t_k})\Delta|^2+|\sigma(y^{\Delta,v^{\epsilon}}_{t_k})|^2\Delta\Big)+2\widetilde{M}_k\Big]^i.
\end{align*}
For the term $\mathcal I^{\epsilon}_1,$ applying  the Young inequality yields   that for $\gamma_1\in(0,1),$
\begin{align*}
&\quad\mathcal I_{1}^{\epsilon}\leq \gamma_1\Delta A^p_{\theta,\Delta}|Z^{\Delta,v^{\epsilon}}(t_k)|^{2p}+K(\gamma_1)\Delta-R\Delta A^{p-1}_{\theta,\Delta}|Z^{\Delta,v^{\epsilon}}(t_k)|^{2(p-1)}|y^{\Delta,v^{\epsilon}}(t_k)|^2\\
&\quad +K(\gamma_1)\Delta \int_{-\tau}^0|y^{\Delta,v^{\epsilon}}_{t_k}(r)|^{2p}\mathrm d\nu_1(r)+K(\gamma_1)\Delta \int_{-\tau}^0|y^{\Delta,v^{\epsilon}}_{t_k}(r)|^{2p}\mathrm d\nu_2(r)\\
&\quad +K\Delta |v^{\epsilon}(t_k)|^2\Big(|Z^{\Delta,v^{\epsilon}}(t_k)|^{2p}+|y^{\Delta,v^{\epsilon}}(t_k)|^{2p}+\int_{-\tau}^0|y^{\Delta,v^{\epsilon}}_{t_k}(r)|^{2p}\mathrm d\nu_1(r)\Big)\\
&\quad+A^{p-1}_{\theta,\Delta}|Z^{\Delta,v^{\epsilon}}(t_k)|^{2(p-1)}(\epsilon|\sigma(y^{\Delta,v^{\epsilon}}_{t_k})\delta W_k|^2+2\widetilde{\mathcal M}_k).
\end{align*}
By Assumption \ref{a4}, we have 
\begin{align}\label{y&Z}
&\quad|y^{\Delta,v^{\epsilon}}(t_k)|^2+2a_3\theta\Delta|y^{\Delta,v^{\epsilon}}(t_k)|^{2+\ell}\nn\\
&\leq |Z^{\Delta,v^{\epsilon}}(t_k)|^2+K\Delta+2a_4\theta\Delta\int_{-\tau}^0|y^{\Delta,v^{\epsilon}}_{t_k}(r)|^2\mathrm d\nu_2(r),
\end{align}
which implies 
\begin{align*}
|y^{\Delta,v^{\epsilon}}(t_k)|^{2p}&=\big(|Z^{\Delta,v^{\epsilon}}(t_k)|^2\!+\!K\Delta\!+\!2a_4\theta\Delta\int_{-\tau}^0\!|y^{\Delta,v^{\epsilon}}_{t_k}(r)|^2\mathrm d\nu_2(r)\big)^{2(p\!-\!1)}|y^{\Delta,v^{\epsilon}}(t_k)|^2\nn\\
&\leq 2|Z^{\Delta,v^{\epsilon}}(t_k)|^{2(p-1)}|y^{\Delta,v^{\epsilon}}(t_k)|^2\\
&\quad+K|y^{\Delta,v^{\epsilon}}(t_k)|^2\Big(K\Delta+2a_4\theta\Delta\int_{-\tau}^0|y^{\Delta,v^{\epsilon}}_{t_k}(r)|^2\mathrm d\nu_2(r)\Big)^{p-1}.
\end{align*}
In consequence, 
\begin{align*}
&-|Z^{\Delta,v^{\epsilon}}(t_k)|^{2(p-1)}|y^{\Delta,v^{\epsilon}}(t_k)|^2\leq -\frac{1}{2}|y^{\Delta,v^{\epsilon}}(t_k)|^{2p}\\
&\quad+K |y^{\Delta,v^{\epsilon}}(t_k)|^2\Big(K\Delta+2a_4\theta\Delta\int_{-\tau}^0|y^{\Delta,v^{\epsilon}}_{t_k}(r)|^2\mathrm d\nu_2(r)\Big)^{p-1}.
\end{align*}
Inserting the above inequality  into $\mathcal I_1^{\epsilon}$, we derive
\begin{align*}
&\mathcal I_1^{\epsilon}\leq \gamma_1\Delta A^p_{\theta,\Delta}|Z^{\Delta,v^{\epsilon}}(t_k)|^{2p}+K(\gamma_1)\Delta-\frac12R\Delta A^{p-1}_{\theta,\Delta}|y^{\Delta,v^{\epsilon}}(t_k)|^{2p}\\
&\quad+K(\gamma_1)\Delta\Big(|y^{\Delta,v^{\epsilon}}(t_k)|^{2p}+\int_{-\tau}^0|y^{\Delta,v^{\epsilon}}_{t_k}(r)|^{2p}\mathrm d\nu_1(r)+\int_{-\tau}^0|y^{\Delta,v^{\epsilon}}_{t_k}(r)|^{2p}\mathrm d\nu_2(r)\Big)\\
&\quad  +K\Delta |v^{\epsilon}(t_k)|^2\Big(|Z^{\Delta,v^{\epsilon}}(t_k)|^{2p}+|y^{\Delta,v^{\epsilon}}(t_k)|^{2p}+\int_{-\tau}^0|y^{\Delta,v^{\epsilon}}_{t_k}(r)|^{2p}\mathrm d\nu_1(r)\Big)\\
&\quad+A^{p-1}_{\theta,\Delta}|Z^{\Delta,v^{\epsilon}}(t_k)|^{2(p-1)}(\epsilon|\sigma(y^{\Delta,v^{\epsilon}}_{t_k})\delta W_k|^2+2\widetilde{\mathcal M}_k).
\end{align*}
For the term $\mathcal I^{\epsilon}_i$ with $i\in\{2,\ldots,p\},$
\begin{align*}
&\mathcal I^{\epsilon}_i
\leq \gamma_1\Delta A^p_{\theta,\Delta}|Z^{\Delta,v^{\epsilon}}(t_k)|^{2p}+K(\gamma_1)\Delta^2\Big(1+|y^{\Delta,v^{\epsilon}}(t_k)|^{2p}+\int_{-\tau}^0|y^{\Delta,v^{\epsilon}}_{t_k}(r)|^{2p}\mathrm d\nu_1(r)\nn\\
&\quad+\int_{-\tau}^0|y^{\Delta,v^{\epsilon}}_{t_k}(r)|^{2p}\mathrm d\nu_2(r)\Big)+K(M)\Delta |v^{\epsilon}(t_k)|^2\Big(|Z^{\Delta,v^{\epsilon}}(t_k)|^{2p}+|y^{\Delta,v^{\epsilon}}(t_k)|^{2p}\nn\\
&\quad+\int_{-\tau}^0|y^{\Delta,v^{\epsilon}}_{t_k}|^{2p}\mathrm d\nu_1(r)\Big)+KA^{p-i}_{\theta,\Delta}|Z^{\Delta,v^{\epsilon}}(t_k)|^{2(p-i)}(\epsilon|\sigma(y^{\Delta,v^{\epsilon}}_{t_k})\delta W_k|^{2i}+|\widetilde{\mathcal M}_k|^i),
\end{align*}
where we used $\Delta|v^{\epsilon}(t_k)|^2\leq M.$
Combining the estimates of the terms $\mathcal I^{\epsilon}_i,i=1,\ldots,p$, we arrive at
\begin{align*}
&\quad|Z^{\Delta,v^{\epsilon}}(t_{k+1})|^{2p}\leq (1+K\gamma_1\Delta)A^p_{\theta,\Delta}|Z^{\Delta,v^{\epsilon}}(t_{k})|^{2p}+K(\gamma_1)\Delta +\mathcal J_{k1}(R)+\mathcal J_{k2}\\
&\quad+K(M)\Delta |v^{\epsilon}(t_k)|^2\Big(|Z^{\Delta,v^{\epsilon}}(t_k)|^{2p}+|y^{\Delta,v^{\epsilon}}(t_k)|^{2p}+\int_{-\tau}^0|y^{\Delta,v^{\epsilon}}_{t_k}|^{2p}\mathrm d\nu_1(r)\Big)\\
&\quad+K|Z^{\Delta,v^{\epsilon}}(t_k)|^{2(p-1)}(\epsilon |\sigma(y^{\Delta,v^{\epsilon}}_{t_k})\delta W_k|^{2}+\widetilde {\mathcal M}_k)\\
&\quad+K\sum_{i=2}^pC_p^i|Z^{\Delta,v^{\epsilon}}(t_k)|^{2(p-i)}(\epsilon |\sigma(y^{\Delta,v^{\epsilon}}_{t_k})\delta W_k|^{2i}+|\widetilde {\mathcal M}_k|^i),
\end{align*}
where
\begin{align*}
&\mathcal J_{k1}(R):=-\Big(\frac p2RA^{p-1}_{\theta,\Delta}-K(\gamma_1)\Big)\Delta|y^{\Delta,v^{\epsilon}}(t_k)|^{2p}\\
&\quad\quad\quad\quad\quad\quad+K(\gamma_1)\Delta \Big(\int_{-\tau}^0|y^{\Delta,v^{\epsilon}}_{t_k}(r)|^{2p}\mathrm d\nu_1(r)+\int_{-\tau}^0|y^{\Delta,v^{\epsilon}}_{t_k}(r)|^{2p}\mathrm d\nu_2(r)\Big),\\
&\mathcal J_{k2}:= K(\gamma_1)\Delta^2 \Big(|y^{\Delta,v^{\epsilon}}(t_k)|^{2p}+\int_{-\tau}^0|y^{\Delta,v^{\epsilon}}_{t_k}(r)|^{2p}\mathrm d\nu_1(r)\nn\\
&\quad\quad\quad\quad\quad\quad+\int_{-\tau}^0|y^{\Delta,v^{\epsilon}}_{t_k}(r)|^{2p}\mathrm d\nu_2(r)\Big).
\end{align*}
Then  for  $\gamma_2>0$, 
\begin{align*}
&e^{\gamma_2t_{k+1}}|Z^{\Delta,v^{\epsilon}}(t_{k+1})|^{2p}\!\leq\! |Z^{\Delta,v^{\epsilon}}(0)|^{2p}\!\!+\!\!\sum_{l=0}^k\!\big((1+K\gamma_1\Delta)A^p_{\theta,\Delta}e^{\gamma_2\Delta}\!\!-\!\!1\big)e^{\gamma_2t_{l}}|Z^{\Delta,v^{\epsilon}}(t_{l})|^{2p}\\
&\quad+K(\gamma_1)e^{\gamma_2 t_{k+1}}+\sum_{l=0}^ke^{\gamma_2t_{l+1}}\big(\mathcal J_{l1}(R)+\mathcal J_{l2}\big)+K(M)\sum_{l=0}^ke^{\gamma_2t_{l+1}} \Delta |v^{\epsilon}(t_l)|^2\times\nn\\
&\quad\Big(|Z^{\Delta,v^{\epsilon}}(t_l)|^{2p}+|y^{\Delta,v^{\epsilon}}(t_l)|^{2p}+\int_{-\tau}^0|y^{\Delta,v^{\epsilon}}_{t_l}(r)|^{2p}\mathrm d\nu_1(r)\Big)\\
&\quad+K\sum_{l=0}^ke^{\gamma_2t_{l+1}}|Z^{\Delta,v^{\epsilon}}(t_l)|^{2(p-1)}(\epsilon |\sigma(y^{\Delta,v^{\epsilon}}_{t_l})\delta W_l|^2+\widetilde{\mathcal M}_l)\\
&\quad+K\sum_{l=0}^ke^{\gamma_2t_{l+1}}\sum_{i=2}^pC_p^i|Z^{\Delta,v^{\epsilon}}(t_l)|^{2(p-i)}(\epsilon |\sigma(y^{\Delta,v^{\epsilon}}_{t_l})\delta W_l|^{2i}+|\widetilde{\mathcal M}_l|^i).
\end{align*}
Similar to the proof of \cref{l4.2}, take sufficiently small $\gamma_1$ and $\gamma_2$ independent of $\Delta$ such that $(1+K\gamma_1\Delta)A^p_{\theta,\Delta}e^{\gamma_2\Delta}-1\leq 0.$ Choose $R$ sufficiently large and further  choose $\Delta^*$ sufficiently small. We  obtain that for $\Delta\leq \Delta^*,$
\begin{align*}
\sum_{l=0}^ke^{\gamma_2t_{l+1}}\big(\mathcal J_{l1}(R)+\mathcal J_{l2}\big)\leq K\|\xi\|^{2p}.
\end{align*} 
Hence, 
\begin{align*}
&e^{\gamma_2t_{k+1}}|Z^{\Delta,v^{\epsilon}}(t_{k+1})|^{2p}\leq |Z^{\Delta,v^{\epsilon}}(0)|^{2p}+K e^{\gamma_2t_{k+1}}+K(M)\int_0^{t_{k+1}}|v^{\epsilon}(\lfloor s\rfloor)|^2\times\nn\\
&\quad\Big(\sup_{r\in[0,s]}e^{\gamma_2 \lfloor r\rfloor}|Z^{\Delta,v^{\epsilon}}(\lfloor r\rfloor)|^{2p}+ \sup_{r\in[0,s]}e^{\gamma_2 \lfloor r\rfloor}|y^{\Delta,v^{\epsilon}}(\lfloor r\rfloor)|^{2p}\Big)\mathrm ds\\
&\quad+K\sum_{l=0}^ke^{\gamma_2t_l}|Z^{\Delta,v^{\epsilon}}(t_l)|^{2(p-1)}(\epsilon |\sigma(y^{\Delta,v^{\epsilon}}_{t_l})\delta W_l|^2+\widetilde{\mathcal M}_l)\\
&\quad+K\sum_{l=0}^ke^{\gamma_2t_l}\sum_{i=2}^pC_p^i|Z^{\Delta,v^{\epsilon}}(t_l)|^{2(p-i)}(\epsilon |\sigma(y^{\Delta,v^{\epsilon}}_{t_l})\delta W_l|^{2i}+|\widetilde{\mathcal M}_l|^i).
\end{align*}
Moreover, it follows from \eqref{y&Z} that
\begin{align*}
\sup_{t\in[0,t_{k}]}e^{\gamma_2\lfloor t\rfloor}|y^{\Delta,v^{\epsilon}}(\lfloor t\rfloor)|^{2p}\leq Ke^{\gamma_2t_k}+\sup_{t\in[0,t_{k}]}e^{\gamma_2\lfloor t\rfloor}|Z^{\Delta,v^{\epsilon}}(\lfloor t\rfloor)|^{2p}.
\end{align*}
We conclude
\begin{align}\label{derive1}
&\sup_{t\in[0,t_{k+1}]}e^{\gamma_2\lfloor t\rfloor}|Z^{\Delta,v^{\epsilon}}(\lfloor t\rfloor)|^{2p}+\sup_{t\in[0,t_{k+1}]}e^{\gamma_2\lfloor t\rfloor}|y^{\Delta,v^{\epsilon}}(\lfloor t\rfloor)|^{2p}\nn\\
&\leq Ke^{\gamma_2t_{k+1}}+K(1+\|\xi\|^{2p(\beta+1)})+K(M)\int_0^{t_{k+1}}|v^{\epsilon}(\lfloor s\rfloor)|^2\times\nn\\
&\quad\Big(\sup_{r\in[0,s]}e^{\gamma_2\lfloor r\rfloor}|Z^{\Delta,v^{\epsilon}}(\lfloor r\rfloor)|^{2p}\!+\!\sup_{r\in[0,s]}e^{\gamma_2\lfloor r\rfloor} |y^{\Delta,v^{\epsilon}}(\lfloor r\rfloor)|^{2p}\Big)\mathrm dt\nn\\
&\quad+K\sup_{j\leq k}\sum_{l=0}^je^{\gamma_2t_l}|Z^{\Delta,v^{\epsilon}}(t_l)|^{2(p-1)}(\epsilon |\sigma(y^{\Delta,v^{\epsilon}}_{t_l})\delta W_l|^2+\widetilde{\mathcal M}_l)\nn\\
&\quad+K\sum_{l=0}^ke^{\gamma_2t_l}\sum_{i=2}^pC_p^i|Z^{\Delta,v^{\epsilon}}(t_l)|^{2(p-i)}(\epsilon |\sigma(y^{\Delta,v^{\epsilon}}_{t_l})\delta W_l|^{2i}+|\widetilde{\mathcal M}_l|^i),
\end{align}
where we used $|Z^{\Delta,v^{\epsilon}}(0)|\leq K(1+\|\xi\|^{\beta+1})$, which was deduced by using \cref{a5}.
By virtue of the Gr\"onwall inequality, one has
\begin{align}\label{yield_eq1}
&\quad\sup_{t\in[0,t_{k+1}]}e^{\gamma_2\lfloor t\rfloor}|Z^{\Delta,v^{\epsilon}}(\lfloor t\rfloor)|^{2p}+\sup_{t\in[0,t_{k+1}]}e^{\gamma_2\lfloor t\rfloor}|y^{\Delta,v^{\epsilon}}(\lfloor t\rfloor)|^{2p}\nn\\
&\leq e^{K(M)M}\Big(K(1+\|\xi\|^{2p(\beta+1)})+Ke^{\gamma_2t_{k+1}}+K\sup_{j\leq k}\sum_{l=0}^je^{\gamma_2t_l}|Z^{\Delta,v^{\epsilon}}(t_l)|^{2(p-1)}\times\nn\\
&\quad(\epsilon |\sigma(y^{\Delta,v^{\epsilon}}_{t_l})\delta W_l|^2+\!\widetilde{\mathcal M}_l)+K\sum_{l=0}^ke^{\gamma_2t_l}\sum_{i=2}^pC_p^i|Z^{\Delta,v^{\epsilon}}(t_l)|^{2(p-i)}\times\nn\\
&\quad(\epsilon |\sigma(y^{\Delta,v^{\epsilon}}_{t_l})\delta W_l|^{2i}+|\widetilde{\mathcal M}_l|^i)\Big).
\end{align}
Applying the Burkholder--Davis--Gundy inequality and the Young inequality, we arrive at that for $\gamma_3\in(0,1)$,
\begin{align*}
&\quad\mathbb E\Big[\sup_{j\leq k}\sum_{l=0}^je^{\gamma_2t_l}|Z^{\Delta,v^{\epsilon}}(t_l)|^{2(p-1)}(\epsilon |\sigma(y^{\Delta,v^{\epsilon}}_{t_l})\delta W_l|^2+\widetilde{\mathcal M}_l)\Big]\\
&\leq \epsilon\Delta \mathbb E\Big[\sum_{l=0}^k e^{\gamma_2t_l}|Z^{\Delta,v^{\epsilon}}(t_l)|^{2(p-1)}|\sigma(y^{\Delta,v^{\epsilon}}_{t_l})|^2\Big]+
K\E\Big[\Big(\sum_{l=0}^k e^{2\gamma_2t_l}|Z^{\Delta,v^{\epsilon}}(t_l)|^{4(p-1)}\times\nn\\
&\quad\epsilon|\sigma (y^{\Delta,v^{\epsilon}}_{t_k})|^2\Big(|\sigma (y^{\Delta,v^{\epsilon}}_{t_k})|^2M+|Z^{\Delta,v^{\epsilon}}(t_k)|^2+|y^{\Delta,v^{\epsilon}}(t_k)|^2\Big)\Delta\Big)^{\frac12}\Big]\\
&\leq \sqrt{\epsilon} K(M,\gamma_3)\int_0^{t_{k+1}}\mathbb E\Big[\sup_{r\in[0,s]}e^{\gamma _3\lfloor r\rfloor}|Z^{\Delta,v^{\epsilon}}(\lfloor r\rfloor)|^{2p}+\sup_{r\in[0,s]}e^{\gamma _3\lfloor r\rfloor}|y^{\Delta,v^{\epsilon}}(\lfloor r\rfloor)|^{2p}\Big]\mathrm ds\\
&\quad+ \gamma_3 K(M)\Big(\sup_{t\in[0,t_{k+1}]} e^{\gamma _3\lfloor t\rfloor}|Z^{\Delta,v^{\epsilon}}(\lfloor t\rfloor)|^{2p}\!+\!\sup_{t\in[0,t_{k+1}]} e^{\gamma _3\lfloor t\rfloor}|y^{\Delta,v^{\epsilon}}(\lfloor t\rfloor)|^{2p}\Big)\!+\!K\|\xi\|^{2p}.
\end{align*}
Similarly, 
\begin{align*}
&\quad \mathbb E\Big[\sum_{l=0}^ke^{\gamma_2t_l}\sum_{i=2}^pC_p^i|Z^{\Delta,v^{\epsilon}}(t_l)|^{2(p-i)}(\epsilon |\sigma(y^{\Delta,v^{\epsilon}}_{t_l})\delta W_l|^{2i}+|\widetilde{\mathcal M}_l|^i)\Big]\\
&\leq \epsilon K(M)\int_0^{t_{k+1}}\mathbb E\Big[\sup_{r\in[0,s]}e^{\gamma _3\lfloor r\rfloor}|Z^{\Delta,v^{\epsilon}}(\lfloor r\rfloor)|^{2p}+\sup_{r\in[0,s]}e^{\gamma _3\lfloor r\rfloor}|y^{\Delta,v^{\epsilon}}(\lfloor r\rfloor)|^{2p}\Big]\mathrm ds\\
&\quad+\gamma_3 K(M)\Big(\sup_{t\in[0,t_{k+1}]} e^{\gamma _3\lfloor t\rfloor}|Z^{\Delta,v^{\epsilon}}(\lfloor t\rfloor)|^{2p}\!+\!\sup_{t\in[0,t_{k+1}]} e^{\gamma _3\lfloor t\rfloor}|y^{\Delta,v^{\epsilon}}(\lfloor t\rfloor)|^{2p}\Big)\!+\!K\|\xi\|^{2p}.
\end{align*}
Taking expectation on both sides of \eqref{yield_eq1}, inserting the above two inequalities into \eqref{yield_eq1}, letting $\gamma_3>0$ be sufficiently small, and using  the Gr\"onwall inequality again, we deduce  
\begin{align*}
&\quad\mathbb E\Big[\sup_{t\in[0,t_{k+1}]}e^{\gamma_2 \lfloor t\rfloor}|Z^{\Delta,v^{\epsilon}}(\lfloor t\rfloor)|^{2p}\Big]+\mathbb E\Big[\sup_{t\in[0,t_{k+1}]}e^{\gamma_2 \lfloor t\rfloor}|y^{\Delta,v^{\epsilon}}(\lfloor t\rfloor)|^{2p}\Big]
\\&\leq  K e^{K(M)M}(1+\|\xi\|^{2p(\beta+1)}+e^{\gamma_2t_{k+1}})e^{\sqrt{\epsilon} K(M)t_{k+1}e^{K(M)M}},
\end{align*}
which implies
\begin{align*}
&\quad\mathbb E\Big[\sup_{t\in[0,t_{k+1}]}|Z^{\Delta,v^{\epsilon}}(\lfloor t\rfloor)|^{2p}\Big]+\mathbb E\Big[\sup_{t\in[0,t_{k+1}]}|y^{\Delta,v^{\epsilon}}(\lfloor t\rfloor)|^{2p}\Big]
\\&\leq  K(M)(1+\|\xi\|^{2p(\beta+1)})e^{\sqrt{\epsilon} K(M)t_{k+1}}.
\end{align*}

 We finishes the proof of \underline{Step 1}. 

\underline{Step 2.} We show that for $\Delta\in[\Delta^*,1],$
\begin{align*}
 &\quad\mathbb E\Big[\sup_{t\in[0,t_{k+1}]}|Z^{\Delta,v^{\epsilon}}(\lfloor t\rfloor)|^{2p}\Big]+\mathbb E\Big[\sup_{t\in[0,t_{k+1}]}|y^{\Delta,v^{\epsilon}}(\lfloor t\rfloor)|^{2p}\Big]
\\&\leq  K(M)(1+\|\xi\|^{2p(\beta+1)})e^{\sqrt{\epsilon} K(M)t_{k+1}}.
\end{align*} 
The proof is similar to that of \underline{Step 2} in Lemma \ref{l4.2}, and thus we only give the sketch of the proof.

 Similar  to the estimate of $\mathcal I^{\epsilon}_i$, we have 
\begin{align*}
&\mathcal I_1^{\epsilon}\leq \gamma_1A^p_{\theta,\Delta}|Z^{\Delta,v^{\epsilon}}(t_k)|^{2p}+K(\gamma_1)\Delta-\frac12R\Delta A^{p-1}_{\theta,\Delta}|y^{\Delta,v^{\epsilon}}(t_k)|^{2p}\\
&\quad+K(\gamma_1)\Delta^p\Big(|y^{\Delta,v^{\epsilon}}(t_k)|^{2p}+\int_{-\tau}^0|y^{\Delta,v^{\epsilon}}_{t_k}(r)|^{2p}\mathrm d\nu_2(r)+\int_{-\tau}^0|y^{\Delta,v^{\epsilon}}_{t_k}(r)|^{2p}\mathrm d\nu_1(r)
\Big)\\
&\quad  +K(M)\Delta |v^{\epsilon}(t_k)|^2\Big(|Z^{\Delta,v^{\epsilon}}(t_k)|^{2p}+|y^{\Delta,v^{\epsilon}}(t_k)|^{2p}+\int_{-\tau}^0|y^{\Delta,v^{\epsilon}}_{t_k}(r)|^{2p}\mathrm d\nu_1(r)\Big)\\
&\quad+A^{p-1}_{\theta,\Delta}|Z^{\Delta,v^{\epsilon}}(t_k)|^{2(p-1)}(\epsilon|\sigma(y^{\Delta,v^{\epsilon}}_{t_k})\delta W_k|^2+2\widetilde{\mathcal M}_k),
\end{align*}
and for $i\in\{2,\ldots,p\},$
\begin{align*}
&\quad\mathcal I^{\epsilon}_i\leq  \gamma_1 A^p_{\theta,\Delta}|Z^{\Delta,v^{\epsilon}}(t_k)|^{2p}
+K(\gamma_1)\Delta^p\Big(|y^{\Delta,v^{\epsilon}}(t_k)|^{2p}+ \int_{-\tau}^0|y^{\Delta,v^{\epsilon}}_{t_k}(r)|^{2p}\mathrm d\nu_1(r)\nn\\
&\quad+ \int_{-\tau}^0|y^{\Delta,v^{\epsilon}}_{t_k}(r)|^{2p}\mathrm d\nu_2(r)\Big)+K(M)\Delta |v^{\epsilon}(t_k)|^2\Big(|Z^{\Delta,v^{\epsilon}}(t_k)|^{2p}+|y^{\Delta,v^{\epsilon}}(t_k)|^{2p}\nn\\
&\quad+\int_{-\tau}^0|y^{\Delta,v^{\epsilon}}_{t_k}|^{2p}\mathrm d\nu_1(r)\Big)+KA^{p-i}_{\theta,\Delta}|Z^{\Delta,v^{\epsilon}}(t_k)|^{2(p-i)}(\epsilon|\sigma(y^{\Delta,v^{\epsilon}}_{t_k})\delta W_k|^{2i}+|\widetilde{\mathcal M}_k|^i).
\end{align*}
It follows from $(|a|+|b|)^p\ge |a|^p+|b|^p$ for $p\mathbb N_+$ that
\begin{align*}
&\quad|Z^{\Delta,v^{\epsilon}}(t_{k+1})|^{2p}+a_3^p\Delta^p|y^{\Delta,v^{\epsilon}}(t_{k+1})|^{p(2+l)}\\
&\leq (1+K\gamma_1)A^p_{\theta,\Delta}|Z^{\Delta,v^{\epsilon}}(t_k)|^{2p}+K\Delta+\widetilde{\mathcal J}_{k1}(R)+\widetilde{\mathcal J}_{k2}\\
&\quad+K\Delta |v^{\epsilon}(t_k)|^2\Big(|Z^{\Delta,v^{\epsilon}}(t_k)|^{2p}+|y^{\Delta,v^{\epsilon}}(t_k)|^{2p}+\int_{-\tau}^0|y^{\Delta,v^{\epsilon}}_{t_k}|^{2p}\mathrm d\nu_1(r)\Big)\\
&\quad+K|Z^{\Delta,v^{\epsilon}}(t_k)|^{2(p-1)}(\epsilon|\sigma(y^{\Delta,v^{\epsilon}}_{t_k})\delta W_k|^2+2\widetilde{\mathcal M}_k)\\
&\quad+K\sum_{i=2}^pC^i_p|Z^{\Delta,v^{\epsilon}}(t_k)|^{2(p-i)}(\epsilon|\sigma(y^{\Delta,v^{\epsilon}}_{t_k})\delta W_k|^{2i}+|\widetilde{\mathcal M}_k|^i),
\end{align*}
where 
\begin{align*}
&\widetilde{\mathcal J}_{k1}(R)=\!-\frac12RA^{p-1}_{\theta,\Delta}\Delta|y^{\Delta,v^{\epsilon}}(t_k)|^{2p},\\
&\widetilde{\mathcal J}_{k2}=\!K(\gamma_1) (R)\Delta^p\big(|y^{\Delta,v^{\epsilon}}(t_k)|^{2p}\!\!+\!\!\int_{-\tau}^0\!|y^{\Delta,v^{\epsilon}}_{t_k}(r)|^{2p}\mathrm d\nu_1(r)\!\!+\!\!\int_{-\tau}^0\!|y^{\Delta,v^{\epsilon}}_{t_k}(r)|^{2p}\mathrm d\nu_2(r)\big).
\end{align*}
By taking $\gamma_1$ and $\gamma_2$ sufficiently small such that $(1+K\gamma_1)A^p_{\theta,\Delta^*}e^{\gamma_2}-1\leq 0,$ we obtain that for $\Delta\in(\Delta^*,1]$,
\begin{align*}
&\quad e^{\gamma_2t_{k+1}}|Z^{\Delta,v^{\epsilon}}(t_{k+1})|^{2p}
\leq |Z^{\Delta,v^{\epsilon}}0)|^{2p}+Ke^{\gamma_2t_{k+1}}+\sum_{l=0}^ke^{\gamma_2t_{l+1}}\widetilde {\mathcal J}_{l1}(R)\nn\\
&\quad+\sum_{l=0}^ke^{\gamma_2t_{l+1}}(\widetilde {\mathcal J}_{l2}-a_3^p\Delta^p|y^{\Delta,v^{\epsilon}}(t_l)|^{p(2+l)})+K(M)\sum_{l=0}^ke^{\gamma_2t_l} \Delta |v^{\epsilon}(t_l)|^2\times\nn\\
&\quad\Big(|Z^{\Delta,v^{\epsilon}}(t_l)|^{2p}+|y^{\Delta,v^{\epsilon}}(t_l)|^{2p}+\int_{-\tau}^0|y^{\Delta,v^{\epsilon}}_{t_l}(r)|^{2p}\mathrm d\nu_1(r)\Big)\nn\\
&\quad+K\sum_{l=0}^ke^{\gamma_2t_l}|Z^{\Delta,v^{\epsilon}}(t_l)|^{2(p-1)}(\epsilon|\sigma(y^{\Delta,v^{\epsilon}}_{t_l})\delta W_l|^2+2\widetilde{\mathcal M}_l)\\
&\quad+K\sum_{l=0}^k\sum_{i=2}^pC^i_p|Z^{\Delta,v^{\epsilon}}(t_l)|^{2(p-i)}(\epsilon|\sigma(y^{\Delta,v^{\epsilon}}_{t_l})\delta W_l|^{2i}+|\widetilde{\mathcal M}_l|^i).
\end{align*}
Then by the similar procedure to the proofs of \eqref{l4.2.13} and \eqref{l4.2.14}, and choosing a sufficiently large number $R$ independent of $\Delta$, we conclude  that
\begin{align*}
&\quad\sum_{l=0}^ke^{\gamma_2t_{l+1}}\widetilde {\mathcal J}_{l1}(R)+\sum_{l=0}^ke^{\gamma_2t_{l+1}}(\widetilde {\mathcal J}_{l2}-a_3^p\Delta^p|y^{\Delta,v^{\epsilon}}(t_l)|^{p(2+l)})\\
&\leq K(\|\xi\|^{2p}+1).
\end{align*}
The remaining proof is similar to that of \underline{Step 1} and is omitted. 

Combining \underline{Steps 1--2}  finishes the proof. 
\end{proof}

\begin{prop}\label{prop_cond}
Let $M>0$ and $\{v^{\epsilon}\}_{\epsilon\in(0,1)} \subset \mathcal P_M$ such that $v^{\epsilon}\overset{d}{\underset{\epsilon\to0}\longrightarrow}v$ as $\mathcal S_M$-valued random variables. Under Assumptions \ref{a1}, \ref{a2}, \ref{a4}, and \ref{a5}, 
it holds that for any $\Delta\in(0,\frac{1}{4a_2\theta}]$,
$$\mathcal G^{\Delta,\epsilon}\Big(\sqrt{\epsilon}W+\int_0^{\cdot}v^{\epsilon}(\lfloor s\rfloor)\mathrm ds\Big)\overset{d}{\underset{\epsilon\to0}\longrightarrow}\mathcal G^{\Delta}\Big(\int_0^{\cdot}v(\lfloor s\rfloor)\mathrm ds\Big).$$
\end{prop}

\begin{proof}
To show $y^{\Delta,v^{\epsilon}}\overset{d}{\underset{\epsilon\to0}\longrightarrow} w^{\Delta,v}$ as  $v^{\epsilon}\overset{d}{\underset{\epsilon\to0}\longrightarrow}v$,
by the duality formula for the Kantorovich–Rubinstein distance (see e.g. \cite[Remark 6.5]{Villani}),
it suffices to prove that 
$\E[\|y^{\Delta,v^{\epsilon}}-w^{\Delta,\nu}\|_{\mathcal C_{\xi}}\wedge1]\to 0$ as $\epsilon\to0$.
Note that
\begin{align*}
&\quad\lim_{\epsilon\rightarrow 0}\E\big[\|y^{\Delta,v^{\epsilon}}-w^{\Delta,\nu}\|_{\mathcal C_{\xi}}\wedge1\big]\nn\\
&\leq \sum_{k=1}^{\infty}\lim_{\epsilon\rightarrow 0}\E\Big[\Big(\Delta e^{-\mathfrak at_{k}}\sup_{t\in[0,t_k]}|y^{\Delta,v^{\epsilon}}(t)-w^{\Delta,v}(t)|\Big)\wedge 1\Big]\nn\\
&\leq \sum_{k=1}^{\infty}\Delta e^{-\mathfrak at_{k}}\lim_{\epsilon\rightarrow 0}\E\Big[\Big(\sup_{t\in[0,t_k]}|y^{\Delta,v^{\epsilon}}(t)-w^{\Delta,v}(t)|\Big)\wedge(e^{\mathfrak at_{k}}\Delta^{-1})\Big].
\end{align*}
Hence, 
we  only need to show that for any  $k\in\mathbb N_+$, $$\lim_{\epsilon\rightarrow 0}\E\Big[\sup_{t\in[0,t_k]}|y^{\Delta,v^{\epsilon}}(t)-w^{\Delta,v}(t)|\wedge(e^{\mathfrak at_{k}}\Delta^{-1})\Big]=0.$$

It follows from  definitions of $Z^{\Delta, v^{\epsilon}}$ and $z^{\Delta, v},$ and  the It\^o formula 
that  
\begin{align*}
&\quad \mathrm d|Z^{\Delta, v^{\epsilon}}(t)-z^{\Delta, v}(t)|^2=\Big[2\big\langle Z^{\Delta, v^{\epsilon}}(t)-z^{\Delta, v}(t), b(y^{\Delta, v^{\epsilon}}_{\lfloor t\rfloor})-b(w^{\Delta, v}_{\lfloor t\rfloor})\big\rangle\nn\\
&\quad+\epsilon |\sigma(y^{\Delta, v^{\epsilon}}_{\lfloor t\rfloor})|^2+2\big\langle Z^{\Delta, v^{\epsilon}}(t)-z^{\Delta, v}(t), \sigma(y^{\Delta, v^{\epsilon}}_{\lfloor t\rfloor})v^{\epsilon}(\lfloor t\rfloor)-\sigma(w^{\Delta, v}_{\lfloor t\rfloor})v (\lfloor t\rfloor)\big\rangle \big]\mathrm dt\nn\\
&\quad+2\sqrt{\epsilon}\<Z^{\Delta, v^{\epsilon}}(t)-z^{\Delta, v}(t),\sigma(y^{\Delta, v^{\epsilon}}_{\lfloor t\rfloor})\mathrm dW(t)\>.
\end{align*}
On account of 
\begin{align*}
&Z^{\Delta, v^{\epsilon}}(t)-z^{\Delta, v}(t)= Z^{\Delta, v^{\epsilon}}(\lfloor t\rfloor)-z^{\Delta, v}(\lfloor t\rfloor)
+\big(b(y^{\Delta, v^{\epsilon}}_{\lfloor t\rfloor})-b(w^{\Delta, v}_{\lfloor t\rfloor})\big)(t-\lfloor t\rfloor)\nn\\
&\quad+\sqrt{\epsilon}\sigma(y^{\Delta, v^{\epsilon}}_{\lfloor t\rfloor})\big(W(t)-W(\lfloor t\rfloor)\big)+\big(\sigma(y^{\Delta, v^{\epsilon}}_{\lfloor t\rfloor})v^{\epsilon}(\lfloor t\rfloor)-\sigma(w^{\Delta, v}_{\lfloor t\rfloor})v (\lfloor t\rfloor)\big)(t-\lfloor t\rfloor)
\end{align*}
and
\begin{align*}
Z^{\Delta, v^{\epsilon}}(\lfloor t\rfloor)-z^{\Delta, v}(\lfloor t\rfloor)=y^{\Delta, v^{\epsilon}}(\lfloor t\rfloor)-w^{\Delta, v}(\lfloor t\rfloor)-\theta\Delta\big(b(y^{\Delta, v^{\epsilon}}_{\lfloor t\rfloor})-b(w^{\Delta, v}_{\lfloor t\rfloor})\big),
\end{align*}
we deduce
\begin{align*}
&\quad \mathrm d|Z^{\Delta, v^{\epsilon}}(t)-z^{\Delta, v}(t)|^2=\Big[2\langle y^{\Delta, v^{\epsilon}}(\lfloor t\rfloor)-w^{\Delta, v}(\lfloor t\rfloor),b(y^{\Delta, v^{\epsilon}}_{\lfloor t\rfloor})-b(w^{\Delta, v}_{\lfloor t\rfloor})\rangle\nn\\
&\quad-\!2\theta\Delta |b(y^{\Delta, v^{\epsilon}}_{\lfloor t\rfloor})\!-\!b(w^{\Delta, v}_{\lfloor t\rfloor})|^2\!
+\!2\<(b(y^{\Delta, v^{\epsilon}}_{\lfloor t\rfloor})\!-\!b(w^{\Delta, v}_{\lfloor t\rfloor}))(t-\lfloor t\rfloor),b(y^{\Delta, v^{\epsilon}}_{\lfloor t\rfloor})\!-\!b(w^{\Delta, v}_{\lfloor t\rfloor})\>\nn\\
&\quad+2\big\<\sigma(y^{\Delta, v^{\epsilon}}_{\lfloor t\rfloor})v^{\epsilon}(\lfloor t\rfloor)-\sigma(w^{\Delta, v}_{\lfloor t\rfloor})v (\lfloor t\rfloor),(b(y^{\Delta, v^{\epsilon}}_{\lfloor t\rfloor})-b(w^{\Delta, v}_{\lfloor t\rfloor}))(t-\lfloor t\rfloor)\nn\\
&\quad+Z^{\Delta, v^{\epsilon}}(t)
-z^{\Delta, v}(t)\big\>+\epsilon |\sigma(y^{\Delta, v^{\epsilon}}_{\lfloor t\rfloor})|^2\Big]\mathrm dt+2\sqrt{\epsilon}\big{\<}b(y^{\Delta, v^{\epsilon}}_{\lfloor t\rfloor})-b(w^{\Delta, v}_{\lfloor t\rfloor}),\nn\\
&\quad\sigma (y^{\Delta, v^{\epsilon}}_{\lfloor t\rfloor})(W(t)-W(\lfloor t\rfloor))\big{\>}\mathrm dt+2\sqrt{\epsilon}\<Z^{\Delta, v^{\epsilon}}(t)-z^{\Delta, v}(t),\sigma(y^{\Delta, v^{\epsilon}}_{\lfloor t\rfloor})\mathrm dW(t)\>.
\end{align*}
By the Young inequality, we have
\begin{align*}
&\quad |Z^{\Delta, v^{\epsilon}}(s)-z^{\Delta, v}(s)|^2
=\int_{0}^{s} 2\big\langle y^{\Delta, v^{\epsilon}}(\lfloor t\rfloor)-w^{\Delta, v}(\lfloor t\rfloor),b(y^{\Delta, v^{\epsilon}}_{\lfloor t\rfloor})-b(w^{\Delta, v}_{\lfloor t\rfloor})\big\rangle\nn\\
&\quad+2\big(t-\lfloor t\rfloor-\theta\Delta\big)\big|b(y^{\Delta, v^{\epsilon}}_{\lfloor t\rfloor})-b(w^{\Delta, v}_{\lfloor t\rfloor})\big|^2\Big]\mathrm dt+\int_{0}^{s} 2\big\<\sigma(y^{\Delta, v^{\epsilon}}_{\lfloor t\rfloor})v^{\epsilon}(\lfloor t\rfloor)\nn\\
&\quad-\sigma(w^{\Delta, v}_{\lfloor t\rfloor})v (\lfloor t\rfloor),\big(b(y^{\Delta, v^{\epsilon}}_{\lfloor t\rfloor})-b(w^{\Delta, v}_{\lfloor t\rfloor})\big)(t-\lfloor t\rfloor)+Z^{\Delta, v^{\epsilon}}(t)-z^{\Delta, v}(t)\big\>\nn\\
&\quad+\epsilon |\sigma(y^{\Delta, v^{\epsilon}}_{\lfloor t\rfloor})|^2\mathrm dt
+2\sqrt{\epsilon}\int_{0}^{s}\big{\<}b(y^{\Delta, v^{\epsilon}}_{\lfloor t\rfloor})-b(w^{\Delta, v}_{\lfloor t\rfloor}),\sigma (y^{\Delta, v^{\epsilon}}_{\lfloor t\rfloor})(W(t)-W(\lfloor t\rfloor))\big{\>}\mathrm dt\nn\\
&\quad+2\sqrt{\epsilon}\int_{0}^{s}\<Z^{\Delta, v^{\epsilon}}(t)-z^{\Delta, v}(t),\sigma(y^{\Delta, v^{\epsilon}}_{\lfloor t\rfloor})\mathrm dW(t)\>\nn\\
&\leq\int_{0}^{s} 2\big\langle y^{\Delta, v^{\epsilon}}(\lfloor t\rfloor)-w^{\Delta, v}(\lfloor t\rfloor),b(y^{\Delta, v^{\epsilon}}_{\lfloor t\rfloor})-b(w^{\Delta, v}_{\lfloor t\rfloor})\big\rangle+|\sigma(y^{\Delta, v^{\epsilon}}_{\lfloor t\rfloor})-\sigma(w^{\Delta, v}_{\lfloor t\rfloor})|^2\nn\\
&\quad+\epsilon |\sigma(y^{\Delta, v^{\epsilon}}_{\lfloor t\rfloor})|^2\mathrm dt+K\int_{0}^{s}|v^{\epsilon}(\lfloor t\rfloor)|^2\Big|\frac{1}{\theta}\big(y^{\Delta, v^{\epsilon}}(\lfloor t\rfloor)-w^{\Delta, v}(\lfloor t\rfloor)\nn\\
&\quad+Z^{\Delta, v^{\epsilon}}(\lfloor t\rfloor)-z^{\Delta, v}(\lfloor t\rfloor)\big)+Z^{\Delta, v^{\epsilon}}(t)-z^{\Delta, v}(t)\Big|^2\mathrm dt\nn\\
&\quad+\int_{0}^{s} 2\big\<\sigma(w^{\Delta, v}_{\lfloor t\rfloor})\big(v^{\epsilon}(\lfloor t\rfloor)-v (\lfloor t\rfloor)\big), \big(b(y^{\Delta, v^{\epsilon}}_{\lfloor t\rfloor})-b(w^{\Delta, v}_{\lfloor t\rfloor})\big)(t-\lfloor t\rfloor)+Z^{\Delta, v^{\epsilon}}(t)\nn\\
&\quad-z^{\Delta, v}(t)\big\>\mathrm dt
+2\sqrt{\epsilon}\int_{0}^{s}\big{\<}b(y^{\Delta, v^{\epsilon}}_{\lfloor t\rfloor})-b(w^{\Delta, v}_{\lfloor t\rfloor}),\sigma (y^{\Delta, v^{\epsilon}}_{\lfloor t\rfloor})(W(t)-W(\lfloor t\rfloor))\big{\>}\mathrm dt\nn\\
&\quad+2\sqrt{\epsilon}\int_{0}^{s}\<Z^{\Delta, v^{\epsilon}}(t)-z^{\Delta, v}(t),\sigma(y^{\Delta, v^{\epsilon}}_{\lfloor t\rfloor})\mathrm dW(t)\>,
\end{align*}
where in the first inequality we used $\int_{t_k}^{t_{k+1}}2(t-\lfloor t\rfloor -\theta\Delta)\mathrm dt=(1-\theta)\Delta^2\leq 0$.
It follows from Assumptions \ref{a1}--\ref{a2} and 
\begin{align*}
\int_{0}^s\int_{-\tau}^0|y^{\Delta, v^{\epsilon}}_{\lfloor t\rfloor}(r)\!-\!w^{\Delta, v}_{\lfloor t\rfloor}(r)|^2\mathrm d\nu_i(r)\mathrm  dt
\leq \int_{0}^s |y^{\Delta, v^{\epsilon}}(\lfloor t\rfloor)-w^{\Delta, v}(\lfloor t\rfloor)|^2\mathrm dt\quad \forall\, i=1,2
\end{align*}
that
\begin{align*}
&\quad |Z^{\Delta, v^{\epsilon}}(s)-z^{\Delta, v}(s)|^2\nn\\
&\leq\int_{0}^{s}\Big[
-2(a_1-a_2-L)|y^{\Delta, v^{\epsilon}}(\lfloor t\rfloor)-w^{\Delta, v}(\lfloor t\rfloor)|^2+\epsilon |\sigma(y^{\Delta, v^{\epsilon}}_{\lfloor t\rfloor})|^2\Big]\mathrm dt\nn\\
&\quad+K\int_{0}^{s} |v^{\epsilon}(\lfloor t\rfloor)|^2\Big(\sup_{r\in[0,t]}\big(|y^{\Delta, v^{\epsilon}}(\lfloor r\rfloor)-w^{\Delta, v}(\lfloor r\rfloor)|^2\!+\!|Z^{\Delta, v^{\epsilon}}(r)-z^{\Delta, v}(r)|^2\big)\Big)\mathrm dt\nn\\
&\quad+E_1(s)+2\sqrt{\epsilon}\int_{0}^{s}\big{\<}b(y^{\Delta, v^{\epsilon}}_{\lfloor t\rfloor})-b(w^{\Delta, v}_{\lfloor t\rfloor}),\sigma (y^{\Delta, v^{\epsilon}}_{\lfloor t\rfloor})(W(t)-W(\lfloor t\rfloor))\big{\>}\mathrm dt\nn\\
&\quad+2\sqrt{\epsilon}\int_{0}^{s}\<Z^{\Delta, v^{\epsilon}}(t)-z^{\Delta, v}(t),\sigma(y^{\Delta, v^{\epsilon}}_{\lfloor t\rfloor})\mathrm dW(t)\>,
\end{align*}
where $E_1(s):=\int_{0}^{s} 2\big\<\sigma(w^{\Delta, v}_{\lfloor t\rfloor})\big(v^{\epsilon}(\lfloor t\rfloor)-v (\lfloor t\rfloor)\big), \big(b(y^{\Delta, v^{\epsilon}}_{\lfloor t\rfloor})-b(w^{\Delta, v}_{\lfloor t\rfloor})\big)(t-\lfloor t\rfloor)+Z^{\Delta, v^{\epsilon}}(t)-z^{\Delta, v}(t)\big\>\mathrm dt.$
By $a_1>a_2+L$, we arrive at
\begin{align*}
&\quad |Z^{\Delta, v^{\epsilon}}(s)-z^{\Delta, v}(s)|^2\nn\\
&\leq \int_{0}^{s} (K|v^{\epsilon}(\lfloor t\rfloor)|^2+\gamma)\Big(\sup_{r\in[0,t]}|y^{\Delta, v^{\epsilon}}(\lfloor r\rfloor)-w^{\Delta, v}(\lfloor r\rfloor)|^2+\sup_{r\in[0,t]}|Z^{\Delta, v^{\epsilon}}(r)\nn\\
&\quad-z^{\Delta, v}(r)|^2\Big)\mathrm dt
+\int_{0}^{s} \epsilon |\sigma(y^{\Delta, v^{\epsilon}}_{\lfloor t\rfloor})|^2\mathrm dt+E_1(s)+2\sqrt{\epsilon}\int_{0}^{s}\big{\<}b(y^{\Delta, v^{\epsilon}}_{\lfloor t\rfloor})-b(w^{\Delta, v}_{\lfloor t\rfloor}),\nn\\
&\quad\sigma (y^{\Delta, v^{\epsilon}}_{\lfloor t\rfloor})(W(t)-W(\lfloor t\rfloor))\big{\>}\mathrm dt+2\sqrt{\epsilon}\int_{0}^{s}\<Z^{\Delta, v^{\epsilon}}(t)-z^{\Delta, v}(t),\sigma(y^{\Delta, v^{\epsilon}}_{\lfloor t\rfloor})\mathrm dW(t)\>.
\end{align*}
Denote $ E^{\epsilon}(t):=\int_0^t\sigma(w^{\Delta, v}_{\lfloor r\rfloor})(v^{\epsilon}(\lfloor r\rfloor)-v(\lfloor r\rfloor))\mathrm dr.$ By virtue of the integration by parts formula, one has
\begin{align*}
&\quad E_1(s)=2\<E^{\epsilon}(s),(b(y^{\Delta, v^{\epsilon}}_{\lfloor s\rfloor})-b(w^{\Delta, v}_{\lfloor s\rfloor}))(s-\lfloor s\rfloor) +(Z^{\Delta, v^{\epsilon}}(s)-z^{\Delta, v}(s))\>\\
&\quad-2\int_0^{s}\Big\<E^{\epsilon}(t),\mathrm d \Big(\big(b(y^{\Delta, v^{\epsilon}}_{\lfloor t\rfloor})-b(w^{\Delta, v}_{\lfloor t\rfloor})\big)(t-\lfloor t\rfloor)+Z^{\Delta, v^{\epsilon}}(t)-z^{\Delta, v}(t)\Big)\Big\>\\
&\leq K|E^{\epsilon}(s)|^2+\frac14 |y^{\Delta, v^{\epsilon}}(\lfloor s\rfloor)-w^{\Delta, v}(\lfloor s\rfloor)|^2+\frac14|Z^{\Delta, v^{\epsilon}}(\lfloor s\rfloor)-z^{\Delta, v}(\lfloor s\rfloor)|^2\\
&\quad +\frac14|Z^{\Delta, v^{\epsilon}}(s)-z^{\Delta, v}(s)|^2
-4\int_0^{s}\<E^{\epsilon}(t),b(y^{\Delta, v^{\epsilon}}_{\lfloor t\rfloor})-\!b(w^{\Delta, v}_{\lfloor t\rfloor}) \>\mathrm dt\nn\\
&\quad-2\int_0^{s}\Big\<E^{\epsilon}(t),
\big(\sigma(y^{\Delta, v^{\epsilon}}_{\lfloor t\rfloor})-\sigma(w^{\Delta, v}_{\lfloor t\rfloor})\big)v^{\epsilon}(\lfloor t\rfloor)\Big\>\mathrm dt-2\int_0^{s}\Big\<E^{\epsilon}(t),
\sigma(w^{\Delta, v}_{\lfloor t\rfloor})\big(v^{\epsilon}(\lfloor t\rfloor)\nn\\
&\quad-v(\lfloor t\rfloor)\big)\Big\>\mathrm dt-2\sqrt{\epsilon}\int_0^{s} e^{\gamma t}\<E^{\epsilon}(t),\sigma(y^{\Delta, v^{\epsilon}}_{\lfloor t\rfloor})\mathrm dW(t)\>\nn\\
&\leq K\sup_{r\in[0,s]}|E^{\epsilon}(r)|^2+\frac14 |y^{\Delta, v^{\epsilon}}(\lfloor s\rfloor)-w^{\Delta, v}(\lfloor s\rfloor)|^2+\frac12 \sup_{r\in[0,s]}|Z^{\Delta, v^{\epsilon}}(r)-z^{\Delta, v}(r)|^2\nn\\
&\quad+K\int_{0}^{s} (|v^{\epsilon}({\lfloor t\rfloor})|^2+1)\Big(\sup_{r\in[0,t]}\big(|y^{\Delta, v^{\epsilon}}(\lfloor r\rfloor)-w^{\Delta, v}(\lfloor r\rfloor)|^2+|Z^{\Delta, v^{\epsilon}}(r)\nn\\
&\quad-z^{\Delta, v}(r)|^2\big)\Big)\mathrm dt
-2\sqrt{\epsilon}\int_0^{s} \<E^{\epsilon}(t),\sigma(y^{\Delta, v^{\epsilon}}_{\lfloor t\rfloor})\mathrm dW(t)\>.
\end{align*}
Note that 
\begin{align*}
&\quad |Z^{\Delta, v^{\epsilon}}(t_k)-z^{\Delta, v}(t_k)|^2\nn\\
&\ge |y^{\Delta, v^{\epsilon}}(t_k)-w^{\Delta, v}(t_k)|^2-2\theta\Delta \<y^{\Delta, v^{\epsilon}}(t_k)-w^{\Delta, v}(t_k),b(y^{\Delta, v^{\epsilon}}_{t_k})-b(w^{\Delta, v}_{t_k})\>\\
&\ge |y^{\Delta, v^{\epsilon}}(t_k)-w^{\Delta, v}(t_k)|^2-2\theta\Delta \Big(-a_1|y^{\Delta, v^{\epsilon}}(t_k)-w^{\Delta, v}(t_k)|^2\nn\\
&\quad+a_2\int_{-\tau}^0|y^{\Delta, v^{\epsilon}}_{t_k}(r)-w^{\Delta, v}_{t_k}(r)|^2\mathrm d\nu_2(r)\Big),
\end{align*}
which implies
\begin{align*}
&\quad |Z^{\Delta, v^{\epsilon}}(t_k)-z^{\Delta, v}(t_k)|^2\nn\\
&\geq |y^{\Delta, v^{\epsilon}}(t_k)-w^{\Delta, v}(t_k)|^2-2a_2\theta \Delta
\sup_{r\in[0,t_k]} |y^{\Delta, v^{\epsilon}}(\lfloor r\rfloor)-w^{\Delta, v}(\lfloor r\rfloor)^2.
\end{align*}
Then,
\begin{align*}
&\quad |y^{\Delta, v^{\epsilon}}(t_{k})\!-\!w^{\Delta, v}(t_{k})|^2
\leq 2a_2\theta \Delta\sup_{t\in[0,t_{k}]}|y^{\Delta, v^{\epsilon}}(\lfloor t\rfloor)\!-\!w^{\Delta, v}(\lfloor t\rfloor)|^2\nn\\
&\quad+K\int_{0}^{t_{k}}(|v^{\epsilon}(\lfloor t\rfloor)|^2+1)\Big(\sup_{r\in[0,t]}\big(|y^{\Delta, v^{\epsilon}}(\lfloor r\rfloor)-w^{\Delta, v}(\lfloor r\rfloor)|^2\!+\!|Z^{\Delta, v^{\epsilon}}(r)\nn\\
&\quad-z^{\Delta, v}(r)|^2\big)\Big)\mathrm dt+\int_{0}^{t_{k}} \epsilon |\sigma(y^{\Delta, v^{\epsilon}}_{\lfloor t\rfloor})|^2\mathrm dt+2\sqrt{\epsilon}\int_{0}^{t_{k}}\big{\<}b(y^{\Delta, v^{\epsilon}}_{\lfloor t\rfloor})-b(w^{\Delta, v}_{\lfloor t\rfloor}),\nn\\
&\quad\sigma (y^{\Delta, v^{\epsilon}}_{\lfloor t\rfloor})(W(t)-W(\lfloor t\rfloor))\big{\>}\mathrm dt+2\sqrt{\epsilon}\int_{0}^{t_{k}}\<Z^{\Delta, v^{\epsilon}}(t)-z^{\Delta, v}(t),\sigma(y^{\Delta, v^{\epsilon}}_{\lfloor t\rfloor})\mathrm dW(t)\>\nn\\
&\quad+K\sup_{r\in[0,t_{k}]}|E^{\epsilon}(r)|^2+\frac14 |y^{\Delta, v^{\epsilon}}(t_{k})-w^{\Delta, v}(t_{k})|^2+\frac12\sup_{r\in[0,t_{k}]}|Z^{\Delta, v^{\epsilon}}(r)-z^{\Delta, v}(r)|^2\nn\\
&\quad-2\sqrt{\epsilon}\int_0^{t_{k}}\<E^{\epsilon}(t),\sigma(y^{\Delta, v^{\epsilon}}_{\lfloor t\rfloor})\mathrm dW(t)\>.
\end{align*}
Therefore,
\begin{align*}
&\quad \frac{1}{2}\sup_{t\in[0,t_{k}]}e^{\gamma t}|Z^{\Delta, v^{\epsilon}}(t)-z^{\Delta, v}(t)|^2+\frac{1}{8}\sup_{t\in\{0,\ldots,t_{k}\}}e^{\gamma t }|y^{\Delta, v^{\epsilon}}(t)-w^{\Delta, v}(t)|^2\\
&\leq K\int_{0}^{t_{k}} \big(1+|v^{\epsilon}(\lfloor t\rfloor)|^2\big)\Big(\sup_{r\in[0,t]}\big(|y^{\Delta, v^{\epsilon}}(\lfloor r\rfloor)-w^{\Delta, v}(\lfloor r\rfloor)|^2+|Z^{\Delta, v^{\epsilon}}(r)\nn\\
&\quad-z^{\Delta, v}(r)|^2\big)\Big)\mathrm dt+2\epsilon \int_0^{t_{k}}\sigma(y^{\Delta, v^{\epsilon}}_{\lfloor t\rfloor})|^2\mathrm dt+4\sqrt{\epsilon}\sup_{r\in[0,t_{k}]}\Big(\int_{0}^{r}\<b(y^{\Delta, v^{\epsilon}}_{\lfloor t\rfloor})-b(w^{\Delta, v}_{\lfloor t\rfloor}), \nn\\
&\quad\sigma(y^{\Delta, v^{\epsilon}}_{\lfloor t\rfloor})(W(t)-W(\lfloor t\rfloor))\>\mathrm dt+\int_{0}^{r}\<Z^{\Delta, v^{\epsilon}}(t)-z^{\Delta, v}(t),\sigma(y^{\Delta, v^{\epsilon}}_{\lfloor t\rfloor})\mathrm dW(t)\>\Big)\nn\\
&\quad+K\sup_{t\in[0,t_{k}]}|E^{\epsilon}(t)|^2+2\sqrt{\epsilon}\sup_{r\in[0,t_{k}]}\int_0^{r}\<E^{\epsilon}(t),\sigma(y^{\Delta, v^{\epsilon}}_{\lfloor t\rfloor})\mathrm dW(t)\>.
\end{align*}
Applying the Gr\"onwall inequality and $v^{\epsilon}(\cdot)\in\mathcal P_{M}$,  we arrive at
\begin{align*}
&\quad \sup_{t\in[0,t_{k}]}|Z^{\Delta, v^{\epsilon}}(t)-z^{\Delta, v}(t)|^2+\sup_{t\in[0,t_{k}]}|y^{\Delta, v^{\epsilon}}(\lfloor t\rfloor)-w^{\Delta, v}(\lfloor t\rfloor)|^2\\
&\leq Ke^{K(M+1)t_{k}}\Big[\epsilon \int_0^{t_{k}}|\sigma(y^{\Delta, v^{\epsilon}}_{\lfloor t\rfloor})|^2\mathrm dt+\sqrt{\epsilon}\sup_{r\in[0,t_{k}]}\Big(\int_{0}^{r}\big\<b(y^{\Delta, v^{\epsilon}}_{\lfloor t\rfloor})-b(w^{\Delta, v}_{\lfloor t\rfloor}),\nn\\
&\quad\sigma(y^{\Delta, v^{\epsilon}}_{\lfloor t\rfloor})(W(t)-W(\lfloor t\rfloor))\big\>\mathrm dt+\int_{0}^{r}\big\<Z^{\Delta, v^{\epsilon}}(t)-z^{\Delta, v}(t),\sigma(y^{\Delta, v^{\epsilon}}_{\lfloor t\rfloor})\mathrm dW(t)\big\>\Big)\nn\\
&\quad+\sup_{t\in[0,t_{k}]}|E^{\epsilon}(t)|^2+\sqrt{\epsilon}\sup_{r\in[0,t_{k}]}\int_0^{r}\<E^{\epsilon}(t),\sigma(y^{\Delta, v^{\epsilon}}_{\lfloor t\rfloor})\mathrm dW(t)\>\Big].
\end{align*}
Utilizing  \cref{bound_control,estimate_y2}, the Burkholder--Davis--Gundy inequality, and the Young inequality, we obtain
\begin{align*}
&\quad\mathbb E\Big[\sup_{t\in[0,t_{k}]}|y^{\Delta, v^{\epsilon}}(\lfloor t\rfloor)-w^{\Delta, v}(\lfloor t\rfloor)|^2\wedge(e^{\mathfrak at_{k}}\Delta^{-1})\Big]\\
&\leq Ke^{K(M+1)t_{k}}\bigg\{\epsilon \int_0^{t_{k}}\mathbb E[1+\|y^{\Delta, v^{\epsilon}}_{\lfloor t\rfloor}\|^2]\mathrm dt+\sqrt{\epsilon}\int_0^{t_{k}}\E\Big[|b(y^{\Delta, v^{\epsilon}}_{\lfloor t\rfloor})-b(w^{\Delta, v}_{\lfloor t\rfloor})|\times\nn\\
&\quad|\sigma(y^{\Delta, v^{\epsilon}}_{\lfloor t\rfloor})||W(t)-W(\lfloor t\rfloor)|\Big]\mathrm dt+\!\sqrt{\epsilon}\mathbb E\Big[\Big(\int_0^{t_{k}}\!\!|Z^{\Delta, v^{\epsilon}}(t)\!-\!z^{\Delta, v}(t)|^2|\sigma(y^{\Delta, v^{\epsilon}}_{\lfloor t\rfloor})|^2\mathrm dt\Big)^{\frac12}\Big]\nn\\
&\quad+\sup_{t\in[0,t_{k}]}\E\big[|E^{\epsilon}(t)|^2\wedge(e^{\mathfrak at_{k}}\Delta^{-1})\big]+\sqrt{\epsilon} \mathbb E\Big[\int_0^{t_{k}}|E^{\epsilon}(t)|^2|\sigma(y^{\Delta, v^{\epsilon}}_{\lfloor t\rfloor})|^2\mathrm dt\Big)^{\frac12}\Big]\bigg\}\\
&\leq K(M)e^{K(M+1)t_{k}}\Big[(1+\|\xi\|^{2(\beta+1)}))(\epsilon+\sqrt\epsilon)+
\E\big[\sup_{t\in[0,t_{k}]}|E^{\epsilon}(t)|^2\wedge(e^{\mathfrak at_{k}}\Delta^{-1})\big]\Big].
\end{align*}
It follows from $v^{\epsilon}\overset{d}{\underset{\epsilon\to0}\longrightarrow}v$ that $\lim_{\epsilon\rightarrow 0}\E\big[\sup_{t\in[0,t_{k}]}|E^{\epsilon}(t)|^2\wedge(e^{\mathfrak at_{k}}\Delta^{-1})\big]=0$.
Consequently, letting $\epsilon\rightarrow 0$ finishes the proof.
\end{proof}

\section{Logarithmic  estimate for numerical density function}
This section is devoted to the logarithmic estimate on the finite time horizon for the  density function of the $\theta$-EM method \eqref{EM_epsilon} of \eqref{FF_small}, based on the smoothness of the numerical density function and the technique of the Malliavin calculus. The constant $T>0$ is fixed in this section.  
For the numerical solution $y^{\Delta,\epsilon}$,
under Assumptions \ref{a1}, \ref{repalceA2}--\ref{assp_den},
it follows from Theorem \ref{trun_num_den} that for any $\Delta\in(0,\frac{1}{4\theta a_5})$ and  $k\in\mathbb N_+$,
 $y^{\Delta,\epsilon}(t_k)$ admits a density function denoted by $\mathfrak p^{\Delta,\epsilon}(t_k, \cdot)$. We introduce the following assumption to ensure the smoothness of $\mathfrak p^{\Delta,\epsilon}(t_k, \cdot)$.

\begin{assp}\label{smooth_assp}
Assume that coefficients $b$ and $\sigma$ are smooth with bounded derivatives of arbitrary orders. 
\end{assp}
The proof of the smoothness of the density $\mathfrak p^{\Delta,\epsilon}(t_k, \cdot)$ depends on the estimate of the Malliavin covariance matrix $\gamma_{k}^{\epsilon}$ of the numerical solution $y^{\Delta,\epsilon}_{t_k}$ and the result that $y^{\Delta,\epsilon}_{t_k}\in \mathbb D^{\infty}(\RR^d)$.

\begin{lemma}\label{lemma_matrix}
Let \cref{a1,repalceA2,assp_den,smooth_assp} hold. Then for any $p\ge 1$ and  $\Delta\in(0, \frac{1}{4\theta a_5})\wedge(0,\frac{1}{4\theta\|\mathcal Db\|_{\mathcal L(\mathcal C^{d};\RR^d)})}]$,
\begin{align*}
\|\det (\gamma_{N^{\Delta}}^{\epsilon})^{-1}\|_{L^p(\Omega)}\leq K\epsilon^{-d}\Delta^{-d}.
\end{align*}
\end{lemma}
\begin{proof}
Since $\det(\gamma^{\epsilon}_{N^{\Delta}})^{-1}=\prod_{i=1}^d\lambda_i^{-1}\leq (\lambda_{\min}^{-1})^{d},$ where $\lambda_i,i=1,\ldots,d$ are the eigenvalues of $\gamma^{\epsilon}_{N^{\Delta}}$ and $\lambda_{\min}:=\min\{\lambda_i,i=1,\ldots,d\}, $ it suffices to estimate the smallest eigenvalue of $\gamma^{\epsilon}_{N^{\Delta}}.$ 
For the Malliavin covariance matrix $\gamma_{k}^{\epsilon}$ of the numerical solution $y^{\Delta,\epsilon}(t_k),$ similar to the recursive relation \eqref{recur_gamma}, for any $\Delta\in(0,\frac{1}{4\theta a_5})$,
 using the equality 
 \begin{align}\label{A+}
 A_{1,N^{\Delta}-1}^{\epsilon}:=(\mathrm{Id}_{d\times d}-\theta \mathcal Db(y^{\Delta,\epsilon}_{t_{N^{\Delta}}})I^0\Delta)^{-1}=\mathrm{Id}_{d\times d}+A^{\epsilon}_{1,N^{\Delta}-1}\theta \mathcal Db(y^{\Delta,\epsilon}_{t_{N^{\Delta}}}) I^0\Delta,
\end{align}
 we derive
\begin{align}\label{A++}
&\quad\lambda_{\min}(\gamma^{\epsilon}_{N^{\Delta}})=\min_{u\in\mathbb R^d,|u|=1}u^{\top}\gamma^{\epsilon}_{N^{\Delta}}u\nn\\
&\ge \epsilon \Delta \min_{u\in\mathbb R^d,|u|=1}u^{\top} A^{\epsilon}_{1,N^{\Delta}-1}\sigma(y^{\Delta,\epsilon}_{t_{N^{\Delta}-1}})\sigma(y^{\Delta,\epsilon}_{t_{N^{\Delta}-1}})^{\top}(A^{\epsilon}_{1,N^{\Delta}-1})^{\top}u\nn\\
&= \epsilon\Delta \min_{u\in\mathbb R^d,|u|=1}u^{\top} \Big[\sigma(y^{\Delta,\epsilon}_{t_{N^{\Delta}-1}})\sigma(y^{\Delta,\epsilon}_{t_{N^{\Delta}-1}})^{\top}+A^{\epsilon}_{1,N^{\Delta}-1}\theta \mathcal Db(y^{\Delta,\epsilon}_{t_{N^{\Delta}}})I^0 \sigma(y^{\Delta,\epsilon}_{t_{N^{\Delta}-1}})\nn\\
&\quad\sigma(y^{\Delta,\epsilon}_{t_{N^{\Delta}-1}})^{\top}\Delta +\sigma(y^{\Delta,\epsilon}_{t_{N^{\Delta}-1}})\sigma(y^{\Delta,\epsilon}_{t_{N^{\Delta}-1}})^{\top} (A^{\epsilon}_{1,N^{\Delta}-1}\theta \mathcal Db(y^{\Delta,\epsilon}_{t_{N^{\Delta}}})I^0\Delta )^{\top}\nn\\
&\quad +A^{\epsilon}_{1,N^{\Delta}-1}\theta \mathcal Db(y^{\Delta,\epsilon}_{t_{N^{\Delta}}})I^0\Delta \sigma(y^{\Delta,\epsilon}_{t_{N^{\Delta}-1}})\sigma(y^{\Delta,\epsilon}_{t_{N^{\Delta}-1}})^{\top} (A^{\epsilon}_{1,N^{\Delta}-1}\theta \mathcal Db(y^{\Delta,\epsilon}_{t_{N^{\Delta}}})I^0\Delta )^{\top}
\Big]u\nn\\
&\ge\frac12\epsilon \Delta \min_{u\in\mathbb R^d,|u|=1}u^{\top}\sigma(y^{\Delta,\epsilon}_{t_{N^{\Delta}-1}})\sigma(y^{\Delta,\epsilon}_{t_{N^{\Delta}-1}})^{\top}u\nn\\
&\quad+\epsilon\Delta \min_{u\in\mathbb R^d,|u|=1}u^{\top} \Big[\big(\frac14\mathrm {Id}_{d\times d}+A^{\epsilon}_{1,N^{\Delta}-1}\theta \mathcal Db(y^{\Delta,\epsilon}_{t_{N^{\Delta}}})I^0\Delta\big)\sigma(y^{\Delta,\epsilon}_{t_{N^{\Delta}-1}})\sigma(y^{\Delta,\epsilon}_{t_{N^{\Delta}-1}})^{\top}\nn\\
&\quad+ \sigma(y^{\Delta,\epsilon}_{t_{N^{\Delta}-1}})\sigma(y^{\Delta,\epsilon}_{t_{N^{\Delta}-1}})^{\top} \big(\frac14\mathrm {Id}_{d\times d}+A^{\epsilon}_{1,N^{\Delta}-1}\theta \mathcal Db(y^{\Delta,\epsilon}_{t_{N^{\Delta}}})I^0\Delta\big)^{\top} \Big]u\nn\\
&\ge \frac12 \epsilon\Delta \min_{u\in\mathbb R^d,|u|=1}u^{\top}\sigma(y^{\Delta,\epsilon}_{t_{N^{\Delta}-1}})\sigma(y^{\Delta,\epsilon}_{t_{N^{\Delta}-1}})^{\top}u+\epsilon\Delta\sigma_0 \min _{u\in\mathbb R^d,|u|=1}u^{\top}\Big(\frac12\mathrm{Id}_{d\times d}\nn\\
&\quad+A^{\epsilon}_{1,N^{\Delta}-1}\theta \mathcal Db(y^{\Delta,\epsilon}_{t_{N^{\Delta}}})I^0\Delta+(A^{\epsilon}_{1,N^{\Delta}-1}\theta \mathcal Db(y^{\Delta,\epsilon}_{t_{N^{\Delta}}})I^0)^{\top}\Delta\Big)u\nn\\
&\ge\frac12 \epsilon \Delta \inf_{\phi\in\mathcal C^d}\min_{u\in\mathbb R^d,|u|=1}u^{\top}\sigma(\phi)\sigma(\phi)^{\top}u\nn\\
&\quad+\epsilon\Delta\sigma_0\min_{u\in\mathbb R^d,|u|=1}\Big(\frac12 |u|^2-2|u|^2|A^{\epsilon}_{1,N^{\Delta}-1}\theta \mathcal Db(y^{\Delta,\epsilon}_{t_{N^{\Delta}}})I^0|\Delta\Big)\\
&\ge\frac12 \epsilon\Delta\sigma_0, 
\end{align}
where we used Assumption \ref{assp_den}, the fact $\|A^{\epsilon}_{1,N^{\Delta}-1}\|_{\mathcal L(\mathbb R^d;\RR^d)}\leq 1$ showed in the proof of \cref{trun_num_den}, and the inequality 
$2|A^{\epsilon}_{1,N^{\Delta}-1}\theta\mathcal Db(y^{\Delta,\epsilon}_{t_{N^{\Delta}}})I^0|\Delta\leq \frac12$ for $\Delta\in(0, \frac{1}{4\theta a_5})\wedge(0,\frac{1}{4\theta\|\mathcal Db\|_{\mathcal L(\mathcal C^{d};\RR^d)})}]$. 
This gives that $\det(\gamma^{\epsilon}_{N^{\Delta}})^{-1}\leq (\frac12\sigma_0\epsilon\Delta)^{-d}.$

\end{proof}

Similar to proofs of \cref{multi_nonlip_nu,bound_control} with $\iota=0$, 
we obtain the following regularity estimates of  $y^{\Delta,\epsilon}$ and $w^{\Delta,v}$. 
\begin{prop}\label{prop_mo_bound}
Under the conditions in \cref{lemma_matrix}, for any $T>0$, $p\in\mathbb N_+$, and $\Delta\in(0,\frac{1}{4\theta a_5})$,
$$\sup_{0\leq t_k\leq T}|w^{\Delta,\epsilon}(t_k)|^{2p}\leq K_{T},\quad\quad\E\Big[\sup_{0\leq t_k\leq T}|y^{\Delta,\epsilon}(t_k)|^{2p}\Big]\leq K_{T}.$$
Moreover, for $n\in\mathbb N_+$,
\begin{align*}
\sup_{r_1,\ldots,r_{n}\in[0,T]}\mathbb E\Big[\sup_{r_1\vee \cdots\vee r_{n}\leq t_k\leq T}|D_{r_1,\ldots,r_{n}}y^{\Delta,\epsilon}(t_k)|^p\Big]\leq K_{T}.
\end{align*}
Therefore, $y^{\Delta,\epsilon}(t_k)\in\mathbb D^{\infty}(\mathbb R^d)$ for $k=1,2,\ldots,N^{\Delta}.$ 
\end{prop}
\begin{prop}\label{smooth_density}
Under the conditions in \cref{lemma_matrix}, for any $k\in\mathbb N_+$ and $\Delta\in(0, \frac{1}{4\theta a_5})\wedge(0,\frac{1}{4\theta\|\mathcal Db\|_{\mathcal L(\mathcal C^{d};\RR^d)})}]$, $y^{\Delta,\epsilon}(t_k)$ admits a smooth density $\mathfrak p^{\Delta,\epsilon}(t_k,\cdot)$.
\end{prop}
\begin{proof}
The proof follows  from  Lemma \ref{lemma_matrix}, Proposition \ref{prop_mo_bound}, and \cite[Theorem 2.1.4]{Nualart}.
\end{proof}
The limit of the logarithmic estimate on the finite time horizon of the numerical density function is closed related to the rate function of the LDP of $\{y^{\Delta,\epsilon}(t)\}_{\epsilon\in
(0,1)}$ for $t\in[0,T]$. To proceed, we list the LDPs of $\{y^{\Delta,\epsilon}(\cdot)\}_{\epsilon\in(0,1)}$ and $\{y^{\Delta,\epsilon}(t)\}_{\epsilon\in
(0,1)}$ in the following proposition, whose proofs are similar to those  of Theorem \ref{main_LDP} and Corollary \ref{coro_LDP1} and thus are omitted. Denote by $\mathcal C_{\xi}([-\tau,T];\mathbb R^d)$ the space of all continuous functions $u:[-\tau,T]\rightarrow \mathbb R^d$ with $u(r)=\xi(r)$ for $r\in[-\tau,0]$, endowed with the norm 
$\|u\|_{\mathcal C_{\xi}([-\tau,T];\mathbb R^d)}:=\sup_{t\in[-\tau,T]}|u(t)|$.

\begin{prop}
 Under the conditions in \cref{lemma_matrix}, it holds that: 
 
(\romannumeral1) The family of random variables  $\{y^{\Delta,\epsilon}(\cdot)\}_{\epsilon\in(0,1)}$ satisfies the LDP in $\mathcal C_{\xi}([-\tau,T];\mathbb R^d)$
with the good rate function given by 
\begin{align*}
I_T(f)=\inf_{\{v\in L^2_{\Delta}([0,T];\mathbb R^m):f=\mathcal G^{\Delta}(\int_0^{\cdot}v(\lfloor s\rfloor)\mathrm ds)\}}\frac1 2\int_0^T|v(\lfloor s\rfloor)|^2\mathrm ds.
\end{align*} 

(\romannumeral2) For any $t\in[0,T],$ the family of random variables  $\{y^{\Delta,\epsilon}(t)\}_{\epsilon\in
(0,1)}$ satisfies the LDP in $\mathbb R^d$ with the good rate function given by
\begin{align}\label{2444}
\tilde I_{t}(z)=\inf \{I_T(f):f\in\mathcal C_{\xi}([-\tau,T];\mathbb R^d),z=f(t)\}.
\end{align}
\end{prop}

Now we state our main result as follows. 

\begin{thm}\label{LDP_density}
Under the conditions in \cref{lemma_matrix}, for any $k\in\mathbb N_+$ and sufficiently small $\Delta>0$,
\begin{align*}
\lim_{\epsilon\to0}\epsilon\ln \mathfrak p^{\Delta,\epsilon}(t,y)= -\tilde I_{t}(y),\quad y\in\mathbb R^d.
\end{align*}
\end{thm}
We prove  Theorem \ref{LDP_density} by  
 introducing  the following equation
\begin{align*}
\left\{
\begin{array}{ll}
y^{\Delta,\epsilon, v}(t_{k+1})=y^{\Delta,\epsilon, v}(t_k)+(1-\theta)b(y^{\Delta,\epsilon, v}_{t_k})\Delta +\theta b(y^{\Delta,\epsilon, v}_{t_{k+1}})\Delta +\sqrt{\epsilon}\sigma(y^{\Delta,\epsilon, v}_{t_k})\delta W_k\\
\qquad \qquad \qquad +\sigma(y^{\Delta,\epsilon, v}_{t_k})v(t_k)\Delta,\quad k\in\mathbb N,\\
y^{\Delta,\epsilon, v}(t_k)=\xi(t_k),\quad k=-N,\ldots,0,
\end{array}
\right.
\end{align*}
where  $y^{\Delta,\epsilon, v}_{t_k}$ is the linear interpolation with respect to $y^{\Delta,\epsilon}(t_{k-N}),\ldots, y^{\Delta,\epsilon}(t_{k})$, and $v\in L^2_{\Delta}$.
Define $\mathbb Z^{\Delta, v}$ as follows:
\begin{align*}
\begin{cases}
\begin{array}{ll}
\mathbb Z^{\Delta, v}(t_{k+1})=\mathbb Z^{\Delta, v}(t_k)+(1-\theta)\mathcal Db(w^{\Delta, v}_{t_k})\mathbb Z^{\Delta, v}_{t_k}\Delta+\theta\mathcal Db(w^{\Delta, v}_{t_{k+1}})\mathbb Z^{\Delta, v}_{t_{k+1}}\Delta\\
\quad\quad\quad\quad\quad\quad+\mathcal D\sigma(w^{\Delta, v}_{t_k})\mathbb Z^{\Delta, v}_{t_k}v(t_k)\Delta +\sigma(w^{\Delta, v}_{t_k})\delta W_k,\quad k\in\mathbb N,\\
\mathbb Z^{\Delta, v}(t_k)=0,\quad k=-N,\ldots,0,
\end{array}
\end{cases}
\end{align*}
where $w^{\Delta, v}$ is the solution of \eqref{skeleton}.

Introduce the auxiliary processes associated to $y^{\Delta,\epsilon, v}$ and $\mathbb Z^{\Delta, v}$, respectively  as
\begin{align*}
\begin{cases}
\begin{array}{lll}
Z^{\Delta,\epsilon, v}(t_k)=\xi(t_k),\quad k=-N,\ldots,-1,\\
Z^{\Delta,\epsilon, v}(t_k)=y^{\Delta,\epsilon, v}(t_k)-\theta b(y^{\Delta,\epsilon, v}_{t_k})\Delta,\quad k=0,\\
Z^{\Delta,\epsilon, v}(t_{k+1})=Z^{\Delta,\epsilon, v}(t_{k})+b(y^{\Delta,\epsilon, v}_{t_{k}})\Delta+\sqrt{\epsilon}\sigma(y^{\Delta,\epsilon, v}_{t_{k}})\delta W_{k}\nn\\
\qquad \qquad \quad \quad\;+\sigma(y^{\Delta,\epsilon, v}_{t_{k}})v(t_{k})\Delta,\quad k\in\mathbb N,
\end{array}
\end{cases}
\end{align*}
and
\begin{align*}
\begin{cases}
\begin{array}{lll}
\widetilde{\mathbb Z}^{\Delta, v}(t_k)=0,\quad k=-N,\ldots,-1,\\
\widetilde{\mathbb Z}^{\Delta, v}(t_k)=\mathbb Z^{\Delta, v}(t_k)-\theta \mathcal Db(w^{\Delta, v}_{t_k})\mathbb Z^{\Delta, v}_{t_k}\Delta,\quad k=0,\\
\widetilde{\mathbb Z}^{\Delta, v}(t_{k+1}) =\widetilde{\mathbb Z}^{\Delta, v}(t_{k})+\mathcal Db(w^{\Delta, v}_{t_{k}}){\mathbb Z}^{\Delta, v}_{t_{k}}\Delta +\mathcal D\sigma(w^{\Delta, v}_{t_{k}})\mathbb Z^{\Delta, v}_{t_{k}}v(t_{k})\Delta\nn\\
\qquad \qquad \quad\quad+\sigma(w^{\Delta, v}_{t_{k}})\delta W_{k},\quad k\in\mathbb N. 
\end{array}
\end{cases}
\end{align*}
Moreover, define the continuous versions  $Z^{\Delta,\epsilon, v}(\cdot)$ and $\widetilde{\mathbb Z}^{\Delta, v}$$(\cdot)$ by
\begin{align*}
Z^{\Delta,\epsilon, v}(t)&:=Z^{\Delta,\epsilon, v}(t_{k})+b(y^{\Delta,\epsilon, v}_{t_{k}})(t-t_k)+\sqrt{\epsilon}\sigma(y^{\Delta,\epsilon, v}_{t_{k}})(W(t)-W(t_k))\nn\\
&\quad+\sigma(y^{\Delta,\epsilon, v}_{t_{k}})v(t_{k})(t-t_k),\\
\widetilde{\mathbb Z}^{\Delta, v}(t)&:=\widetilde{\mathbb Z}^{\Delta, v}(t_k)+\mathcal Db(w^{\Delta, v}_{t_{k}}){\mathbb Z}^{\Delta, v}_{t_{k}}(t-t_k)+\mathcal D\sigma(w^{\Delta, v}_{t_{k}})\mathbb Z^{\Delta, v}_{t_{k}}v(t_{k})(t-t_k)\nn\\
&\quad+\sigma(w^{\Delta, v}_{t_{k}})(W(t)-W(t_k))
\end{align*}
for $t\in[t_k,t_{k+1})$ with $k\in\mathbb N$, and 
$Z^{\Delta,\epsilon, v}(t):=\xi(t)$, $\widetilde{\mathbb Z}^{\Delta, v}(t):=0$ for $t\in[-\tau,0]$.



Similar to the proof of \cref{multi_nonlip_nu}, 
we obtain the following regularity estimates of  $y^{\Delta,\epsilon,v}$. 
\begin{prop}\label{Chp5Y+1}
Let \cref{a1,repalceA2,assp_den,smooth_assp}  hold. Then for any $T>0$, $p\in\mathbb N_+$, and
$\Delta\in(0,\frac{1}{a\theta a_5})$,
$$\E\Big[\sup_{0\leq t_k\leq T}|y^{\Delta,\epsilon,v}(t_k)|^{2p}\Big]\leq K_{T}.$$
\end{prop}

\begin{proof}[\textbf{Proof of Theorem \ref{LDP_density}.}]
We first give the proof of 
\begin{align}\label{proof_logupp}\limsup_{\epsilon\to0}\epsilon\ln \mathfrak p^{\Delta,\epsilon}(t,y)\leq -\tilde I_t(y).
\end{align}
By Lemma \ref{lemma_matrix} and the boundedness of the $p$th moment of $y^{\Delta,\epsilon}$, we obtain  
\begin{align*}
\|(\gamma^{\epsilon}_{N^{\Delta}})^{-1}\|_{L^p(\Omega)}\leq K\|\det(\gamma^{\epsilon}_{N^{\Delta}})^{-1}\|_{L^{2p}(\Omega)}\leq K\epsilon^{-d}\Delta^{-d}.
\end{align*}
This, combining \eqref{2444}, Proposition \ref{prop_mo_bound}, and \cite[Proposition 4.1]{Ma_functional} finishes the proof of \eqref{proof_logupp}. 

Next, we show that $\liminf_{\epsilon\to0}\epsilon\ln \mathfrak p^{\Delta,\epsilon}(t,y)\ge -\tilde I_t(y)$, whose proof is divided  
 into the following five steps.

\underline{Step 1.} Prove that for any $p\ge 1$ and $\Delta\in(0, \frac{1}{8\theta a_5}]\wedge(0,\frac{1}{4\theta\|\mathcal Db\|_{\mathcal L(\mathcal C^{d};\RR^d)})}]$,
$$\lim_{\epsilon\to 0}\mathbb E\Big[\sup_{0\leq k\leq N^{\Delta}}|y^{\Delta,\epsilon, v}(t_k)-w^{\Delta, v}(t_k)|^p\Big]=0.$$

By the definitions  of $Z^{\Delta,\epsilon, v}$ and $z^{\Delta, v},$ we have that for $p\ge 2,$
\begin{align*}
|Z^{\Delta,\epsilon, v}(t)-z^{\Delta, v}(t)|^p&\leq K\Big|\int_0^tb(y^{\Delta,\epsilon, v}_{\lfloor s\rfloor})-b(w^{\Delta, v}_{\lfloor s\rfloor})\mathrm ds\Big|^p+K\Big|\int_0^t\sigma(y^{\Delta,\epsilon, v}_{\lfloor s\rfloor})v(\lfloor s\rfloor)\nn\\
&\quad-\sigma(w^{\Delta, v}_{\lfloor s\rfloor})v(\lfloor s\rfloor)\mathrm ds\Big|^p+K\Big|\sqrt{\epsilon}\int_0^t\sigma(y^{\Delta,\epsilon, v}_{\lfloor s\rfloor})\mathrm dW(s)\Big|^p.
\end{align*}
It follows from the boundedness of derivatives of $b$, $\sigma$, and  the Burkholder--Davis--Gundy inequality that
 \begin{align*}
 &\mathbb  E\Big[\sup_{t\in[0,T]}|Z^{\Delta,\epsilon, v}(t)-z^{\Delta, v}(t)|^p\Big]\leq K_T\int_0^T\mathbb E[\|y^{\Delta,\epsilon, v}_{\lfloor s\rfloor}\!-\!w^{\Delta, v}_{\lfloor s\rfloor}\|^p]\mathrm ds
 +K_T\int_0^T\mathbb E[\|y^{\Delta,\epsilon, v}_{\lfloor s\rfloor}\nn\\
 &\quad-w^{\Delta, v}_{\lfloor s\rfloor}\|^p]\mathrm ds\big(\int_0^T|v(\lfloor s\rfloor)|^2\mathrm ds\big)^{\frac p2}+K_T\epsilon^{\frac p2}\int_0^T\big(1+\mathbb E[\|y^{\Delta,\epsilon, v}_{\lfloor s\rfloor}\|^p]\big)\mathrm ds\\
 &\leq K_T\int_0^T\mathbb E[\sup_{r\in[0,s]}|y^{\Delta,\epsilon, v}(\lfloor r\rfloor)-w^{\Delta, v}(\lfloor r\rfloor)|^p]\mathrm ds+K_T\epsilon^{\frac p2}\int_0^T\big(1+\mathbb E[\|y^{\Delta,\epsilon, v}_{\lfloor s\rfloor}\|^p]\big)\mathrm ds.
 \end{align*}
Moreover, we have
\begin{align*}
&\quad|y^{\Delta,\epsilon, v}(t_{k+1})-w^{\Delta, v}(t_{k+1})|^2\leq |Z^{\Delta,\epsilon, v}(t_{k+1})-z^{\Delta, v}(t_{k+1})|^2\\
&\quad+2\theta\Delta \<y^{\Delta,\epsilon, v}(t_{k+1})-w^{\Delta, v}(t_{k+1}),b(y^{\Delta,\epsilon, v}_{t_{k+1}})-b(w^{\Delta, v}_{t_{k+1}})\>\\
&\leq |Z^{\Delta,\epsilon, v}(t_{k+1})-z^{\Delta, v}(t_{k+1})|^2+4\theta a_5\Delta \sup_{0\leq k\leq N^{\Delta}}|y^{\Delta,\epsilon, v}(t_{k})-w^{\Delta, v}(t_{k})|^2.
\end{align*}
By  $\Delta\leq \frac{1}{8\theta a_5}$, we obtain 
\begin{align}\label{inver_ineq}
\sup_{0\leq k\leq N^{\Delta}}|y^{\Delta,\epsilon, v}(t_{k})-w^{\Delta, v}(t_{k})|^2&\leq 2\sup_{0\leq k\leq N^{\Delta}}|Z^{\Delta,\epsilon, v}(t_{k})-z^{\Delta, v}(t_{k})|^2,
\end{align}
which implies that
\begin{align*}
\mathbb E\Big[\sup_{0\leq k\leq N^{\Delta}}|y^{\Delta,\epsilon, v}(t_{k})\!-\!w^{\Delta, v}(t_{k})|^p\Big]&\leq K_T\int_0^T\mathbb E[\sup_{r\in[0,s]}|y^{\Delta,\epsilon, v}(\lfloor r\rfloor)\!-\!w^{\Delta, v}(\lfloor r\rfloor)|^p]\mathrm ds\\
&\quad +K_T\epsilon^{\frac p2}\int_0^T\big(1+\mathbb \sup_{s\in[0,T]}E[\|y^{\Delta,\epsilon, v}_{\lfloor s\rfloor}\|^p]\big)\mathrm ds.
\end{align*}
According to \cref{Chp5Y+1} and the Gr\"onwall inequality, one has 
\begin{align*}
\mathbb E\Big[\sup_{0\leq k\leq N^{\Delta}}|y^{\Delta,\epsilon, v}(t_{k})-w^{\Delta, v}(t_{k})|^p\Big]\leq K_T\epsilon^{\frac p2},
\end{align*}
which finishes the proof of \underline{Step 1} by taking $\epsilon\to0.$

\underline{Step 2.} Prove that for any $p\ge 1,$  $$\lim_{\epsilon\to0}\mathbb E\Big[\sup_{0\leq k\leq N^{\Delta}}|Y^{\Delta,\epsilon,v}(t_k)-\mathbb Z^{\Delta, v}(t_k)|^p\Big]=0,$$ where $Y^{\Delta,\epsilon, v}(t_k):=\frac{1}{\sqrt{\epsilon}}(y^{\Delta,\epsilon, v}(t_k)-w^{\Delta, v}(t_k)).$

By the definitions of $Z^{\Delta,\epsilon, v}(t)$, $z^{\Delta, v}(t)$, and $\widetilde{\mathbb Z}^{\Delta, v}(t)$, we arrive at 
\begin{align*}
\frac{1}{\sqrt{\epsilon}}(Z^{\Delta,\epsilon, v}(t)-z^{\Delta, v}(t))\!-\!\widetilde{\mathbb Z}^{\Delta, v}(t)=\mathscr A_1(t)+\mathscr A_2(t)+\mathscr A_3(t),
\end{align*}
where
\begin{align*}
\mathscr A_1(t)&:=\int_0^t\Big[\frac{1}{\sqrt{\epsilon}}\big(b(y^{\Delta,\epsilon, v}_{\lfloor s\rfloor})\!-\!b(w^{\Delta, v}_{\lfloor s\rfloor})\big)\!-\!\mathcal Db(w^{\Delta, v}_{\lfloor s\rfloor})\mathbb Z^{\Delta, v}_{\lfloor s\rfloor}\Big]\mathrm ds,\\
\mathscr A_2(t)&:=\int_0^t\Big[\frac{1}{\sqrt{\epsilon}}\big(\sigma(y^{\Delta\epsilon,v}_{\lfloor s\rfloor})v(\lfloor s\rfloor)-\sigma(w^{\Delta, v}_{\lfloor s\rfloor})v(\lfloor s\rfloor)\big)-\mathcal D\sigma(w^{\Delta, v}_{\lfloor s\rfloor})\mathbb Z^{\Delta, v}_{\lfloor s\rfloor}v(\lfloor s\rfloor)\Big]\mathrm ds,\\
\mathscr A_3(t)&:=\int_0^t\Big[\sigma(y^{\Delta,\epsilon, v}_{\lfloor s\rfloor})-\sigma(w^{\Delta, v}_{\lfloor s\rfloor})\Big]\mathrm dW(s)=:\mathscr A_1(t)+\mathscr A_2(t)+\mathscr A_3(t).
\end{align*}
For the term $\mathscr A_3,$ utilizing the Burkholder--Davis--Gundy inequality, we deduce that for $p\ge 2,$
\begin{align*}
\mathbb E\Big[\sup_{t\leq T}|\mathscr A_3(t)|^p\Big]&\leq K_T\mathbb E\Big[\int_0^T\|y^{\Delta,\epsilon, v}_{\lfloor s\rfloor}-w^{\Delta, v}_{\lfloor s\rfloor}\|^p\mathrm ds\Big]\\
&\leq K_T\int_0^T\mathbb E\Big[\sup_{0\leq k\leq N^{\Delta}}|y^{\Delta,\epsilon, v}(t_k)-w^{\Delta, v}(t_k)|^p\Big]\mathrm ds\leq K_T\epsilon^{\frac p2},
\end{align*}
where in the last step we used  \underline{Step 1}.

For the term $\mathscr A_1,$ by the Taylor formula,  we obtain
\begin{align*}
b(y^{\Delta,\epsilon, v}_{\lfloor s\rfloor})-b(w^{\Delta, v}_{\lfloor s\rfloor})&=\mathcal Db(w^{\Delta, v}_{\lfloor s\rfloor})(y^{\Delta,\epsilon, v}_{\lfloor s\rfloor}-w^{\Delta, v}_{\lfloor s\rfloor})+\int_0^1(1-\varsigma)\mathcal D^2b\big((1-\varsigma)w^{\Delta, v}_{\lfloor s\rfloor}\nn\\
&\quad+\varsigma y^{\Delta,\epsilon, v}_{\lfloor s\rfloor}\big)\mathrm d\varsigma(y^{\Delta,\epsilon, v}_{\lfloor s\rfloor}-w^{\Delta, v}_{\lfloor s\rfloor},y^{\Delta,\epsilon, v}_{\lfloor s\rfloor}-w^{\Delta, v}_{\lfloor s\rfloor}),
\end{align*}
which implies 
\begin{align}\label{taylor1}
&\quad \frac{1}{\sqrt{\epsilon}}\big(b(y^{\Delta,\epsilon, v}_{\lfloor s\rfloor})-b(w^{\Delta, v}_{\lfloor s\rfloor})\big)-\mathcal Db(w^{\Delta, v}_{\lfloor s\rfloor})\mathbb Z^{\Delta, v}_{\lfloor s\rfloor}\nn\\
&=\mathcal Db(w^{\Delta, v}_{\lfloor s\rfloor})(Y^{\Delta,\epsilon, v}_{\lfloor s\rfloor}-\mathbb Z^{\Delta, v}_{\lfloor s\rfloor})+\int_0^1(1-\varsigma)\mathcal D^2b\big((1-\varsigma)w^{\Delta, v}_{\lfloor s\rfloor}\nn\\
&\quad+\varsigma y^{\Delta,\epsilon, v}_{\lfloor s\rfloor}\big)\mathrm d\varsigma(Y^{\Delta,\epsilon, v}_{\lfloor s\rfloor},y^{\Delta,\epsilon, v}_{\lfloor s\rfloor}-w^{\Delta, v}_{\lfloor s\rfloor}).
\end{align}
It follows  from the H\"older inequality, \cref{smooth_assp}, and \underline{Step 1} that for $p\ge 2,$
\begin{align*}
&\quad\mathbb E\Big[\sup_{t\leq T}|\mathscr A_1(t)|^p\Big]\leq K_T\E\Big[\int_0^T\mathbb \|Y^{\Delta,\epsilon, v}_{\lfloor s\rfloor}-\mathbb Z^{\Delta, v}_{\lfloor s\rfloor}\|^p+\|Y^{\Delta,\epsilon, v}_{\lfloor s\rfloor}\|^p\|y^{\Delta,\epsilon, v}_{\lfloor s\rfloor}-w^{\Delta, v}_{\lfloor s\rfloor}\|^p\mathrm ds\Big]\\
&\leq K_T\!\int_0^T\!\mathbb E\big[\sup_{r\leq s}|Y^{\Delta,\epsilon, v}(\lfloor r\rfloor)\!-\!\mathbb Z^{\Delta, v}(\lfloor r\rfloor)|^p\big]\mathrm ds\!+\!K_T\Big(\!\int_0^T\!\mathbb E\big[\sup_{0\leq k\leq N^{\Delta}}|Y^{\Delta,\epsilon, v}(t_k)|^{2p}\big]\mathrm ds\!\\
&\quad\times \int_0^T\mathbb E\big[\sup_{0\leq k\leq N^{\Delta}}|y^{\Delta,\epsilon, v}(t_k)-w^{\Delta, v}(t_k)|^{2p}\big]\mathrm ds\Big)^{\frac12}\\
&\leq K_T\int_0^T\mathbb E\Big[\sup_{r\leq s}|Y^{\Delta,\epsilon, v}(\lfloor r\rfloor)-\mathbb Z^{\Delta, v}(\lfloor r\rfloor)|^p\Big]\mathrm ds+K_T\epsilon^{\frac p2}.
\end{align*}
Similarly, for the term $\mathscr A_2,$ by virtue of the Taylor formula and the H\"older inequality,  we have that for $p\ge 2,$
\begin{align*}
&\quad \mathbb E\Big[\sup_{t\leq T}|\mathscr A_2(t)|^p\Big]\leq \mathbb E \Big[\Big(\int_0^T\Big|\mathcal D\sigma (w^{\Delta, v}_{\lfloor s\rfloor})(Y^{\Delta,\epsilon, v}_{\lfloor s\rfloor}-\mathbb Z^{\Delta, v}_{\lfloor s\rfloor})+\int_0^1(1-\varsigma)\times\nn\\
&\quad\mathcal D ^2\sigma\big((1-\varsigma)w^{\Delta, v}_{\lfloor s\rfloor}+\lambda y^{\Delta,\epsilon, v}_{\lfloor s\rfloor}\big)\mathrm d\varsigma(Y^{\Delta,\epsilon, v}_{\lfloor s\rfloor},y^{\Delta,\epsilon, v}_{\lfloor s\rfloor}-w^{\Delta, v}_{\lfloor s\rfloor})\Big||v(\lfloor s\rfloor)|\mathrm ds\Big)^p\Big]\\
&\leq K_T\int_0^T\mathbb E\Big[\sup_{r\leq s}|Y^{\Delta,\epsilon, v}(\lfloor r\rfloor)-\mathbb Z^{\Delta, v}(\lfloor r\rfloor)|^p\Big]\mathrm ds\big(\int_0^T|v(\lfloor s\rfloor)|^2\mathrm ds\big)^{\frac p2}\nn\\
&\quad+K\!\Big(\int_0^T\!\mathbb E\Big[\sup_{0\leq k\leq N^{\Delta}}|Y^{\Delta,\epsilon, v}(t_k)|^{2p}\Big]\mathrm ds\int_0^T\mathbb E\Big[\sup_{0\leq k\leq N^{\Delta}}|y^{\Delta,\epsilon, v}(t_k)\nn\\
&\quad-w^{\Delta, v}(t_k)|^{2p}\Big]\mathrm ds\Big)^{\frac12}\big(\int_0^T|v(\lfloor s\rfloor)|^2\mathrm ds\big)^{\frac p2}\\
&\leq K_T\int_0^T\mathbb E\Big[\sup_{r\leq s}|Y^{\Delta,\epsilon, v}(\lfloor r\rfloor)-\mathbb Z^{\Delta, v}(\lfloor r\rfloor)|^p\Big]\mathrm ds+K_T\epsilon^{\frac p2}.
\end{align*}
Combining estimates of terms $\mathscr A_1,\mathscr A_2,$ and $\mathscr A_3$, we arrive at that for $p\ge 2,$
\begin{align}\label{inter_1}
&\quad\mathbb E\Big[\sup_{t\leq T}\Big|\frac{1}{\sqrt{\epsilon}}(Z^{\Delta,\epsilon, v}(t)-z^{\Delta, v}(t))-\widetilde{\mathbb Z}^{\Delta, v}(t)\Big|^p\Big]\nn\\
&\leq K_T\int_0^T\mathbb E\Big[\sup_{r\leq s}|Y^{\Delta,\epsilon, v}(\lfloor r\rfloor)-\mathbb Z^{\Delta, v}(\lfloor r\rfloor)|^p\Big]\mathrm ds+K_T\epsilon^{\frac p2}.
\end{align}
Noting that
\begin{align*}
&\quad\Big|\frac{1}{\sqrt{\epsilon}}(Z^{\Delta,\epsilon, v}(t_k)-z^{\Delta, v}(t_k))-\widetilde{\mathbb Z}^{\Delta, v}(t_k)\Big|^2\nn\\
&\ge \Big| \frac{1}{\sqrt{\epsilon}}(y^{\Delta,\epsilon, v}(t_k)-w^{\Delta, v}(t_k))-{\mathbb Z}^{\Delta, v}(t_k)\Big|^2-2\theta\Delta \Big\<\frac{1}{\sqrt{\epsilon}}(y^{\Delta,\epsilon, v}(t_k)-w^{\Delta, v}(t_k))\nn\\
&\quad-{\mathbb Z}^{\Delta, v}(t_k), \frac{1}{\sqrt{\epsilon}}(b(y^{\Delta,\epsilon, v}_{t_k})-b(w^{\Delta, v}_{t_k}))-\mathcal Db(w^{\Delta, v}_{t_k}){\mathbb Z}^{\Delta, v}_{t_k}\Big\>,
\end{align*}
we derive from \eqref{taylor1}, \cref{smooth_assp}, and the Young inequality that for $\Delta\leq\frac{1}{8\theta a_5}$,
\begin{align*}
&\quad\Big| \frac{1}{\sqrt{\epsilon}}(y^{\Delta,\epsilon, v}(t_k)-w^{\Delta, v}(t_k))-{\mathbb Z}^{\Delta, v}(t_k)\Big|^2\nn\\
&\leq \Big|\frac{1}{\sqrt{\epsilon}}(Z^{\Delta,\epsilon, v}(t_k)-z^{\Delta, v}(t_k))-\widetilde{\mathbb Z}^{\Delta, v}(t_k)\Big|^2+5\theta a_5\Delta \sup_{0\leq k\leq N^{\Delta}}|Y^{\Delta,\epsilon, v}(t_k)\nn\\
&\quad-\mathbb Z^{\Delta, v}(t_k)|^2+K\Delta \sup_{0\leq k\leq N^{\Delta}}|Y^{\Delta,\epsilon, v}(t_k)|^2|y^{\Delta,\epsilon, v}(t_k)-w^{\Delta, v}(t_k)|^2.
\end{align*}
This indicates 
\begin{align*}
&\quad\sup_{0\leq k\leq N^{\Delta}}\Big| Y^{\Delta,\epsilon, v}(t_k)-{\mathbb Z}^{\Delta, v}(t_k)\Big|^2\nn\\
&\leq 3 \sup_{0\leq k\leq N^{\Delta}}\Big|\frac{1}{\sqrt{\epsilon}}(Z^{\Delta,\epsilon, v}(t_k)-z^{\Delta, v}(t_k))-\widetilde{\mathbb Z}^{\Delta, v}(t_k)\Big|^2\\
&\quad+K\sup_{0\leq k\leq N^{\Delta}}|Y^{\Delta,\epsilon, v}(t_k)|^2|y^{\Delta,\epsilon, v}(t_k)-w^{\Delta, v}(t_k)|^2,
\end{align*}
which together with \eqref{inter_1}, and \cref{prop_mo_bound,Chp5Y+1}, yields  that for $p\ge 2,$
\begin{align*}
&\quad\mathbb E\Big[\sup_{0\leq k\leq N^{\Delta}}\Big| Y^{\Delta,\epsilon, v}(t_k)-{\mathbb Z}^{\Delta, v}(t_k)\Big|^p\Big]\nn\\
&\leq K_T\int_0^T\mathbb E\Big[\sup_{r\leq s}|Y^{\Delta,\epsilon, v}(\lfloor r\rfloor)-\mathbb Z^{\Delta, v}(\lfloor r\rfloor)|^p\Big]\mathrm ds+K_T\epsilon^{\frac p2}\\
&\quad+K\Big(\mathbb E\Big[\sup_{0\leq k\leq N^{\Delta}}|Y^{\Delta,\epsilon, v}(t_k)|^{2p}\Big]\mathbb E\Big[\sup_{0\leq k\leq N^{\Delta}}|y^{\Delta,\epsilon, v}(t_k)-w^{\Delta, v}(t_k)| ^{2p}\Big]\Big)^{\frac12}\\
&\leq K_T\int_0^T\mathbb E\Big[\sup_{r\leq s}|Y^{\Delta,\epsilon, v}(\lfloor r\rfloor)-\mathbb Z^{\Delta, v}(\lfloor r\rfloor)|^p\Big]\mathrm ds+K_T\epsilon^{\frac p2}.
\end{align*}
Applying the Gr\"owall inequality completes the proof of \underline{Step 2}.

\underline{Step 3.} Prove that for any $p\ge 1,$ 
$\lim_{\epsilon\to0}\mathbb E\Big[\sup_{\alpha\leq t_k\leq T}|D_{\alpha}y^{\Delta,\epsilon, v}(t_k)|^p\Big]=0.$

It follows from the definition of $Z^{\Delta,\epsilon, v}$ that $D_{\alpha}Z^{\Delta,\epsilon, v}$ satisfies  
\begin{align}\label{DZ}
D_{\alpha}Z^{\Delta,\epsilon, v}(t)&=\int_0^t\mathcal Db(y^{\Delta,\epsilon, v}_{\lfloor s\rfloor})D_{\alpha}y^{\Delta,\epsilon, v}_{\lfloor s\rfloor}\mathrm ds+\int_0^t\sqrt{\epsilon}\mathcal D\sigma(y^{\Delta,\epsilon, v}_{\lfloor s\rfloor})D_{\alpha}y^{\Delta,\epsilon, v}\mathrm dW(s)\nn\\
&\quad +\int_0^t\mathcal D\sigma(y^{\Delta,\epsilon, v}_{\lfloor s\rfloor})D_{\alpha}y^{\Delta,\epsilon, v}_{\lfloor s\rfloor}v(\lfloor s\rfloor)\mathrm ds+\sqrt{\epsilon}\sigma(y^{\Delta,\epsilon, v}_{\lfloor \alpha\rfloor})\mathbf 1_{[0,t]}(\alpha). 
\end{align}
We obtain from the Burkholder--Davis--Gundy inequality and the H\"older inequality that for $p\ge 2,$
\begin{align*}
&\quad\mathbb E\Big[\sup_{t\leq T}|D_{\alpha}Z^{\Delta,\epsilon, v}(t)|^p\Big]\nn\\
&\leq K_T\int_0^T\mathbb E[\|D_{\alpha}y^{\Delta,\epsilon, v}_{\lfloor s\rfloor}\|^p]\mathrm ds+K_T\epsilon^{\frac p2}\int_0^T\mathbb E[\|D_{\alpha}y^{\Delta,\epsilon, v}_{\lfloor s\rfloor}\|^p]\mathrm ds+K\epsilon^{\frac p2}\big(1+\mathbb E[\|y^{\Delta,\epsilon, v}_{\lfloor \alpha\rfloor}\|^p]\big)\\
&\leq K_T\int_0^T\mathbb E\Big[\sup_{r\leq s}|D_{\alpha}y^{\Delta,\epsilon, v}(\lfloor r\rfloor)|^p\Big]\mathrm ds+K\epsilon^{\frac p2}\big(1+\mathbb E[\|y^{\Delta,\epsilon, v}_{\lfloor \alpha\rfloor}\|^p]\big).
\end{align*}
Moreover, similar to the proof of \eqref{inver_ineq},  we have
\begin{align*}
\sup_{0\leq k\leq N^{\Delta}}|D_{\alpha}y^{\Delta,\epsilon, v}(t_k)|^2\leq 2\sup_{0\leq k\leq N^{\Delta}}|D_{\alpha}Z^{\Delta,\epsilon, v}(t_k)|^2,
\end{align*}
 which together with the Gr\"onwall inequality finishes the proof of \underline{Step 3}.

\underline{Step 4.} Prove that for any $p\ge 1$ and sufficiently small $\Delta>0$,  $$\lim_{\epsilon\to0}\mathbb E\Big[\sup_{ \alpha\leq t_k\leq T}|D_{\alpha}Y^{\Delta,\epsilon, v}(t_k)-D_{\alpha}\mathbb Z^{\Delta, v}(t_k)|^p\Big]=0.$$

Since  $D_{\alpha}\widetilde {\mathbb Z}^{\Delta, v}$ satisfies 
\begin{align*}
D_{\alpha}\widetilde{\mathbb Z}^{\Delta, v}(t)=&\int_0^t\mathcal Db(w^{\Delta, v}_{\lfloor s\rfloor})D_{\alpha}\mathbb Z^{\Delta, v}_{\lfloor s\rfloor}\mathrm ds+\int_0^t\mathcal D\sigma(w^{\Delta, v}_{\lfloor s\rfloor})D_{\alpha}\mathbb Z^{\Delta, v}_{\lfloor s\rfloor}v(\lfloor s\rfloor)\mathrm ds\nn\\
&+\sigma(w^{\Delta, v}_{\lfloor \alpha\rfloor})\mathbf 1_{[0,t]}(\alpha),
\end{align*}
we derive from \eqref{DZ} that
\begin{align*}
&\frac{1}{\sqrt{\epsilon}}D_{\alpha}Z^{\Delta,\epsilon, v}(t)-D_{\alpha}\widetilde{\mathbb Z}^{\Delta, v}(t)=\tilde {\mathscr A}_1(t)+\tilde {\mathscr A}_2(t)+\tilde {\mathscr A}_3(t)+\tilde {\mathscr A}_4(t),
\end{align*}
where
\begin{align*}
\tilde {\mathscr A}_1(t)&:=\int_0^t\Big(\frac{1}{\sqrt{\epsilon}}\mathcal Db(y^{\Delta,\epsilon, v}_{\lfloor s\rfloor})D_{\alpha}y^{\Delta,\epsilon, v}_{\lfloor s\rfloor}-\mathcal Db(w^{\Delta, v}_{\lfloor s\rfloor})D_{\alpha}\mathbb Z^{\Delta, v}_{\lfloor s\rfloor}\Big)\mathrm ds,\\
\tilde {\mathscr A}_2(t)&:=\int_0^t\Big(\frac{1}{\sqrt{\epsilon}}\mathcal D\sigma(y^{\Delta,\epsilon, v}_{\lfloor s\rfloor})D_{\alpha}y^{\Delta,\epsilon, v}_{\lfloor s\rfloor}-\mathcal D\sigma(w^{\Delta, v}_{\lfloor s\rfloor})D_{\alpha}\mathbb Z^{\Delta, v}_{\lfloor s\rfloor}\Big)v(\lfloor s\rfloor)\mathrm ds,\\
\tilde {\mathscr A}_3(t)&:=(\sigma (y^{\Delta,\epsilon, v}_{\lfloor \alpha\rfloor})-\sigma(w^{\Delta, v}_{\lfloor \alpha \rfloor}))\mathbf 1_{[0,t]}(\alpha),\nn\\
\tilde {\mathscr A}_4(t)&:=\int_0^t\mathcal D\sigma(y^{\Delta,\epsilon, v}_{\lfloor s\rfloor})D_{\alpha}y^{\Delta,\epsilon, v}_{\lfloor s\rfloor}\mathrm dW(s).
\end{align*}
By the Taylor formula, we obtain
\begin{align*}
&\quad\big(\mathcal Db(y^{\Delta,\epsilon, v}_{\lfloor s\rfloor})-\mathcal Db(w^{\Delta, v}_{\lfloor s\rfloor})\big)D_{\alpha}y^{\Delta,\epsilon, v}_{\lfloor s\rfloor}=\mathcal D^2b(w^{\Delta, v}_{\lfloor s\rfloor})(y^{\Delta,\epsilon, v}_{\lfloor s\rfloor}-w^{\Delta, v}_{\lfloor s\rfloor},D_{\alpha}y^{\Delta,\epsilon, v}_{\lfloor s\rfloor})\\
&\quad +\int_0^1\mathcal D^3b((1-\varsigma)w^{\Delta, v}_{\lfloor s\rfloor}+\varsigma y^{\Delta,\epsilon, v}_{\lfloor s\rfloor})\mathrm d\varsigma\big(y^{\Delta,\epsilon, v}_{\lfloor s\rfloor}-w^{\Delta, v}_{\lfloor s\rfloor},y^{\Delta,\epsilon, v}_{\lfloor s\rfloor}-w^{\Delta, v}_{\lfloor s\rfloor},D_{\alpha}y^{\Delta,\epsilon, v}_{\lfloor s\rfloor}\big).
\end{align*}
Hence,  for $p\ge 2, $
\begin{align*}
&\quad\mathbb  E\Big[\sup_{t\leq T}|\tilde {\mathscr A}_1(t)|^p\Big]
= \mathbb E\Big[\sup_{t\leq T}\Big|\int_0^t\frac{1}{\sqrt{\epsilon}}(\mathcal Db(y^{\Delta,\epsilon, v}_{\lfloor s\rfloor})-\mathcal Db(w^{\Delta, v}_{\lfloor s\rfloor}))D_{\alpha}y^{\Delta,\epsilon, v}_{\lfloor s\rfloor}\nn\\
&\quad+\mathcal Db(w^{\Delta, v}_{\lfloor s\rfloor})\big(\frac{1}{\sqrt{\epsilon}}D_{\alpha}y^{\Delta,\epsilon, v}_{\lfloor s\rfloor}-D_{\alpha}\mathbb Z^{\Delta, v}_{\lfloor s\rfloor}\big)\mathrm ds\Big|^p\Big]\\
&\leq K_T\int_0^T\mathbb E\Big[\|Y^{\Delta,\epsilon, v}_{\lfloor s\rfloor}\|^p\|D_{\alpha}y^{\Delta,\epsilon, v}_{\lfloor s\rfloor}\|^p+\|Y^{\Delta,\epsilon, v}_{\lfloor s\rfloor}\|^p\|y^{\Delta,\epsilon, v}_{\lfloor s\rfloor}-w^{\Delta, v}_{\lfloor s\rfloor}\|^p\|D_{\alpha}y^{\Delta,\epsilon, v}_{\lfloor s\rfloor}\|^p\nn\\
&\quad+\|\frac{1}{\sqrt{\epsilon}}D_{\alpha}y^{\Delta,\epsilon, v}_{\lfloor s\rfloor}-D_{\alpha}\mathbb Z^{\Delta, v}_{\lfloor s\rfloor}\|^p\Big]\mathrm ds\\
&\leq K_T\Big(\mathbb E\Big[\sup_{0\leq k\leq N^{\Delta}}|Y^{\Delta,\epsilon, v}(t_k)|^{2p}\Big]\mathbb E\Big[\sup_{0\leq k\leq N^{\Delta}}|D_{\alpha}y^{\Delta,\epsilon, v}(t_k)|^{2p}\Big]\Big)^{\frac12}\\
&\quad+K_T\Big(\mathbb E\Big[\sup_{0\leq k\leq N^{\Delta}}|Y^{\Delta,\epsilon, v}(t_k)|^{2p}\Big]\Big)^{\frac12}\Big(\mathbb E\Big[\sup_{0\leq k\leq N^{\Delta}}|y^{\Delta,\epsilon, v}(t_k)-w^{\Delta, v}(t_k)|^{4p}\Big]\times\nn\\
&\quad\E\Big[\sup_{0\leq k\leq N^{\Delta}}|Y^{\Delta,\epsilon, v}(t_k)|^{4p}\Big]\Big)^{\frac14}\nn\\
&+K_T\int_0^T\mathbb E\Big[\sup_{r\leq s}|\frac{1}{\sqrt{\epsilon}}D_{\alpha}y^{\Delta,\epsilon, v}(\lfloor r\rfloor)-D_{\alpha}\mathbb Z^{\Delta, v}(\lfloor r\rfloor)|^p\Big]\mathrm ds. 
\end{align*}
Similarly, 
\begin{align*}
&\quad\mathbb E\Big[\sup_{t\leq T}|\tilde {\mathscr A}_2(t)|^p\Big]\leq K_T\Big(\mathbb E\Big[\sup_{0\leq k\leq N^{\Delta}}|Y^{\Delta,\epsilon, v}(t_k)|^{2p}\Big]\mathbb E\Big[\sup_{0\leq k\leq N^{\Delta}}|D_{\alpha}y^{\Delta,\epsilon, v}(t_k)|^{2p}\Big]\Big)^{\frac12}\\
&\quad+K_T\Big(\mathbb E\Big[\sup_{0\leq k\leq N^{\Delta}}|Y^{\Delta,\epsilon, v}(t_k)|^{2p}\Big]\Big)^{\frac12}\Big(\mathbb E\Big[\sup_{0\leq k\leq N^{\Delta}}|y^{\Delta,\epsilon, v}(t_k)-w^{\Delta, v}(t_k)|^{4p}\Big]\times\nn\\
&\quad\E\Big[\sup_{0\leq k\leq N^{\Delta}}|Y^{\Delta,\epsilon, v}(t_k)|^{4p}\Big]\Big)^{\frac14}\nn\\
&+K_T\int_0^T\mathbb E\Big[\sup_{r\leq s}|\frac{1}{\sqrt{\epsilon}}D_{\alpha}y^{\Delta,\epsilon, v}(\lfloor r\rfloor)-D_{\alpha}\mathbb Z^{\Delta, v}(\lfloor r\rfloor)|^p\Big]\mathrm ds.
\end{align*}
Consequently, we arrive at
\begin{align}\label{inter_DZeq}
&\quad\mathbb E\Big[\sup_{t\leq T}\Big|\frac{1}{\sqrt{\epsilon}}D_{\alpha}Z^{\Delta,\epsilon, v}(t)-D_{\alpha}\widetilde{\mathbb Z}^{\Delta, v}(t)\Big|^p\Big]\nn\\
&\leq K_T\Big(\mathbb E\Big[\sup_{0\leq k\leq N^{\Delta}}|Y^{\Delta,\epsilon, v}(t_k)|^{2p}\Big]\mathbb E\Big[\sup_{0\leq k\leq N^{\Delta}}|D_{\alpha}y^{\Delta,\epsilon, v}(t_k)|^{2p}\Big]\Big)^{\frac12}\\
&\quad+K_T\Big(\mathbb E\Big[\sup_{0\leq k\leq N^{\Delta}}|Y^{\Delta,\epsilon, v}(t_k)|^{2p}\Big]\Big)^{\frac12}\Big(\mathbb E\Big[\sup_{0\leq k\leq N^{\Delta}}|y^{\Delta,\epsilon, v}(t_k)-w^{\Delta, v}(t_k)|^{4p}\Big]\times\nn\\
&\quad\E\Big[\sup_{0\leq k\leq N^{\Delta}}|Y^{\Delta,\epsilon, v}(t_k)|^{4p}\Big]\Big)^{\frac14}+K_T\int_0^T\mathbb E\Big[\sup_{r\leq s}|\frac{1}{\sqrt{\epsilon}}D_{\alpha}y^{\Delta,\epsilon, v}(\lfloor r\rfloor)\nn\\
&\quad-D_{\alpha}\mathbb Z^{\Delta, v}(\lfloor r\rfloor)|^p\Big]\mathrm ds +\mathbb E\Big[\!\sup_{0\leq k\leq N^{\Delta}}|y^{\Delta,\epsilon, v}(t_k)-w^{\Delta, v}(t_k)|^p\Big]\nn\\
&\quad+K_T\int_0^T\mathbb E\Big[\sup_{0\leq k\leq N^{\Delta}}|D_{\alpha}y^{\Delta,\epsilon, v}(t_k)|^p\Big]\mathrm ds.
\end{align}
Moreover, we have 
\begin{align*}
&\Big|\frac{1}{\sqrt{\epsilon}}D_{\alpha}y^{\Delta,\epsilon, v}(t_k)-D_{\alpha}\mathbb Z^{\Delta, v}(t_k)\Big|^2\leq \Big|\frac{1}{\sqrt{\epsilon}}D_{\alpha}Z^{\Delta,\epsilon, v}(t_k)-D_{\alpha}\widetilde{\mathbb Z}^{\Delta, v}(t_k)\Big|^2\\
&\quad+2\theta\Delta \Big|\mathcal Db(y^{\Delta,\epsilon, v}_{t_k})\frac{1}{\sqrt{\epsilon}}D_{\alpha}y^{\Delta,\epsilon, v}_{t_k}-\mathcal Db(w^{\Delta, v}_{t_k})D_{\alpha}\mathbb Z^{\Delta, v}_{t_k}\Big|^2,
\end{align*}
According to the Taylor formula, we have
\begin{align*}
&\quad\sup_{0\leq k\leq N^{\Delta}}\Big|\frac{1}{\sqrt{\epsilon}}D_{\alpha}y^{\Delta,\epsilon, v}(t_k)-D_{\alpha}\mathbb Z^{\Delta, v}(t_k)\Big|^2\nn\\
&\leq \sup_{0\leq k\leq N^{\Delta}}\Big|\frac{1}{\sqrt{\epsilon}}D_{\alpha}Z^{\Delta,\epsilon, v}(t_k)-D_{\alpha}\widetilde{\mathbb Z}^{\Delta, v}(t_k)\Big|^2\\
&\quad +K\Delta \sup_{0\leq k\leq N^{\Delta}}|Y^{\Delta,\epsilon, v}(t_k)|^2|D_{\alpha}y^{\Delta,\epsilon, v}(t_k)|^2\\
&\quad +K\Delta \sup_{0\leq k\leq N^{\Delta}}|Y^{\Delta,\epsilon, v}(t_k)|^2|y^{\Delta,\epsilon, v}(t_k)-w^{\Delta, v}(t_k)|^2|D_{\alpha}y^{\Delta,\epsilon, v}(t_k)|^2\nn\\
&\quad+4\theta\Delta\|^2\mathcal Db\|_{\mathcal L(\mathcal C^{d};\RR^d)}\sup_{0\leq k\leq N^{\Delta}}\Big|\frac{1}{\sqrt{\epsilon}}D_{\alpha}y^{\Delta,\epsilon, v}(t_k)-D_{\alpha}\mathbb Z^{\Delta, v}(t_k)\Big|^2.
\end{align*}
Then, for sufficiently small $\Delta$,
\begin{align*}
&\quad\sup_{0\leq k\leq N^{\Delta}}\Big|\frac{1}{\sqrt{\epsilon}}D_{\alpha}y^{\Delta,\epsilon, v}(t_k)-D_{\alpha}\mathbb Z^{\Delta, v}(t_k)\Big|^2\nn\\
&\leq K \sup_{0\leq k\leq N^{\Delta}}\Big|\frac{1}{\sqrt{\epsilon}}D_{\alpha}Z^{\Delta,\epsilon, v}(t_k)-D_{\alpha}\widetilde{\mathbb Z}^{\Delta, v}(t_k)\Big|^2\\
&\quad +K\Delta \sup_{0\leq k\leq N^{\Delta}}|Y^{\Delta,\epsilon, v}(t_k)|^2|D_{\alpha}y^{\Delta,\epsilon, v}(t_k)|^2\\
&\quad +K\Delta \sup_{0\leq k\leq N^{\Delta}}|Y^{\Delta,\epsilon, v}(t_k)|^2|y^{\Delta,\epsilon, v}(t_k)-w^{\Delta, v}(t_k)|^2|D_{\alpha}y^{\Delta,\epsilon, v}(t_k)|^2.
\end{align*}
Making use of \eqref{inter_DZeq}, the Gr\"onwall inequality and \underline{Steps 1--3} finishes the proof of \underline{Step 4}.

\underline{Step 5.} Prove that for any $p\ge 1$, $l\in \mathbb N_+$, and sufficiently small $\Delta>0$,
$$\lim_{\epsilon\to0}\mathbb E\Big[\sup_{\alpha_1\vee\cdots\vee\alpha_l\leq t_k\leq T}|D_{\alpha_1\cdots\alpha_l}Y^{\Delta,\epsilon, v}(t_k)-D_{K\cdots\alpha_l}\mathbb Z^{\Delta, v}(t_k)|^p\Big]=0.$$

The case of $l=1$ is proved in \underline{Step 4}, and the case of $l>1$ can be proved by the induction argument on $l$. We omit the proof. 

Combining \underline{Steps 1--5}, and according to  \cite[Proposition 5.1]{Ma_functional}, we obtain
\begin{align*}
\liminf_{\epsilon\to0}\epsilon \ln\mathfrak p^{\Delta,\epsilon}(t,y)\ge -\tilde d_{t}(y)
\end{align*}
with  $\tilde d_t(y):=\inf\{I_T(f):f\in \mathcal C_{\xi}([-\tau,T];\mathbb R^d),y=f(t),\det(\bar{\gamma} _{w^{\Delta, v}(t)})>0\}$, where $\bar{\gamma}_{\Phi(v)}=(\<\mathscr D(\Phi(v))^{i},\mathscr  D(\Phi(v))^j\>_H)_{1\leq i,j\leq d}$ with functional $\Phi:L^2([0,T];\mathbb R^m)\to\mathbb R^d$.

At last, we  show that $\tilde I_t=\tilde d_t,$ which is equivalent  to prove that $\det(\bar{\gamma}_{w^{\Delta, v}(t)})>0$ for each $t\in(0,T]$. Noticing   
\begin{align*}
\mathscr  D_{\alpha}w^{\Delta, v}(t_{k+1})&=\mathscr {D}_{\alpha}w^{\Delta, v}(t_{k})+(1-\theta)\mathscr  Db(w^{\Delta, v}_{t_k})\mathscr {D}_{\alpha}w^{\Delta, v}_{t_k}\Delta+\theta \mathscr  Db(w^{\Delta, v}_{t_{k+1}})\mathscr {D}_{\alpha}w^{\Delta, v}_{t_{k+1}}\Delta\\
&\quad+\mathscr  D\sigma (w^{\Delta, v}_{t_k})v(t_k)\Delta+\sigma(w^{\Delta, v}_{t_{k}})\mathbf 1_{[t_k,t_{k+1})}(\alpha)\Delta,
\end{align*}
the proof of $\det(\bar{\gamma}_{w^{\Delta, v}(t)})>0$ is similar to that of Theorem \ref{trun_num_den} and is omitted. 
Hence the proof of Theorem \ref{LDP_density} is finished. 
\end{proof}

\section{Summary and outlook}
This chapter investigates  the LDP for the $\theta$-EM method of the SFDE. 
For the SFDE with small noise, under the superlinearly growing drift coefficient condition, 
we first present  the Freidlin--Wentzell LDP for the $\theta$-EM solution    on the infinite time horizon.   
In fact, the exact solution of the considered equation  can also be proved to satisfy the LDP by the approach used in this chapter. Natural questions arise: What is the relation between  rate functions for the numerical  solution case and  the exact solution case?  How to apply the result on the LDP  of the $\theta$-EM method to solving  some problems in statistics, engineering, statistical mechanics, and applied probability?

Another result obtained in this chapter is the logarithmic estimate for the density function of the $\theta$-EM  solution  provided that drift and diffusion coefficients of the SFDE are  smooth and have arbitrary order bounded derivatives. The relation between the logarithmic limit and the rate function of the LDP is given.  It would be  interesting to further  study the logarithmic estimate for the numerical  density function for the SFDE with  superlinearly growing coefficients. 

    \cleardoublepage
    \renewcommand{\sectionmark}[1]{\markright{#1}}
    \addcontentsline{toc}{chapter}{Bibliography}
    \bibliographystyle{splncs04.bst}
    \bibliography{sfde}
    \cleardoublepage
\end{document}